\documentclass[10pt]{siamltex}
\usepackage{amsmath,amsfonts,amscd,amssymb,bm,cite}
\usepackage{mathrsfs,listings,url}
\usepackage{graphics,color,ulem}
\usepackage{float}
\usepackage[titletoc,page]{appendix}
\usepackage{subfigure}
\usepackage[colorlinks=true, bookmarksopen,
pdfauthor={CST},
pdfcreator={pdftex},
pdfsubject={algorithms},
linkcolor={blue},
anchorcolor={black},
citecolor={firebrick},
filecolor={magenta},
menucolor={black},
pagecolor={red},
plainpages=false,pdfpagelabels,
urlcolor={db} ]{hyperref}
\usepackage[capitalise]{cleveref}
\usepackage{comment}
\usepackage{verbatim}
\usepackage{marginnote}
\usepackage{float}
\usepackage[makeroom]{cancel}
\usepackage[pdftex]{graphicx}
\usepackage[section]{placeins}
\usepackage{mathptmx}
\usepackage{cases}
\usepackage{subeqnarray}
\usepackage{hhline}
\usepackage[linesnumbered,ruled,vlined]{algorithm2e}
\usepackage{xcolor}
\usepackage{tikz}
\usepackage{caption} 
\usepackage{lmodern}

\usepackage{comment}
\includecomment{confidential}
\excludecomment{confidential}

\usetikzlibrary{arrows}
\usetikzlibrary{positioning,shapes,snakes,calc,decorations,decorations.markings}

\newtheorem{remark}{Remark}
\restylefloat{table}
\allowdisplaybreaks
\definecolor{db}{rgb}{0.0470,0,0.5294}
\definecolor{dg}{rgb}{0.0,0.392,0.0}
\definecolor{firebrick}{rgb}{0.698,0.133,0.133}
\definecolor{bl}{rgb}{0.0,0.0,0.0}
\definecolor{linen}{rgb}{0.980,0.941,0.902}
\definecolor{ivory}{rgb}{1.0,1.0,0.941}
\definecolor{aliceblue}{rgb}{0.941,0.973,1.0}
\definecolor{beige}{rgb}{0.961,0.961,0.863}
\definecolor{tan}{rgb}{0.824,0.706,0.549}
\definecolor{lightsteelblue}{rgb}{0.690,0.769,0.871}
\definecolor{paleturquoise}{rgb}{0.686,0.933,0.933}
\definecolor{lightblue}{rgb}{0.678,0.847,0.902}
\definecolor{skyblue}{rgb}{0.529,0.808,0.922}
\definecolor{palegoldenrod}{rgb}{0.933,0.910,0.667}
\definecolor{lightgoldenrod}{rgb}{0.933,0.867,0.510}
\definecolor{lightyellow}{rgb}{1.0,1.0,0.878}
\definecolor{yellow}{rgb}{1.0,1.0,0.0}
\definecolor{lightyellow1}{rgb}{1.0,1.0,0.878}
\definecolor{lemonchiffon}{rgb}{1.0,0.980,0.804}
\definecolor{myyellow}{rgb}{1,1,.9}
\definecolor{darkgreen}{rgb}{0.0,0.392,0.0}
\definecolor{darkviolet}{rgb}{0.580,0.0,0.827}
\definecolor{lightsalmon}{rgb}{1.0,0.627,0.478}
\definecolor{orange}{rgb}{1.0,0.647,0.0}
\definecolor{darkblue}{rgb}{0.00,0.00,0.55}
\numberwithin{equation}{section}
\Crefname{table}{Table}{Tables}
\Crefname{figure}{Figure}{Figures}

\newcommand\titlelowercase[1]{\texorpdfstring{\lowercase{#1}}{#1}}

\begin{document}
	
	\title{\large{T\titlelowercase{he} V\titlelowercase{ariable} T\titlelowercase{ime-stepping} DLN-E\titlelowercase{nsemble} A\titlelowercase{lgorithms for} i\titlelowercase{ncompressible} N\titlelowercase{avier}-S\titlelowercase{tokes} E\titlelowercase{quations}}} 
	\author{
			Wenlong Pei\thanks{
			Department of Mathematics, The Ohio State University, Columbus, OH 43210,
			USA. Email: \href{mailto:pei.176@osu.edu}{pei.176@osu.edu}. } 
		    }
	\date{\emty}
	\maketitle
	
	\begin{abstract}
		In the report, we propose a family of variable time-stepping ensemble algorithms for solving multiple incompressible Navier-Stokes equations (NSE) at one pass.
        The one-leg, two-step methods designed by Dahlquist, Liniger, and Nevanlinna (henceforth the DLN method) are non-linearly stable and second-order accurate under arbitrary time grids.
        We design the family of variable time-stepping DLN-Ensemble algorithms for multiple systems of NSE and prove that its numerical solutions are stable and second-order accurate in velocity under moderate time-step restrictions.
        Meanwhile, the family of algorithms can be equivalently implemented by a simple refactorization process: adding time filters on the backward Euler ensemble algorithm.
        In practice, we raise one time adaptive mechanism (based on the local truncation error criterion) for the family of DLN-Ensemble algorithms to balance accuracy and computational costs.
        Several numerical tests are to support the main conclusions of the report. The constant step test confirms the second-order convergence and time efficiency. 
        The variable step test verifies the stability of the numerical solutions and the time efficiency of the adaptive mechanism.

	\end{abstract}
	
	\begin{keywords}
		Ensemble, variable time-stepping, $G$-stability, second-order, time adaptivity, NSE, Refactorization
	\end{keywords}
	
	\begin{AMS}
		65M12, 35Q30, 76D05
	\end{AMS}

    \section{Introduction}
    Ensemble simulation of flow equations with different input data is highly involved in uncertainty quantification, weather prediction, sensitivity analysis, and many other applications in computational fluid dynamics \cite{TK93_AMS, LP08_JCP, Lew05_MWR, MX06}.
    However, this procedure for complex flow problems, especially nonlinear, time-dependent partial differential equations, would result in huge computational costs and yield challenges in accurate calculation if reliable approximations or useful statistical data are required.
    Recently Jiang and Layton proposed an efficient time-stepping ensemble algorithm (BEFE-Ensemble) for fast computation of $J$ ($J>1$) time-dependent Navier-Stokes equations (NSE) \cite{JL14_IJUQ}.
    In the algorithm, they use the backward Euler (BE) time-stepping scheme for time discretization, combined with any standard finite element (FE) space (meeting the discrete inf-sup condition) for spatial discretization. 
    They decompose the non-linear term of each NSE into two parts: the ensemble mean (independent of the ensemble index) and the fully explicit fluctuation (lagged to the previous time level).
    As a result, the simulations of $J$ separate realizations are reduced to solving $J$ linear systems with the same coefficient matrix and $J$ distinct right-hand vectors at each time step. 
    Meanwhile, numerical solutions are stable and first-order accurate in time under Courent-Friedrichs-Lewy (CFL) like conditions.
    The efficiency of the BEFE-ensemble algorithm can be further improved by theories of iterative methods, such as block CG, block QMR, and block GMRES solvers \cite{FOP95_CMAME, FM97_LAA, GS96_LAA}.
    Later, Jiang replaced the first-order BE scheme with blended multi-step backward differentiation formulas with uniform time grids to obtain second-order ensemble solutions of NSE in time \cite{Jia17_NMPDE, Jia15_JSC}.

    Time adaptivity (adjusting the time step based on certain criteria) is an essential approach to balance the conflict between computational cost and time efficiency in the numerical approximation of differential equations, especially in stiff and unstable problems \cite{SW06_JCAM, HW10_Springer}. 
    As linear multistep methods with unfavorable time steps will lead to instability in the numerical simulation \cite{NL79_BIT}, the variable time-stepping ensemble algorithms with time adaptive mechanisms are little explored.
    Here, we refer to a one-parameter family of one-leg, two-step schemes proposed by Dahlquist, Liniger, and Nevanlinna (henceforth the DLN method) \cite{DLN83_SIAM_JNA} and propose the corresponding family of variable time-stepping DLN-Ensemble algorithms for incompressible NSE. 
    Since the DLN method is $G$-stable (nonlinear stable) and second-order accurate under \textit{arbitrary sequence of time steps} \cite{Dah78_BIT,Dah76_Tech_RIT,Dah78_AP_NYL}, its application to many fluid models have confirmed its potential \cite{QCWLL23_ANM,QHPL21_JCAM,LPQT21_NMPDE,Pei24_NM,CQGWLL24_CMA,SP23_arXiv}.

    We denote $\{ t_{n} \}_{n=0}^{N}$ time grids on time interval $[0,T]$, and 
	$k_{n} = t_{n+1} - t_{n}$ local time steps. 
    Given the following initial value problem 
	\begin{gather}
		y'(t) = g(t,y(t)), \quad y(0) = y_{0}, \quad 0 \leq t \leq T, 
		\label{eq:IVP}
	\end{gather}
	where $y: [0,T] \rightarrow \mathbb{R}^{d}$ and $g: [0,T] \times \mathbb{R}^{d} \rightarrow \mathbb{R}^{d}$ ($d \in \mathbb{N}$). 
    The variable time-stepping DLN method (with the parameter $\theta \in [0,1]$) for the above initial value problem in \eqref{eq:IVP} is 
	\begin{gather}
		\sum_{\ell =0}^{2}{\alpha _{\ell }}y_{n-1+\ell }
		= \widehat{k}_{n} g \Big( \sum_{\ell =0}^{2}{\beta _{\ell }^{(n)}}t_{n-1+\ell } ,
		\sum_{\ell =0}^{2}{\beta _{\ell }^{(n)}}y_{n-1+\ell} \Big), \qquad n=1,\ldots,N-1.
		\label{eq:1leg-DLN}
	\end{gather}
    $y_{n}$ represents the DLN solution of $y(t)$ at time $t_{n}$. The coefficients of the DLN method 
	in \eqref{eq:1leg-DLN} are 
	\begin{gather*}
		\begin{bmatrix}
			\alpha _{2} \vspace{0.2cm} \\
			\alpha _{1} \vspace{0.2cm} \\
			\alpha _{0} 
		\end{bmatrix}
		= 
		\begin{bmatrix}
			\frac{1}{2}(\theta +1) \vspace{0.2cm} \\
			-\theta \vspace{0.2cm} \\
			\frac{1}{2}(\theta -1)
		\end{bmatrix}, \ \ \ 
		\begin{bmatrix}
			\beta _{2}^{(n)}  \vspace{0.2cm} \\
			\beta _{1}^{(n)}  \vspace{0.2cm} \\
			\beta _{0}^{(n)}
		\end{bmatrix}
		= 
		\begin{bmatrix}
			\frac{1}{4}\Big(1+\frac{1-{\theta }^{2}}{(1+{%
					\varepsilon _{n}}{\theta })^{2}}+{\varepsilon _{n}}^{2}\frac{\theta (1-{%
					\theta }^{2})}{(1+{\varepsilon _{n}}{\theta })^{2}}+\theta \Big)\vspace{0.2cm%
			} \\
			\frac{1}{2}\Big(1-\frac{1-{\theta }^{2}}{(1+{\varepsilon _{n}}{%
					\theta })^{2}}\Big)\vspace{0.2cm} \\
			\frac{1}{4}\Big(1+\frac{1-{\theta }^{2}}{(1+{%
					\varepsilon _{n}}{\theta })^{2}}-{\varepsilon _{n}}^{2}\frac{\theta (1-{%
					\theta }^{2})}{(1+{\varepsilon _{n}}{\theta })^{2}}-\theta \Big)%
		\end{bmatrix}.
	\end{gather*}
    $\varepsilon _{n} = (k_n - k_{n-1})/(k_n + k_{n-1}) \in (-1,1)$ is the step variability.
    $\widehat{k}_{n} = {\alpha _{2}}k_{n}-{\alpha _{0}}k_{n-1}$ is the average time step to ensure the second order accuracy of numerical solutions. 
	If $\theta = 1$ or $\theta = 0$, the DLN method is reduced to the midpoint rule or two-step midpoint rule. 
	To our knowledge, the DLN method is the \textit{only} multistep time-stepping method possessing $G$-stability and second-order accuracy under \textit{arbitrary time step sequence}.

    Herein we propose the variable time-stepping DLN-ensemble algorithm for incompressible NSE and provide detailed numerical analysis. 
    We consider $J$ ($J > 1$) systems of incompressible NSE on the domain $\Omega \subset \mathbb{R}^{d}$ ($d = 2,3$) over time interval $[0,T]$.
    For the $j$-th ($j = 1, \cdots, J$) system of NSE, the velocity $u^{j}(x,t)$, pressure $p^{j}(x,t)$ and body force $f^{j}(x,t)$ satisfy
    \begin{gather}
		\begin{cases}
			u_{t}^{j} + (u^{j} \cdot \nabla) u^{j} - \nu \Delta u^{j} + \nabla p^{j} = f^{j} \\
			\qquad \nabla \cdot u^{j} = 0, \quad \displaystyle \int_{\Omega} p^{j} dx = 0 
		\end{cases}, \qquad x \in \Omega, \quad 0 < t < T,
		\label{eq:jth-NSE}
	\end{gather}
    with the following homogeneous Dirichlet boundary condition $u^{j}|_{\partial{\Omega}} = 0$ 
	and initial condition $u^{j}(x,0) = u_{0}^{j}(x)$ for $x \in \Omega$. 
    Let $\tilde{u}_{n}^{j}$, $\tilde{p}_{n}^{j}$ be the semi-discrete approximations of velocity and pressure in the $j$-th system of NSE in \eqref{eq:jth-NSE} at time $t_{n}$ ($n = 0,1 \cdots, N$) respectively.
    Given any sequence $\{ z_{n} \}_{n=0}^{N}$, we adopt the following notations for convenience
	\begin{gather*}
		z_{n,\alpha} = \sum_{\ell =0}^{2} \alpha _{\ell } z_{n-1+\ell}, \quad 
		z_{n,\beta} = \sum_{\ell =0}^{2} \beta _{\ell }^{(n)} z_{n-1+\ell}, \\
		z_{n,\ast} = \Big( \beta _{2}^{(n)} \big( 1 + \frac{k_{n}}{k_{n-1}} \big) + \beta _{1}^{(n)} \Big) z_{n}
		+ \Big( \beta _{0}^{(n)} - \beta _{2}^{(n)} \frac{k_{n}}{k_{n-1}} \Big) z_{n-1}.
	\end{gather*}
    The essential part of the DLN-Ensemble algorithm at time level $t_{n+1}$ is to decompose the non-linear term into the sum of mean and fully explicit fluctuation, i.e.
    \begin{gather*}
		\big((u^{j} \cdot \nabla) u^{j} \big) (t_{n,\beta}) \approx
		\big(  \langle \tilde{u} \rangle_{n,\ast} \cdot \nabla \big) \tilde{u}_{n,\beta}^{j}
		+ \big( (\tilde{u}_{n,\ast}^{j} -  \langle \tilde{u} \rangle_{n,\ast}) \cdot \nabla \big) \tilde{u}_{n,\ast}^{j}.
	\end{gather*}
    $\tilde{u}_{n,\beta}^{j}$ is the implicit, second-order approximation of $u^{j}(t_{n,\beta})$ in time.
	$\tilde{u}_{n,\ast}^{j}$ and $\langle  \tilde{u} \rangle_{n,\ast}$ are explicit, second-order approximations of 
	$u^{j}(t_{n,\beta})$ and average of $u^{j}(t_{n,\beta})$ in time respectively.

    The report is organized as follows. 
	Necessary preliminaries and notations are given in Section \ref{sec:Prelim}. In Section \ref{sec:Alg}, we propose the variable time-stepping DLN-Ensemble algorithm for $J$ systems of NSE and discuss how the algorithm is implemented equivalently by adding a few lines of codes to the BEFE-Ensemble algorithm. 
	Detailed numerical analysis for the algorithm is provided in Section \ref{sec:Num-Analysis}.
	In Subsection \ref{subsec:Stab-Error-L2}, we prove the stability and error convergence of velocity in $L^2$-norm.
	Stability and error convergence of velocity in $H^{1}$-norm are derived in 
	Subsection \ref{subsec:Stab-Error-H1}. Numerical analysis for pressure is given in 
	Subsection \ref{subsec:Stab-Error-Pressure}.
	With the assumption of certain CFL-like conditions, 
	all conclusions in Section \ref{sec:Num-Analysis} hold under \textit{non-uniform time grids}.
	In Section \ref{sec:Adapt-Imple}, we design one time adaptive mechanism for the DLN-Ensemble based on the local truncation error (LTE) criterion.
	We offer several numerical tests in Section \ref{sec:Num-Tests} to confirm the stability, accuracy, and efficiency of the DLN-Ensemble algorithm.

    \subsection{Recent Works on Ensemble Algorithms} 
    In the last ten years, various time-stepping ensemble algorithms have been applied to numerous fluid models.
	Takhirov, Neda, and Waters extend the BEFE-Ensemble algorithm to time relaxation NSE models \cite{SAK01_PF} and show the stability of the algorithm provided fluctuations are small enough \cite{TNW16_NMPDE}.
	Jiang and Layton analyze an efficient ensemble regulation algorithm for under-resolved and convection-dominated flow and reconsider an old but not as well-developed definition of the mixing length, which increases stability and improves flow prediction \cite{JL15_NMPDE}.
	Gunzburger, Jiang, and Schneier utilize the proper orthogonal decomposition (POD) method and establish the ensemble-POD approach for multiple non-stationary NSEs \cite{GJS17_SIAM_JNA,GJS18_IJNAM}. As a result, the computational cost is further reduced by solving a reduced linear system at each time step. 
	Gunzburger, Jiang, and Wang allow uncertainties in viscosity coefficients of the parameterized flow problems and decompose the viscosity terms into the mean and fluctuations besides the non-linear terms in ensemble calculations \cite{GJW19_CMAM,GJW19_IMA_JNA}.
	Mohebujjaman and Rebholz apply the BEFE-Ensemble algorithm to incompressible magnetohydrodynamic (MHD) flows in Els\"asser variable formulation and increase the efficiency by solving a stable decoupling of each MHD system at each time step \cite{MR17_CMAM}.
	Mohebujjaman designs a stable, second-order time-stepping algorithm for computing MHD ensemble flow by writing the viscosity and magnetic diffusivity pair as one-leg $\theta$-scheme \cite{Moh22_AAMM}.
	Connors develops a statistical turbulence model with ensemble calculation for two fluids motivated by atmosphere-ocean interaction \cite{Con18_IJNAM}.
	Fiordilino and Khankan calculate an ensemble of solutions to laminar natural convection by solving two coupled linear systems, each involving a shared matrix \cite{Fio18_SIAM_JNA,FK18_IJNAM}. 
    In return, storage requirements and computational costs are largely reduced.
	Jiang, Li, and Yang explore an unconditionally stable ensemble algorithm based on the Crank-Nicolson leap-frog (CNLFAC) method for the Stokes-Darcy model and reduce the size of the linear system by decoupling Stokes-Darcy system into two smaller subphysics problems \cite{JLY21_SIAM_JNA}.
	Recently, the scalar auxiliary variable (SAV) approach for gradient flows \cite{SXY18_JCP,SX18_SIAMJNA}, constructing the scalar auxiliary function to avoid solving non-linear systems, is adopted in ensemble simulation to achieve stability and convergence of numerical solutions without CFL like conditions \cite{CHJ23_JSC,JY21_SIAM_JSC,JY24_JMAA}.

    \section{Notations and Preliminaries} 
	\label{sec:Prelim}
    Let $\Omega \subset \mathbb{R}^{d}$ ($d = 2,3$) be the domain of $J$ systems of NSE. 
	$H^{\ell}(\Omega)$ ($\ell \in \mathbb{N} \cup \{ 0 \}$) is Sobolev space $W^{\ell,2}(\Omega)$ with 
	norm $\| \cdot \|_{\ell}$ and semi-norm $| \cdot |_{\ell}$.
	$H^{0}(\Omega)$ is exact $L^{2}(\Omega)$ inner product space with norm $\| \cdot \|$ and inner product $(\cdot, \cdot)$. The velocity space $X$ and pressure space $Q$ for NSE are
    \begin{gather*}
		X := \Big\{ v \in \big( H^{1}(\Omega) \big)^{d}: v|_{\partial \Omega} = 0 \Big\}, \qquad
		Q := \Big\{ q \in L^{2}(\Omega): (q,1) = 0 \Big\}.
	\end{gather*}
	The diverence-free space for velocity is 
    $V = \big\{ v \in X: ( q,\nabla \cdot v) = 0, \ \forall q \in Q \big\}$.
    $X^{-1}$ is the dual space of $X$ with dual norm 
	\begin{gather}
		\| g \|_{-1} := \sup_{0 \neq v \in X} \frac{(g, v )}{\| \nabla v \|}, \qquad \forall g \in X^{-1}. 
		\label{eq:dual-norm}
	\end{gather}
	We need the following Bochner spaces on the time interval $[0,T]$: for $p \geq 1$,
	\begin{align*}
		L^{p}\big( 0,T; \big(H^{\ell}(\Omega) \big)^{d} \big) 
		\!&=\! \Big\{  v(\cdot,t) \in \big(H^{\ell}(\Omega) \big)^{d}:  0 < t < T, \   \| v \|_{p,\ell} \!=\! \Big( \int_{0}^{T} \| v(\cdot,t) \|_{\ell}^{p} \Big)^{\frac{1}{p}} \!<\! \infty  \Big\}, \\
		L^{\infty}\big( 0,T; \big(H^{\ell}(\Omega) \big)^{d} \big) 
		\!&=\! \Big\{  v(\cdot,t) \in \big(H^{\ell}(\Omega) \big)^{d}: 0 < t < T, \    \| v \|_{\infty,\ell} \!=\!  \sup_{0<t<T} \| v(\cdot,t) \|_{\ell} \!<\! \infty  \Big\}, \\
		L^{p}\big( 0,T; X^{-1} \big) 
		\!&=\! \Big\{  v(\cdot,t) \in X^{-1}: 0 < t < T, \ \| v \|_{p,-1} \!=\! \Big( \int_{0}^{T} \| v(\cdot,t) \|_{-1}^{p} \Big)^{\frac{1}{p}} \!<\! \infty  \Big\}.
	\end{align*}
    We define the skew-symmetric, non-linear operator
	\begin{gather*}
		b(u,v,w) := \frac{1}{2} \big( (u \cdot \nabla) v , w \big) - \frac{1}{2} \big( (u \cdot \nabla) w , v \big), \qquad \forall u,v,w \in \big(H^{1}(\Omega) \big)^{d}.
	\end{gather*}
	We use divergence theorem to obtain 
	\begin{gather*}
		b(u,v,w) = \big( (u \cdot \nabla) v , w \big) + \frac{1}{2} \big( (\nabla \cdot u) v, w \big), 
		\qquad \forall u,v,w \in X. 
	\end{gather*}
	Therefore $b(u,v,w) = \big( (u \cdot \nabla) v , w \big)$ for any $u \in V$ and $v,w \in X$.
    We need the following lemma about the operator $b$ for numerical analysis.
	\begin{lemma}
		\label{lemma:b-bound}
		For any $u,v,w \in H^{1}(\Omega)$,
		\begin{align}
			\big| \big( (u \cdot \nabla )v, w \big) \big|
			\leq& C(\Omega)
			\begin{cases}
				\| u \|_{1} |v|_{1} \| w \|_{1} \\
				\| u \| \| v \|_{2} \| w \|_{1} & v \in H^{2}(\Omega) \\
				\| u \|_{2} |v|_{1} \| w \| & u \in H^{2}(\Omega)
			\end{cases},
			\label{eq:b-bound-1} \\
			\big| b(u,v,w) \big| \leq& C(\Omega) \big( \| u \| \| u \|_{1} \big)^{1/2} \| v \|_{1} \| w \|_{1}. 
			\label{eq:b-bound-2}
		\end{align}
		If $u,v,w \in X$, then 
		\begin{gather}
			\big| b(u,v,w) \big| \leq C(\Omega)
			\begin{cases}
				\| u \|_{1} \| v \|_{1} \big( \| w \| \| w \|_{1} \big)^{1/2} \\
				\| u \|_{1} \| v \|_{2} \| w \| & v \in H^{2}(\Omega)
			\end{cases}.
			\label{eq:b-bound-3}
		\end{gather}
	\end{lemma}
	\begin{proof}
		See \cite[p.273-275]{Ing13_IJNAM}.
	\end{proof}

    Let $\{\mathcal{T}_{h}\}$ be the family of edge-to-edge triangulations of domain $\Omega$ 
    with diameter $h \in (0,1]$. 
	$X^{h} \subset X$ and $Q^{h} \subset Q$ are certain finite element spaces for velocity and pressure  respectively.
    The divergence-free space for $X^{h}$ and $Q^{h}$ is 
	\begin{gather*}
		V^{h} := \Big\{ v^{h} \in X^{h}: \big( q^{h}, \nabla \cdot v^{h} \big) = 0, \ \ 
		\forall q^{h} \in Q^{h} \Big\}.
	\end{gather*}
    We assume that $X^{h}$ is $C^{m}$-space containing polynomials of highest degree $r$ ($r \in \mathbb{N}$), 
    $Q^{h}$ is $C^{m}$-space containing polynomials of highest degree $s$ ($s \in \mathbb{N}$) and 
    have the following approximations for velocity $u \in (H^{r+1})^{d} \cap X$ 
    and pressure $p \in H^{s+1} \cap Q$ 
    \begin{equation}
        \label{eq:approx-thm}
	    \begin{aligned}
		\inf_{v^{h} \in X^{h}} \| u - v^{h} \|_{\ell_{1}} &\leq C h^{r+1-\ell_{1}} |w|_{r+1}, \ \ 
		0 \leq  \ell_{1} \leq \min\{ m+1, r+1 \},  \\
		\inf_{q^{h} \in Q^{h}} \| p - q^{h} \|_{\ell_{2}} &\leq C h^{s+1-\ell_{2}} |q|_{s+1}, \ \ \ 
		0 \leq  \ell_{2} \leq \min\{ m+1, s+1 \}.
	    \end{aligned}
    \end{equation}
    The inverse inequality for $X^{h}$ is 
	\begin{gather}
		|v^{h}|_{1} \leq C(\Omega) h^{-1} \| v^{h} \|, \qquad \forall v^{h} \in X^{h}.
		\label{eq:inv-inequal}
	\end{gather}
    We refer to \cite{BS07,Cia02_SIAM} for proofs of \eqref{eq:approx-thm} and  \eqref{eq:inv-inequal}.
    To ensure the uniqueness of the numerical solutions, we assume the pair 
    $(X^{h}, Q^{h})$ satisfies the discrete inf-sup condition (Ladyzhenskaya-Babu\v{s}ka-Brezzi condition)
    \begin{gather}
		\inf_{q^{h} \in Q^{h}} \sup_{v^{h} \in X^{h}} \frac{ \big( \nabla \cdot v^{h}, q^{h} \big)}{ \| \nabla v^{h} \| \| q^{h} \| } \geq C_{\tt{is}} > 0,
		\label{eq:inf-sup-cond}
	\end{gather}
    where $C_{\tt{is}}$ is independent of $h$.
    Taylor-Hood ($\mathbb{P}$2-$\mathbb{P}$1) space and Mini element space are typical examples meeting this criterion.
    For any pair $(u,p) \in V \times Q$, the Stokes projection $\big( I_{\rm{St}}^{h} u, I_{\rm{St}}^{h} p \big) \in V^{h} \times Q^{h}$ is defined to be the unique solution to the following Stokes problem 
	\begin{gather}
		\begin{cases}
			\nu \big( \nabla u - \nabla I_{\rm{St}}^{h} u, \nabla v^{h} \big) = \big( p - I_{\rm{St}}^{h} p , \nabla \cdot v^{h} \big) \\
			\qquad \quad \  - \big( q^{h}, \nabla \cdot I_{\rm{St}}^{h} u \big) = 0
		\end{cases}, 
		\quad \forall (v^{h}, q^{h}) \in X^{h} \times Q^{h}.
		\label{eq:Stokes-def}
	\end{gather}
    If the pair $(X^{h}, Q^{h})$ satisfies the discrete inf-sup condition in \eqref{eq:inf-sup-cond}, we have the following approximation property of the Stokes projection (See \cite{GR86_Springer,Joh16_Springer} for proof)
    \begin{equation}
        \label{eq:Stoke-Approx}
        \begin{aligned}
            | u - I_{\rm{St}}^{h} u |_{1} \leq& 
            2 \Big( 1 + \frac{1}{C_{\rm{is}}^{h}} \Big) \inf_{v^{h} \in X^{h}} | u - v^{h} |_{1}
            + \nu^{-1} \inf_{q^{h} \in Q^{h}} \| p - q^{h} \|.
        \end{aligned}
    \end{equation}

	\section{Algorithm}
	\label{sec:Alg} 
	Let $u_{n}^{j,h} \in X^{h}$ and $p_{n}^{j,h} \in Q^{h}$ be fully discrete solutions of $u(x,t_{n})$ and $p(x,t_{n})$ in the $j$-th NSE in \eqref{eq:jth-NSE} respectively.
	$\langle u^{h} \rangle_{n}$ represents average value of $\{ u_{n}^{j,h} \}_{j=1}^{J}$.
	The family of variable time-stepping DLN-ensemble algorithms (with the parameter $\theta \in [0,1]$) for the $j$-th system of NSE is:   
	given $u_{n-1}^{j,h}, u_{n}^{j,h}$ and $p_{n-1}^{j,h}, p_{n}^{j,h}$, we find $u_{n+1}^{j,h} \in X^{h}$ 
	and $p_{n+1}^{j,h} \in Q^{h}$ satisfying: 
	\begin{gather}
		\begin{cases}
			\Big( \frac{u_{n,\alpha}^{j,h}}{\widehat{k}_{n}}, v^{h}  \Big) 
			\!+\! b \big( \langle u^{h} \rangle_{n,\ast} , u_{n,\beta}^{j,h} , v^{h} \big) 
			\!-\! \big( p_{n,\beta}^{j,h}, \nabla \cdot v^{h} \big) \!+\! {\nu} \big( \nabla u_{n,\beta}^{j,h} , \nabla v^{h}  \big)  \\
			\qquad \qquad \qquad 
			= \big( f^{j}(t_{n,\beta}), v^{h} \big) - b \big( u_{n,\ast}^{j,h} - \langle u^{h} \rangle_{n,\ast}, u_{n,\ast}^{j,h}, v^{h} \big),  \\
			\ \\
			\big( \nabla \cdot u_{n,\beta}^{j,h} , q^{h} \big) = 0,
		\end{cases} 
		\label{eq:DLN-Ensemble-Alg}
	\end{gather}
	for all $v^{h} \in X^{h}$, $q^{h} \in Q^{h}$, $j = 1,2, \cdots, J$ and $n = 1,2, \cdots, N-1$.
	Here 
	\begin{gather*}
		\langle u^{h} \rangle_{n,\ast} 
		\!=\! \Big( \beta _{2}^{(n)} \big( 1 + \frac{k_{n}}{k_{n-1}} \big) + \beta _{1}^{(n)} \Big) \langle u^{h} \rangle_{n}
		\!+\! \Big( \beta _{0}^{(n)} - \beta _{2}^{(n)} \frac{k_{n}}{k_{n-1}} \Big) \langle u^{h} \rangle_{n-1}
		\!=\! \frac{1}{J} \Big( \sum_{j=1}^{J} u_{n,\ast}^{j,h} \Big).
	\end{gather*}
	\begin{remark}
		\begin{itemize}
			\item At time step $t_{n+1}$, the algorithm in \eqref{eq:DLN-Ensemble-Alg} is to solve the following $J$ linear systems with the same left matrix, i.e.
			\begin{gather*}
				A [ \mathbf{x}_{1}, \mathbf{x}_{2}, \cdots, \mathbf{x}_{J} ] 
				= [ \mathbf{b}_{1}, \mathbf{b}_{2}, \cdots, \mathbf{b}_{J} ].
			\end{gather*}
			where $A$ is the left coefficient matrix in \eqref{eq:DLN-Ensemble-Alg} for all $j$, 
			$\mathbf{x}_{j}$ and $\mathbf{b}_{j}$ are the numerical solutions and right vector of the algorithm for $j$-th system of NSE respectively. 
			\item For $J=1$,  the algorithm in \eqref{eq:DLN-Ensemble-Alg} is exactly the semi-implicit DLN algorithm for NSE (see \cite{Pei24_NM} for more details).
		\end{itemize}
	\end{remark}
	The DLN-Ensemble algorithm in \eqref{eq:DLN-Ensemble-Alg} can be equivalently implemented by the refactorization process on the (BEFE-Ensemble)-like algorithm: \\
	\textit{Step 1}. Pre-process:
	\begin{gather*}
		\begin{cases}
			u_{n,\tt{old}}^{j,h} = \Big(\beta_{1}^{(n)} \!-\! \frac{\alpha_{1} \beta_{2}^{(n)}}{\alpha_{2}}  \Big) u_{n}^{j,h} 
			+ \Big(\beta_{0}^{(n)} \!-\! \frac{\alpha_{0} \beta_{2}^{(n)}}{\alpha_{2}}  \Big) u_{n\!-\!1}^{j,h}, \\
			p_{n,\tt{old}}^{j,h} = \Big(\beta_{1}^{(n)} \!-\! \frac{\alpha_{1} \beta_{2}^{(n)}}{\alpha_{2}}  \Big) p_{n}^{j,h} 
			+ \Big(\beta_{0}^{(n)} \!-\! \frac{\alpha_{0} \beta_{2}^{(n)}}{\alpha_{2}}  \Big) p_{n\!-\!1}^{j,h}, \\
			k_{n}^{\tt{BE}} = \frac{\beta_{2}^{(n)}}{\alpha_{2}} \widehat{k}_{n}.
		\end{cases}
	\end{gather*}
	\textit{Step 2}. (BEFE-Ensemble)-like solver on time interval 
	$[t_{n,\beta} - k_{n}^{\tt{BE}},t_{n,\beta}]$: solving $u_{n+1}^{j,h,\tt{BE}}$ and $p_{n+1}^{j,h,\tt{BE}}$ such that for all
	$(v^{h}, q^{h})$ in $X^{h} \times Q^{h}$
	\begin{gather*}
	\begin{cases}
		\Big( \frac{u_{n+1}^{j,h,\tt{BE}} - u_{n,\tt{BE}}^{j,h}}{k_{n}^{\tt{BE}}}, v^{h}  \Big) 
		\!+\! b \big( \langle u^{h} \rangle_{n,\ast}, u_{n+1}^{j,h,\tt{BE}} , v^{h} \big) 
		\!-\! \big( p_{n+1}^{j,h,\tt{BE}}, \nabla \cdot v^{h} \big) 
		\!+\! {\nu} \big( \nabla u_{n+1}^{j,h,\tt{BE}}, \nabla v^{h}  \big)  \\
		\qquad \qquad \qquad \qquad \qquad 
		= \big( f^{j}(t_{n,\beta}), v^{h} \big) 
		- b \big( u_{n,\ast}^{j,h} - \langle u^{h} \rangle_{n,\ast} , u_{n,\ast}^{j,h}, v^{h} \big),  \\
		\ \\
		\qquad \qquad \qquad \qquad  \qquad \big( \nabla \cdot u_{n+1}^{j,h,\tt{BE}} , q^{h} \big) = 0.
	\end{cases} 
	\end{gather*}
	\textit{Step 3}. Post-process to obtain DLN-Ensemble solutions:
	\begin{gather*}
		\begin{cases}
			u_{n+1}^{j,h} = \frac{1}{\beta_{2}^{(n)}} u_{n+1}^{j,h,\tt{BE}} - \frac{\beta_{1}^{(n)}}{\beta_{2}^{(n)}} u_{n}^{j,h}
				- \frac{\beta_{0}^{(n)}}{\beta_{2}^{(n)}} u_{n-1}^{j,h}, \\
			p_{n+1}^{j,h} = \frac{1}{\beta_{2}^{(n)}} p_{n+1}^{j,h,\tt{BE}} - \frac{\beta_{1}^{(n)}}{\beta_{2}^{(n)}} p_{n}^{j,h}
				- \frac{\beta_{0}^{(n)}}{\beta_{2}^{(n)}} p_{n-1}^{j,h}.
		\end{cases}
	\end{gather*}
	\begin{remark} 
		\textit{Step 2} just revises the non-linear terms of the BEFE-Ensemble algorithm a little: 
		\begin{align*}
			\text{BEFE-Ensemble :  } 
			&b \big( \langle u_{\tt{old}}^{h} \rangle_{n}, u_{n+1}^{j,h,\tt{BE}} , v^{h} \big)
			+ b \big( u_{n,\tt{old}}^{j,h} - \langle u_{\tt{old}}^{h} \rangle_{n}, u_{n,\tt{old}}^{j,h}, v^{h} \big), \\
			\text{(BEFE-Ensemble)-like :  }
			&b \big( \langle u^{h} \rangle_{n,\ast}, u_{n+1}^{j,h,\tt{BE}} , v^{h} \big)
			+ b \big( u_{n,\ast}^{j,h} - \langle u^{h} \rangle_{n,\ast} , u_{n,\ast}^{j,h}, v^{h} \big),
		\end{align*}
		where $\langle u_{\tt{old}}^{h} \rangle_{n}$ is average value of $\{ u_{n,\tt{old}}^{j,h} \}_{j=1}^{J}$.
	\end{remark}
	We refer to \cite{LPT21_AML} for the equivalence of the DLN-Ensemble algorithms and refactorization process on the (BEFE-Ensemble)-like algorithm. We summarize the refactorization process in \cref{tikz-BEFE-Ensemble-refact}.

	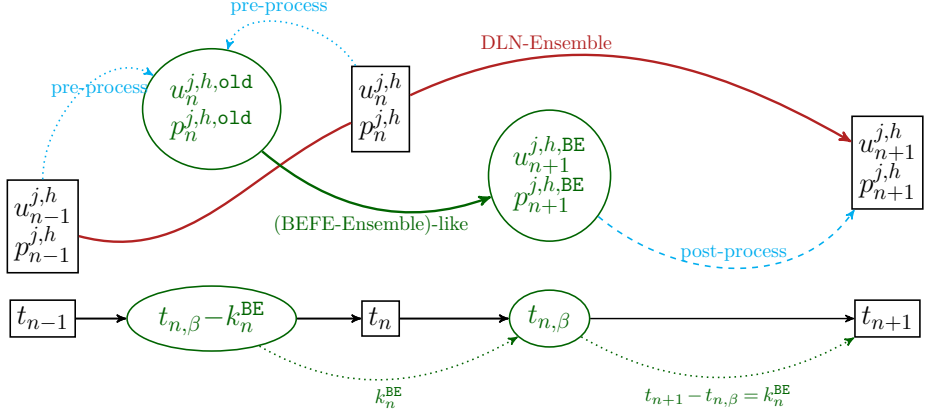
\begin{figure}[]
		\resizebox{.95\textwidth}{!}{
			\begin{tikzpicture}
			[
			->,
			>=stealth',
			auto,node distance=3cm,
			thick,
			main node/.style={ draw, black,font=\sffamily\Large\bfseries}
			]
			
			\node[main node] (1)   {$t_{n-1}$};
			\node[main node] (2) [darkgreen,ellipse] [right of=1] {$t_{n,\beta} \!-\! k_n^{\tt{BE}}$};
			\node[main node] (3) [right of=2] {$t_{n}$} ;
			\node[main node] (4) [darkgreen,ellipse,right of=3] {$t_{n,\beta}$};
			\node[main node] (5) [white,right of=4] {};
			\node[main node] (6) [right of=5] {$t_{n+1}$};
			\node[main node] (11) [above=0.8] {$\begin{matrix} u_{n-1}^{j,h} \\ p_{n-1}^{j,h} \end{matrix}$};
			\node[main node] (12) [right of=11,darkgreen,ellipse,above=1.0] {$\begin{matrix} u_{n}^{j,h,\tt{old}} \\ p_{n}^{j,h,\tt{old}} \end{matrix}$};
			\node[main node] (13) [right of=12] {$\begin{matrix} u_{n}^{j,h} \\ p_{n}^{j,h} \end{matrix}$};
			\node[main node] (14) [darkgreen,ellipse,right of =13, below=0.005] {$\begin{matrix} u_{n+1}^{j,h,\tt{BE}} \\ p_{n+1}^{j,h,\tt{BE}} \end{matrix}$} ;
			\node[main node] (15) [white,right of=14, above=0.10] {};
			\node[main node] (16) [right of=15] {$\begin{matrix} u_{n+1}^{j,h} \\ p_{n+1}^{j,h} \end{matrix}$};

			\path[every node/.style={font=\sffamily\small}]
			(1) edge node [right] {} (2)
			(2) edge node [right] {} (3)
			(3) edge node [right] {} (4)
			(4) edge[line width=.1mm] node [right] {} (6);
			
			\draw[darkgreen, dotted, thick]
			(2) [out=-30, in=-150] to  node[midway,below] {$k_n^{\tt{BE}}$} (4);    			
			\draw[darkgreen, dotted, thick]
			(4) [out=-30, in=-150] to  node[midway,below] {$t_{n+1} - t_{n,\beta} = k_n^{\tt{BE}}$} (6);
			
			\draw[firebrick,very thick] (11) to[out=-15,in=205] 
			(13) to[out=25,in=150] node[pos=0.3,above] {DLN-Ensemble} (16);     			    			
			\draw[cyan, dotted, thick]
			(11) [out=90, in=150] to  node[pos=0.65,below] {pre-process} (12);
			\draw[cyan, dotted, thick]
			(13) [out=120, in=75] to node[pos=0.55,above] {pre-process} (12);
			\draw[darkgreen,  very thick]    			
			(12) [out=-40, in=-160] to node[midway,below] {(BEFE-Ensemble)-like} (14);
			\draw[cyan, dashed, thick]
			(14) [out=-40, in=-125] to  node[midway,above] {post-process} (16);

			\path    			(1) edge node [right] {} (2)
			(2) edge node [right] {} (3)
			(3) edge node [right] {} (4)
			(4) edge[line width=.25mm] node [right] {} (6);
			\end{tikzpicture}
		}  
		\caption{\textbf{Refactorization Process on (BEFE-Ensemble)-like algorithm}}
		\label{tikz-BEFE-Ensemble-refact}
	\end{figure}

	\section{Numerical Analysis}
	\label{sec:Num-Analysis}
	Throughout this section, we assume that the finite element spaces $X^{h}$, $Q^{h}$ satisfy the approximations in \eqref{eq:approx-thm}, inverse inequality \eqref{eq:inv-inequal} and the discrete inf-sup condition in \eqref{eq:inf-sup-cond}.
	$u_{n}^{j}$, $p_{n}^{j}$ represent exact velocity and pressure of $j$-th NSE in \eqref{eq:jth-NSE} at time $t_{n}$.
	We need following discrete Bochner spaces with the time grids $\{t_{n} \}_{n=0}^{N}$ on the 
	time interval $[0,T]$
	\begin{align*}
		\ell^{\infty} \big( \{ t_{n}\}_{n=0}^{N};\big(H^{\ell}(\Omega) \big)^{d}  \big)
		&= \big\{ f(\cdot,t) \in \big(H^{\ell}(\Omega) \big)^{d}:  0 < t < T, \  \| |f| \|_{\infty,\ell} < \infty \big\}, \\
		\ell^{\infty,\beta} \big( \{ t_{n}\}_{n=0}^{N};\big(H^{\ell}(\Omega) \big)^{d} \big)
		&= \big\{ f(\cdot,t) \in \big(H^{\ell}(\Omega) \big)^{d}: 0 < t < T, \  \| |f| \|_{\infty,\ell,\beta} < \infty \big\}, \\
		\ell^{p,\beta} \big( \{ t_{n}\}_{n=0}^{N};\big(H^{\ell}(\Omega) \big)^{d} \big) 
		&= \big\{ f(\cdot,t) \in \big(H^{\ell}(\Omega) \big)^{d} : 0 < t < T, \ \| |f| \|_{p,\ell,\beta} < \infty \big\}, \\
		\ell^{p,\beta} \big( \{ t_{n}\}_{n=0}^{N}; X^{-1} \big) &= \big\{ f(\cdot,t) \in X^{-1}: 0 < t < T, \ \| |f| \|_{p,-1,\beta} < \infty \big\},
		\end{align*}
	where the corresponding discrete norms are 
	\begin{equation}
		\label{eq:def-norm-discrete}
		\begin{aligned}
			&\| |f| \|_{\infty,\ell} := \max_{0 \leq n \leq N} \| f(\cdot, t_{n}) \|_{\ell}, \ \ \ 
			\| |f| \|_{\infty,\ell,\beta} := \max_{1 \leq n \leq N-1} \| f(\cdot, t_{n,\beta}) \|_{\ell},  \\
			&\| |f| \|_{p,\ell,\beta} 
			:= \Big( \sum_{n=1}^{N-1} (k_{n} + k_{n-1})\| f(\cdot,t_{n,\beta}) \|_{\ell}^{p} \Big)^{1/p},  \\
			&\| |f| \|_{p,-1,\beta} 
			:= \Big( \sum_{n=1}^{N-1} (k_{n} + k_{n-1})\| f(\cdot,t_{n,\beta}) \|_{-1}^{p} \Big)^{1/p}.
		\end{aligned}
	\end{equation}
	The discrete norm $\| | \cdot | \|_{p,\ell,\beta}$ in \eqref{eq:def-norm-discrete} is the form of Riemann sum in which the function $f$ is evaluated at $t_{n,\beta} \in [t_{n-1},t_{n+1}]$.
	Since $\varepsilon_{n} \in (-1,1)$, it's easy to check 
	that coefficients $\{ \beta_{\ell}^{(n)} \}_{\ell=0}^{2}$ have the following bounds for $\theta \in [0,1)$
	\begin{gather}
	\frac{ 2 \!+\! \theta \!+\! \theta^{2} }{4(1 \!+\! \theta)} \!< \!\beta_{2}^{(n)}\! <\! \frac{1 \!+\! \theta}{2 (1 \!-\! \theta)}, \ \ 
	\frac{ - \theta }{ 1 \!-\! \theta} \!< \!\beta_{1}^{(n)}\!  <\! \frac{ \theta}{1 \!+\! \theta}, \ \ 
	\frac{ 1  \!-\! 2 \theta \!-\! \theta^2}{ 2 (1 \!-\! \theta) (1 \!+\! \theta) } \!< \! \beta_{0}^{(n)} \!
	<\! \frac{  2 \!-\! \theta \!+\! \theta^2 }{4 (1 \!-\! \theta)}.
	\label{eq:bound-beta}
	\end{gather}
	\begin{definition}
		\label{G-matrix}
		For $0\leq \theta\leq 1$, define the semi-positive symmetric definite matrix 
		$G(\theta) \in \mathbb{R}^{2d \times 2d}$ by
		\begin{equation*}
		G(\theta) = 
		\begin{bmatrix}
		\frac{1}{4}(1+\theta)\mathbb{I}_d & 0 \\
		0 & \frac{1}{4}(1-\theta)\mathbb{I}_d 
		\end{bmatrix},
		\end{equation*}
	\end{definition}
	where $\mathbb{I}_d \in \mathbb{R}^{d \times d}$ is the identity matrix. It's easy to see that $G(\theta)$
	is a positive definite matrix for $\theta \in [0,1)$.
	The following two Lemmas play essential roles in numerical analysis.
	\begin{lemma}  
		\label{lemma:G-stab}
		For any sequence $\{ y_{n} \}_{n=0}^{N}$ in $\big( L^{2}(\Omega) \big)^{d}$ and any $\theta \in [0,1]$, we have the following identity: for any $n = 1,2,\cdots N-1$
		\begin{gather} 
		\Big(\sum_{\ell\!=\!0}^{2} {\alpha_{\ell}} y_{n\!-\!1\!+\!\ell}, \sum_{\ell\!=\!0}^{2} {\beta_{\ell}^{(n)}} y_{n\!-\!1\!+\!\ell} \Big) \!=\! \begin{Vmatrix} {y_{n\!+\!1}} \\ {y_{n}} \end{Vmatrix}^{2}_{G(\theta)}\!-\!\begin{Vmatrix} {y_{n}} \\ {y_{n\!-\!1}} \end{Vmatrix}^{2}_{G(\theta)}
		\!+\! \Big\| \sum_{\ell\!=\!0}^{2} \gamma_{\ell}^{(n)} y_{n\!-\!1\!+\!\ell} \Big\|^{2}, 
		\label{eq:Gstab-Id}
		\end{gather}
		where the $\| \cdot \|_{G(\theta)}$-norm with respect to $L^{2}(\Omega)$ space is 
		\begin{gather}  
		\begin{Vmatrix} u \\ v \end{Vmatrix}^{2}_{G(\theta)} \!:=\! 
		[u^{\rm{tr}},v^{\rm{tr}}] G(\theta) \begin{bmatrix} u \\ v \end{bmatrix}
		\!=\! {\frac{1}{4}}(1 \!+\! {\theta}) \| u \|^{2} 
		+ {\frac{1}{4}}(1 \!-\! {\theta}) \| v \|^{2}, \quad \forall u,v \in \big( L^{2}(\Omega) \big)^{d}, 
		\label{def:G-norm}
		\end{gather}
		where $\rm{tr}$ means transpose of any vector and the $\gamma$-coefficients in numerical dissipation are
		\begin{align*}   
		\gamma_{1}^{(n)}=-\frac{\sqrt{\theta ( 1-{\theta }^{2})}}{\sqrt{2}
			(1+\varepsilon _{n} \theta)}, \quad 
		\gamma_{2}^{(n)}=-\frac{1-\varepsilon_{n}}{2}\gamma_{1}^{(n)}, \quad 
		\gamma_{0}^{(n)}=-\frac{1+\varepsilon_{n}}{2}\gamma_{1}^{(n)},
		\end{align*}
	\end{lemma}
	\begin{proof}
		Just algebraic calculation.
	\end{proof}
	\begin{remark}
		If we replace $\big( L^{2}(\Omega) \big)^{d}$ with $\mathbb{R}^{d}$ 
		and $L^{2}$-norm with Euclidean $\ell^{2}$-norm, 
		the variable time-stepping DLN method in \eqref{eq:1leg-DLN} is unconditional $G$-stable from the $G$-stability identity in \eqref{eq:Gstab-Id}. 
		Recall the definition of $G$-stability for the one-leg, $m$-step scheme with constant 
		step $k$ \cite{Dah78_BIT}: 
		\begin{gather*}
			\sum_{\ell=0}^{m} \alpha_{\ell} y_{n+1-\ell} = k f 
			\big( \sum_{\ell=0}^{m} \beta_{\ell} t_{n+1-\ell}, \sum_{\ell=0}^{m} \beta_{\ell} y_{n+1-\ell} \big).
		\end{gather*}
		The above scheme satisfies $G$-stability condition
		if there exists a real symmetric positive definite matrix 
		$G = [g_{ij}]_{i,j=1}^{m} \in \mathbb{R}^{m \times m}$ such that for all $n$
		\begin{gather*}
			Y_{n+1}^{tr} \mathbb{G} Y_{n+1} - Y_{n}^{tr} \mathbb{G} Y_{n} \leq 2 k \Big( f 
			\big( \sum_{\ell=0}^{m} \beta_{\ell} t_{n+1-\ell}, \sum_{\ell=0}^{m} \beta_{\ell} y_{n+1-\ell} \big), \sum_{\ell=0}^{m} \beta_{\ell} y_{n+1-\ell} \Big).
		\end{gather*}
		where $Y_{n} = [y_{n}^{\rm{tr}}, y_{n-1}^{\rm{tr}}, \cdots, y_{n-m+1}^{\rm{tr}}]^{\rm{tr}}$ and 
		$\mathbb{G} = G \otimes \mathbb{I}_{d}$.  
		The above $G$-stability inequality ensures that the deviation from the initial condition in (in $G$-norm) controls the deviations from the sequence of solutions at later times based on that initial condition.
 		From the $G$-stability identity in \eqref{eq:Gstab-Id}, the DLN method with $\theta \in [0,1)$ is $G$-stable. 
		For the case $\theta = 1$, the DLN method with $\theta = 1$ is reduced to one-step midpoint rule and its $G$-stability property is easy to check by definition.
	\end{remark}

	\begin{lemma}
		\label{lemma:DLN-consistency}
		Let $\{ t_{n} \}_{n=0}^{N}$ be the time grids on the time interval $[0,T]$, $u(\cdot,t)$ the mapping from $[0,T]$ to $H^{\ell}(\Omega)$ and $u_{n}$ the function $u(\cdot,t_{n})$ in $H^{\ell}(\Omega)$. Assuming the mapping $u(\cdot,t)$ is smooth about $t$, then for any $\theta \in [0,1)$
		\begin{align}
		\| u_{n,\beta} - u(t_{n,\beta}) \|_{\ell}^{2} 
		\leq& C(\theta) (k_{n} + k_{n-1})^{3} \int_{t_{n-1}}^{t_{n+1}} \| u_{tt} \|_{\ell}^{2} dt, 
		\label{eq:consist-2nd-eq1} \\
		\| u_{n,\ast} - u(t_{n,\beta}) \|_{\ell}^{2} 
		\leq& C(\theta) (k_{n} + k_{n-1})^{3} \int_{t_{n-1}}^{t_{n+1}} \| u_{tt} \|_{\ell}^{2} dt, 
		\label{eq:consist-2nd-eq2} \\
		\Big\| \frac{u_{n,\alpha}}{\widehat{k}_{n}} - u_{t}(t_{n,\beta})\Big\|_{\ell}^{2}  
		\leq& C(\theta) (k_{n} + k_{n-1})^{3} \displaystyle\int_{t_{n-1}}^{t_{n+1}} \| u_{ttt} \|_{\ell}^{2} dt. \label{eq:consist-2nd-eq3} 
		\end{align}
		For $\theta = 1$, the corresponding conclusions for the midpoint rule method are 
		\begin{align*}
		\big\| u_{n,\beta} - u(t_{n,\beta}) \big\|_{\ell}^{2} 
		=& \Big\| \frac{u_{n+1} + u_{n}}{2} - u \big( \frac{t_{n+1} + t_{n}}{2} \big)\Big\|_{\ell}
		\leq C k_{n}^{3} \int_{t_{n}}^{t_{n+1}} \| u_{tt} \|_{\ell}^{2} dt, 
		\\
		\| u_{n,\ast} - u(t_{n,\beta}) \|_{\ell}^{2}
		=& \Big\| \Big( 1 + \frac{k_{n}}{2k_{n-1}} \Big) u_{n} - \frac{k_{n}}{2k_{n-1}} u_{n-1} 
		    - u \big( \frac{t_{n+1} + t_{n}}{2} \big) \Big\|_{\ell} \\
		\leq& C (k_{n} + k_{n-1})^{3} \int_{t_{n-1}}^{t_{n+1}} \| u_{tt} \|_{\ell}^{2} dt, \\
		\Big\| \frac{u_{n,\alpha}}{\widehat{k}_{n}} - u_{t}(t_{n,\beta})\Big\|_{\ell}^{2}  
		=& \Big\| \frac{u_{n+1} - u_{n}}{k_{n}} - u_{t} \big( \frac{t_{n+1} + t_{n}}{2} \big) \Big\|_{\ell}
		\leq C k_{n}^{3} \displaystyle\int_{t_{n}}^{t_{n+1}} \| u_{ttt} \|_{\ell}^{2} dt. 
		\end{align*}
	\end{lemma}
	\begin{proof}
		See Appendix \ref{appendixA}.
	\end{proof}

	\noindent All conclusions in Section \ref{sec:Num-Analysis} need the following CFL-like conditions:
	for $j = 1, \cdots , J$, $\theta \in (0,1)$ and $n = 1, \cdots, N-1$, 
	\begin{gather} 
		1 - C (\Omega,\theta) \Big( \frac{1 + {\varepsilon_{n}}{\theta}}{1 - {\varepsilon_{n}}} \Big)^{2} \frac{\widehat{k}_{n}}{h{\nu}} \Big\| \nabla \big( u_{n,\ast}^{j,h} - \langle u^{h} \rangle_{n,\ast} \big)  \Big\|^{2} 
		\geq 0, 
		\label{eq:CFL-like-cond}
	\end{gather}
	for certain positive constant $C (\Omega,\theta)$ independent of time steps and diameter $h$.

	\subsection{Stability and Error Analysis in $L^{2}$-norm} \label{subsec:Stab-Error-L2}
	\ 
	\begin{theorem}
		\label{thm:Stab-L2}
		If the source function in the $j$-th NSE $f^{j} \in \ell^{2,\beta}\big( \{ t_{n}\}_{n=0}^{N}; X^{-1} \big)$ for all $j$, then for all $\theta \in (0,1)$, under CFL-like conditions in \eqref{eq:CFL-like-cond},
		the family of variable time-stepping DLN-Ensemble Algorithms in \eqref{eq:DLN-Ensemble-Alg} satisfy: 
		for $M = 2, 3, \cdots, N$, 
		\begin{gather}
			\frac{1 + \theta}{4} \| u_{M}^{j,h} \|^{2} + \frac{1 - \theta}{4} \| u_{M-1}^{j,h} \|^{2} 
			+ \sum_{n = 1}^{M-1} \frac{{\theta} (1 - {\theta}^{2})}{4 ( 1 + {\varepsilon_{n}}{\theta})^{2}} 
			\big\| U_{n}^{j,h} \big\|^{2}
			+ \sum_{n=1}^{M-1} \frac{\nu \widehat{k}_{n}}{4} \| \nabla u_{n,\beta}^{j,h} \|^{2} \notag \\
			\leq \frac{C}{\nu} \| |f^{j}| \|_{2,-1,\beta}^{2} 
			+ \frac{1 + \theta}{4}  \| u_{1}^{j,h} \|^{2} + \frac{1 - \theta}{4} \| u_{0}^{j,h} \|^{2},
			\label{eq:L2-Stab-conclusion}
		\end{gather}
		where 
		\begin{gather*}
			U_{n}^{j,h} = \frac{1 - \varepsilon_{n}}{2} u_{n+1}^{j,h} - u_{n}^{j,h} + \frac{1 + \varepsilon_{n}}{2} u_{n-1}^{j,h}.
		\end{gather*}
	\end{theorem}
	\begin{proof}
		We set $v_{h} = u_{n,\beta}^{j,h}$, $q^{h} = p_{n,\beta}^{j,h}$ in \eqref{eq:DLN-Ensemble-Alg}, add two equations together and use the skew-symmetric property of $b$
		\begin{gather}  
			\frac{1}{\widehat{k}_{n}} \big( u_{n,\alpha}^{j,h},u_{n,\beta}^{j,h} \big) 
			+ {\nu} \| \nabla u_{n,\beta}^{j,h} \|^{2} \!=\! \big(f^{j}(t_{n,\beta}), u_{n,\beta}^{j,h} \big)
			- b \big( u_{n,\ast}^{j,h} - \langle u^{h} \rangle_{n,\ast}, u_{n,\ast}^{j,h}, u_{n,\beta}^{j,h} \big).
			\label{eq:DLN-Stab-Eq1}
		\end{gather}
		By the skew-symmetric property of $b$, \eqref{eq:b-bound-3} in Lemma \ref{lemma:b-bound}, 
		Poincar\'e inequalty and the inverse inequality in \eqref{eq:inv-inequal} 
		\begin{align}
			|b \big( u_{n,\ast}^{j,h} \!-\! \langle u^{h} \rangle_{n,\ast}, u_{n,\ast}^{j,h}, u_{n,\beta}^{j,h} \big)| 
			\!=& \!|b \big( u_{n,\ast}^{j,h} \!-\! \langle u^{h} \rangle_{n,\ast}, u_{n,\beta}^{j,h}, u_{n,\beta}^{j,h} - u_{n,\ast}^{j,h} \big)| 
			\label{eq:b-Stab-1} \\
			\leq&\! \frac{C(\Omega)}{\sqrt{h}} \big\| \nabla \big( u_{n,\ast}^{j,h} \!-\! \langle u^{h} \rangle_{n,\ast}, u_{n,\ast}^{j,h} \big) \big\| \| \nabla u_{n,\beta}^{j,h} \| \| u_{n,\beta}^{j,h} \!-\! u_{n,\ast}^{j,h} \|. \notag 
		\end{align}
		\begin{confidential}
			\color{darkblue}
			\begin{align*}
				&|b \big( u_{n,\ast}^{j,h} \!-\! \langle u^{h} \rangle_{n,\ast}, u_{n,\ast}^{j,h}, u_{n,\beta}^{j,h} \big)| \\
				\leq& |b \big( u_{n,\ast}^{j,h} \!-\! \langle u^{h} \rangle_{n,\ast}, u_{n,\beta}^{j,h}, u_{n,\beta}^{j,h} - u_{n,\ast}^{j,h} \big)| \\
				\leq& C(\Omega) \| \nabla \big(u_{n,\ast}^{j,h} \!-\! \langle u^{h} \rangle_{n,\ast} \big) \| 
				\| \nabla u_{n,\beta}^{j,h} \| \| \nabla \big( u_{n,\beta}^{j,h} - u_{n,\ast}^{j,h}\big) \|^{1/2} 
				\| u_{n,\beta}^{j,h} - u_{n,\ast}^{j,h} \|^{1/2} \\
				\leq& C(\Omega) h^{-1/2} \| \nabla \big( u_{n,\ast}^{j,h} \!-\! \langle u^{h} \rangle_{n,\ast} \big) \| \| \nabla u_{n,\beta}^{j,h} \| \| u_{n,\beta}^{j,h} - u_{n,\ast}^{j,h} \| .
			\end{align*}
			\normalcolor
		\end{confidential}
		It's easy to check
		\begin{gather}
			u_{n,\beta}^{j,h} - u_{n,\ast}^{j,h}
			= \frac{2 \beta_{2}^{(n)}}{1 - \varepsilon_{n}} U_{n}^{j,h},
			\label{eq:DLN-2nd-approx}
		\end{gather}
		\begin{confidential}
			\color{darkblue}
			\begin{align*}
				&u_{n,\beta}^{j,h} - u_{n,\ast}^{j,h} \\
				=& \big( \beta_{2}^{(n)} u_{n+1}^{j,h} + \beta_{1}^{(n)} u_{n}^{j,h} + \beta_{0}^{(n)} u_{n-1}^{j,h} \big)
				- \Big\{ \Big[ \beta_{2}^{(n)} \big( 1 + \frac{k_{n}}{k_{n-1}} \big) 
				+ \beta_{1}^{(n)} \Big] u_{n}^{j,h} + \big( \beta_{0}^{(n)} - \beta_{2}^{(n)} \frac{k_{n}}{k_{n-1}} \big) u_{n-1}^{j,h}  \Big\}  \\
				=& \beta_{2}^{(n)} u_{n+1}^{j,h} + \Big\{ \beta_{1}^{(n)} - \Big[ \beta_{2}^{(n)} \big( 1 + \frac{k_{n}}{k_{n-1}} \big) + \beta_{1}^{(n)} \Big] \Big\} u_{n}^{j,h}
				+ \Big[ \beta_{0}^{(n)} - \big( \beta_{0}^{(n)} - \beta_{2}^{(n)} \frac{k_{n}}{k_{n-1}} \big) \Big] u_{n-1}^{j,h}  \\
				=& \beta_{2}^{(n)} u_{n+1}^{j,h} - \beta_{2}^{(n)} \big( 1 + \frac{k_{n}}{k_{n-1}} \big) u_{n}^{j,h} + \beta_{2}^{(n)} \frac{k_{n}}{k_{n-1}} u_{n-1}^{j,h}  \\
				=& \beta_{2}^{(n)} \Big[ u_{n+1}^{j,h} - \big( 1 + \frac{k_{n}}{k_{n-1}} \big) u_{n}^{j,h} 
				+ \frac{k_{n}}{k_{n-1}}  u_{n-1}^{j,h} \Big] .
			\end{align*}
			By definition of $\varepsilon_{n}$ 
			\begin{gather*}
				\varepsilon_{n} = \frac{k_{n} - k_{n-1}}{k_{n} + k_{n-1}}, \ \ \ \Leftrightarrow \ \ \  \left( k_{n} + k_{n-1} \right) \varepsilon_{n} = k_{n} - k_{n-1}  \\
				k_{n-1} \varepsilon_{n} + k_{n-1} = k_{n} - k_{n}\varepsilon_{n}, \ \ \ \Leftrightarrow \ \ \ \left( 1 + \varepsilon_{n} \right)k_{n-1} = \left(1- \varepsilon_{n} \right)k_{n} \\
				\frac{k_{n}}{k_{n-1}} = \frac{1 + \varepsilon_{n}}{1 - \varepsilon_{n}}.
			\end{gather*}
			Thus
			\begin{align*}
				u_{n,\beta}^{j,h} - u_{n,\ast}^{j,h} 
				&= \beta_{2}^{(n)} \Big[ u_{n+1}^{j,h} - \big( 1 + \frac{k_{n}}{k_{n-1}} \big) u_{n}^{j,h} 
				+ \frac{k_{n}}{k_{n-1}}  u_{n-1}^{j,h} \Big] \\
				&= \beta_{2}^{(n)} \Big[ u_{n+1}^{j,h} - \big( 1 + \frac{1 + \varepsilon_{n}}{1 - \varepsilon_{n}} \big) u_{n}^{j,h} + \frac{1 + \varepsilon_{n}}{1 - \varepsilon_{n}} u_{n-1}^{j,h} \Big] \\
				&= \beta_{2}^{(n)} \Big[ u_{n+1}^{j,h} - \big(\frac{1-\varepsilon_{n}+ 1 + \varepsilon_{n}}{1 - \varepsilon_{n}} \big) u_{n}^{j,h} + \frac{1 + \varepsilon_{n}}{1 - \varepsilon_{n}} u_{n-1}^{j,h} \Big] \\
				&= \beta_{2}^{(n)} \Big[ u_{n+1}^{j,h} - \big(\frac{2}{1 - \varepsilon_{n}} \big) u_{n}^{j,h} 
				+ \frac{1 + \varepsilon_{n}}{1 - \varepsilon_{n}} u_{n-1}^{j,h} \Big] \\
				&= \beta_{2}^{(n)} \frac{2}{1 - \varepsilon_{n}} \Big[ \frac{1 - \varepsilon_{n}}{2} u_{n+1}^{j,h} - u_{n}^{j,h} + \frac{1 + \varepsilon_{n}}{2} u_{n-1}^{j,h} \Big] \\
				&= \frac{2 \beta_{2}^{(n)}}{1 - \varepsilon_{n}} U_{n}^{j,h}.
			\end{align*}
			\normalcolor
		\end{confidential}
		By \eqref{eq:b-Stab-1}, \eqref{eq:DLN-2nd-approx}, \eqref{eq:bound-beta} and Young's inequality
		\begin{align}
			&\widehat{k}_{n} |b \big( u_{n,\ast}^{j,h} \!-\! \langle u^{h} \rangle_{n,\ast}, u_{n,\ast}^{j,h}, u_{n,\beta}^{j,h} \big)|  \label{eq:b-Stab-2} \\
			\leq&  C (\Omega,\theta) \Big( \frac{1 + \varepsilon_{n} \theta}{1 - \varepsilon_{n}} \Big)^{2} 
			\frac{\widehat{k}_{n}^{2}}{h}
			\big\| \nabla \big( u_{n,\ast}^{j,h} \!-\! \langle u^{h} \rangle_{n,\ast}  \big) \big\|^{2} \| \nabla u_{n,\beta}^{j,h} \|^{2}
			+ \frac{ \theta (1 - {\theta}^{2})}{4( 1 + \varepsilon_{n} \theta )^{2}} \| U_{n}^{j,h} \|^{2}.
			\notag 
		\end{align}
		\begin{confidential}
			\color{darkblue}
			\begin{align*}
				&\widehat{k}_{n} |b \big( u_{n,\ast}^{j,h} \!-\! \langle u^{h} \rangle_{n,\ast}, u_{n,\ast}^{j,h}, u_{n,\beta}^{j,h} \big)|  \\
				\leq& C (\Omega) \frac{2 \beta_{2}^{(n)} \widehat{k}_{n}}{1 - \varepsilon_{n}} h^{-1/2} \| \nabla \big( u_{n,\ast}^{j,h} \!-\! \langle u^{h} \rangle_{n,\ast} \big) \| \| \nabla u_{n,\beta}^{j,h} \| \| U_{n}^{j,h} \| \\
				=& C(\Omega,\theta) \Big[ \Big(\frac{1 + {\varepsilon_{n}}{\theta}}{1 - \varepsilon_{n}} \Big) \widehat{k}_{n} h^{-1/2} \big\| \nabla \big( u_{n,\ast}^{j,h} \!-\! \langle u^{h} \rangle_{n,\ast} \big) \big\| \| \nabla u_{n,\beta}^{j,h} \| \Big] \Big( \frac{1}{1 + {\varepsilon_{n}}{\theta}} \| U_{n}^{j,h} \| \Big) \\
				\leq&  C (\Omega,\theta) \Big( \frac{1 + {\varepsilon_{n}}{\theta}}{1 - \varepsilon_{n}} \Big)^{2} \widehat{k}_{n}^{2} h^{-1} \big\| \nabla \big( u_{n,\ast}^{j,h} \!-\! \langle u^{h} \rangle_{n,\ast}  \big) \big\|^{2} \| \nabla u_{n,\beta}^{j,h} \|^{2} + \frac{ \theta (1 - {\theta}^{2} )}{4 ( 1 + {\varepsilon_{n}} {\theta} )^{2}} \| U_{n}^{j,h} \|^{2}
			\end{align*}
			By the fact 
			\begin{align*}
			\Big\| \sum_{\ell\!=\!0}^{2} \gamma_{\ell}^{(n)} u_{n\!-\!1\!+\!\ell}^{j,h} \Big\|^{2}
			= {\gamma_{1}^{(n)}}^{2} \Big\| \frac{1 - \varepsilon_{n}}{2} u_{n+1}^{j,h} - u_{n}^{j,h} 
			+ \frac{1 + \varepsilon_{n}}{2} u_{n-1}^{j,h} \Big\|^{2} 
			= \frac{\theta (1 - \theta^{2})}{2 (1 + \varepsilon_{n} \theta)^{2}} \| U_{n}^{j,h} \|^{2}
			\end{align*}
			\normalcolor
		\end{confidential}
		By the identity in \eqref{eq:Gstab-Id}, definition of dual norm in \eqref{eq:dual-norm}, Young's inequality, \eqref{eq:b-Stab-2} and the fact 
		\begin{align*}
			\Big\| \sum_{\ell\!=\!0}^{2} \gamma_{\ell}^{(n)} u_{n\!-\!1\!+\!\ell}^{j,h} \Big\|^{2}
			= \frac{\theta (1 - \theta^{2})}{2 (1 + \varepsilon_{n} \theta)^{2}} \| U_{n}^{j,h} \|^{2},
		\end{align*}
		\eqref{eq:DLN-Stab-Eq1} becomes 
		\begin{confidential}
			\color{darkblue}
			\begin{gather*}
				\begin{Vmatrix} u_{n\!+\!1}^{j,h} \\ u_{n}^{j,h} \end{Vmatrix}^{2}_{G(\theta)}\!-\!\begin{Vmatrix} u_{n}^{j,h} \\ u_{n\!-\!1}^{j,h} \end{Vmatrix}^{2}_{G(\theta)}
				\!+\! \frac{\theta (1 - \theta^{2})}{2 (1 + \varepsilon_{n} \theta)^{2}} \| U_{n}^{j,h} \|^{2} 
				+ {\nu} \widehat{k}_{n} \| \nabla u_{n,\beta}^{j,h} \|^{2} \\
				- C (\Omega,\theta) \Big( \frac{1 + \varepsilon_{n} \theta}{1 - \varepsilon_{n}} \Big)^{2} 
				\frac{\widehat{k}_{n}^{2}}{h}
				\big\| \nabla \big( u_{n,\ast}^{j,h} \!-\! \langle u^{h} \rangle_{n,\ast} \big) \big\|^{2} \| \nabla u_{n,\beta}^{j,h} \|^{2}
				- \frac{ \theta (1 - {\theta}^{2})}{4( 1 + \varepsilon_{n} \theta )^{2}} \| U_{n}^{j,h} \|^{2} \\
				\leq \widehat{k}_{n} \big(f^{j}(t_{n,\beta}), u_{n,\beta}^{j,h} \big)
				\leq \widehat{k}_{n} \| f^{j}(t_{n,\beta}) \|_{-1} \| \nabla u_{n,\beta}^{j,h} \|
				\leq \frac{C}{\nu} \widehat{k}_{n} \| f^{j}(t_{n,\beta}) \|_{-1}^{2} 
				+ \frac{\nu \widehat{k}_{n}}{4} \| \nabla u_{n,\beta}^{j,h} \|^{2}.
			\end{gather*}
			\normalcolor
		\end{confidential}
		\begin{align}
			&\begin{Vmatrix} u_{n\!+\!1}^{j,h} \\ u_{n}^{j,h} \end{Vmatrix}^{2}_{G(\theta)}\!-\!\begin{Vmatrix} u_{n}^{j,h} \\ u_{n\!-\!1}^{j,h} \end{Vmatrix}^{2}_{G(\theta)}
			\!+\! \frac{\theta (1 - \theta^{2})}{4 (1 + \varepsilon_{n} \theta)^{2}} \| U_{n}^{j,h} \|^{2}
			+ \frac{\nu \widehat{k}_{n}}{4} \| \nabla u_{n,\beta}^{j,h} \|^{2}  
			\label{eq:DLN-Stab-Eq2} \\
			& + \frac{\nu \widehat{k}_{n}}{2} \Big[ 1 - C (\Omega,\theta) \Big( \frac{1 + \varepsilon_{n} \theta}{1 - \varepsilon_{n}} \Big)^{2} \frac{\widehat{k}_{n}}{ \nu h} \big\| \nabla \big( u_{n,\ast}^{j,h} \!-\! \langle u^{h} \rangle_{n,\ast} \big) \big\|^{2} \Big]
			\| \nabla u_{n,\beta}^{j,h} \|^{2} 
			\leq \frac{C}{\nu} \widehat{k}_{n} \| f^{j}(t_{n,\beta}) \|_{-1}^{2} \notag .
		\end{align}
		We sum \eqref{eq:DLN-Stab-Eq2} over $n$ from $1$ to $M-1$ and use the definition of the discrete norm in \eqref{eq:def-norm-discrete}, CLF-like conditions in \eqref{eq:CFL-like-cond} and 
		the fact $\widehat{k}_{n} \leq k_{n} + k_{n-1}$ to obtain \eqref{eq:L2-Stab-conclusion}.
		\begin{confidential}
			\color{darkblue}
			\begin{align*}
			\begin{Vmatrix} u_{n\!+\!1}^{j,h} \\ u_{n}^{j,h} \end{Vmatrix}^{2}_{G(\theta)}\!-\!\begin{Vmatrix} u_{n}^{j,h} \\ u_{n\!-\!1}^{j,h} \end{Vmatrix}^{2}_{G(\theta)}
			\!+\! \frac{\theta (1 - \theta^{2})}{4 (1 + \varepsilon_{n} \theta)^{2}} \| U_{n}^{j,h} \|^{2}
			+ \frac{\nu \widehat{k}_{n}}{4} \| \nabla u_{n,\beta}^{j,h} \|^{2}  
			\leq  \frac{C}{\nu} (k_{n}+k_{n-1}) \| f^{j}(t_{n,\beta}) \|_{-1}^{2}.
			\end{align*}
			\begin{gather*}
			\frac{1 + \theta}{4} \| u_{M}^{j,h} \|^{2} + \frac{1 - \theta}{4} \| u_{M-1}^{j,h} \|^{2} 
			+ \sum_{n = 1}^{M-1} \frac{{\theta} (1 - {\theta}^{2})}{4 ( 1 + {\varepsilon_{n}}{\theta})^{2}} 
			\big\| U_{n}^{j,h} \big\|^{2}
			+ \sum_{n=1}^{M-1} \frac{\nu \widehat{k}_{n}}{4} \| \nabla u_{n,\beta}^{j,h} \|^{2} \notag \\
			\leq \frac{C}{\nu} \| |f| \|_{2,-1,\beta}^{2}
			+ \frac{1 + \theta}{4}  \| u_{1}^{j,h} \|^{2} + \frac{1 - \theta}{4} \| u_{0}^{j,h} \|^{2}
			\end{gather*}
			\normalcolor
		\end{confidential}

	\end{proof}

	\begin{remark}
		For $\theta = 0,1$, the above proof for the case $\theta \in (0,1)$ doesn't work. If $\theta = 0,1$,
		we have 
		\begin{align*}
			&\Big\| \sum_{\ell\!=\!0}^{2} \gamma_{\ell}^{(n)} u_{n\!-\!1\!+\!\ell}^{j,h} \Big\|^{2} = 0, \\
			&u_{n,\beta}^{j,h} - u_{n,\ast}^{j,h} =
			\frac{1}{2 } \Big[ u_{n+1}^{j,h} - \Big( 1 + \frac{k_{n}}{k_{n-1}} \Big)u_{n}^{j,h}
				                   + \frac{k_{n}}{k_{n-1}} u_{n-1}^{j,h} \Big] .
		\end{align*} 
		It's impossible to hide $\| u_{n,\beta}^{j,h} - u_{n,\ast}^{j,h} \|^{2}$ (generally non-zero) by the numerical dissipation 
		$\displaystyle  \Big\| \sum_{\ell\!=\!0}^{2} \gamma_{\ell}^{(n)} u_{n\!-\!1\!+\!\ell}^{j,h} \Big\|^{2}$. 
		How to prove Theorem \ref{thm:Stab-L2} for the case $\theta = 0,1$ leaves an open question. \\
	\end{remark}

	For error analysis, we denote $k_{\rm{max}} = \displaystyle \max_{0 \leq n \leq N-1} \{ k_{n} \}$ and assume the following time ratio bounds
	\begin{gather}
		C_{L} < \frac{k_{n}}{k_{n-1}} < C_{U}, \qquad n = 1, 2, \cdots, N-1,
		\label{eq:time-ratio-cond}
	\end{gather}
	for some $C_{L},C_{U} >0$.
	\begin{theorem}
		\label{thm:Error-L2}
		If the velocity $u^{j}(x,t)$ and pressure $p^{j}(x,t)$ in the $j$-th NSE satisfy the following regularities:
		\begin{gather}
		u \in \ell^{\infty}(\{t_{n}\}_{n=0}^{N};H^{r}(\Omega))  \cap \ell^{\infty,\beta}(\{t_{n}\}_{n=0}^{N};H^{r+1}(\Omega)), \ p \in \ell^{2,\beta}\big( \{t_{n}\}_{n=0}^{N};H^{s+1}(\Omega) \big),
		\notag \\
		u_{t} \in L^{2}\big( 0,T;H^{r+1}(\Omega) \big),  \ \ 
		u_{tt} \in L^{2}\big(0,T;H^{r+1}(\Omega) \big), \ \ 
		u_{ttt} \in L^{2}\big( 0,T;X^{-1} \big),
		\label{eq:L2-regularity}
		\end{gather}
		for all $j$, then under CFL-like conditions in \eqref{eq:CFL-like-cond} and time ratio bounds in \eqref{eq:time-ratio-cond}, the numerical solutions by the DLN-Ensemble algorithms in \eqref{eq:DLN-Ensemble-Alg} for all $\theta \in (0,1)$ satisfy
		\begin{gather}
		\max_{0 \leq n \leq N} \| u_{n}^{j} - u_{n}^{j,h} \| + \Big(  \nu  \sum_{n=1}^{N-1} \widehat{k}_{n} 
		\| \nabla \big( u_{n,\beta}^{j} - u_{n,\beta}^{j,h} \big) \|^{2} \Big)^{1/2} 
		\label{eq:error-L2-conclusion} \\
		\leq  \mathcal{O} \big( h^{r} ,h^{s+1} , k_{\rm{max}}^{2}, h^{1/2} k_{\rm{max}}^{3/2} \big).
		\notag
		\end{gather}
	\end{theorem}
	\begin{proof}
		The proof is relatively long, thus we leave it to Appendix \ref{appendixB-L2}.
	\end{proof}
	\ \\

	\subsection{Stability and Error Analysis in $H^{1}$-norm}
	\label{subsec:Stab-Error-H1}
	\ 
	\begin{theorem}
		If the source function in the $j$-th NSE $f^{j} \in \ell^{2,\beta}\big( \{ t_{n}\}_{n=1}^{N}; L^{2}(\Omega) \big)$,
		the velocity $u^{j}(x,t)$ and pressure $p^{j}(x,t)$ in the $j$-th NSE satisfy the regularities 
		in \eqref{eq:L2-regularity} for all $j$, the time step $k_{\rm{max}}$ and diameter $h$ satisfy 
		\begin{gather}
			k_{\rm{max}} \leq h^{1/4},
			\label{eq:time-diameter-relation}
		\end{gather}
		then for all $\theta \in (0,1)$, under CFL-like conditions in \eqref{eq:CFL-like-cond} and time ratio bounds in \eqref{eq:time-ratio-cond}, 
		the family of variable time-stepping DLN-Ensemble Algorithms in \eqref{eq:DLN-Ensemble-Alg} satisfy: 
		for $M = 2, 3, \cdots, N$, 
		\begin{gather}
			\| \nabla u_{M}^{j,h} \|^{2}  
			+ \sum_{n=1}^{M-1} \frac{ \theta (1 - {\theta}^{2})}{4 ( 1 + \varepsilon_{n} \theta )^{2}} \big\| \nabla U_{n}^{j,h} \big\|^{2}
			+ \frac{4}{\nu(1+\theta)} \sum_{n=1}^{M-1} \widehat{k}_{n} 
			\Big\| \widehat{k}_{n}^{-1} u_{n,\alpha}^{j,h} \Big\|^{2}        
			\label{eq:DLN-Stab-H1-conclusion} \\
			\leq  C(\frac{1}{\nu}, \theta,T)
			\Big[ \frac{C}{\nu} \| |f| \|_{2,0,\beta}^{2} \!+\! C(\theta) \big( \| \nabla u_{1}^{j,h} \|^{2} \!+\! \| \nabla u_{0}^{j,h} \|^{2} \big) \Big].
			\notag 
		\end{gather}
	\end{theorem}
	\begin{proof}
		We set $v^{h} = u_{n,\alpha}^{j,h}/\widehat{k}_{n} \in V^{h}$ in \eqref{eq:DLN-Ensemble-Alg}.
		According to the definition of $V^{h}$
		\begin{gather*}
			\big( p_{n,\beta}^{j,h}, \nabla \cdot \widehat{k}_{n}^{-1} u_{n,\alpha}^{j,h} \big) =0.
		\end{gather*}
		Since 
		\begin{align*}
		&b \big( \langle u^{h} \rangle_{n,\ast} - u_{n,\ast}^{j,h}, u_{n,\ast}^{j,h}, \widehat{k}_{n}^{-1} u_{n,\alpha}^{j,h} \big) 
		- b \big( \langle u^{h} \rangle_{n,\ast}, u_{n,\beta}^{j,h} , \widehat{k}_{n}^{-1} u_{n,\alpha}^{j,h} \big) \\
		=& b \big( \langle u^{h} \rangle_{n,\ast} - u_{n,\ast}^{j,h}, u_{n,\ast}^{j,h}, \widehat{k}_{n}^{-1} u_{n,\alpha}^{j,h} \big) - b \big( \langle u^{h} \rangle_{n,\ast} - u_{n,\ast}^{j,h}, u_{n,\beta}^{j,h} , \widehat{k}_{n}^{-1} u_{n,\alpha}^{j,h} \big)  \\
		&- b \big( u_{n,\ast}^{j,h}, u_{n,\beta}^{j,h} , \widehat{k}_{n}^{-1} u_{n,\alpha}^{j,h} \big) \\
		=& b \big( \langle u^{h} \rangle_{n,\ast} - u_{n,\ast}^{j,h}, u_{n,\ast}^{j,h} - u_{n,\beta}^{j,h}, \widehat{k}_{n}^{-1} u_{n,\alpha}^{j,h} \big) 
		- b \big( u_{n,\ast}^{j,h}, u_{n,\beta}^{j,h} - u_{n,\beta}^{j} , \widehat{k}_{n}^{-1} u_{n,\alpha}^{j,h} \big) \\
		&- b \big( u_{n,\ast}^{j,h}, u_{n,\beta}^{j} , \widehat{k}_{n}^{-1} u_{n,\alpha}^{j,h} \big),
		\end{align*}
		we have
		\begin{align}
			&\Big\| \widehat{k}_{n}^{-1} u_{n,\alpha}^{j,h} \Big\|^{2} + 
			\frac{\nu}{\widehat{k}_{n}} \Big(\begin{Vmatrix}
			\nabla {u_{n+1}^{j,h}} \\
			\nabla {u_{n}^{j,h}}
			\end{Vmatrix}_{G(\theta)}^{2} \!\!\!-\!
			\begin{Vmatrix}
			\nabla {u_{n}^{j,h}} \\
			\nabla {u_{n-1}^{j,h}}
			\end{Vmatrix}%
			_{G(\theta)}^{2} 
			\!\!\!+\! \frac{ \theta (1 - {\theta}^{2})}{2 ( 1 + \varepsilon_{n} \theta )^{2}} \big\| \nabla U_{n}^{j,h} \big\|^{2}  \Big) 
			\label{eq:DLN-Stab-H1-Eq1} \\
			=& b \big( \langle u^{h} \rangle_{n,\ast} - u_{n,\ast}^{j,h}, u_{n,\ast}^{j,h} - u_{n,\beta}^{j,h}, \widehat{k}_{n}^{-1} u_{n,\alpha}^{j,h} \big) 
			- b \big( u_{n,\ast}^{j,h}, u_{n,\beta}^{j,h} - u_{n,\beta}^{j} , \widehat{k}_{n}^{-1} u_{n,\alpha}^{j,h} \big)
			\notag \\
			&- b \big( u_{n,\ast}^{j,h}, u_{n,\beta}^{j} , \widehat{k}_{n}^{-1} u_{n,\alpha}^{j,h} \big)
			+ \big( f^{j}( t_{n,\beta}), \widehat{k}_{n}^{-1} u_{n,\alpha}^{j,h} \big). \notag 
		\end{align}
		By Cauchy-Schwarz inequality and Young's equality,
		\begin{gather}
		\big( f^{j}( t_{n,\beta}), \widehat{k}_{n}^{-1} u_{n,\alpha}^{j,h} \big)
		\!\leq\! \| f^{j}( t_{n,\beta}) \| \big\| \widehat{k}_{n}^{-1} u_{n,\alpha}^{j,h} \big\|
		\!\leq\! C \| f^{j}( t_{n,\beta}) \|^{2} \!+\! \frac{1}{16} \big\| \widehat{k}_{n}^{-1} u_{n,\alpha}^{j,h} \big\|^{2}.
		\label{eq:DLN-Stab-H1-RHS-term1}
		\end{gather}
		By \eqref{eq:b-bound-3}, Poincar\'e inequality, inverse inequality in \eqref{eq:inv-inequal}, 
		\eqref{eq:DLN-2nd-approx}, Young's equality, bounds of $\{\beta_{\ell}^{(n)}\}_{\ell=0}^{2}$ in \eqref{eq:bound-beta} and time ratio restrictions in \eqref{eq:time-ratio-cond}
		\begin{confidential}
			\color{darkblue}
			\begin{align*}
			&b \big( \langle u^{h} \rangle_{n,\ast} - u_{n,\ast}^{j,h}, u_{n,\ast}^{j,h} - u_{n,\beta}^{j,h}, \widehat{k}_{n}^{-1} u_{n,\alpha}^{j,h} \big) \notag \\
			\leq& C(\Omega) \big\| \nabla \big( \langle u^{h} \rangle_{n,\ast} - u_{n,\ast}^{j,h} \big) \big\|
			\big\| \nabla \big( u_{n,\ast}^{j,h} - u_{n,\beta}^{j,h} \big) \big\| 
			\big\| \widehat{k}_{n}^{-1} u_{n,\alpha}^{j,h} \big\|^{1/2} \big\| \nabla \big( \widehat{k}_{n}^{-1} u_{n,\alpha}^{j,h} \big) \big\|^{1/2} \\
			\leq& \frac{ C(\Omega) }{\sqrt{h}} \big\| \nabla \big( \langle u^{h} \rangle_{n,\ast} - u_{n,\ast}^{j,h} \big) \big\| \frac{2 \beta_{2}^{(n)}}{ 1 - \varepsilon_{n}} \big\| \nabla U_{n}^{j,h} \big\|
			\big\| \widehat{k}_{n}^{-1} u_{n,\alpha}^{j,h} \big\| \\
			\leq& C(\Omega,\theta) \sqrt{ \frac{\widehat{k}_{n} }{ h \nu } } \Big( \frac{  1 + \varepsilon_{n} \theta}{ 1 - \varepsilon_{n} } \Big) \big\| \big\| \nabla \big( \langle u^{h} \rangle_{n,\ast} - u_{n,\ast}^{j,h} \big) \big\|
			\big\| \widehat{k}_{n}^{-1} u_{n,\alpha}^{j,h} \big\|
			\frac{ \sqrt{\nu}}{ \sqrt{\widehat{k}_{n} }  \big( 1 + \varepsilon_{n} \theta \big) } \big\| \nabla U_{n}^{j,h} \big\| \\
			\leq& \frac{ C(\Omega,\theta) \widehat{k}_{n} }{h \nu} \Big( \frac{1+ \varepsilon_{n} \theta}{1- \varepsilon_{n} } \Big)^{2}
			\big\| \nabla \big( \langle u^{h} \rangle_{n,\ast} - u_{n,\ast}^{j,h} \big) \big\|^{2} 
			\big\| \widehat{k}_{n}^{-1} u_{n,\alpha}^{j,h} \big\|^{2}
			+ \frac{\nu \theta (1 - \theta^{2}) }{4 \widehat{k}_{n} \big( 1 + \varepsilon_{n} \theta \big)^{2} }
			\big\| \nabla U_{n}^{j,h} \big\|^{2}.
			\end{align*}
			\normalcolor
		\end{confidential}
		\begin{align}
			&b \big( \langle u^{h} \rangle_{n,\ast} - u_{n,\ast}^{j,h}, u_{n,\ast}^{j,h} - u_{n,\beta}^{j,h}, \widehat{k}_{n}^{-1} u_{n,\alpha}^{j,h} \big) 
			\label{eq:DLN-Stab-H1-RHS-term2} \\
			\leq& C(\Omega) \big\| \nabla \big( \langle u^{h} \rangle_{n,\ast} - u_{n,\ast}^{j,h} \big) \big\|
			\big\| \nabla \big( u_{n,\ast}^{j,h} - u_{n,\beta}^{j,h} \big) \big\| 
			\big\| \widehat{k}_{n}^{-1} u_{n,\alpha}^{j,h} \big\|^{1/2} \big\| \nabla \big( \widehat{k}_{n}^{-1} u_{n,\alpha}^{j,h} \big) \big\|^{1/2} \notag \\
			\leq& \frac{ C(\Omega) }{\sqrt{h}} \big\| \nabla \big( \langle u^{h} \rangle_{n,\ast} - u_{n,\ast}^{j,h} \big) \big\| \frac{2 \beta_{2}^{(n)}}{ 1 - \varepsilon_{n}} \big\| \nabla U_{n}^{j,h} \big\|
			\big\| \widehat{k}_{n}^{-1} u_{n,\alpha}^{j,h} \big\| \notag \\
			\leq& \frac{ C(\Omega,\theta) \widehat{k}_{n} }{h \nu} \Big( \frac{1+ \varepsilon_{n} \theta}{1- \varepsilon_{n} } \Big)^{2}
			\big\| \nabla \big( \langle u^{h} \rangle_{n,\ast} - u_{n,\ast}^{j,h} \big) \big\|^{2} 
			\big\| \widehat{k}_{n}^{-1} u_{n,\alpha}^{j,h} \big\|^{2}
			+ \frac{\nu \theta (1 - \theta^{2}) }{4 \widehat{k}_{n} \big( 1 + \varepsilon_{n} \theta \big)^{2} }
			\big\| \nabla U_{n}^{j,h} \big\|^{2}. \notag 
		\end{align}
		\begin{confidential}
			\color{darkblue}
			\begin{align*}
			&b \big( u_{n,\ast}^{j,h}, u_{n,\beta}^{j,h} - u_{n,\beta}^{j} , \widehat{k}_{n}^{-1} u_{n,\alpha}^{j,h} \big) \\
			\leq& C(\Omega) \| \nabla u_{n,\ast}^{j,h} \| \| \nabla e_{n,\beta}^{j} \| 
			\big\| \widehat{k}_{n}^{-1} u_{n,\alpha}^{j,h} \big\|^{1/2} \big\| \nabla \big( \widehat{k}_{n}^{-1} u_{n,\alpha}^{j,h} \big) \big\|^{1/2} \notag \\
			\leq& \frac{C(\Omega)}{\sqrt{h}} \| \nabla u_{n,\ast}^{j,h} \| \| \nabla e_{n,\beta}^{j} \| 
			\big\| \widehat{k}_{n}^{-1} u_{n,\alpha}^{j,h} \big\| \\
			\leq& \frac{C(\Omega) }{h} \| \nabla u_{n,\ast}^{j,h} \|^{2} \| \nabla e_{n,\beta}^{j} \|^{2}
			+ \frac{1}{16} \big\| \widehat{k}_{n}^{-1} u_{n,\alpha}^{j,h} \big\|^{2} \\
			\leq& \frac{C(\Omega,\theta) }{h} \| \nabla e_{n,\beta}^{j} \|^{2} \big( \| \nabla u_{n}^{j,h} \|^{2} 
			+ \| \nabla u_{n-1}^{j,h} \|^{2}  \big) + \frac{1}{16} \big\| \widehat{k}_{n}^{-1} u_{n,\alpha}^{j,h} \big\|^{2}. 
			\end{align*}
			\normalcolor
		\end{confidential}
		\begin{align}
			&b \big( u_{n,\ast}^{j,h}, u_{n,\beta}^{j,h} \!-\! u_{n,\beta}^{j} , \widehat{k}_{n}^{-1} u_{n,\alpha}^{j,h} \big) 
			\label{eq:DLN-Stab-H1-RHS-term3} \\
			\leq& C(\Omega) \| \nabla u_{n,\ast}^{j,h} \| \| \nabla e_{n,\beta}^{j} \| 
			\big\| \widehat{k}_{n}^{-1} u_{n,\alpha}^{j,h} \big\|^{1/2} \big\| \nabla \big( \widehat{k}_{n}^{-1} u_{n,\alpha}^{j,h} \big) \big\|^{1/2} \notag \\
			\leq& \frac{C(\Omega,\theta) }{h} \| \nabla e_{n,\beta}^{j} \|^{2} \big( \| \nabla u_{n}^{j,h} \|^{2} 
			+ \| \nabla u_{n-1}^{j,h} \|^{2}  \big) + \frac{1}{16} \big\| \widehat{k}_{n}^{-1} u_{n,\alpha}^{j,h} \big\|^{2}.  \notag 
		\end{align}
		By \eqref{eq:b-bound-3}, Poincar\'e inequality and Young's inequality
		\begin{align}
			&b \big( u_{n,\ast}^{j,h}, u_{n,\beta}^{j} , \widehat{k}_{n}^{-1} u_{n,\alpha}^{j,h} \big) 
			\label{eq:DLN-Stab-H1-RHS-term4} \\
			\leq& C(\Omega) \| \nabla u_{n,\ast}^{j,h} \| \| u_{n,\beta}^{j} \|_{2} \| \widehat{k}_{n}^{-1} u_{n,\alpha}^{j,h} \|
			\notag \\
			\leq& C(\Omega,\theta) \| \nabla u_{n,\beta}^{j} \|_{2}^{2} \big( \| \nabla u_{n}^{j,h} \|^{2} 
			+ \| \nabla u_{n-1}^{j,h} \|^{2}  \big) + \frac{1}{16} \big\| \widehat{k}_{n}^{-1} u_{n,\alpha}^{j,h} \big\|^{2}. \notag 
		\end{align}
		We combine \eqref{eq:DLN-Stab-H1-RHS-term1}, \eqref{eq:DLN-Stab-H1-RHS-term2}, \eqref{eq:DLN-Stab-H1-RHS-term3} and \eqref{eq:DLN-Stab-H1-RHS-term4} 
		\begin{confidential}
			\color{darkblue}
			\begin{align*}
				& \frac{\nu}{\widehat{k}_{n}} \Big(\begin{Vmatrix}
				\nabla {u_{n+1}^{j,h}} \\
				\nabla {u_{n}^{j,h}}
				\end{Vmatrix}_{G(\theta)}^{2} \!\!\!-\!
				\begin{Vmatrix}
				\nabla {u_{n}^{j,h}} \\
				\nabla {u_{n-1}^{j,h}}
				\end{Vmatrix}%
				_{G(\theta)}^{2} 
				\!\!\!+\! \frac{ \theta (1 - {\theta}^{2})}{4 ( 1 + \varepsilon_{n} \theta )^{2}} \big\| \nabla U_{n}^{j,h} \big\|^{2}  \Big)  \\
				&+ \frac{1}{2} \Big( 1 - \frac{ C(\Omega,\theta) \widehat{k}_{n} }{h \nu} \Big( \frac{1+ \varepsilon_{n} \theta}{1- \varepsilon_{n} } \Big)^{2}
				\big\| \nabla \big( \langle u^{h} \rangle_{n,\ast} - u_{n,\ast}^{j,h} \big) \big\|^{2} \Big)
				\Big\| \widehat{k}_{n}^{-1} u_{n,\alpha}^{j,h} \Big\|^{2} 
				+ \frac{1}{4} \Big\| \widehat{k}_{n}^{-1} u_{n,\alpha}^{j,h} \Big\|^{2}   \\
				\leq&  C \| f^{j}( t_{n,\beta}) \|^{2} + C(\Omega,\theta) \Big( \frac{1}{h} \| \nabla e_{n,\beta}^{j} \|^{2} 
				+ \| \nabla u_{n,\beta}^{j} \|_{2}^{2} \Big) \big( \| \nabla u_{n}^{j,h} \|^{2} 
				+ \| \nabla u_{n-1}^{j,h} \|^{2}  \big)
			\end{align*}
			\normalcolor
		\end{confidential}
		\begin{align}
			& \begin{Vmatrix}
			\nabla {u_{n+1}^{j,h}} \\
			\nabla {u_{n}^{j,h}}
			\end{Vmatrix}_{G(\theta)}^{2} \!\!\!-\!
			\begin{Vmatrix}
			\nabla {u_{n}^{j,h}} \\
			\nabla {u_{n-1}^{j,h}}
			\end{Vmatrix}%
			_{G(\theta)}^{2} 
			\!\!\!+\! \frac{ \theta (1 - {\theta}^{2})}{4 ( 1 + \varepsilon_{n} \theta )^{2}} \big\| \nabla U_{n}^{j,h} \big\|^{2} 
			\label{eq:DLN-Stab-H1-Eq2} \\
			& + \frac{\widehat{k}_{n}}{2 \nu} \Big( 1 - \frac{ C(\Omega,\theta) \widehat{k}_{n} }{h \nu} \Big( \frac{1+ \varepsilon_{n} \theta}{1- \varepsilon_{n} } \Big)^{2}
			\big\| \nabla \big( \langle u^{h} \rangle_{n,\ast} - u_{n,\ast}^{j,h} \big) \big\|^{2} \Big)
			\Big\| \widehat{k}_{n}^{-1} u_{n,\alpha}^{j,h} \Big\|^{2} 
			+ \frac{\widehat{k}_{n}}{4 \nu} \Big\| \widehat{k}_{n}^{-1} u_{n,\alpha}^{j,h} \Big\|^{2}   \notag \\
			\leq& \frac{C \widehat{k}_{n} }{\nu}  \| f^{j}( t_{n,\beta}) \|^{2} 
			+ \frac{C(\Omega,\theta)}{\nu} \Big( \frac{1}{h \nu} \nu \widehat{k}_{n} \| \nabla e_{n,\beta}^{j} \|^{2} 
			+ \widehat{k}_{n} \| \nabla u_{n,\beta}^{j} \|_{2}^{2} \Big) \big( \| \nabla u_{n}^{j,h} \|^{2} 
			+ \| \nabla u_{n-1}^{j,h} \|^{2}  \big). \notag 
		\end{align}
		We sum \eqref{eq:DLN-Stab-H1-Eq2} over $n$ from $1$ to $M-1$, use CFL-like conditions 
		in \eqref{eq:CFL-like-cond}, time-diameter conditions in  \eqref{eq:time-ratio-cond} and $G$-norm 
		in \eqref{def:G-norm} 
		\begin{align}
			&\| \nabla u_{M}^{j,h} \|^{2} 
			+ \sum_{n=1}^{M-1} \frac{ \theta (1 - {\theta}^{2})}{4 ( 1 + \varepsilon_{n} \theta )^{2}} \big\| \nabla U_{n}^{j,h} \big\|^{2}
			+ \frac{4}{\nu(1 + \theta)} \sum_{n=1}^{M-1} \widehat{k}_{n} \Big\| \widehat{k}_{n}^{-1} u_{n,\alpha}^{j,h} \Big\|^{2} 
			\label{eq:DLN-Stab-H1-eq3}   \\
			\leq& \frac{C(\Omega,\theta)}{\nu} \sum_{n=1}^{M-1}\big( \widehat{k}_{n} \| u_{n,\beta}^{j} \|_{2}^{2} 
			+ \frac{1}{h \nu} \nu \widehat{k}_{n} \| e_{n,\beta}^{j} \|_{1}^{2} \big)
			\big( \| \nabla u_{n}^{j,h} \|^{2} + \| \nabla u_{n-1}^{j,h} \|^{2} \big) 
			\notag \\
			& + \sum_{n=1}^{M-1} \frac{C (k_{n} + k_{n-1}) }{\nu} \| f^{j}( t_{n,\beta}) \|^{2}
			+ C(\theta) \big( \| \nabla u_{1}^{j,h} \|^{2} + \| \nabla u_{0}^{j,h} \|^{2} \big). \notag 
		\end{align}
		By \eqref{eq:error-L2-inf-final-L2} and \eqref{eq:error-L2-inf-final-H1} in the proof of 
		Theorem \ref{thm:Error-L2} to obtain
		\begin{align}
		\sum_{n=1}^{M-1} \nu \widehat{k}_{n} \| e_{n,\beta}^{j} \|_{1}^{2}
		=& \nu \sum_{n=1}^{M-1} \widehat{k}_{n} \| e_{n,\beta}^{j} \|^{2}
		+ \nu \sum_{n=1}^{M-1} \widehat{k}_{n} \| \nabla e_{n,\beta}^{j} \|^{2} 
		\label{eq:error-DLN-L2H1} \\
		\leq& C \nu T \max_{0 \leq n \leq M} \| e_{n}^{j} \|^{2} 
		+ \nu \sum_{n=1}^{M-1} \widehat{k}_{n} \| \nabla e_{n,\beta}^{j} \|^{2} \notag \\
		\leq& \big( C \nu T + 1 \big) F_{2}^{2}, \notag 
		\end{align}
		where
		\begin{align*}
		F_{2} 
		=& \exp \Big[ \frac{C(\Omega,\theta)}{\nu} \big( k_{\rm{max}}^{4} \| u_{tt} \|_{2,2}^{2} 
		+ \| |u| \|_{2,2,\beta}^{2} \big) \Big] 
		\sqrt{F_{1}} + C(\Omega)h^{r} \| |u| \|_{\infty,r} 
		\notag \\
		& + C(\Omega,\theta) \sqrt{\nu} h^{r} \Big( k_{\rm{max}}^{2} \| u_{tt} \|_{2,r+1}
		+ \| |u| \|_{2,r+1,\beta} \Big),
		\end{align*}
		and $F_{1}$ is in \eqref{eq:def-F1}.
		We apply \eqref{eq:error-DLN-L2H1} and discrete Gronwall inequality without restrictions (\cite[p.369]{HR90_SIAM_NA}) to \eqref{eq:DLN-Stab-H1-eq3} 
		\begin{confidential}
			\color{darkblue}
			\begin{align*}
			&\| \nabla u_{M}^{j,h} \|^{2}   
			+ \sum_{n=1}^{M-1} \frac{ \theta (1 - {\theta}^{2})}{4 ( 1 + \varepsilon_{n} \theta )^{2}} \big\| \nabla U_{n}^{j,h} \big\|^{2}
			+ \frac{4}{\nu (1 + \theta)} \sum_{n=1}^{N-1} \widehat{k}_{n} 
			\Big\| \widehat{k}_{n}^{-1} u_{n,\alpha}^{j,h} \Big\|^{2} \\
			\leq & \frac{C(\Omega,\theta)}{\nu} 
			\Big[ \big( \widehat{k}_{M-1} \| u_{M-1,\beta}^{j} \|_{2}^{2} 
			+ \frac{1}{h \nu} \nu \widehat{k}_{M-1} \| e_{M-1,\beta}^{j} \|_{1}^{2}  \big) 
			\| \nabla u_{M-1}^{j,h} \|^{2} \\
			+& \sum_{n=2}^{M-2}  \big(\widehat{k}_{n+1} \| u_{n+1,\beta}^{j} \|_{2}^{2} 
			+ \widehat{k}_{n} \| u_{n,\beta}^{j} \|_{2}^{2} 
			+ \frac{1}{h \nu} \nu \widehat{k}_{n+1} \| e_{n+1,\beta}^{j} \|_{1}^{2} 
			+ \frac{1}{h \nu} \nu \widehat{k}_{n} \| e_{n,\beta}^{j} \|_{1}^{2} \big) 
			\| \nabla u_{n}^{j,h} \|^{2} \\
			+& \big( \widehat{k}_{1} \| u_{1,\beta}^{j} \|_{2}^{2} 
			+ \frac{1}{h \nu} \nu \widehat{k}_{1} \| e_{1,\beta}^{j} \|_{1}^{2}  \big) 
			\| \nabla u_{0}^{j,h} \|^{2} \Big] \\
			+& \frac{C}{\nu} \| |f| \|_{2,0,\beta}^{2} + C(\theta) \big( \| \nabla u_{1}^{j,h} \|^{2} + \| \nabla u_{0}^{j,h} \|^{2} \big).
			\end{align*}
			\normalcolor
		\end{confidential}
		\begin{align}
			&\| \nabla u_{M}^{j,h} \|^{2}  
			+ \sum_{n=1}^{M-1} \frac{ \theta (1 - {\theta}^{2})}{4 ( 1 + \varepsilon_{n} \theta )^{2}} \big\| \nabla U_{n}^{j,h} \big\|^{2}
			+ \frac{4}{\nu(1+\theta)} \sum_{n=1}^{M-1} \widehat{k}_{n} 
			\Big\| \widehat{k}_{n}^{-1} u_{n,\alpha}^{j,h} \Big\|^{2}        
			\label{eq:DLN-Stab-H1-eq4} \\
			\leq&\!\!  \exp \!\! \Big[ \frac{C(\Omega,\theta)}{\nu} \!\! 
			\Big( k_{\rm{max}}^{4} \| u_{tt}^{j} \|_{2,2}^{2} 
			\!+\! \| |u^{j}|\|_{2,2,\beta}^{2} \!+\! \frac{ (C \nu T \!+\! 1)F_{2}^{2} }{h \nu} \Big)\!\! \Big] \!\!
			\Big[ \frac{C\| |f| \|_{2,0,\beta}^{2}}{\nu} \notag \\
			&\qquad \qquad \qquad \qquad \qquad \qquad \qquad \qquad \qquad \qquad \qquad \quad \ \ +\! C(\theta) \big( \| \nabla u_{1}^{j,h} \|^{2} \!+\! \| \nabla u_{0}^{j,h} \|^{2} \big) \!\!\Big].
			\notag 
		\end{align}
		By time ratio conditions in \eqref{eq:time-ratio-cond}, $h^{-1} F_{2}^{2}$ is bounded. 
		Thus we have \eqref{eq:DLN-Stab-H1-conclusion} from \eqref{eq:DLN-Stab-H1-eq4}.
	\end{proof}

	\begin{theorem}
		\label{thm:Error-H1}
		We assume that for the $j$-th NSE in \eqref{eq:jth-NSE}, the velocity $u^{j}(x,t)$ satisfies
		\begin{gather*}
		u^{j} \!\in\! \ell^{\infty}(\{ t_{n} \}_{n=0}^{N};H^{r}(\Omega) \cap H^{2}(\Omega) ) 
		\cap \ell^{\infty,\beta}(\{ t_{n} \}_{n=0}^{N};H^{2}(\Omega))
		\cap \ell^{2,\beta}(\{ t_{n} \}_{n=0}^{N};H^{r+1}(\Omega)), \\
		u_{t}^{j} \in L^{2}(0,T;H^{r+1}(\Omega)),  \ \ 
		u_{tt}^{j} \in L^{2}(0,T;H^{r+1}(\Omega)), \ \ 
		u_{ttt}^{j} \in L^{2}(0,T;X^{-1} \cap L^{2}(\Omega)), 
		\end{gather*} 
		and the pressure $p^{j}(x,t)$ satisfies
		\begin{gather*}
			p^{j} \in \ell^{\infty}(\{ t_{n} \}_{n=0}^{N};H^{s+1}(\Omega)) \cap \ell^{2,\beta}(\{ t_{n} \}_{n=0}^{N};H^{s+1}(\Omega)), \\
				p_{t}^{j} \in L^{2}(0,T;H^{s+1}(\Omega)), \ p_{tt}^{j} \in L^{2}(0,T;H^{s+1}(\Omega)),
		\end{gather*}
		for all $j$.  
		Under the CFL-like conditions in \eqref{eq:CFL-like-cond}, time ratio bounds in \eqref{eq:time-ratio-cond} and the time-diameter condition in \eqref{eq:time-diameter-relation},
		the numerical solutions of the DLN-Ensemble algorithms in \eqref{eq:DLN-Ensemble-Alg} for all $\theta \in (0,1)$ satisfy 
		\begin{align}
		&\max_{0 \leq n \leq N}\| u_{n}^{j} - u_{n}^{j,h} \|_{1}  
		+ \sum_{n=1}^{N-1} \frac{\widehat{k}_{n}}{\nu} \Big\| \frac{ u_{n,\alpha}^{j} - u_{n,\alpha}^{j,h} }{ \widehat{k}_{n} } \Big\|^{2}
		\leq \mathcal{O} \big( h^{r},h^{s+1}, k_{\rm{max}}^{2}, h^{1/2}k_{\rm{max}}^{3/2} \big). \label{eq:error-H1-conclusion} 
		\end{align}
	\end{theorem}
	\begin{proof}
		The proof is relatively long, thus we leave it to Appendix \ref{appendixB-H1}.
	\end{proof}
	\ \\

	\subsection{Numerical Analysis for Pressure} \ 
	\label{subsec:Stab-Error-Pressure}
	\begin{theorem}
		If the source function $f^{j}$  in the $j$-th NSE satisfy 
		\begin{gather*}
			f^{j} \in \ell^{2,\beta}\big( \{ t_{n}\}_{n=1}^{N}; L^{2}(\Omega) \cap X^{-1} \big),
		\end{gather*}
		the velocity $u^{j}(x,t)$ and pressure $p^{j}(x,t)$ in the $j$-th NSE satisfy the regularities 
		in \eqref{eq:L2-regularity} for all $j$, the time step $k_{\rm{max}}$ and diameter $h$ satisfy the restriction in \eqref{eq:time-diameter-relation},
		then for all $\theta \in (0,1)$, under CFL-like conditions in \eqref{eq:CFL-like-cond} and time ratio bounds in \eqref{eq:time-ratio-cond}, 
		the family of variable time-stepping DLN-Ensemble Algorithms in \eqref{eq:DLN-Ensemble-Alg} satisfy: 
		for  $M = 2, 3, \cdots, N$, 
		\begin{align}
			&C_{\tt{is}}^{2} \sum_{n=1}^{M-1} \widehat{k}_{n} \| p_{n,\beta}^{j,h} \|^{2} 
			\label{eq:DLN-Pressure-Stab-L2-conclusion} \\
			\leq& C(\Omega, \theta, \nu) \Big( \frac{1}{\nu} \| |f^{j}| \|_{2,-1,\beta}^{2} 
			\!+\! \| u_{1}^{j,h} \|^{2} \!+\! \| u_{0}^{j,h} \|^{2} \Big) \notag \\
			+& C(\Omega, \theta, \nu,\frac{1}{\nu},T)   
			\Big( \frac{1}{\nu} \| |f| \|_{2,0,\beta}^{2} \!+\! \| \nabla u_{1}^{j,h} \|^{2} \!+\! \| \nabla u_{0}^{j,h} \|^{2} \Big) \notag \\
			+& C(\Omega,\theta,\frac{1}{\nu},T) \Big( \frac{1}{\nu} \| |f^{j}| \|_{2,-1,\beta}^{2} 
			\!+\! \| u_{1}^{j,h} \|^{2} \!+\! \| u_{0}^{j,h} \|^{2} \Big)
			\Big( \frac{1}{\nu} \| |f| \|_{2,0,\beta}^{2} \!+\! \| \nabla u_{1}^{j,h} \|^{2} \!+\! \| \nabla u_{0}^{j,h} \|^{2} \Big). \notag
		\end{align}
	\end{theorem}
	\begin{proof}
		By the first equation of DLN-Ensemble Algorithms in \eqref{eq:DLN-Ensemble-Alg}, 
		\begin{align}
			\big( p_{n,\beta}^{j,h}, \nabla \cdot v^{h} \big) 
			=&  \Big( \frac{u_{n,\alpha}^{j,h}}{\widehat{k}_{n}}, v^{h}  \Big)  + {\nu} \big( \nabla u_{n,\beta}^{j,h} , \nabla v^{h}  \big)
			- \big( f^{j}(t_{n,\beta}), v^{h} \big)  
			\label{eq:DLN-Pressure-Stab-Eq1}   \\
			&+ b \big( \langle u^{h} \rangle_{n,\ast} , u_{n,\beta}^{j,h} , v^{h} \big) 
			+ b \big( u_{n,\ast}^{j,h} - \langle u^{h} \rangle_{n,\ast}, u_{n,\ast}^{j,h}, v^{h} \big), 
			\quad \forall v^{h} \in X^{h}. \notag 
		\end{align}
		By Cauchy-Schwarz inequality, Poincar\'e inequality, Young's inequality and definition of dual norm in \eqref{eq:dual-norm}
		\begin{gather}
			\Big( \frac{u_{n,\alpha}^{j,h}}{\widehat{k}_{n}}, v^{h}  \Big)
			\leq C(\Omega) \Big\| \widehat{k}_{n}^{-1} u_{n,\alpha}^{j,h} \Big\| \| \nabla v^{h} \|,  \qquad
			\nu \big( \nabla u_{n,\beta}^{j,h} , \nabla v^{h} \big)
			\leq \nu \big\| \nabla u_{n,\beta}^{j,h} \big\| \| \nabla v^{h} \|, \notag \\
			\big( f^{j}(t_{n,\beta}), v^{h} \big) 
			\leq \| f^{j}(t_{n,\beta}) \|_{-1} \| \nabla v^{h} \|. 
			\label{eq:DLN-Pressure-Stab-Eq1-term1-2-3} 
		\end{gather}
		For non-linear terms, we have
		\begin{align*}
		&b \big( u_{n,\ast}^{j,h} - \langle u^{h} \rangle_{n,\ast}, u_{n,\ast}^{j,h}, v^{h} \big) 
		+ b \big( \langle u^{h} \rangle_{n,\ast}, u_{n,\beta}^{j,h} , v^{h} \big) \\
		=& b \big( u_{n,\ast}^{j,h} - \langle u^{h} \rangle_{n,\ast}, u_{n,\ast}^{j,h}, v^{h} \big) 
		+ b \big( \langle u^{h} \rangle_{n,\ast} - u_{n,\ast}^{j,h}, u_{n,\beta}^{j,h} , v^{h} \big)  
		+ b \big( u_{n,\ast}^{j,h}, u_{n,\beta}^{j,h} , v^{h} \big) \\
		=& b \big( u_{n,\ast}^{j,h} - \langle u^{h} \rangle_{n,\ast}, u_{n,\ast}^{j,h} - u_{n,\beta}^{j,h}, v^{h} \big) 
		+ b \big( u_{n,\ast}^{j,h}, u_{n,\beta}^{j,h} , v^{h} \big),
		\end{align*}
		By skew-symmetric property of $b$, \eqref{eq:b-bound-2}, \eqref{eq:b-bound-3} in Lemma \ref{lemma:b-bound}, inverse inequality in \eqref{eq:inv-inequal}, the fact in \eqref{eq:DLN-2nd-approx}, CFL-like conditions in \eqref{eq:CFL-like-cond} and bound of $\beta_{\ell}^{(n)}$ in \eqref{eq:bound-beta}
		\begin{align}
			b \big( u_{n,\ast}^{j,h} - \langle u^{h} \rangle_{n,\ast}, u_{n,\ast}^{j,h} - u_{n,\beta}^{j,h}, v^{h} \big) 
			=& b \big( u_{n,\ast}^{j,h} - \langle u^{h} \rangle_{n,\ast}, v^{h}, u_{n,\beta}^{j,h} - u_{n,\ast}^{j,h} \big) 
			\label{eq:DLN-Pressure-Stab-Eq1-term4} \\
			\leq& C(\Omega) h^{-1/2} \big\| \nabla ( u_{n,\ast}^{j,h} - \langle u^{h} \rangle_{n,\ast} ) \big\|
			\frac{2 \beta_{2}^{(n)}}{1 - \varepsilon_{n}} \| U_{n}^{j,h} \| \| \nabla v^{h} \| \notag \\
			\leq& C(\Omega,\theta) h^{-1/2} \sqrt{\frac{h \nu}{\widehat{k}_{n} } } \frac{1 - \varepsilon_{n}}{1 + \varepsilon_{n} \theta} \frac{1}{ 1 - \varepsilon_{n} } \| U_{n}^{j,h} \| \| \nabla v^{h} \| \notag \\
			=& C(\Omega,\theta) \sqrt{\frac{\nu}{\widehat{k}_{n} } } \frac{1}{ 1 + \varepsilon_{n} \theta } 
			\| U_{n}^{j,h} \| \| \nabla v^{h} \|. \notag \\
			b \big( u_{n,\ast}^{j,h}, u_{n,\beta}^{j,h}, v^{h} \big) 
			\leq& C(\Omega) \| \nabla u_{n,\ast}^{j,h} \| \| \nabla u_{n,\beta}^{j,h} \| \| \nabla v^{h} \| 
			\label{eq:DLN-Pressure-Stab-Eq1-term5} \\
			\leq& C(\Omega,\theta) \big( \| \nabla u_{n}^{j,h} \| + \| \nabla u_{n-1}^{j,h}  \| \big)  \| \nabla u_{n,\beta}^{j,h} \| 
			\| \nabla v^{h} \|. \notag 
		\end{align}
		We combine \eqref{eq:DLN-Pressure-Stab-Eq1}, \eqref{eq:DLN-Pressure-Stab-Eq1-term1-2-3}, \eqref{eq:DLN-Pressure-Stab-Eq1-term4}, \eqref{eq:DLN-Pressure-Stab-Eq1-term5} and obtain: 
		$\forall v^{h} \in X_{r}^{h}$
		\begin{align}
			\frac{ \big( p_{n,\beta}^{j,h}, \nabla \cdot v^{h} \big) }{ \| \nabla v^{h} \| }
			\leq& C(\Omega) \Big\| \widehat{k}_{n}^{-1} u_{n,\alpha}^{j,h} \Big\| + \nu \big\| \nabla u_{n,\beta}^{j,h} \big\|
			+ \| f^{j}(t_{n,\beta}) \|_{-1} 
			\label{eq:DLN-Pressure-Stab-Eq2}  \\
			&+ C(\Omega,\theta) \sqrt{\frac{\nu}{\widehat{k}_{n} } } \frac{1}{ 1 + \varepsilon_{n} \theta } 
			\| U_{n}^{j,h} \| + C(\Omega,\theta) \big( \| \nabla u_{n}^{j,h} \| + \| \nabla u_{n-1}^{j,h}  \| \big)  \| \nabla u_{n,\beta}^{j,h} \|. \notag 
		\end{align}
		By \eqref{eq:inf-sup-cond}, we have: for $M=2,\cdots, N$
		\begin{align}
			C_{\tt{is}}^{2} \sum_{n=1}^{M-1} \widehat{k}_{n} \| p_{n,\beta}^{j,h} \|^{2}
			\leq& \frac{C(\Omega,\theta) \nu}{\nu (1+\theta)}\sum_{n=1}^{M-1} \widehat{k}_{n} \Big\| \widehat{k}_{n}^{-1} u_{n,\alpha}^{j,h} \Big\|^{2}
			+ \nu \sum_{n=1}^{M-1} \nu \widehat{k}_{n} \big\| \nabla u_{n,\beta}^{j,h} \big\|^{2} 
			\label{eq:DLN-Pressure-Stab-Eq3} \\
			+&  \sum_{n=1}^{M-1} (k_{n} + k_{n-1}) \| f^{j}(t_{n,\beta}) \|_{-1}^{2} 
			+ C(\Omega,\theta) \nu  \sum_{n=1}^{M-1} \frac{\theta(1 - \theta^{2})}{ (1 + \varepsilon_{n} \theta )^{2} } \| U_{n}^{j,h} \|^{2} \notag \\
			+& \frac{C(\Omega,\theta)}{\nu} \big( \max_{0 \leq n \leq M} \| \nabla u_{n}^{j,h} \|^{2} \big) 
			\sum_{n=1}^{M-1} \nu \widehat{k}_{n} \| \nabla u_{n,\beta}^{j,h} \|^{2}
		\end{align}
		By \eqref{eq:L2-Stab-conclusion} and \eqref{eq:DLN-Stab-H1-conclusion}, we have \eqref{eq:DLN-Pressure-Stab-L2-conclusion}.

	\end{proof}

	\begin{theorem}
		\label{thm:Error-Pressure}
		We assume that for the $j$-th NSE in \eqref{eq:jth-NSE}, the velocity $u^{j}(x,t)$ satisfies
		\begin{gather*}
		u \in \ell^{\infty}(\{ t_{n} \}_{n=0}^{N};H^{r+1}(\Omega)) \cap \ell^{\infty,\beta}(\{ t_{n} \}_{n=0}^{N};H^{1}(\Omega))
		\cap \ell^{2,\beta}(\{ t_{n} \}_{n=0}^{N};H^{r+1}(\Omega)), \\
		u_{t} \in L^{2}(0,T;H^{r+1}(\Omega)),  \ \ 
		u_{tt} \in L^{2}(0,T;H^{r+1}(\Omega)), \ \ 
		u_{ttt} \in L^{2}(0,T;X^{-1} \cap L^{2}(\Omega)), 
		\end{gather*} 
		and the pressure $p^{j}(x,t)$ satisfies
		\begin{gather*}
		p \in \ell^{\infty}(\{ t_{n} \}_{n=0}^{N};H^{s+1}) \cap \ell^{2,\beta}(\{ t_{n} \}_{n=0}^{N};H^{s+1}), \\
		p_{t} \in L^{2}(0,T;H^{s+1}(\Omega)), \ \ p_{tt} \in L^{2}(0,T;H^{s+1}(\Omega)),
		\end{gather*}
		for all $j$.  
		Under the CFL-like conditions in \eqref{eq:CFL-like-cond}, time ratio conditions in \eqref{eq:time-ratio-cond} and the time-diameter condition in \eqref{eq:time-diameter-relation},
		the numerical solutions of the DLN-Ensemble algorithms in \eqref{eq:DLN-Ensemble-Alg} for all $\theta \in (0,1)$ satisfy
		\begin{align}
		\Big( \sum_{n=1}^{N-1} \widehat{k}_{n} \| p_{n,\beta}^{j} - p_{n,\beta}^{j,h} \|^{2} \Big)^{1/2}
		\leq \mathcal{O} \big( h^{r},h^{s+1}, k_{\rm{max}}^{2}, h^{1/2}k_{\rm{max}}^{3/2}  \big). 
		\label{eq:pressure-L2-conclusion} 
		\end{align}
	\end{theorem}
	\begin{proof}
		The proof is relatively long, thus we leave it to Appendix \ref{appendixB-Pressure}.
	\end{proof}
	\ \\

	\section{Adaptive Implementation}
	\label{sec:Adapt-Imple}
	The time adaptive mechanism for the DLN-Ensemble algorithm depends on two essential parts
	\begin{itemize}
		\item the estimator of LTE through a certain explicit scheme,
		\item the time step controller for next step computing. 
	\end{itemize}
	We consider the fully explicit, variable time-stepping AB2-like scheme\footnote{The derivation of scheme is similar to that of two-step Adam-Bashforth (AB2) scheme, hence we call it AB2-like scheme.} to estimate LTE of the DLN-Ensemble algorithm. The AB2-like scheme for the initial value problem in \eqref{eq:IVP} is
	\begin{align}
		y_{n\!+\!1}^{\tt{AB2-like}} \!&=\! y_{n} \!+\! \frac{t_{n\!+\!1} \!-\!t_{n}}{2 (t_{n\!-\!1,\beta} \!-\!t_{n\!-\!2,\beta})} 
		\Big[ (t_{n\!+\!1} \!+\!t_{n} \!-\!2 t_{n\!-\!2,\beta} )  g^{\tt{DLN}} (t_{n\!-\!1,\beta}, y_{n\!-\!1,\beta})   
		\label{eq:AB2-like} \\
		& \qquad \qquad \qquad \qquad \qquad \qquad 
		\!-\! (t_{n\!+\!1} \!+\!t_{n} \!-\! 2 t_{n\!-\!1,\beta} ) g^{\tt{DLN}} (t_{n\!-\!2,\beta}, y_{n\!-\!2,\beta})  \Big], \notag 
	\end{align}
	where $g^{\tt{DLN}}(t_{n\!-\!1,\beta}, y_{n\!-\!1,\beta}) $ and $g^{\tt{DLN}} (t_{n\!-\!2,\beta}, y_{n\!-\!2,\beta})$ are calculated by the DLN scheme in \eqref{eq:1leg-DLN}
	\begin{align*}
		g^{\tt{DLN}}(t_{n-1,\beta}, y_{n-1,\beta}) 
		&= \frac{y_{n-1,\alpha}}{\widehat{k}_{n-1}}, \quad
		g^{\tt{DLN}} (t_{n-2,\beta}, y_{n-2,\beta}) =
		\frac{y_{n-2,\alpha}}{\widehat{k}_{n-2}}.
	\end{align*}
	From \eqref{eq:AB2-like}, we see that $y_{n\!+\!1}^{\tt{AB2-like}}$	is certain interpolant of the previous four DLN solutions $\{ y_{n}, y_{n-1}, y_{n-2}, y_{n-3} \}$. Hence, the AB2-like scheme in \eqref{eq:AB2-like} is fully explicit.
	We refer to \cite{LPT23_ACSE} for the derivation of the explicit AB2-like scheme in \eqref{eq:AB2-like} and the following estimator of LTE $\widehat{T}_{n+1}$ by AB2-like scheme
	\begin{align*}
		\widehat{T}_{n+1}
		=&\frac{-G^{(n)}}{G^{(n)}+\mathcal{R}^{(n)}} \big( y_{n+1}^{\tt{DLN}} - y_{n+1}^{\tt{AB2-like}} \big),
		\notag \\
		G^{(n)} \!=\!& \Big( \frac{1}{2} \!-\! \frac{ \alpha_{0}}{2 \alpha_{2}} \frac{1}{\tau_{n}} \Big)
		\Big( \beta_{2}^{(n)} \!-\! \beta_{0}^{(n)} \frac{1}{\tau_{n}}  \Big)^{2} 
		\!+\! \frac{ \alpha_{0}}{6 \alpha_{2}} \Big( \frac{1}{\tau_{n}} \Big)^{3} \!-\! \frac{1}{6}, 
		\notag \\
		\mathcal{R}^{(n)} \!=\! & \frac{1}{12} \Big[ 2 \!+\! \frac{3}{\tau_n} \Big( 1 \!-\! \beta_2^{(n\!-\!2)} \frac{1}{\tau_{n\!-\!1}} \!+\!\beta_0^{(n\!-\! 2)} \frac{1}{\tau_{n\!-\!2}} \frac{1}{\tau_{n\!-\!1}} \Big)
		\Big( 1\!-\!\beta_2^{(n\!-\!1)} \frac{1}{\tau_n}\!+\!\beta_0^{(n\!-\!1)} \frac{1}{\tau_{n\!-\!1}} \frac{1}{\tau_n} \Big) \notag \\
		&+\! \frac{3}{\tau_n} \Big( 1\!+\!\frac{1}{\tau_n} \!-\!\beta_2^{(n\!-\!2)} \frac{1}{\tau_{n\!-\!1}} \frac{1}{\tau_n}\!+\!\beta_0^{(n\!-\!2)} \frac{1}{\tau_{n\!-\!2}} \frac{1}{\tau_{n\!-\!1}} \frac{1}{\tau_n} \Big)
		\Big( -\beta_2^{(n\!-\!1)} \!+\! \beta_0^{(n\!-\!1)} \frac{1}{\tau_{n\!-\!1}} \Big) \Big], \notag 
	\end{align*}
	and $\tau_n = k_{n}/k_{n-1}$ is ratio of time step.
	For the DLN-Ensemble algorithm, the AB2-like solutions for $j$-th system at time $t_{n+1}$ is
	\begin{align}
		&u_{n\!+\!1,\tt{AB2-like}}^{j,h} 
		\label{eq:AB2-like-Ensemble} \\
		=& u_{n}^{j,h} \!+\! \frac{t_{n\!+\!1} \!-\!t_{n}}{2 (t_{n\!-\!1,\beta} \!-\!t_{n\!-\!2,\beta})} 
		\Big[ (t_{n\!+\!1} \!+\!t_{n} \!-\!2 t_{n\!-\!2,\beta} )  
		\frac{u_{n-1,\alpha}^{j,h}}{\widehat{k}_{n-1}}   
		\!-\! (t_{n\!+\!1} \!+\!t_{n} \!-\! 2 t_{n\!-\!1,\beta} ) 
		\frac{u_{n-2,\alpha}^{j,h}}{\widehat{k}_{n-2}}  \Big], \notag 
	\end{align}
	and we use the following estimator for LTE of the DLN-Ensemble algorithm 
	\begin{gather}
		\widehat{T}_{n+1}
		= \max_{1 \leq j \leq J}\bigg\{\frac{|G^{(n)}|}{|G^{(n)}+\mathcal{R}^{(n)}|} 
		\frac{\big\| u_{n+1,\tt{DLN}}^{j,h} - u_{n+1,\tt{AB2-like}}^{j,h} \big\|}{\big\| u_{n+1,\tt{DLN}}^{j,h} \big\|} \bigg\} .
		\label{eq:Estimator-LTE-Ensemble}
	\end{gather}
	We adopt the time step controller proposed in \cite{HNW93_Springer}
	\begin{gather}
    	k_{n+1} = k_{n} \cdot \min \Big\{ 1.5, \max \Big\{ 0.2, \kappa \big( {\tt{Tol}}/\| \widehat{T}_{n+1} \|  \big)^{\frac{1}{3}} \Big\} \Big\},
    	\label{eq:improve-controller}
    \end{gather}
    where $\kappa \in (0, 1]$ is the safety factor and ${\tt{Tol}}$ is the required tolerance for the LTE.
	From \eqref{eq:improve-controller}, the next time step $k_{n+1}$ is adjusted smaller if the estimator for LTE $\widehat{T}_{n+1}$ is large with respect to ${\tt{Tol}}$. Meanwhile, $k_{n+1}$ is no larger than $1.5 k_{n}$ for robust computing and no smaller than $0.2 k_{n}$ for efficiency. 
	The time adaptivity mechanism for the DLN-Ensemble algorithm is summarized in Algorithm \ref{alg:Adap-DLN}.
	\LinesNumberedHidden
    \begin{algorithm}[ptbh]
    	\caption{Adaptivity of  DLN (estimator of LTE by AB2-like scheme)}
    	\label{alg:Adap-DLN}
    	\KwIn{tolerance ${\tt{Tol}}$, previous DLN solutions $u_{n}^{j,h},u_{n-1}^{j,h},u_{n-2}^{j,h},u_{n-3}^{j,h}$ for all $j$,
    		current time step $k_{n}$, three previous time step $k_{n-1},k_{n-2},k_{n-3}$, safety factor $\kappa$, \;}
    	compute the DLN solution $u_{n+1}^{h,\tt DLN}$ and $p_{n+1}^{h,\tt DLN}$ 
    	by \eqref{eq:DLN-Ensemble-Alg} or the refactorization process \;
    	compute the AB2-like solution $u_{n+1}^{h,\tt{AB2-like}}$ by \eqref{eq:AB2-like-Ensemble} \;
    	use $k_{n},k_{n-1},k_{n-2},k_{n-3}$ to update $\tau_{n},\tau_{n-1},\tau_{n-2}$ \;
    	compute $\widehat{T}_{n+1}$ by \eqref{eq:Estimator-LTE-Ensemble}      \tcp*{relative estimator}
    	\uIf{$ \widehat{T}_{n+1}  < \rm{Tol}$}
    	{
    		$u_{n+1}^{h} \Leftarrow u_{n+1}^{h,\tt DLN}$  \tcp*{accept the result}
    		$k_{n\!+\!1} \!\Leftarrow \! k_{n} \cdot \min \big\{ \!1.5, \max \big\{\!0.2, \kappa \big(\!\frac {\text{Tol}}{ \widehat{T}_{n+1}  }\!\big)^{1/3} \big\} \!\big\}$  
    		\tcp*{adjust step by \eqref{eq:improve-controller}}
    	}\Else
    	{
    		\tt{// adjust current step to recompute}  
    		$k_{n} \!\Leftarrow \! k_{n} \cdot \min \big\{ 1.5, \max \big\{0.2, \kappa \big(\frac {\text{Tol}}{ \widehat{T}_{n+1}  }\big)^{1/3} \big\} \big\}$ \;
    	}
    \end{algorithm}

	\section{Numerical Tests}
	\label{sec:Num-Tests}
	In this section, we test the DLN-Ensemble Algorithm in \eqref{eq:DLN-Ensemble-Alg} with $\theta = 2/3$ and $\theta = 2/\sqrt{5}$. 
	$\theta = 2/3$ is mentioned \cite{DLN83_SIAM_JNA} to balance the magnitude of LTE and fine stability properties.
	$\theta = 2/\sqrt{5}$ is suggested in \cite{KS05} to have the best stability at infinity.
	We use Taylor-Hood ($\mathbb{P}2-\mathbb{P}1$) finite element space for spatial discretization and the software MATLAB for programming.

	\subsection{Convergence Test}
	\label{subsec:Conv-test}

	To verify second-order convergence of DLN-ensemble algorithm in \eqref{eq:DLN-Ensemble-Alg}, we refer to the 2D Taylor-Green vortex problem  \cite{GQ98_IJNMF} with the following revised analytical solutions of the NSE on the domain $\Omega = [0,1]^{2}$ over time interval $[0,1]$. 
	\begin{gather*}
		u_{1}(x,y,t) = - \cos(\pi x) \sin( \pi y) \sin( \omega t), \qquad u_{2}(x,y,t) = \sin(\pi x) \cos( \pi y) \sin(\omega t), \\ 
		p(x,y,t) = - \frac{1}{4} \big[ \cos(2 \pi x) + \cos(2 \pi y) \big] \sin^{2}(\omega t).
	\end{gather*}
	We use the constant time-stepping DLN-ensemble algorithm in \eqref{eq:DLN-Ensemble-Alg} with the refactorization process on BE-solver to solve ten systems of NSE ($J=10$) simultaneously. The exact solutions for velocity are  
	\begin{gather*}
		u_{1}^{j}(x,y,t) = \big( 1 + \delta_{j} \big) u_{1}(x,y,t), \quad u_{2}^{j}(x,y,t) = \big( 1 + \delta_{j} \big) u_{2}(x,y,t),
		\quad j = 1, \cdots, 10,
	\end{gather*}
	where the random array $\{ \delta_{j} \}_{j=1}^{10}$ are uniformly distributed between $-1.\rm{e}-2$ and $1.\rm{e}-2$ in ascending order.
	The pressure function $p^{j}(x,y,t)$ is the same as $p(x,y,t)$ for all $j$. 
	We set $\nu = 5.\rm{e}-3$ ($\rm{Re} = 200$), $\omega = 10$ and constant time step $k = \frac{1}{2}h$.
	We measure the following error for each $j$-th systems
	\begin{gather*}
		\| u^{j} - u^{j,h} \|_{\infty,0} = \max_{0 \leq n \leq N} \| u_{n}^{j} - u_{n}^{j,h} \|, \qquad 
		\| u^{j} - u^{j,h} \|_{\infty,1} = \max_{0 \leq n \leq N} \| u_{n}^{j} - u_{n}^{j,h} \|_{1}, \\
		\| p^{j} - p^{j,h} \|_{2,0} = \Big( \sum_{n=0}^{N} k \| p_{n}^{j} - p_{n}^{j,h} \|^{2} \Big)^{1/2}, \qquad
		j = 1,2,\cdots, J,
	\end{gather*}
	and average of the above errors
	\begin{gather}
		\mathbb{E}\big[\| u - u^{h} \|_{\infty,0}\big]
		\!=\! \frac{1}{J} \Big( \sum_{j=1}^{J} \| u^{j} - u^{j,h} \|_{\infty,0} \Big), \  
		\mathbb{E}\big[\| u - u^{h} \|_{\infty,1}\big]
		\!=\! \frac{1}{J} \Big( \sum_{j=1}^{J} \| u^{j} - u^{j,h} \|_{\infty,1} \Big), \notag \\
		\mathbb{E}\big[\| p - p^{h} \|_{2,0}\big]
		\!=\! \frac{1}{J} \Big( \sum_{j=1}^{J} \| p^{j} - p^{j,h} \|_{2,0} \Big), \qquad
		j = 1,2,\cdots, J.
		\label{eq:Error-Avg}
	\end{gather}
	The convergence rate $R$ is given by 
	\begin{gather*}
		R = \log_{2} \Big( \frac{\rm{error}(k,h)}{\rm{error}(k/2,h/2)} \Big).
	\end{gather*}
	We provide errors of the first system ($j=1$), the last system ($j=10$), average errors and corresponding convergence rate in \Cref{table:Const-DLN23-Ensemble-Re200-001,table:Const-DLN25-Ensemble-Re200-001}. 
	We observe that all the errors are second-order convergent, consistent with theories in Section \ref{sec:Num-Analysis}.
	\begin{table}[ptbh]
		\centering
		\renewcommand\arraystretch{1.25}
		\caption{Error and convergence rate for constant time-stepping DLN-Ensemble algorithm with $\theta = 2/3$, $\rm{Re}=200$} 
		\begin{tabular}{cccccccc}
			\hline
			\hline
			$h\!$  & $\!\| u^{1} \!-\! u^{1,h} \|_{\infty,0}\!$ & $\!R\!$
			& $\!\| u^{1} \!-\! u^{1,h} \|_{\infty,1}\!$ & $\!R\!$ 
			& $\!\| p^{1} \!-\! p^{1,h} \|_{2,0}\!$ & $\!R\!$
			\\
			\hline 
			$\frac{1}{8}$         & 1.4321e-2   & -        & 4.6046e-1   & -        & 7.3065e-2    & -
			\\
			$\frac{1}{16}$        & 3.3161e-3   & 2.1106   & 1.0836e-1   & 2.0878   & 1.4523e-2    & 2.3309
			\\
			$\frac{1}{32}$        & 7.9477e-4   & 2.0609   & 2.6300e-2   & 2.0421   & 3.4100e-3    & 2.0905
			\\
			$\frac{1}{64}$        & 1.9550e-4   & 2.0234   & 6.4898e-3   & 2.0188   & 8.3585e-4    & 2.0285 
			\\
			$\frac{1}{128}$       & 4.8595e-5   & 2.0083   & 1.6276e-3   & 1.9954   & 2.0754e-4    & 2.0098 
			\\
			\hline
			\hline 
			$h\!$   & $\!\| u^{10} \!-\! u^{10,h} \|_{\infty,0}\!$ & $\!R\!$
			& $\!\| u^{10} \!-\! u^{10,h} \|_{\infty,1}\!$ & $\!R\!$ 
			& $\!\| p^{10} \!-\! p^{10,h} \|_{2,0}\!$ & $\!R\!$
			\\
			\hline 
			$\frac{1}{8}$          & 1.4616e-2    & -        & 4.6995e-1   & -        & 7.3733e-2    & -
			\\
			$\frac{1}{16}$         & 3.3856e-3    & 2.1101   & 1.1027e-1   & 2.0915   & 1.4653e-2    & 2.3311
			\\
			$\frac{1}{32}$         & 8.1171e-4    & 2.0604   & 2.6754e-2   & 2.0432   & 3.4409e-3    & 2.0903
			\\
			$\frac{1}{64}$         & 1.9981e-4    & 2.0223   & 6.6069e-3   & 2.0177   & 8.4349e-4    & 2.0283
			\\
			$\frac{1}{128}$        & 4.9665e-5    & 2.0084   & 1.6582e-3   & 1.9943   & 2.0945e-4    & 2.0098
			\\
			\hline
			\hline 
			$h\!$   & $\!\mathbb{E}\!\big[\| u \!-\! u^{h} \|_{\infty,0}\big]\!$ & $\!R\!$
			& $\!\mathbb{E}\!\big[\| u \!-\! u^{h} \|_{\infty,1}\big]\!$ & $\!R\!$ 
			& $\!\mathbb{E}\!\big[\| p \!-\! p^{h} \|_{2,0}\big]\!$ & $\!R\!$
			\\
			\hline 
			$\frac{1}{8}$          & 1.4446e-2    & -        & 4.6447e-1   & -        & 7.3348e-2    & -
			\\
			$\frac{1}{16}$         & 3.3456e-3    & 2.1104   & 1.0914e-1   & 2.0894   & 1.4578e-2    & 2.3310
			\\
			$\frac{1}{32}$         & 8.0196e-4    & 2.0607   & 2.6492e-2   & 2.0426   & 3.4231e-3    & 2.0904
			\\
			$\frac{1}{64}$         & 1.9733e-4    & 2.0229   & 6.5394e-3   & 2.0183   & 8.3908e-4    & 2.0284
			\\
			$\frac{1}{128}$        & 4.9049e-5    & 2.0083   & 1.6406e-3   & 1.9950   & 2.0835e-4    & 2.0098
			\\
			\hline
			\hline 
		\end{tabular}
		\label{table:Const-DLN23-Ensemble-Re200-001}
	\end{table}
	
	\begin{table}[ptbh]
		\centering
		\renewcommand\arraystretch{1.25}
		\caption{Error and convergence rate for constant time-stepping DLN-Ensemble algorithm 
		with $\theta = 2/\sqrt{5}$, $\rm{Re}=200$} 
		\begin{tabular}{cccccccc}
			\hline
			\hline
			$h\!$  & $\!\| u^{1} \!-\! u^{1,h} \|_{\infty,0}\!$ & $\!R\!$
			& $\!\| u^{1} \!-\! u^{1,h} \|_{\infty,1}\!$ & $\!R\!$ 
			& $\!\| p^{1} \!-\! p^{1,h} \|_{2,0}\!$ & $\!R\!$
			\\
			\hline 
			$\frac{1}{8}$         & 9.2835e-3   & -        & 3.8115e-1   & -        & 4.3617e-2    & -
			\\
			$\frac{1}{16}$        & 2.0535e-3   & 2.1766   & 8.3308e-2   & 2.1938   & 9.8705e-3    & 2.1437
			\\
			$\frac{1}{32}$        & 4.8267e-4   & 2.0890   & 2.0362e-2   & 2.0326   & 2.3891e-3    & 2.0467
			\\
			$\frac{1}{64}$        & 1.1783e-4   & 2.0344   & 4.9924e-3   & 2.0281   & 5.9022e-4    & 2.0171 
			\\
			$\frac{1}{128}$       & 2.9213e-5   & 2.0120   & 1.2570e-0   & 1.9897   & 1.4688e-4    & 2.0066 
			\\
			\hline
			\hline 
			$h\!$   & $\!\| u^{10} \!-\! u^{10,h} \|_{\infty,0}\!$ & $\!R\!$
			& $\!\| u^{10} \!-\! u^{10,h} \|_{\infty,1}\!$ & $\!R\!$ 
			& $\!\| p^{10} \!-\! p^{10,h} \|_{2,0}\!$ & $\!R\!$
			\\
			\hline 
			$\frac{1}{8}$          & 9.4639e-3    & -        & 3.8753e-1   & -        & 4.4468e-2    & -
			\\
			$\frac{1}{16}$         & 2.0950e-3    & 2.1755   & 8.4202e-2   & 2.2024   & 1.0055e-2    & 2.1449
			\\
			$\frac{1}{32}$         & 4.9279e-4    & 2.0879   & 2.0565e-2   & 2.0337   & 2.4341e-3    & 2.0464
			\\
			$\frac{1}{64}$         & 1.2041e-4    & 2.0330   & 5.0451e-3   & 2.0273   & 6.0146e-4    & 2.0169
			\\
			$\frac{1}{128}$        & 2.9856e-5    & 2.0119   & 1.2709e-3   & 1.9890   & 1.4970e-4    & 2.0064
			\\
			\hline
			\hline 
			$h\!$   & $\!\mathbb{E}\!\big[\| u \!-\! u^{h} \|_{\infty,0}\big]\!$ & $\!R\!$
			& $\!\mathbb{E}\!\big[\| u \!-\! u^{h} \|_{\infty,1}\big]\!$ & $\!R\!$ 
			& $\!\mathbb{E}\!\big[\| p \!-\! p^{h} \|_{2,0}\big]\!$ & $\!R\!$
			\\
			\hline 
			$\frac{1}{8}$          & 9.3600e-3    & -      & 3.8384e-1   & -      & 4.3977e-2    & -
			\\
			$\frac{1}{16}$         & 2.0711e-3    & 2.1761   & 8.3686e-2   & 2.1975   & 9.9484e-3    & 2.1442
			\\
			$\frac{1}{32}$         & 4.8696e-4    & 2.0885   & 2.0448e-2   & 2.0330   & 2.4081e-3    & 2.0466
			\\
			$\frac{1}{64}$         & 1.1892e-4    & 2.0338   & 5.0147e-3   & 2.0277   & 5.9496e-4    & 2.0170
			\\
			$\frac{1}{128}$        & 2.9486e-5    & 2.0119   & 1.2629e-3   & 1.9894   & 1.4807e-4    & 2.0065
			\\
			\hline
			\hline 
		\end{tabular}
		\label{table:Const-DLN25-Ensemble-Re200-001}
	\end{table}
	Then we increase the difficulty by setting $\nu = 1.\rm{e}-3$ ($\rm{Re} = 1000$) and the random array $\{ \delta_{j} \}_{j=1}^{10}$ uniformly distributed between $-0.1$ and $0.1$ in ascending order. From \eqref{eq:CFL-like-cond}, the CFL-like conditions are more likely to be violated with a larger deviation of velocity (arising from the larger magnitude of $\delta_{j}$) and smaller viscosity $\nu$.  
	From \Cref{table:Const-DLN23-Ensemble-Re1000-01}, we observe that the constant DLN-Ensemble algorithm with $\theta = 2/3$ still has robust simulation under such challenging conditions. 
	We skip the results of the case $\theta = 2/\sqrt{5}$ since its performance is relatively poor in terms of the convergence rate: the errors are large under $h = 1/8, 1/16$ and decrease rapidly as $h \leq 1/32$. 

	\begin{table}[ptbh]
		\centering
		\renewcommand\arraystretch{1.25}
		\caption{Error and convergence rate for constant time-stepping DLN-Ensemble algorithm with $\theta = 2/3$, $\rm{Re}=1000$} 
		\begin{tabular}{cccccccc}
			\hline
			\hline
			$h\!$  & $\!\| u^{1} \!-\! u^{1,h} \|_{\infty,0}\!$ & $\!R\!$
			& $\!\| u^{1} \!-\! u^{1,h} \|_{\infty,1}\!$ & $\!R\!$ 
			& $\!\| p^{1} \!-\! p^{1,h} \|_{2,0}\!$ & $\!R\!$
			\\
			\hline 
			$\frac{1}{8}$          & 1.8574e-2   & -            & 1.0816         & -            & 7.0583e-2    & -
			\\
			$\frac{1}{16}$        & 3.8152e-3   & 2.2835   & 3.1321e-1   & 1.7880   & 1.3998e-2    & 2.3341
			\\
			$\frac{1}{32}$        & 8.5635e-4   & 2.1555   & 8.0781e-2   & 1.9550   & 3.2828e-3    & 2.0922
			\\
			$\frac{1}{64}$        & 2.0441e-4   & 2.0667   & 2.0712e-2   & 1.9635   & 8.0435e-4    & 2.0291 
			\\
			$\frac{1}{128}$      & 5.0309e-5   & 2.0226   & 5.4966e-3   & 1.9139   & 1.9969e-4    & 2.0101 
			\\
			\hline
			\hline 
			$h\!$   & $\!\| u^{10} \!-\! u^{10,h} \|_{\infty,0}\!$ & $\!R\!$
			& $\!\| u^{10} \!-\! u^{10,h} \|_{\infty,1}\!$ & $\!R\!$ 
			& $\!\| p^{10} \!-\! p^{10,h} \|_{2,0}\!$ & $\!R\!$
			\\
			\hline 
			$\frac{1}{8}$           & 2.1742e-2    & -            & 1.3737         & -            & 7.7259e-2    & -
			\\
			$\frac{1}{16}$         & 4.5712e-3    & 2.2499   & 3.9628e-1   & 1.7935   & 1.5289e-2    & 2.3373
			\\
			$\frac{1}{32}$         & 1.0482e-3    & 2.1246   & 9.1134e-2   & 2.1204   & 3.5878e-3    & 2.0913
			\\
			$\frac{1}{64}$         & 2.5525e-4    & 2.0380   & 2.2929e-2   & 1.9908   & 8.7976e-4    & 2.0279
			\\
			$\frac{1}{128}$       & 6.3720e-5    & 2.0021   & 6.0159e-3   & 1.9303   & 2.1850e-4    & 2.0095
			\\
			\hline
			\hline 
			$h\!$   & $\!\mathbb{E}\!\big[\| u \!-\! u^{h} \|_{\infty,0}\big]\!$ & $\!R\!$
			& $\!\mathbb{E}\!\big[\| u \!-\! u^{h} \|_{\infty,1}\big]\!$ & $\!R\!$ 
			& $\!\mathbb{E}\!\big[\| p \!-\! p^{h} \|_{2,0}\big]\!$ & $\!R\!$
			\\
			\hline 
			$\frac{1}{8}$           & 2.0069e-2    & -            & 1.2074         & -            & 7.3653e-2    & -
			\\
			$\frac{1}{16}$         & 4.1617e-3    & 2.2698   & 3.4243e-1   & 1.8180   & 1.4593e-2    & 2.3355
			\\
			$\frac{1}{32}$         & 9.4837e-4    & 2.1336   & 8.5446e-2   & 2.0027   & 3.4239e-3    & 2.0916
			\\
			$\frac{1}{64}$         & 2.2822e-4    & 2.0550   & 2.1726e-2   & 1.9756   & 8.3922e-4    & 2.0285
			\\
			$\frac{1}{128}$       & 5.6693e-5    & 2.0092   & 5.7345e-3   & 1.9217   & 2.0838e-4    & 2.0098
			\\
			\hline
			\hline 
		\end{tabular}
		\label{table:Const-DLN23-Ensemble-Re1000-01}
	\end{table}

	\subsection{Efficiency Test}
	We use the same test problem in Subsection \ref{subsec:Conv-test} to verify the time efficiency of the DLN-Ensemble algorithm in \eqref{eq:DLN-Ensemble-Alg}. We apply the constant time-stepping DLN-Ensemble algorithm with refactorization process on the BE-Ensemble algorithm to $J$ systems of NSE with $J = 1, 10, 100$. We set $\nu = 1.\rm{e}-3$ ($\rm{Re} = 1000$), $\omega = 10$, $h = \frac{1}{64}$ and constant time step $k = \frac{h}{2}$. 
	The random array $\{ \delta_{j} \}_{j=1}^{J}$ are uniformly distributed between $-1.\rm{e}-1$ and $1.\rm{e}-1$ in ascending order. But for different value $J$ and $\theta$, the random array $\{ \delta_{j} \}_{j=1}^{J}$ is re-computed before the simulation. We measure the average errors in \eqref{eq:Error-Avg} and the following maximum errors with respect to $j$
	\begin{gather*}
		\max_{j} \{ \| u^{j} - u^{j,h} \|_{\infty,0} \}, \quad
		\max_{j} \{ \| u^{j} \!-\! u^{j,h} \|_{\infty,1} \}, \quad
		\max_{j} \{ \| p^{j} \!-\! p^{j,h} \|_{2,0} \}.
	\end{gather*}

	From \Cref{table:Const-DLN23-Ensemble-Efficiency,table:Const-DLN23-Ensemble-Efficiency-max}, we observe that 
	the accuracy of Algorithm \eqref{eq:DLN-Ensemble-Alg} is ensured even if the number of NSE systems being solved substantially increases. 
	Since the $J$ linear systems in the DLN-Ensemble algorithm \eqref{eq:DLN-Ensemble-Alg} share the same coefficient matrix, the rising time due to a larger number $J$ comes from linear system solving. Here, we use the direct method \cite{QSS07_Springer} to solve the linear system at each time step. If the coefficient matrix has size $N_d \times N_d$, the complexity for LU factorization is about $\frac{2}{3} N_d^{3}$ FLOPS, and the complexity for forward and backward substitutions are both $J \times N_d^{2}$ FLOPS. 
	\Cref{fig:DLN-Efficiency} show that the CPU time for the simulation has an almost linear relation with the number of system $J$, which confirms our analysis.
	\begin{table}[ptbh]
		\centering
		\renewcommand\arraystretch{1.25}
		\caption{Average error of constant time-stepping DLN-Ensemble algorithm with $\theta = 2/3, 2/\sqrt{5}$, $\rm{Re}=1000$} 
		\begin{tabular}{ccccc}
			\hline
			\hline
			$\theta = \frac{2}{3}$ & $\!\mathbb{E}\!\big[\| u \!-\! u^{h} \|_{\infty,0}\big]\!$ 
			& $\!\mathbb{E}\!\big[\| u \!-\! u^{h} \|_{\infty,1}\big]\!$ 
			& $\!\mathbb{E}\!\big[\| p \!-\! p^{h} \|_{2,0}\big]\!$ & CPU Time(s)
			\\
			\hline 
			$J = 1$      & 2.5216e-4   & 2.2785e-2   & 8.5573e-4   & 3430.34          
			\\
			$J = 10$     & 2.3719e-4   & 2.2116e-2   & 8.4545e-4   & 6596.29   
			\\
			$J = 100$    & 2.3028e-4   & 2.1815e-2   & 8.4072e-4   & 38756.76   
			\\
			\hline
			\hline
			$\theta = \frac{2}{\sqrt{5}}$ & $\!\mathbb{E}\!\big[\| u \!-\! u^{h} \|_{\infty,0}\big]\!$ 
			& $\!\mathbb{E}\!\big[\| u \!-\! u^{h} \|_{\infty,1}\big]\!$ 
			& $\!\mathbb{E}\!\big[\| p \!-\! p^{h} \|_{2,0}\big]\!$ & CPU Time(s)
			\\
			\hline 
			$J = 1$      & 1.3737e-4   & 1.8546e-2   & 5.9141e-4   & 3460.38          
			\\
			$J = 10$     & 1.4144e-4   & 1.8657e-2   & 6.0108e-4   & 6815.91   
			\\
			$J = 100$    & 1.4138e-4   & 1.8653e-2   & 6.0077e-4   & 40340.54   
			\\
			\hline
			\hline
		\end{tabular}
		\label{table:Const-DLN23-Ensemble-Efficiency}
	\end{table}

	\begin{table}[ptbh]
		\centering
		\renewcommand\arraystretch{1.25}
		\caption{Maximum error of constant time-stepping DLN-Ensemble algorithm with $\theta = 2/3, 2/\sqrt{5}$, $\rm{Re}=1000$} 
		\begin{tabular}{cccc}
			\hline
			\hline
			$\theta = \frac{2}{3}$ & $\max \big\{ \| u^{j} - u^{j,h} \|_{\infty,0}\big\}$ 
			& $\max \big\{\| u^{j} - u^{j,h} \|_{\infty,1}\big\}$ 
			& $\max \big\{\| p^{j} - p^{j,h} \|_{2,0}\big\}$ 
			\\
			\hline 
			$J = 1$      & 2.5216e-4   & 2.2785e-2   & 8.5573e-4             
			\\
			$J = 10$     & 2.5760e-4   & 2.3039e-2   & 8.7556e-4      
			\\
			$J = 100$    & 2.5888e-4   & 2.3099e-2   & 8.8368e-4    
			\\
			\hline
			\hline
			$\theta = \frac{2}{\sqrt{5}}$ & $\max \big\{ \| u^{j} - u^{j,h} \|_{\infty,0}\big\}$ 
			& $\max \big\{\| u^{j} - u^{j,h} \|_{\infty,1}\big\}$ 
			& $\max \big\{\| p^{j} - p^{j,h} \|_{2,0}\big\}$ 
			\\
			\hline 
			$J = 1$      & 1.3737e-4   & 1.8546e-2   & 5.9141e-4           
			\\
			$J = 10$     & 1.5484e-4   & 1.8982e-2   & 6.5313e-4    
			\\
			$J = 100$    & 1.5705e-4   & 1.9043e-2   & 6.6205e-4   
			\\
			\hline
			\hline
		\end{tabular}
		\label{table:Const-DLN23-Ensemble-Efficiency-max}
	\end{table}

	\begin{figure}[ptbh]
		\subfigure[CPU time for $\theta = \frac{2}{3}$]{ \label{fig:DLN23_Efficiency}
			\hspace{-0.2cm}
			\begin{minipage}[t]{0.45\linewidth}
				\centering
				\includegraphics[width=2.8in,height=1.8in]{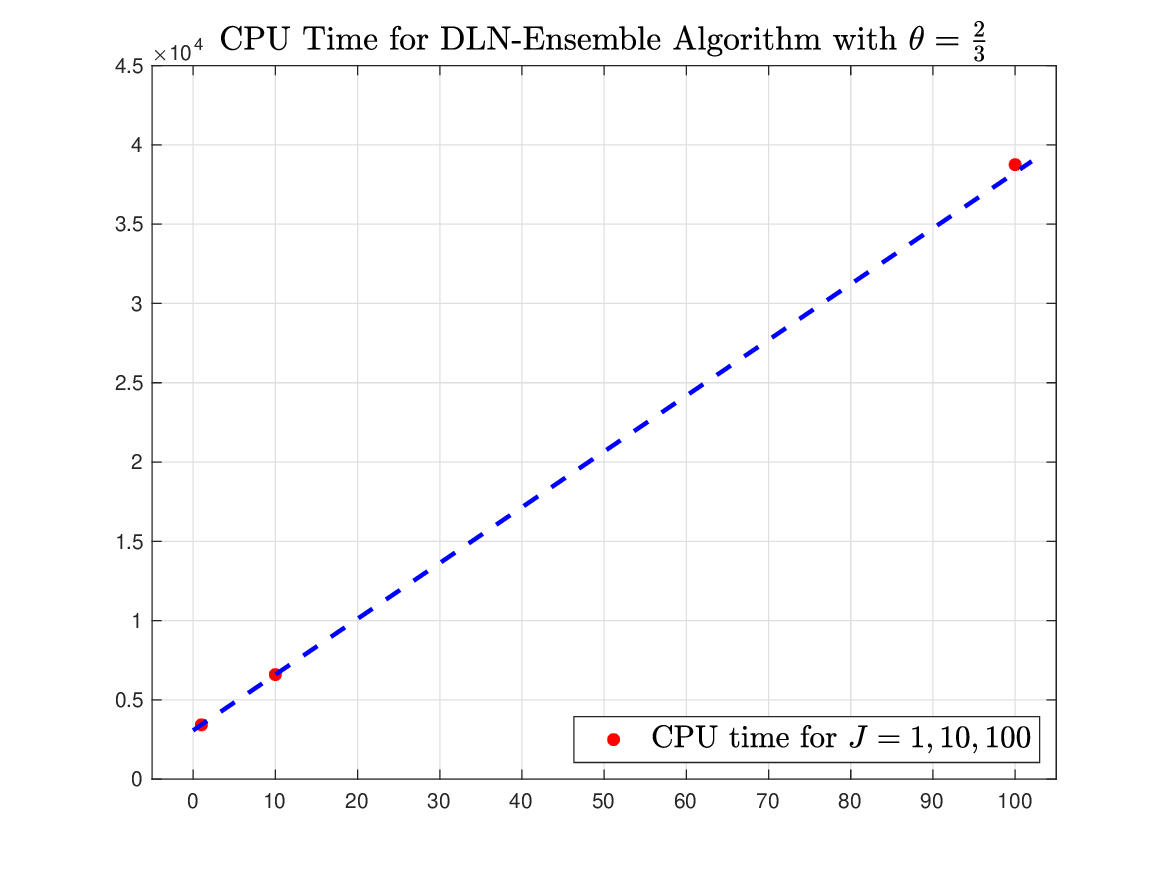}\\
				\vspace{0.02cm}
			\end{minipage}
			\quad}%
		\subfigure[CPU time for $\theta = \frac{2}{\sqrt{5}}$]{ \label{fig:DLN25_Efficiency}
			\hspace{-0.2cm}
			\begin{minipage}[t]{0.45\linewidth}
				\centering
				\includegraphics[width=2.8in,height=1.8in]{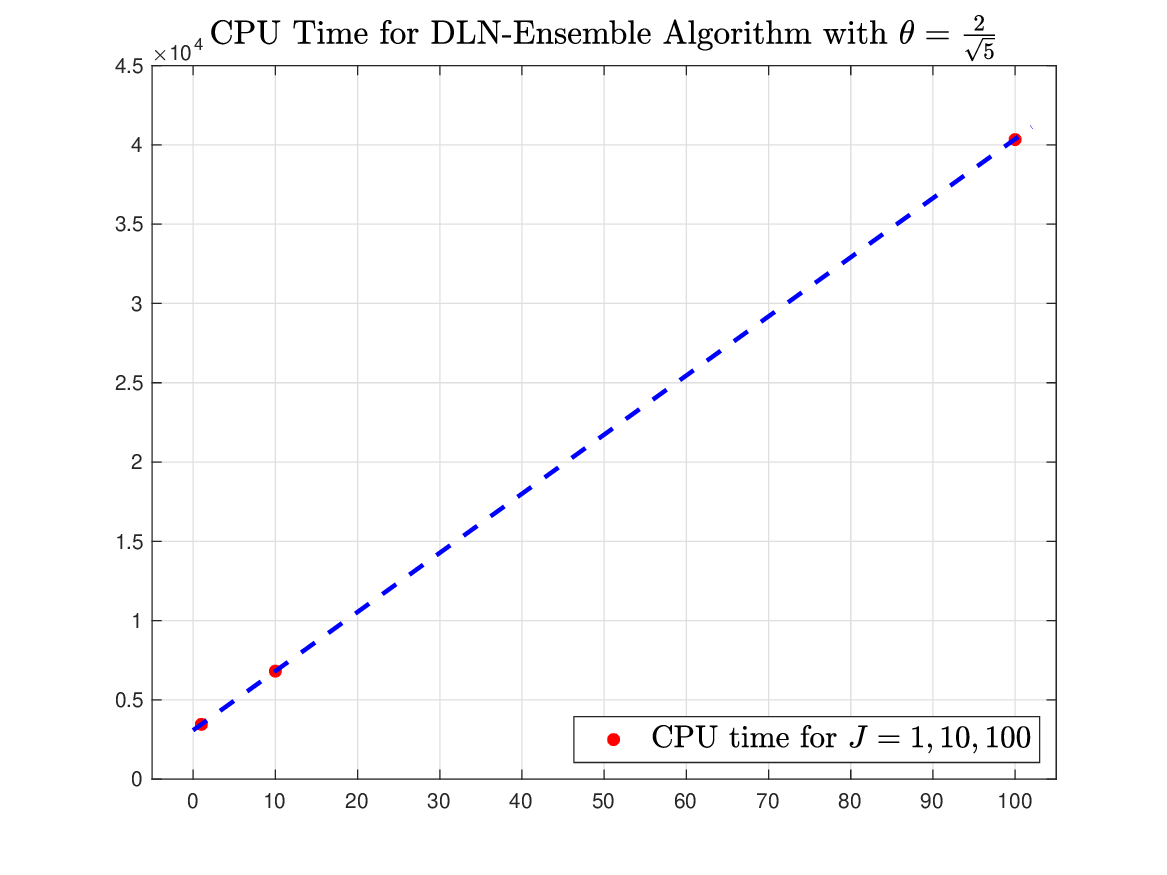}\\
				\vspace{0.02cm}
		\end{minipage}}  	
		\par
		\centering
		\vspace{-0.2cm}
		\caption{CPU time for constant DLN-Ensemble algorithm with $\theta = \frac{2}{3},\frac{2}{\sqrt{5}}$
		}
		\label{fig:DLN-Efficiency}
	\end{figure}

	\subsection{Time Adaptive Test}
	In this subsection, we will show that the variable time-stepping DLN-Ensemble 
	scheme \eqref{eq:DLN-Ensemble-Alg} with time adaptive algorithm \ref{alg:Adap-DLN} outperforms the DLN-Ensemble scheme with uniform time grids in extremely stiff problems.
	To construct a stiff problem concerning time, 
	we change the time component function in the revised Taylor-Green problem on domain $\Omega = [0,1] \times [0,1]$ to be 
	\begin{align*}
		&\begin{cases}
			G_{1}(t) = e^{g_{1}(t)} \big[ \cos(g_{2}(t)) + \sin(g_{2}(t)) \big] \\
			G_{2}(t) = e^{g_{1}(t)} \big[ \cos(g_{2}(t)) - \sin(g_{2}(t)) \big]
		\end{cases}, \\
		&\begin{cases}
			g_{1}(t) = 10^{\omega} ( t + 2 e^{-t} -2) \\
			g_{2}(t) = 10^{\omega} ( 1 - e^{-t} - te^{-1})
		\end{cases},  \qquad \qquad  t_{0} \leq t \leq T.
	\end{align*}
	$G_{1}(t)$, $G_{2}(t)$ are the first and second components of the solution vector to an extremely stiff ordinary differential system proposed by Lindberg \cite{Lin74_BIT}.  
	Hence exact solutions are based on the following functions 
	\begin{gather}
		u_{1}(x,y,t) = - \cos(\pi x) \sin( \pi y) G_{i}(t), \qquad 
		u_{2}(x,y,t) = \sin(\pi x) \cos( \pi y) G_{i}(t), \\ 
		p(x,y,t) = - \frac{1}{4} \big[ \cos(2 \pi x) + \cos(2 \pi y) \big] G_{i}(t), \qquad i = 1,2.
		\label{eq:refer-sol}
	\end{gather}
	As in Subsection \ref{subsec:Conv-test}, we use the DLN-Ensemble algorithms with $\theta = 2/3, 2/\sqrt{5}$ to solve ten systems of NSE ($J=10$). The true velocity functions are 
	\begin{gather*}
	u_{1}^{j}(x,y,t) = \big( 1 + \delta_{j} \big) u_{1}(x,y,t), \quad u_{2}^{j}(x,y,t) = \big( 1 + \delta_{j} \big) u_{2}(x,y,t),
	\quad j = 1, \cdots, 10,
	\end{gather*}
	where the random array $\{ \delta_{j} \}_{j=1}^{10}$ are uniformly distributed between $-0.1$ and $0.1$ in ascending order.
	Pressure function $p^{j}(x,y,t)$ is equal to $p(x,y,t)$ for all $j$. The source function $f^{j}(x,y,t)$, initial and boundary conditions are decided by exact solutions. 
	We define the numerical kinetic energy $\mathcal{E}_{n+1}^{j,h}$, numerical kinetic energy dissipation rate 
	$\mathcal{E}_{n+1,\tt{VD}}^{j,h}$ and numerical dissipation $\mathcal{E}_{n+1,\tt{ND}}^{j,h}$ of $j$-th NSE at time $t_{n+1}$
	\begin{gather*}
		\mathcal{E}_{n+1}^{j,h} = \frac{1}{2} \| u_{n+1}^{j,h} \|^{2}, \ \ 
		\mathcal{E}_{n+1,\tt{VD}}^{j,h} = \nu \| \nabla u_{n+1}^{j,h} \|^{2}, \ \ 
		\mathcal{E}_{n+1,\tt{ND}}^{j,h} = 
		\frac{1}{\widehat{k}_{n}} \Big\| \sum_{\ell=0}^{2} \gamma_{\ell}^{(n)} u_{n-1+\ell}^{j,h} \Big\|^{2},
	\end{gather*}
	and evaluate performance of the DLN-Ensemble algorithms via the following variables
	\begin{align*}
		&\text{Average and maximum of $\{ \mathcal{E}_{n+1}^{j,h} \}_{j=1}^{J}$:} \ \ 
		\begin{cases}
			\mathbb{E}[\mathcal{E}_{n+1}^{h}] = \displaystyle \frac{1}{J} \sum_{j=1}^{J} \mathcal{E}_{n+1}^{j,h} \\
			\mathcal{E}_{n+1}^{h,\infty} = \displaystyle \max_{1 \leq j \leq J} \{ \mathcal{E}_{n+1}^{j,h} \}
		\end{cases},  \\
		&\text{Average and maximum of $\{ \mathcal{E}_{n+1,\tt{VD}}^{j,h} \}_{j=1}^{J}$:} 
		\ \
		\begin{cases}
			\mathbb{E}[\mathcal{E}_{n+1,\tt{VD}}^{h}] = \displaystyle \frac{1}{J} \sum_{j=1}^{J} \mathcal{E}_{n+1,\tt{VD}}^{j,h}\\
			\mathcal{E}_{n+1,\tt{VD}}^{h,\infty} = \displaystyle \max_{1 \leq j \leq J} \{ \mathcal{E}_{n+1,\tt{VD}}^{j,h} \} 
		\end{cases}, \\
		&\text{Average and maximum of $\{ \mathcal{E}_{n+1,\tt{ND}}^{j,h} \}_{j=1}^{J}$:} \ \ 
		\begin{cases}
			\mathbb{E}[\mathcal{E}_{n+1,\tt{ND}}^{h}] = \displaystyle \frac{1}{J} \sum_{j=1}^{J} \mathcal{E}_{n+1,\tt{ND}}^{j,h} \\
			\mathcal{E}_{n+1,\tt{ND}}^{h,\infty} = \displaystyle \max_{1 \leq j \leq J} \{ \mathcal{E}_{n+1,\tt{ND}}^{j,h} \}
		\end{cases}.
	\end{align*}
	We define the exact average exaxt kinetic energy $\mathbb{E}[\mathcal{E}_{n+1}]$ and maximum kinetic energy 
	$\mathcal{E}_{n+1}^{\infty}$ at $t_{n+1}$
	\begin{gather*}
		\mathbb{E}[\mathcal{E}_{n+1}] = \frac{1}{J} \sum_{j=1}^{J} \big( \frac{1}{2} \| u_{n+1}^{j} \|^{2} \big),
		\qquad 
		\mathcal{E}_{n+1}^{\infty} = \max_{1 \leq j \leq J} \big\{ \frac{1}{2} \| u_{n+1}^{j} \|^{2} \big\},
	\end{gather*}
	and also evaluate the error of average kinetic energy: 
	$\log_{10} \big| \mathbb{E}[\mathcal{E}_{n+1}] - \mathbb{E}[\mathcal{E}_{n+1}^{h}] \big|$, 
	and the error of maximum kinetic energy: 
	$\log_{10} \big| \mathcal{E}_{n+1}^{\infty} - \mathcal{E}_{n+1}^{h,\infty} \big|$.

	We set $\rm{Re} = 1000$ for all ten systems of NSE and diameter $h = 0.02$ for mesh generation of $\Omega$. For the adaptive DLN-Ensemble algorithm in \eqref{eq:DLN-Ensemble-Alg}, we set required tolerance 
	${\tt{Tol}} = 1.\rm{e}-4$, the safety factor $\kappa = 0.95$, the maximum time step 
	$k_{\rm{max}} = 1.\rm{e}-4$ for stability and the minimum time step $k_{\rm{min}} = 1.\rm{e}-6$ for efficiency.
	Two initial time steps are the same as $k_{\rm{min}}$ and initial conditions at the first two steps are decided by exact solutions. 

	We first set time component function in \eqref{eq:refer-sol} to be $G_{1}(t)$ with $\omega = 3.1$ and simulate
	10 systems of NSE over time interval $[1.59 \ \  1.6032]$. 
	From \Cref{fig:Lindberg1st}, we see that $G_{1}(t)$ increases rapidly from 0 ($t = 1.596$) to 300 ($t = 1.6015$) and then declines sharply to $-200$ ($t = 1.6032$). 
	We expect that the adaptive algorithms outperform the constant time-stepping algorithms since the adaptive mechanism assigns more steps for drastic changes.
	Number of steps, number of rejections, and total cost in steps (number of steps plus number of rejections) of the adaptive DLN-Ensemble algorithms with $\theta = 2/3, 2/\sqrt{5}$ are offered 
	in \Cref{table:Num-Steps-1st}. 
	We also apply the corresponding constant time-stepping algorithms with the same total cost in steps. 

	From \Cref{fig:KE-avg-1st,fig:KE-max-1st}, we observe that all the algorithms have the true pattern of average kinetic energy and maximum kinetic energy. 
	However \Cref{fig:Error-avg-1st,fig:Error-max-1st} show that adaptive algorithms obtain relatively small errors at the end of simulation since more time steps are assigned in the simulations of extremely stiff part of true solutions ($t \geq 1.602$). 
	From \Cref{fig:ND-avg-1st,fig:ND-max-1st}, the numerical dissipation of adaptive algorithms is relatively large before $t = 1.602$ but grows slower at the end.  
	Meanwhile, all DLN-Ensemble algorithms have similar patterns of kinetic energy dissipation rates in \Cref{fig:VD-avg-1st,fig:VD-max-1st}.
	\Cref{fig:EST-LTE-1st,fig:Step-1st} also verify that the highly stiff part ($t \geq 1.6$) is hard to simulate since the estimator of LTE $\widehat{T}_{n+1}$ exceeds the required tolerance ${\tt{Tol}} = 1.\rm{e}-4$ frequently and time step size oscillates near the minimum step size $k_{\rm{min}} = 1.\rm{e}-6$ after $t = 1.602$.
	\begin{figure}[ptbh]
		\subfigure[$G_{1}(t)$ with $\omega = 3.1$]{ \label{fig:Lindberg1st}
			\hspace{-0.2cm}
			\begin{minipage}[t]{0.45\linewidth}
				\centering
				\includegraphics[width=2.8in,height=1.8in]{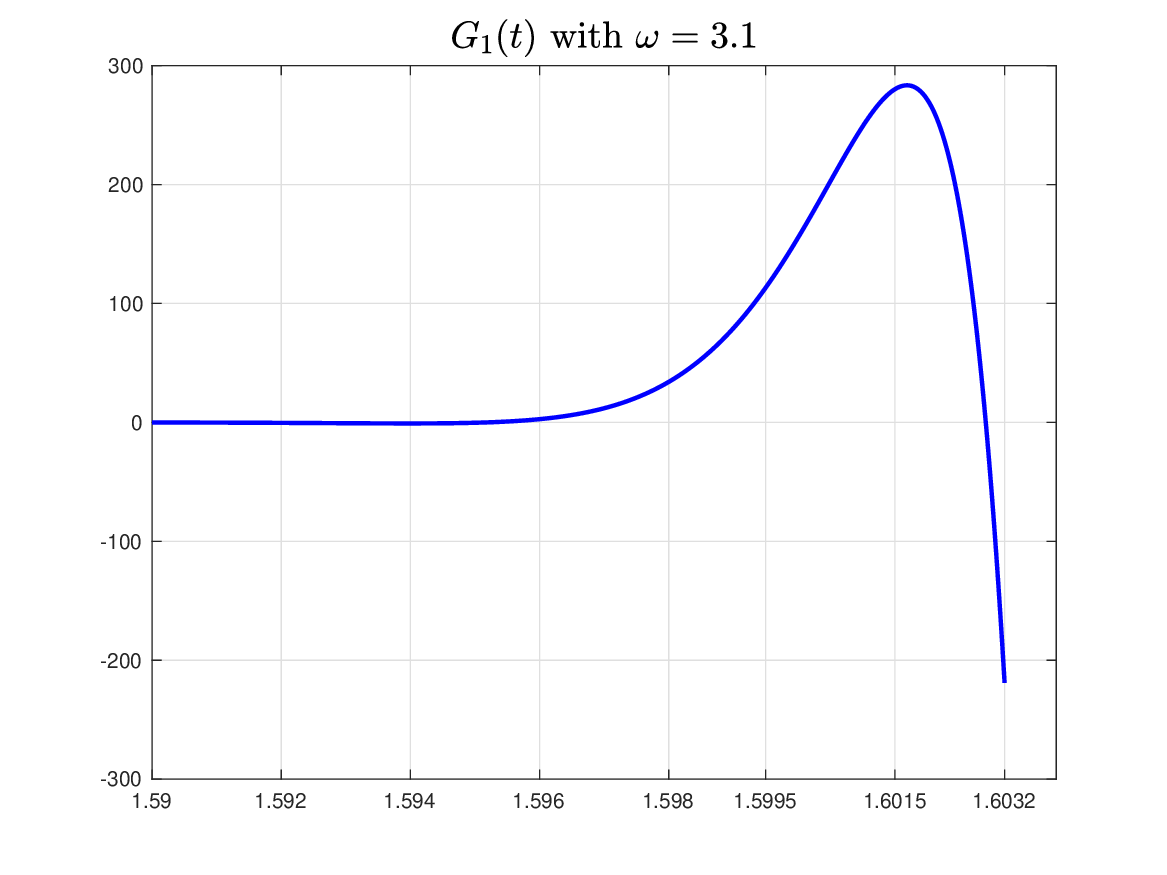}\\
				\vspace{0.02cm}
			\end{minipage}
			\quad}%
		\subfigure[$G_{2}(t)$ with $\omega = 3.1$]{ \label{fig:Lindberg2nd}
			\hspace{-0.2cm}
			\begin{minipage}[t]{0.45\linewidth}
				\centering
				\includegraphics[width=2.8in,height=1.8in]{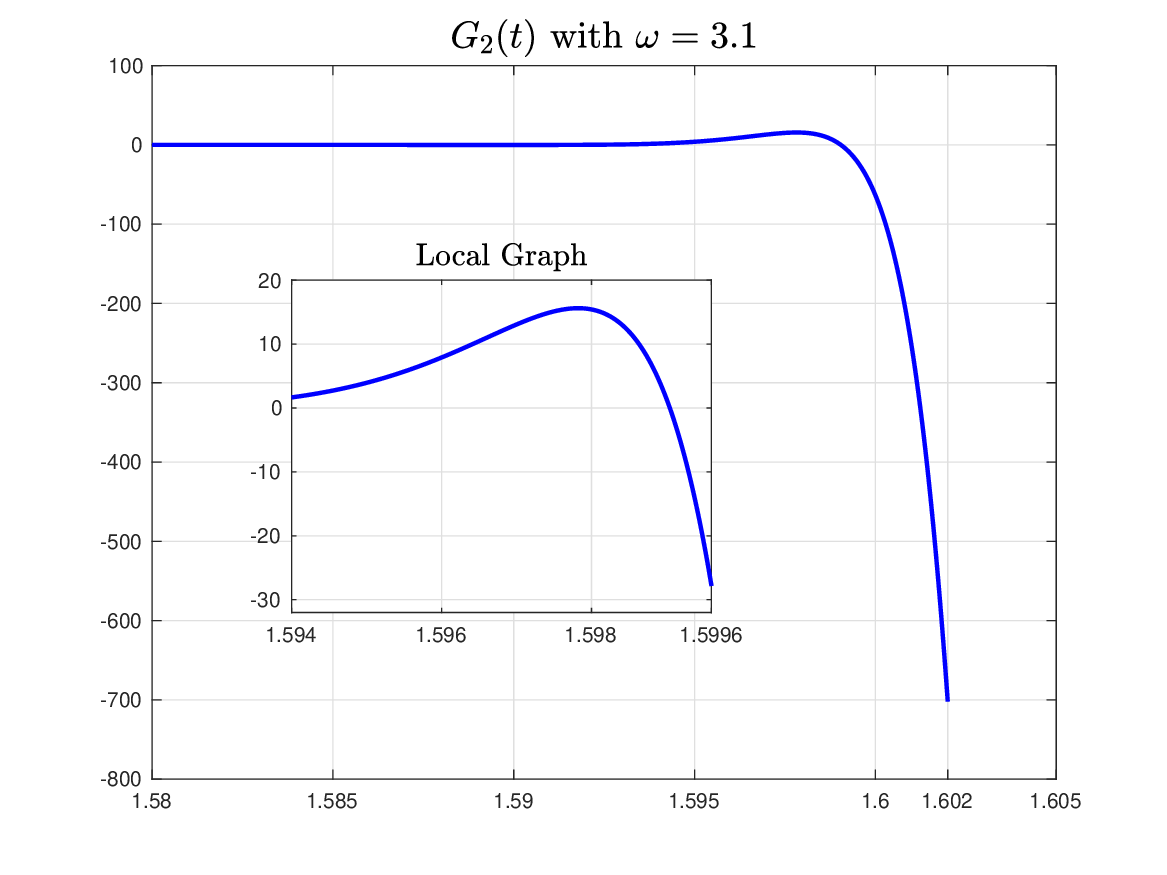}\\
				\vspace{0.02cm}
		\end{minipage}}  	
		\par
		\centering
		\vspace{-0.2cm}
		\caption{Time component functions proposed by Lindberg}
		\label{fig:Time-component}
	\end{figure}

	\begin{table}[ptbh]
		\centering
		\renewcommand\arraystretch{1.25}
		\caption{Number of steps for the DLN-Ensemble algorithms with $\theta = 2/3, 2/\sqrt{5}$ for the test with time component function $G_{1}(t)$.} 
		\begin{tabular}{lccc}
			\hline
			\hline
			DLN-Ensemble algorithms & \# Steps 
			& \# Rejections 
			& Total cost in steps
			\\
			\hline 
			Adaptive with $\theta = 2/3$            & 2495   & 1926   & 4421             
			\\
			Adaptive with $\theta = 2/\sqrt{5}$     & 2710   & 2000   & 4710      
			\\
			Constant with $\theta = 2/3$            & 4421   & -      & 4421    
			\\
			Constant with $\theta = 2/\sqrt{5}$     & 4710   & -      & 4710    
			\\
			\hline 
		\end{tabular}
		\label{table:Num-Steps-1st}
	\end{table}

	\begin{figure}[ptbh]
		\subfigure[$\mathbb{E} \text{[} \mathcal{E}_{n+1}^{h} \text{]}$ 
		and $\mathbb{E} \text{[}\mathcal{E}_{n+1} \text{]}$]{ \label{fig:KE-avg-1st}
			\hspace{-0.2cm}
			\begin{minipage}[t]{0.45\linewidth}
				\centering
				\includegraphics[width=2.8in,height=2.0in]{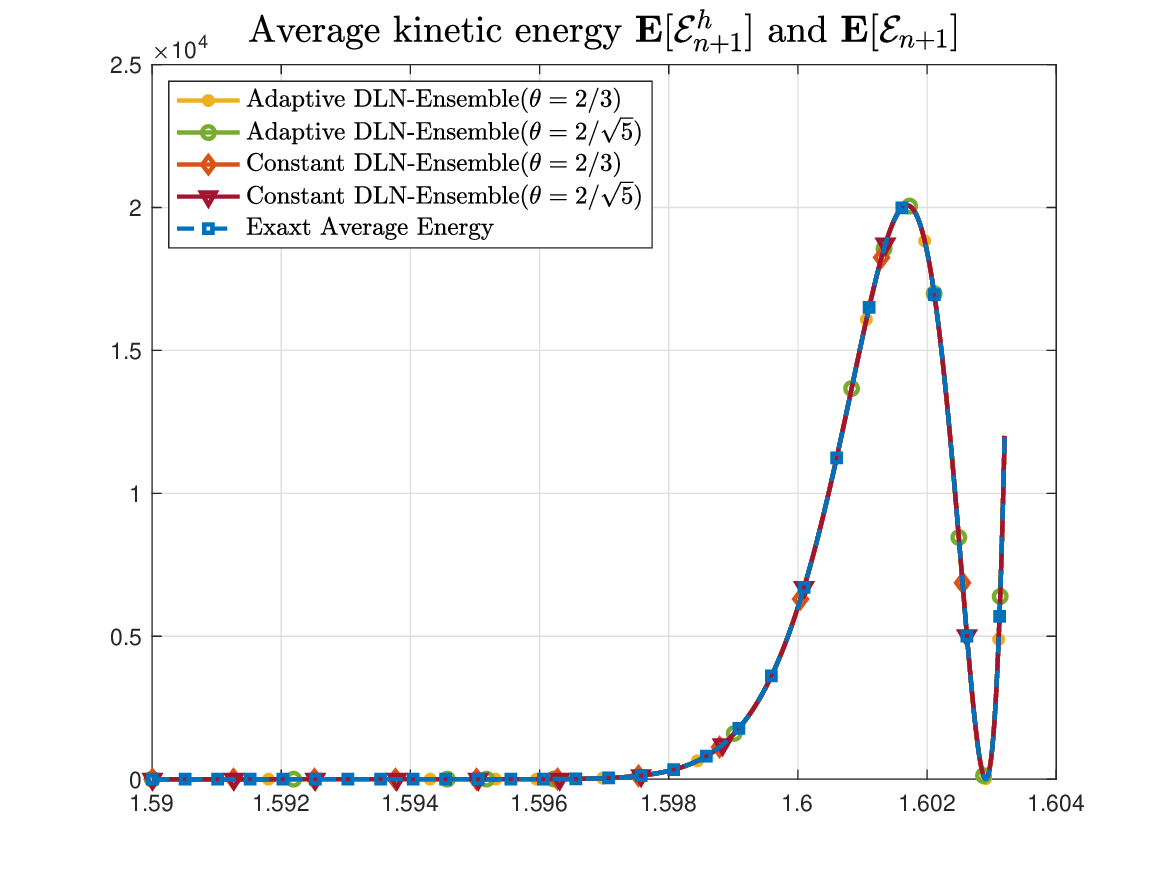}\\
				\vspace{0.02cm}
			\end{minipage}
			\quad}%
		\subfigure[$\log_{10} \big( \big| \mathbb{E} \text{[}\mathcal{E}_{n+1} \text{]} - 
					\mathbb{E} \text{[} \mathcal{E}_{n+1}^{h} \text{]} \big| \big)$]{ \label{fig:Error-avg-1st}
			\hspace{-0.2cm}
			\begin{minipage}[t]{0.45\linewidth}
				\centering
				\includegraphics[width=2.8in,height=2.0in]{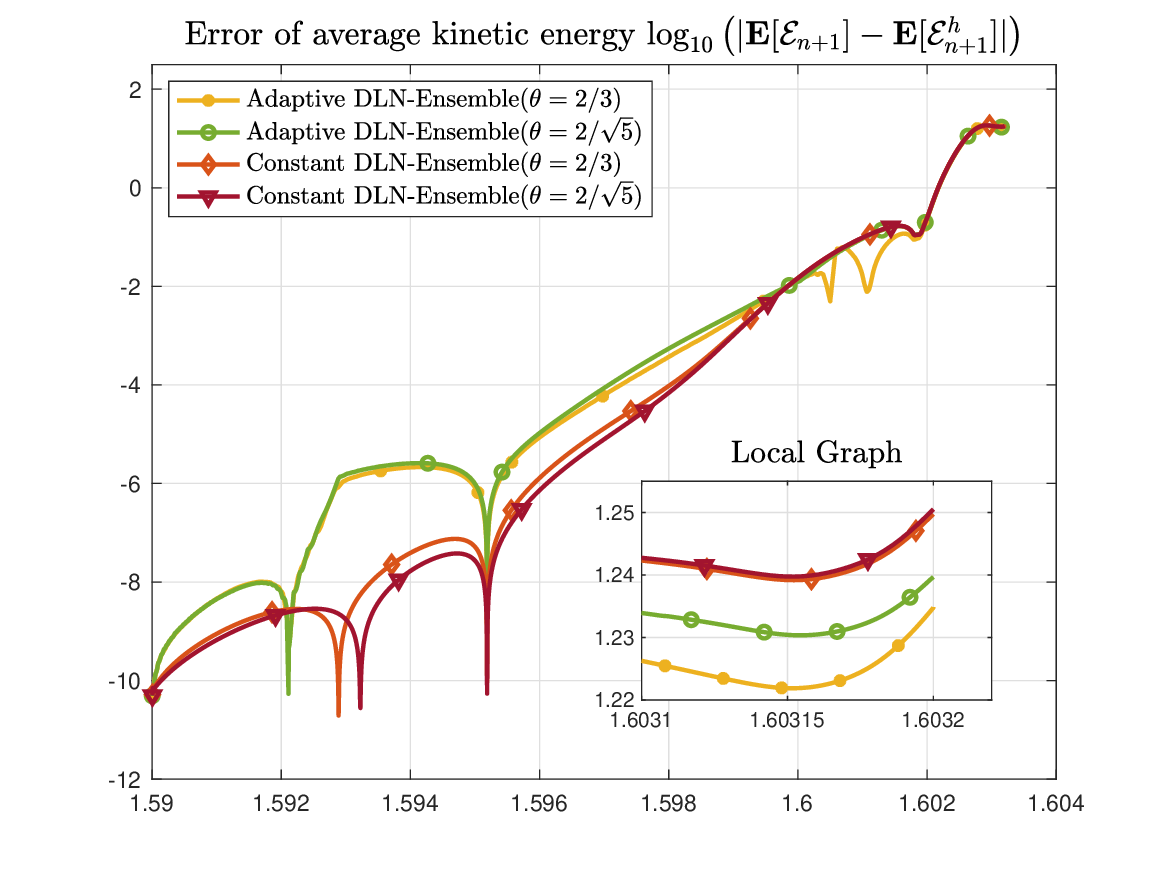}\\
				\vspace{0.02cm}
		\end{minipage}}  	
		\subfigure[$\mathcal{E}_{n+1}^{h,\infty}$ and $\mathcal{E}_{n+1}^{\infty}$]{ \label{fig:KE-max-1st}
			\hspace{-0.2cm}
			\begin{minipage}[t]{0.45\linewidth}
				\centering
				\includegraphics[width=2.8in,height=2.0in]{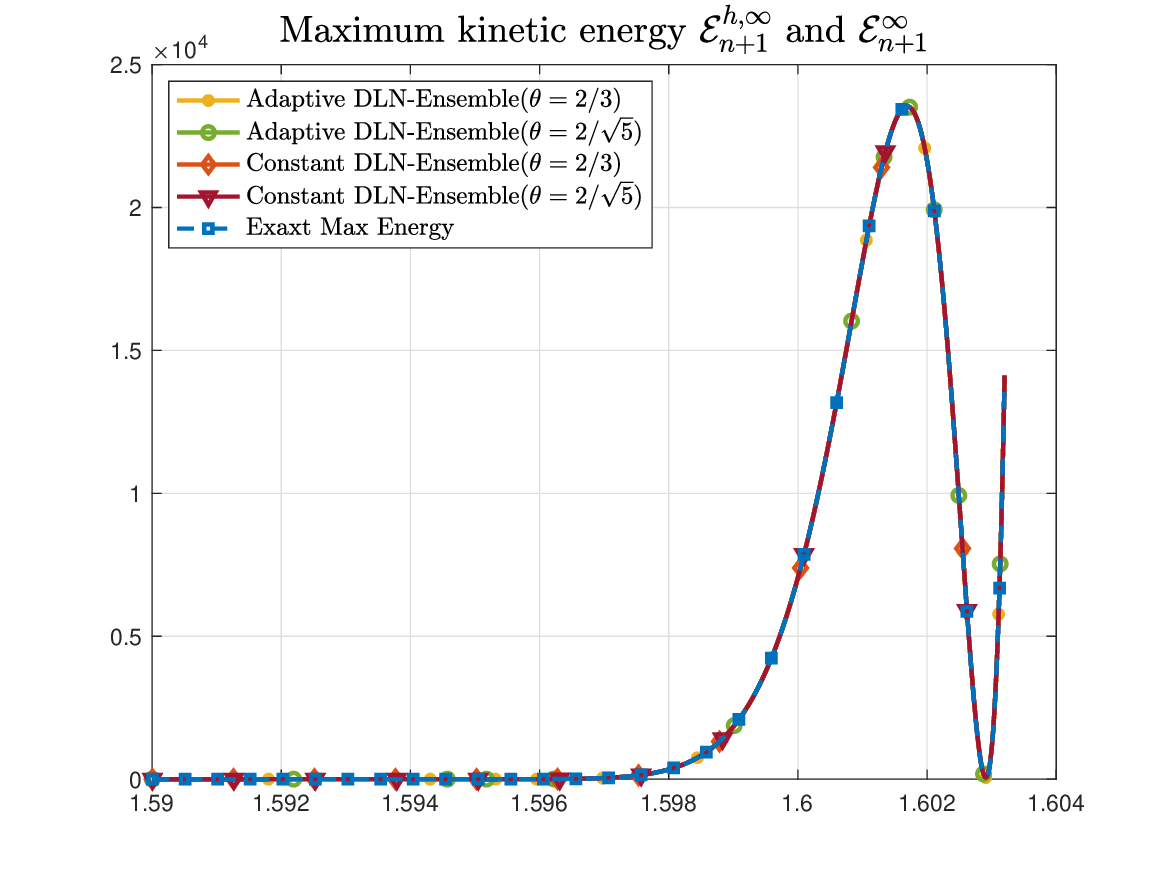}\\
				\vspace{0.02cm}
			\end{minipage}
			\quad}%
		\subfigure[$\log_{10} \big( \big| \mathcal{E}_{n+1}^{\infty} - \mathcal{E}_{n+1}^{h,\infty} \big| \big)$]{ \label{fig:Error-max-1st}
			\hspace{-0.2cm}
			\begin{minipage}[t]{0.45\linewidth}
				\centering
				\includegraphics[width=2.8in,height=2.0in]{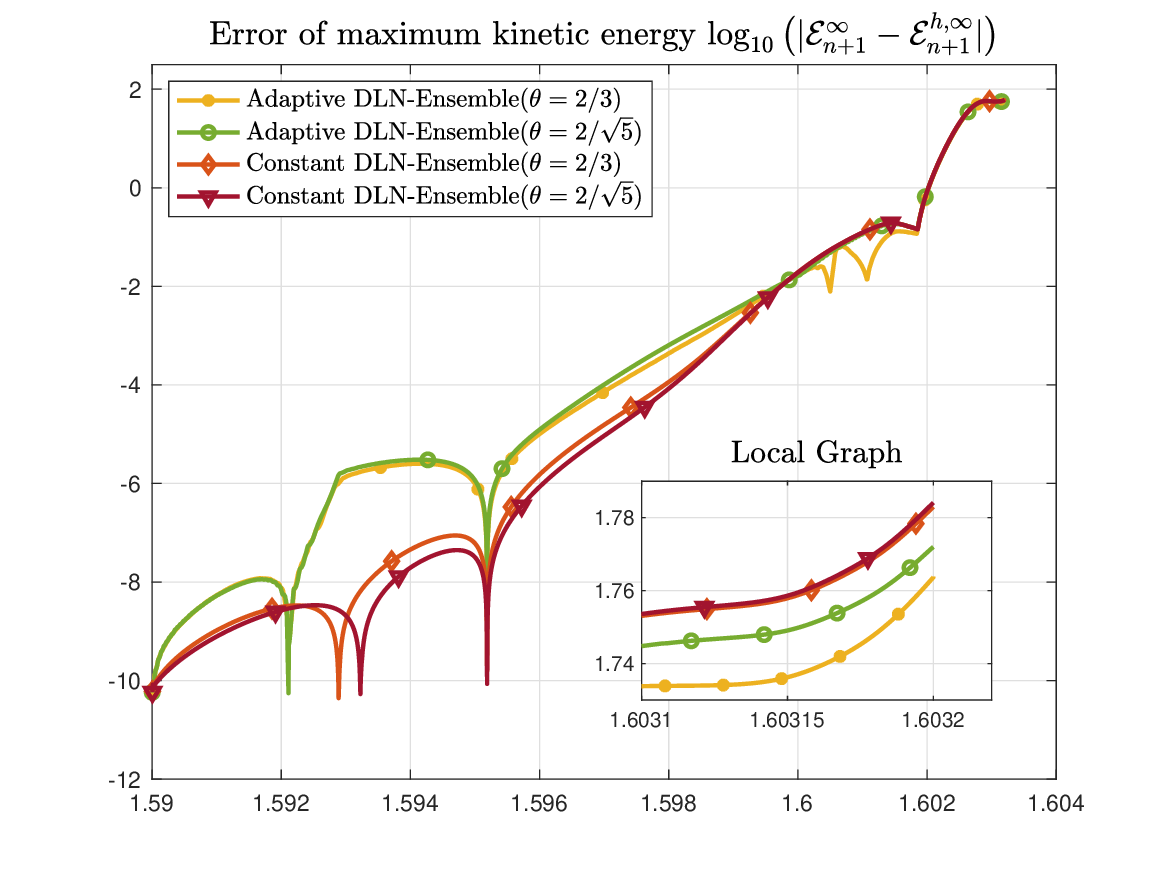}\\
				\vspace{0.02cm}
		\end{minipage}}  	
		\par
		\centering
		\vspace{-0.2cm}
		\caption{All the algorithms have the true pattern of average kinetic energy and maximum kinetic energy.
		However adaptive algorithms obtain relatively small errors at the end of the simulation since more time steps are assigned in the simulations of extremely stiff parts of true solutions ($t \geq 1.602$).}
		\label{fig:Energy-1st}
	\end{figure}

	\begin{figure}[ptbh]
		\subfigure[$\log_{10} ( \mathbb{E} \text{[} \mathcal{E}_{n+1,\tt{ND}}^{h} \text{]} )$]{ \label{fig:ND-avg-1st}
			\hspace{-0.2cm}
			\begin{minipage}[t]{0.45\linewidth}
				\centering
				\includegraphics[width=2.8in,height=2.0in]{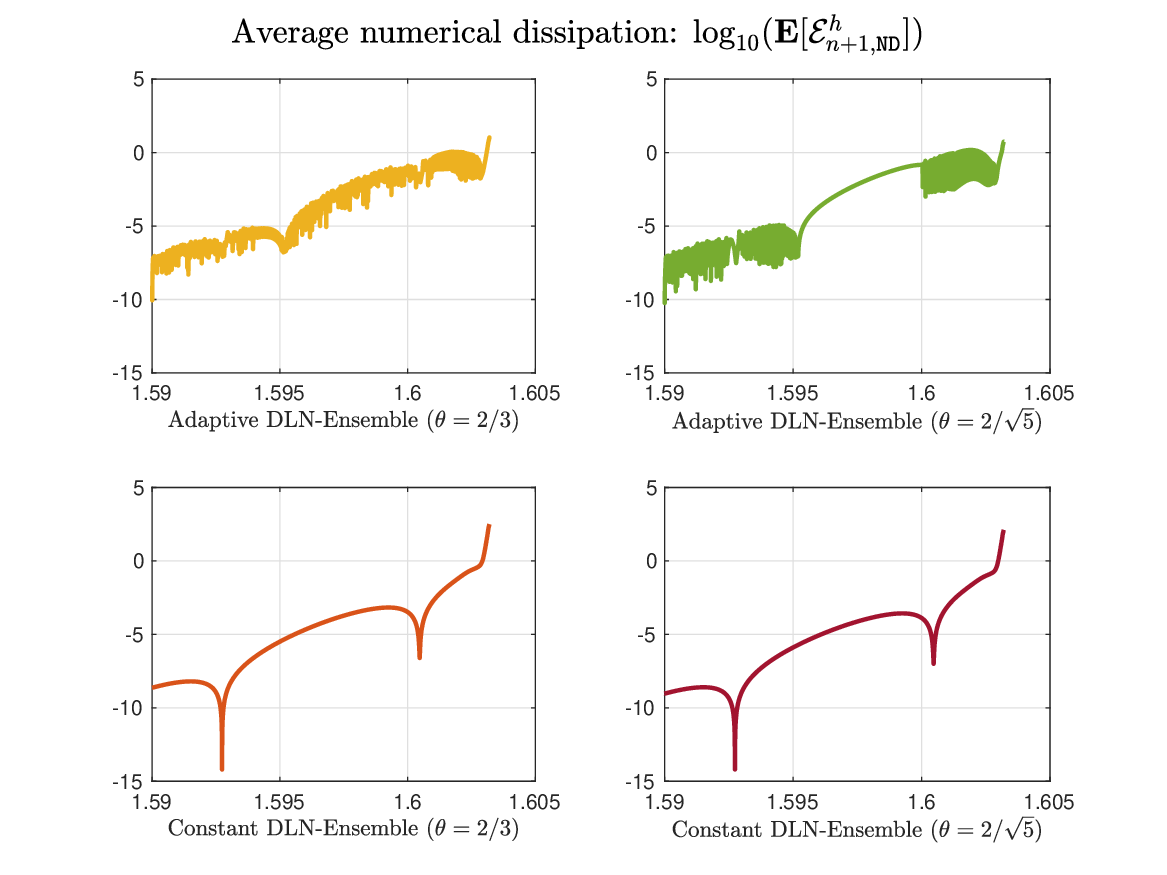}\\
				\vspace{0.02cm}
			\end{minipage}
			\quad}%
		\subfigure[$\log_{10} ( \mathcal{E}_{n+1,\tt{ND}}^{h,\infty})$]{ \label{fig:ND-max-1st}
			\hspace{-0.2cm}
			\begin{minipage}[t]{0.45\linewidth}
				\centering
				\includegraphics[width=2.8in,height=2.0in]{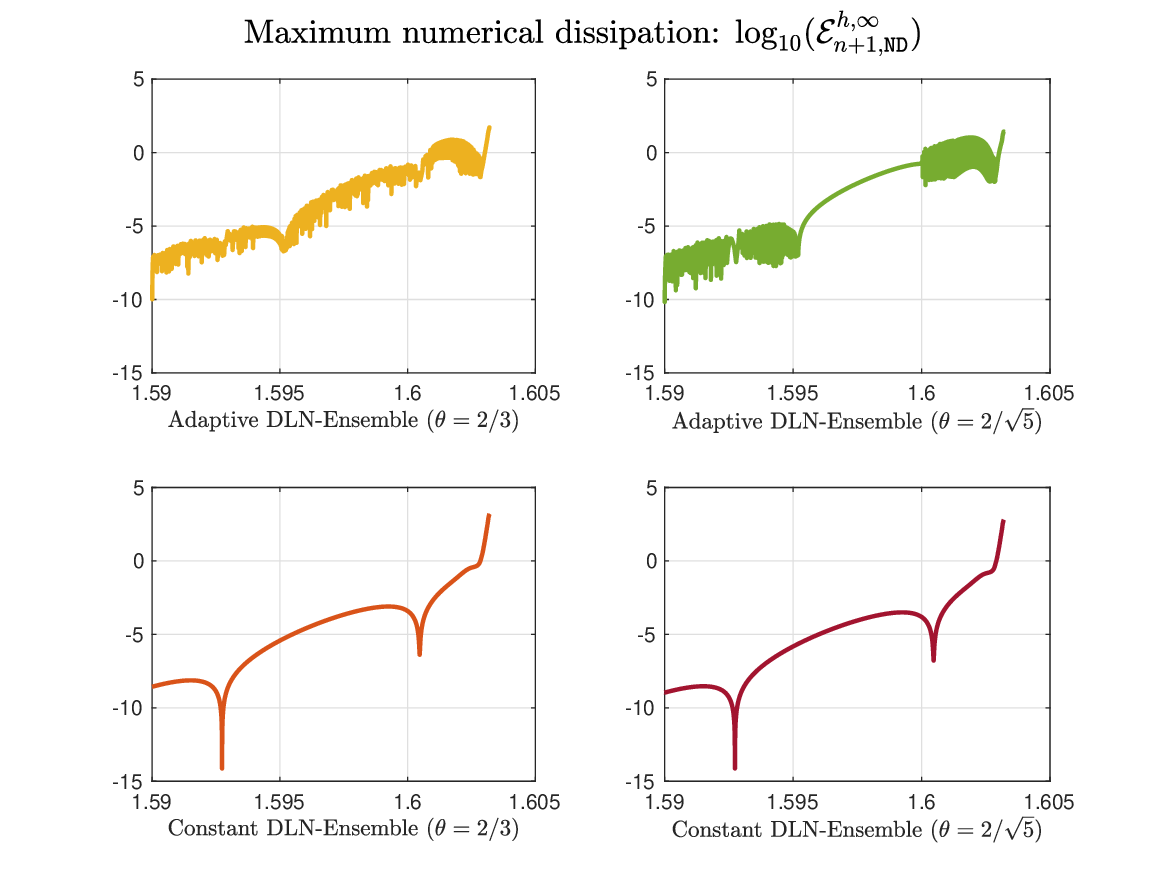}\\
				\vspace{0.02cm}
		\end{minipage}}  	
		\subfigure[$\log_{10} ( \mathbb{E} \text{[} \mathcal{E}_{n+1,\tt{VD}}^{h} \text{]} )$]{ \label{fig:VD-avg-1st}
			\hspace{-0.2cm}
			\begin{minipage}[t]{0.45\linewidth}
				\centering
				\includegraphics[width=2.8in,height=2.0in]{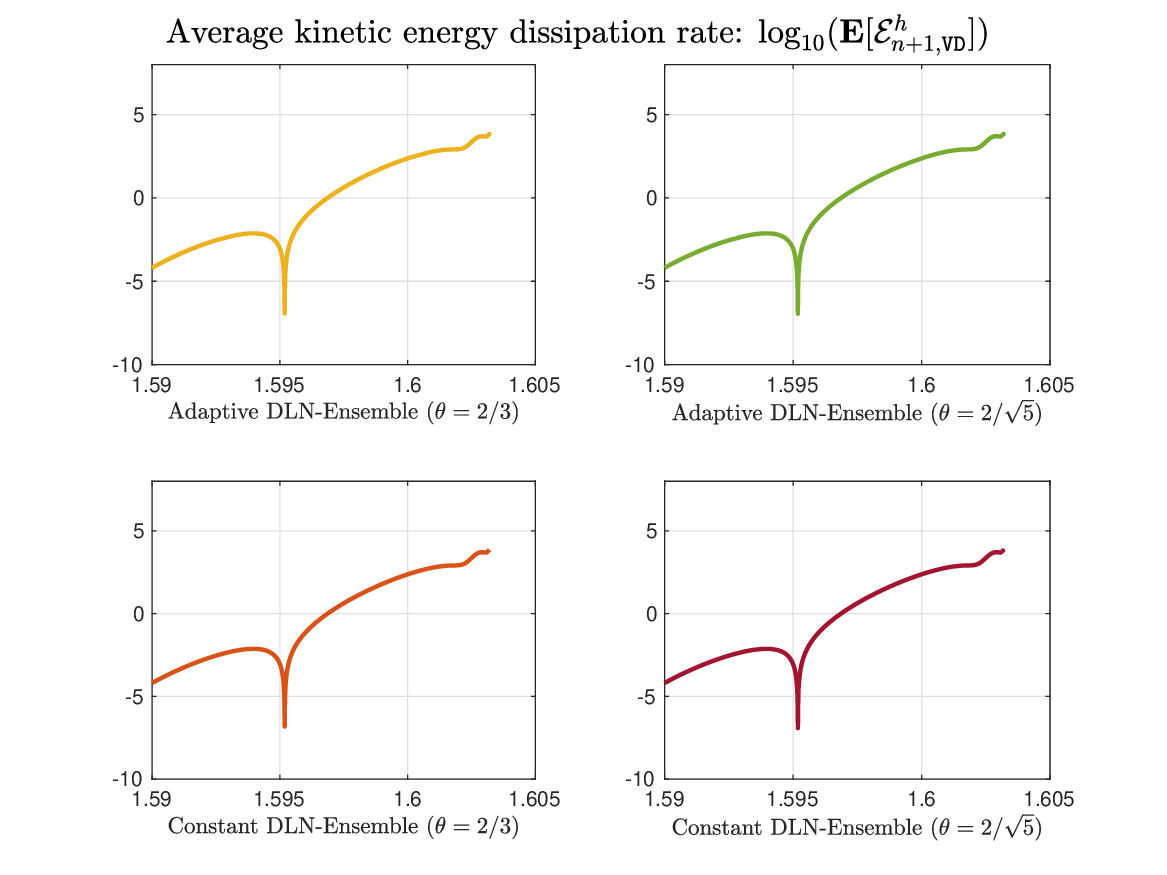}\\
				\vspace{0.02cm}
			\end{minipage}
			\quad}%
		\subfigure[$\log_{10} ( \mathcal{E}_{n+1,\tt{VD}}^{h,\infty})$]{ \label{fig:VD-max-1st}
			\hspace{-0.2cm}
			\begin{minipage}[t]{0.45\linewidth}
				\centering
				\includegraphics[width=2.8in,height=2.0in]{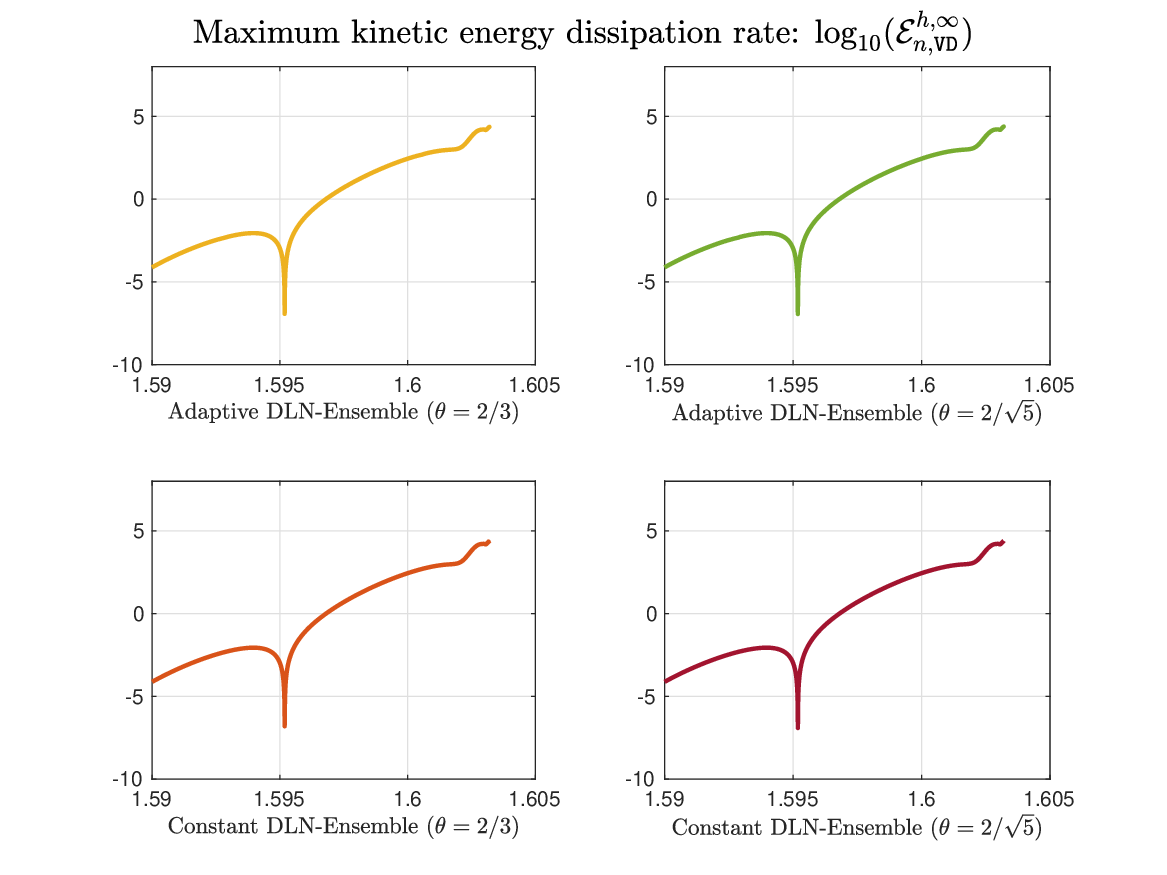}\\
				\vspace{0.02cm}
		\end{minipage}}  	
		\subfigure[$\widehat{T}_{n+1}$]{ \label{fig:EST-LTE-1st}
			\hspace{-0.2cm}
			\begin{minipage}[t]{0.45\linewidth}
				\centering
				\includegraphics[width=2.8in,height=2.0in]{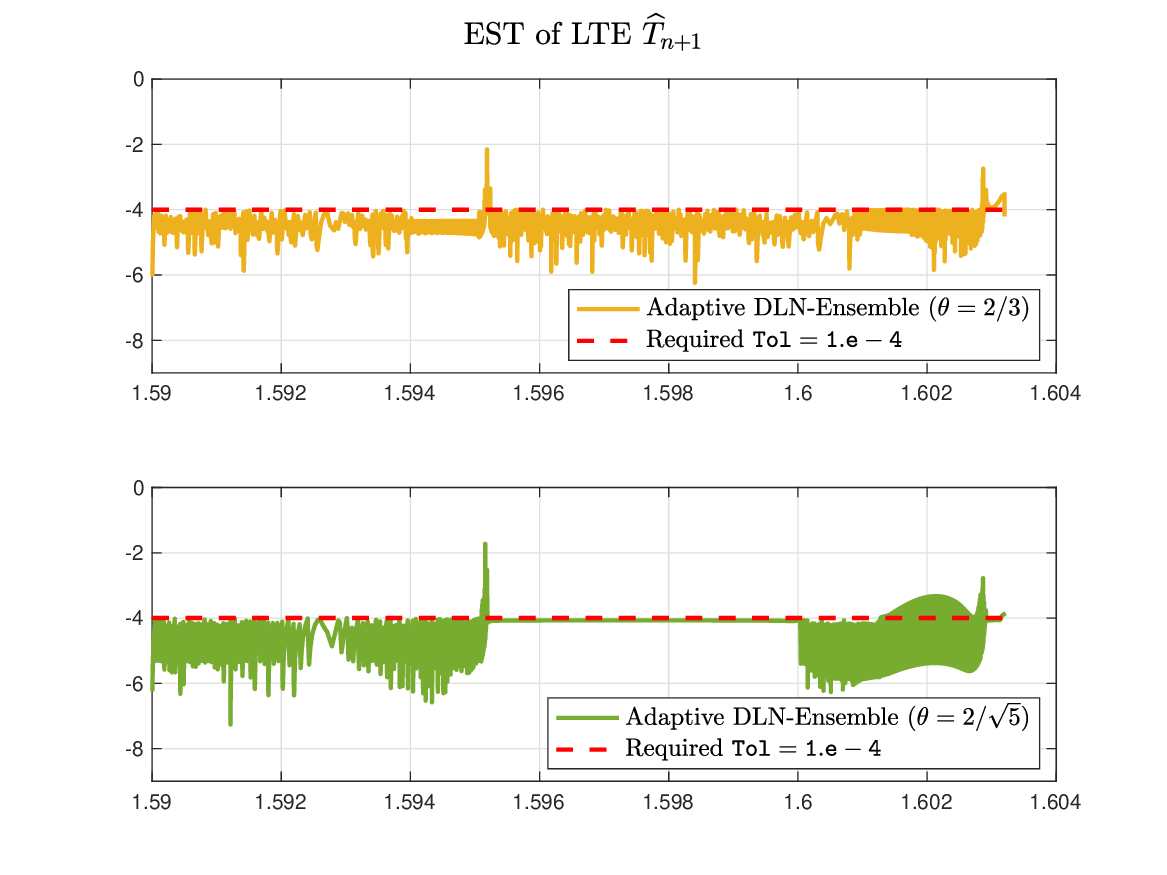}\\
				\vspace{0.02cm}
			\end{minipage}
			\quad}%
		\subfigure[$\log_{10} (k_{n})$]{ \label{fig:Step-1st}
			\hspace{-0.2cm}
			\begin{minipage}[t]{0.45\linewidth}
				\centering
				\includegraphics[width=2.8in,height=2.0in]{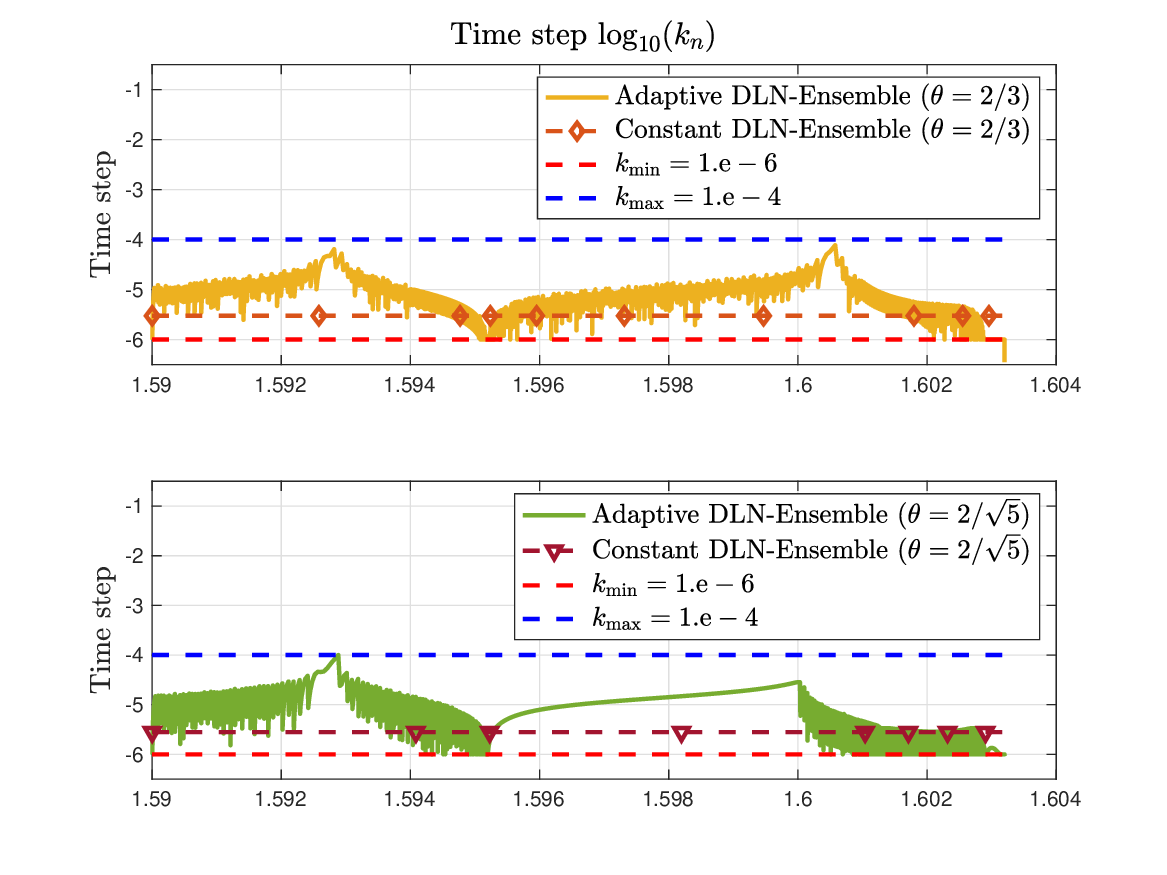}\\
				\vspace{0.02cm}
		\end{minipage}}  	
		\par
		\centering
		\vspace{-0.2cm}
		\caption{The numerical dissipation of adaptive algorithms is relatively large before $t = 1.602$ but grows slower at the end.  
		Meanwhile, all DLN-Ensemble algorithms have similar patterns of kinetic energy dissipation rates.
		The highly stiff part ($t \geq 1.6$) is hard to simulate since the estimator of LTE $\widehat{T}_{n+1}$ exceeds the required tolerance ${\tt{Tol}} = 1.\rm{e}-4$ frequently and time step size oscillates near the minimum step size $k_{\rm{min}} = 1.\rm{e}-6$ after $t = 1.602$.}
		\label{fig:ND-VD-LTE-Step-1st}
	\end{figure}

	Then we set time component function to be $G_{2}(t)$ with $\omega=3.1$ and simulate 10 systems of NSE over time interval $[1.58 \ \ 1.602]$. 
	From \Cref{fig:Lindberg2nd}, $G_{2}(t)$ has a slight growth to $16$ at $t = 1.598$ and declines rapidly to $-700$ in a period of $0.002$. 
	The number of steps, number of rejections, and total cost in steps for the adaptive algorithms are summarized in \Cref{table:Num-Steps-2nd}.
	The adaptive algorithms have small errors of energy in \Cref{fig:Error-avg-2nd,fig:Error-max-2nd} even though the exact average energy and maximum energy are above $10^{5}$ in \Cref{fig:KE-avg-2nd,fig:KE-max-2nd}. 
	The adaptive algorithm with $\theta = 2/\sqrt{5}$ outperforms all other algorithms since it uses the least number of steps to achieve minimum errors at the end. 
	The constant time-stepping algorithm with $\theta = 2/\sqrt{5}$ is unstable and obtains abnormal errors of energy in \Cref{fig:Error-Const25-avg-2nd,fig:Error-Const25-max-2nd}, which implies the advantage of time adaptivity in extremely stiff test problems.
	From \Cref{fig:ND-avg-2nd,fig:ND-max-2nd,fig:VD-avg-2nd,fig:VD-max-2nd}, all the algorithms except the constant time-stepping algorithm with $\theta = 2/\sqrt{5}$ have similar patterns of numerical dissipation and kinetic energy dissipation rate. 
	Adaptive algorithm with $\theta = 2/3$ is less efficient than adaptive algorithm with $\theta = 2/\sqrt{5}$ in that its $\widehat{T}_{n+1}$ reaches the required tolerance ${\tt{Tol}} = 1.\rm{e}-4$ more often and use the minimum time step size frequently. 
 
	\begin{table}[ptbh]
		\centering
		\renewcommand\arraystretch{1.25}
		\caption{Number of steps for the DLN-Ensemble algorithms with $\theta = 2/3, 2/\sqrt{5}$ for the test with time component function $G_{2}(t)$} 
		\begin{tabular}{lccc}
			\hline
			\hline
			DLN-Ensemble algorithms & \# Steps 
			& \# Rejections 
			& Total cost in steps
			\\
			\hline 
			Adaptive with $\theta = 2/3$            & 3851   & 3467   & 7318             
			\\
			Adaptive with $\theta = 2/\sqrt{5}$     & 3121   & 1875   & 4996     
			\\
			Constant with $\theta = 2/3$            & 7318   & -      & 7318    
			\\
			Constant with $\theta = 2/\sqrt{5}$     & 4996   & -      & 4996    
			\\
			\hline 
		\end{tabular}
		\label{table:Num-Steps-2nd}
	\end{table}

	\begin{figure}[ptbh]
		\subfigure[]{ \label{fig:Error-avg-2nd}
			\hspace{-0.2cm}
			\begin{minipage}[t]{0.45\linewidth}
				\centering
				\includegraphics[width=2.8in,height=2.0in]{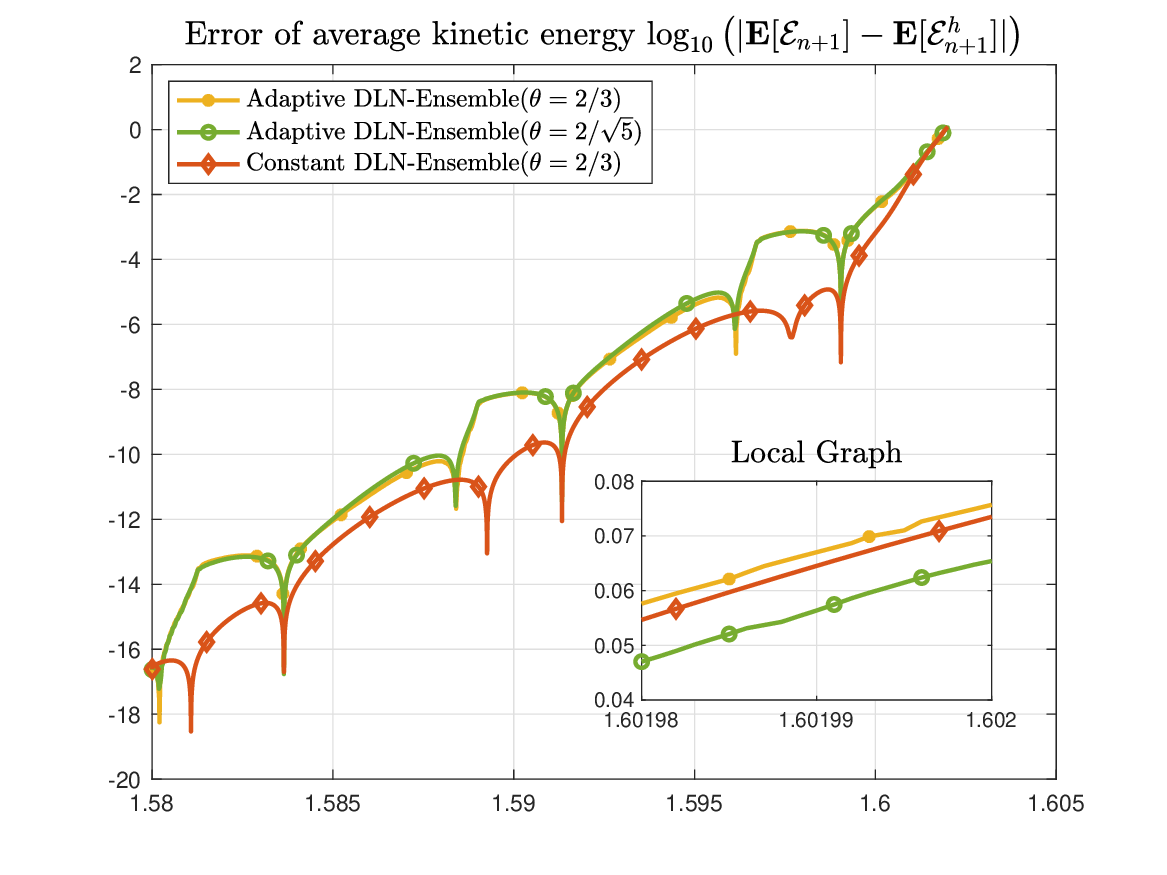}\\
				\vspace{0.02cm}
			\end{minipage}
			\quad}%
		\subfigure[]{ \label{fig:Error-Const25-avg-2nd}
			\hspace{-0.2cm}
			\begin{minipage}[t]{0.45\linewidth}
				\centering
				\includegraphics[width=2.8in,height=2.0in]{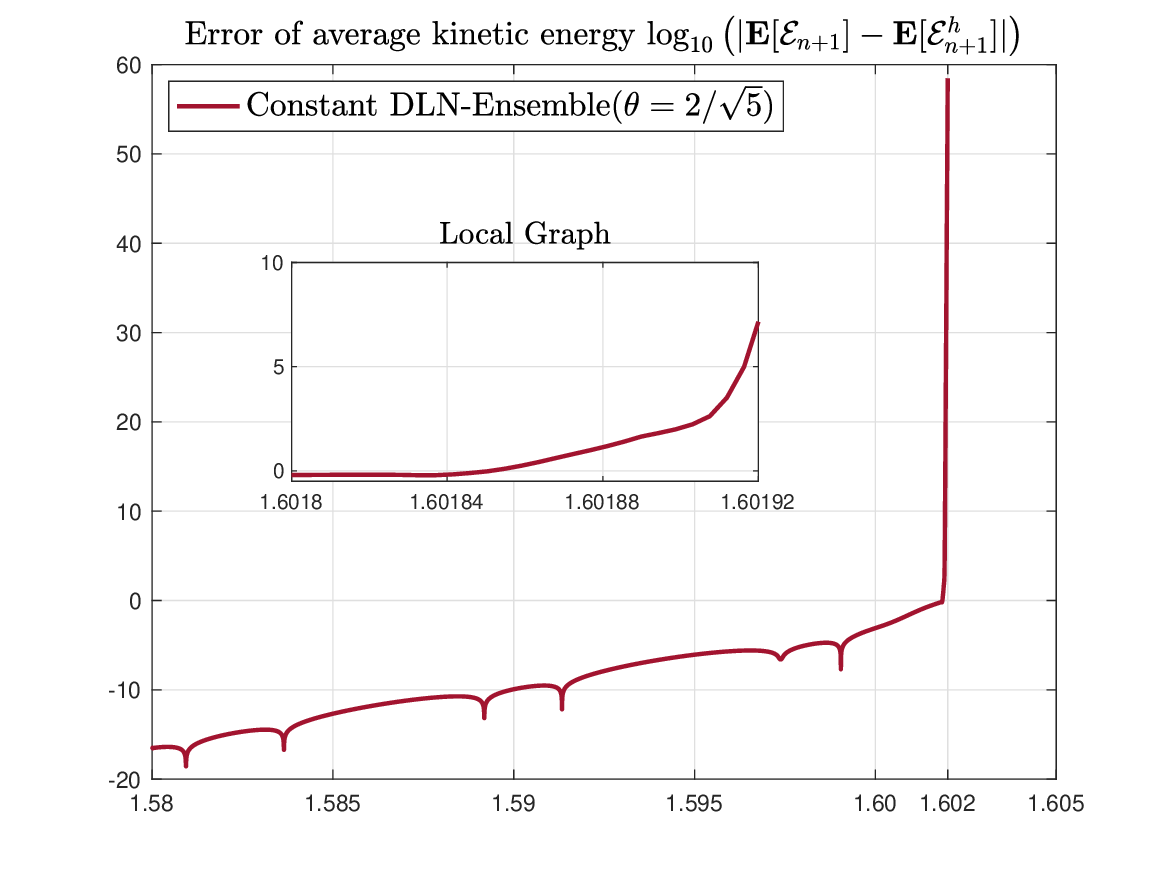}\\
				\vspace{0.02cm}
		\end{minipage}}  	
		\subfigure[]{ \label{fig:Error-max-2nd}
			\hspace{-0.2cm}
			\begin{minipage}[t]{0.45\linewidth}
				\centering
				\includegraphics[width=2.8in,height=2.0in]{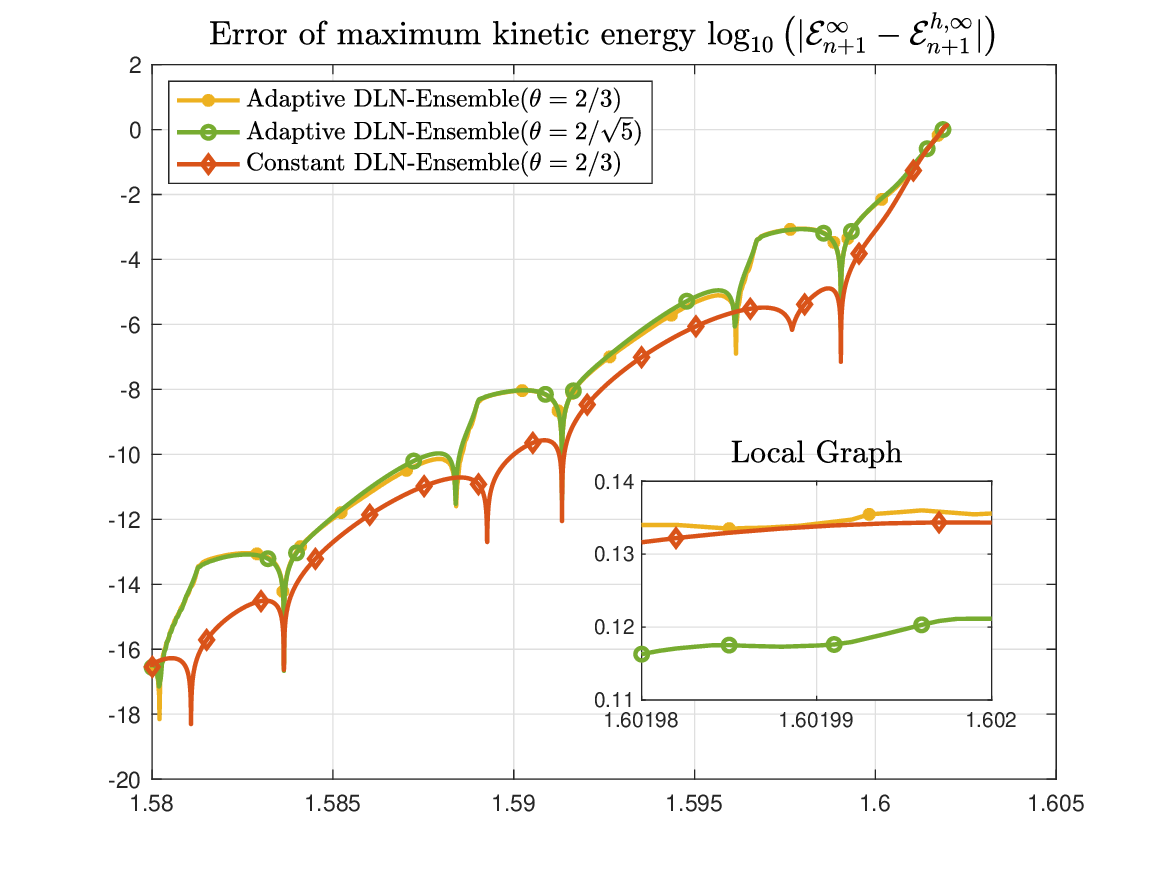}\\
				\vspace{0.02cm}
			\end{minipage}
			\quad}%
		\subfigure[]{ \label{fig:Error-Const25-max-2nd}
			\hspace{-0.2cm}
			\begin{minipage}[t]{0.45\linewidth}
				\centering
				\includegraphics[width=2.8in,height=2.0in]{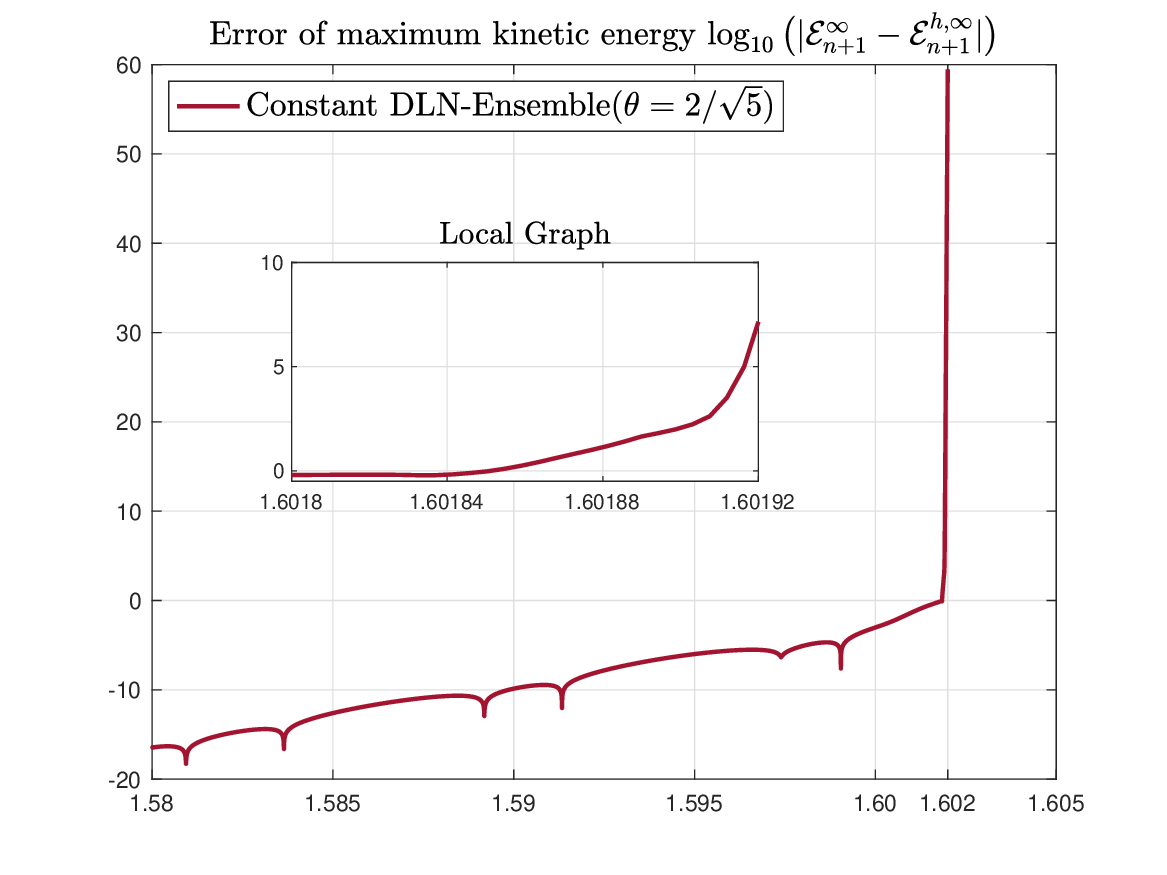}\\
				\vspace{0.02cm}
		\end{minipage}}  	
		\subfigure[$\mathbb{E} \text{[} \mathcal{E}_{n+1}^{h} \text{]}$ 
		and $\mathbb{E} \text{[}\mathcal{E}_{n+1} \text{]}$]{ \label{fig:KE-avg-2nd}
			\hspace{-0.2cm}
			\begin{minipage}[t]{0.45\linewidth}
				\centering
				\includegraphics[width=2.8in,height=2.0in]{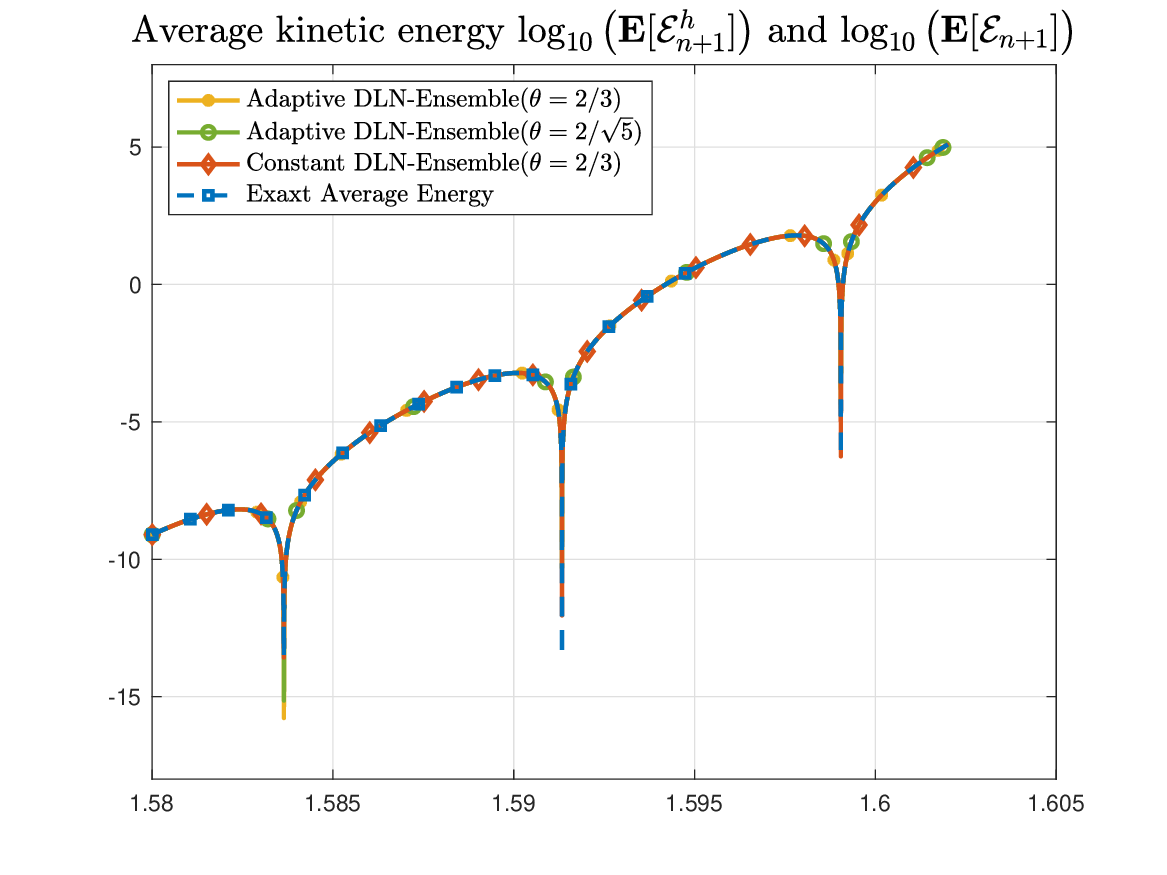}\\
				\vspace{0.02cm}
			\end{minipage}
			\quad}%
		\subfigure[$\mathcal{E}_{n+1}^{h,\infty}$ and $\mathcal{E}_{n+1}^{\infty}$]{ \label{fig:KE-max-2nd}
			\hspace{-0.2cm}
			\begin{minipage}[t]{0.45\linewidth}
				\centering
				\includegraphics[width=2.8in,height=2.0in]{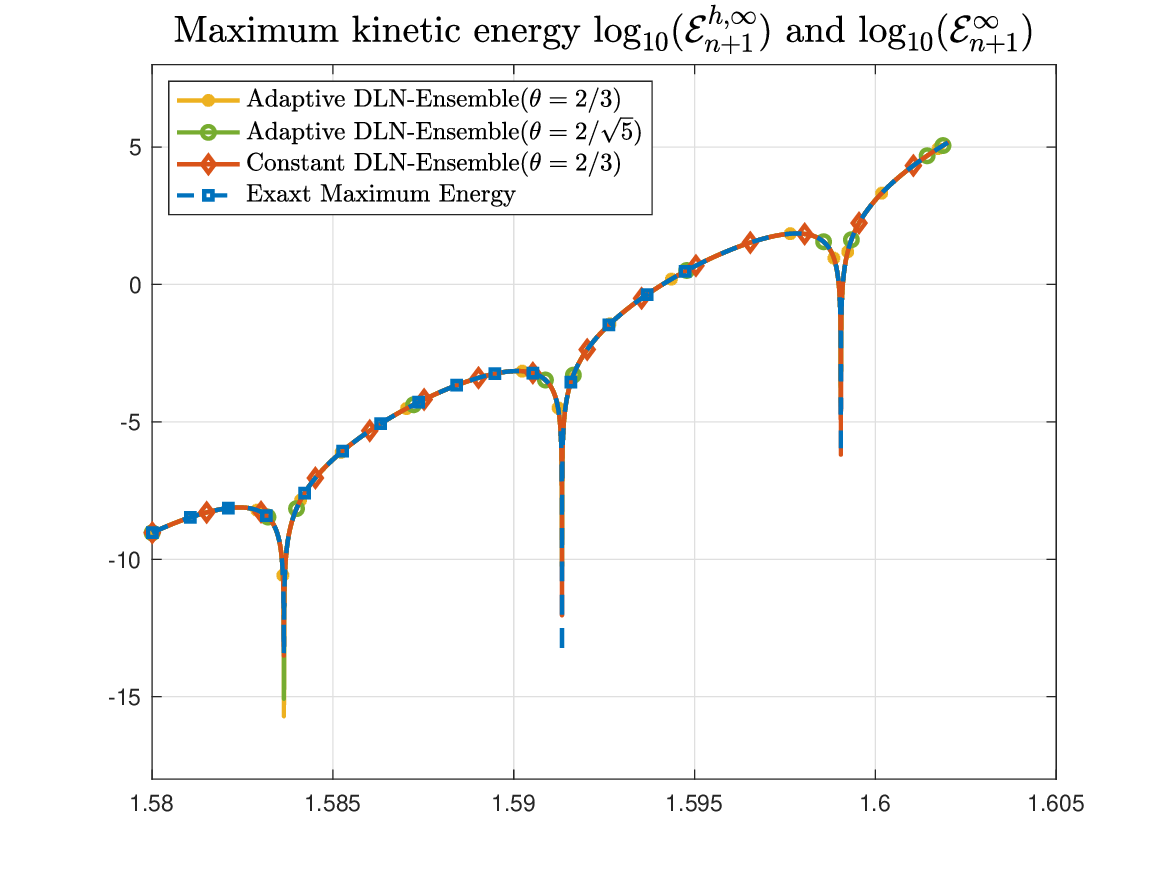}\\
				\vspace{0.02cm}
		\end{minipage}}  
		\par
		\centering
		\vspace{-0.2cm}
		\caption{The adaptive algorithms have small energy errors even if the exact average and maximum energy are above $10^{5}$. The adaptive algorithm with $\theta = 2/\sqrt{5}$ outperforms all other algorithms since it uses the least number of steps to achieve minimum errors at the end. The constant time-stepping algorithm with $\theta = 2/\sqrt{5}$ is unstable and obtains abnormal energy errors, which implies the advantage of time adaptivity in extremely stiff test problems.}
		\label{fig:Energy-error-2nd}
	\end{figure}

	\begin{figure}[ptbh]
		\subfigure[$\log_{10} ( \mathbb{E} \text{[} \mathcal{E}_{n+1,\tt{ND}}^{h} \text{]} )$]{ \label{fig:ND-avg-2nd}
			\hspace{-0.2cm}
			\begin{minipage}[t]{0.45\linewidth}
				\centering
				\includegraphics[width=2.8in,height=2.0in]{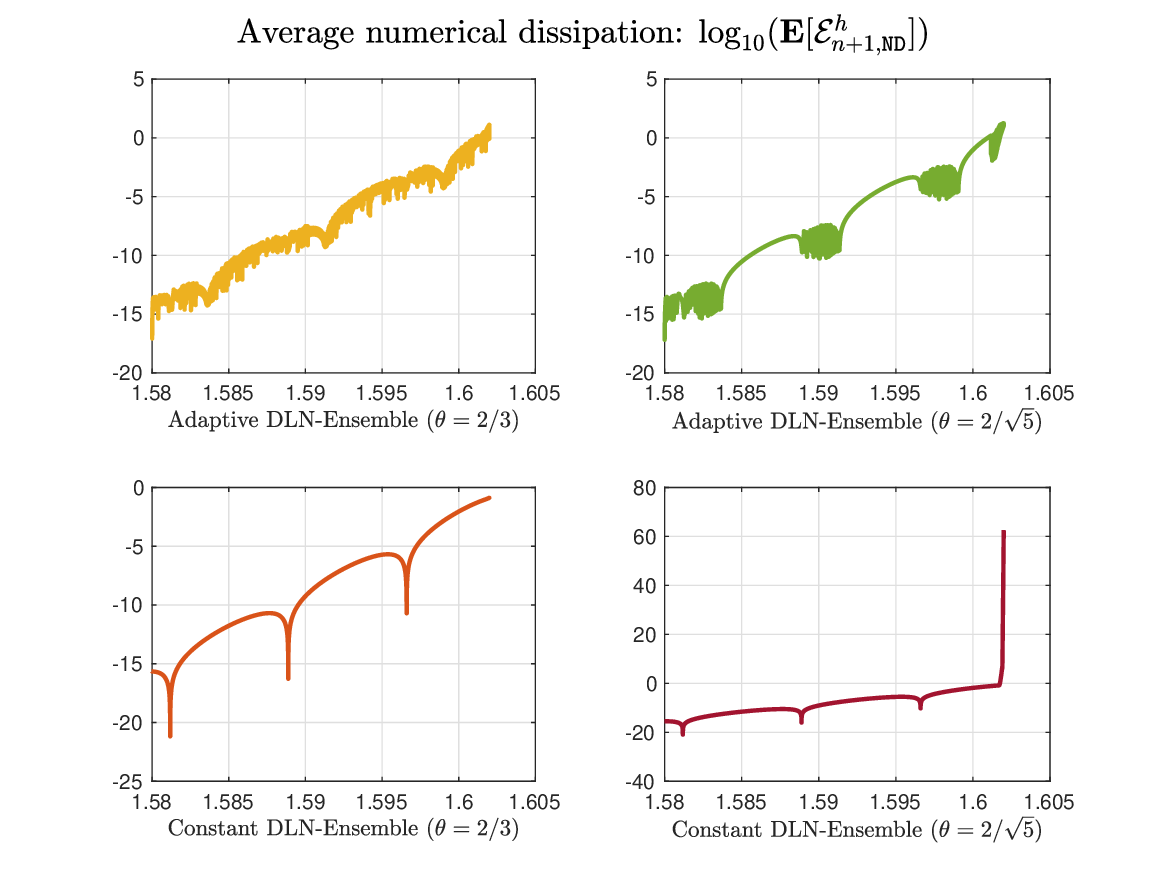}\\
				\vspace{0.02cm}
			\end{minipage}
			\quad}%
		\subfigure[$\log_{10} ( \mathcal{E}_{n+1,\tt{ND}}^{h,\infty})$]{ \label{fig:ND-max-2nd}
			\hspace{-0.2cm}
			\begin{minipage}[t]{0.45\linewidth}
				\centering
				\includegraphics[width=2.8in,height=2.0in]{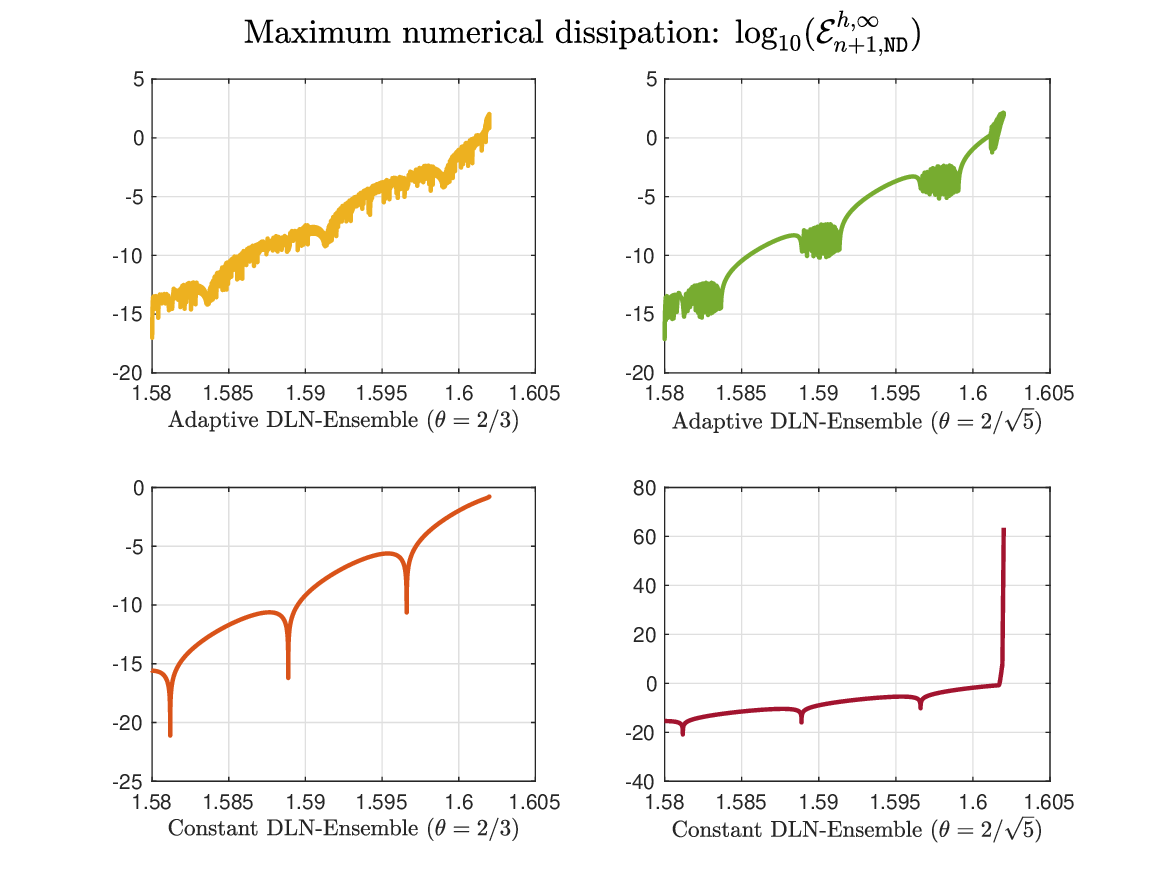}\\
				\vspace{0.02cm}
		\end{minipage}}  	
		\subfigure[$\log_{10} ( \mathbb{E} \text{[} \mathcal{E}_{n+1,\tt{VD}}^{h} \text{]} )$]{ \label{fig:VD-avg-2nd}
			\hspace{-0.2cm}
			\begin{minipage}[t]{0.45\linewidth}
				\centering
				\includegraphics[width=2.8in,height=2.0in]{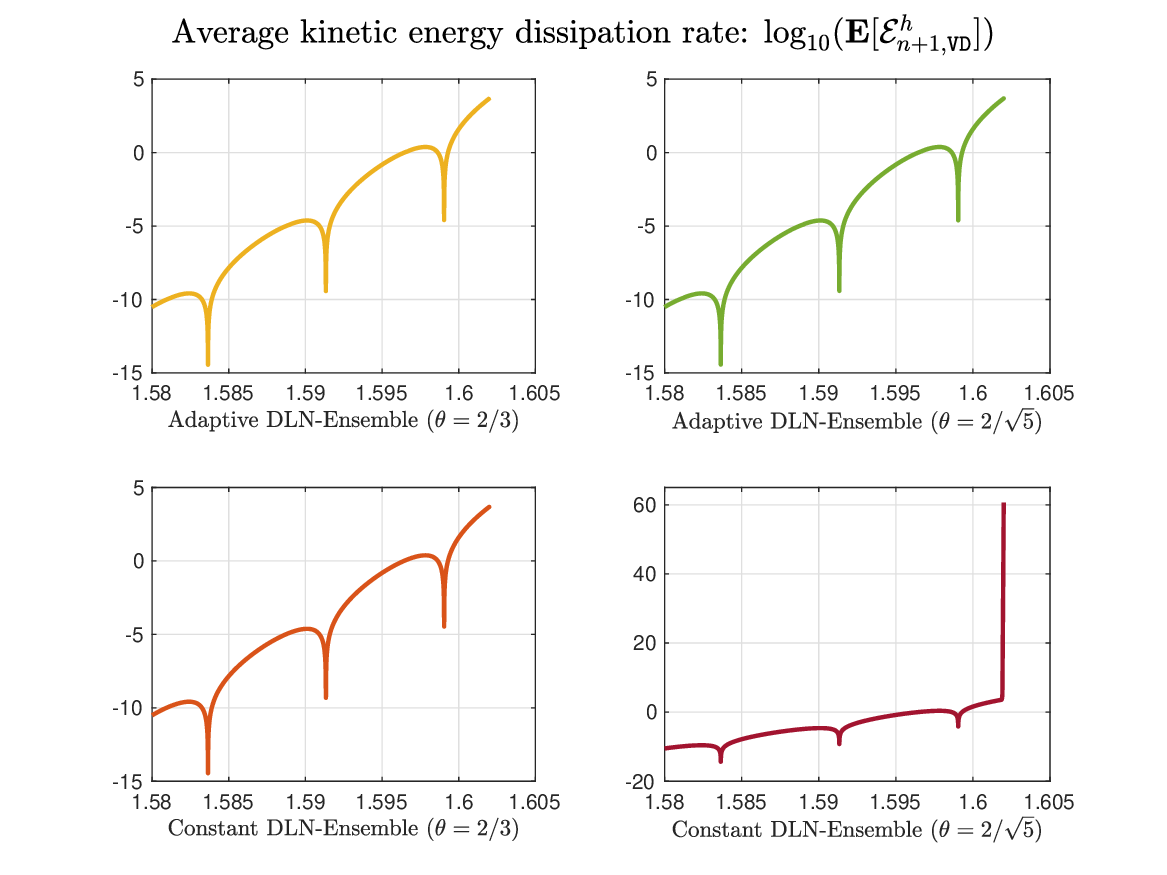}\\
				\vspace{0.02cm}
			\end{minipage}
			\quad}%
		\subfigure[$\log_{10} ( \mathcal{E}_{n+1,\tt{VD}}^{h,\infty})$]{ \label{fig:VD-max-2nd}
			\hspace{-0.2cm}
			\begin{minipage}[t]{0.45\linewidth}
				\centering
				\includegraphics[width=2.8in,height=2.0in]{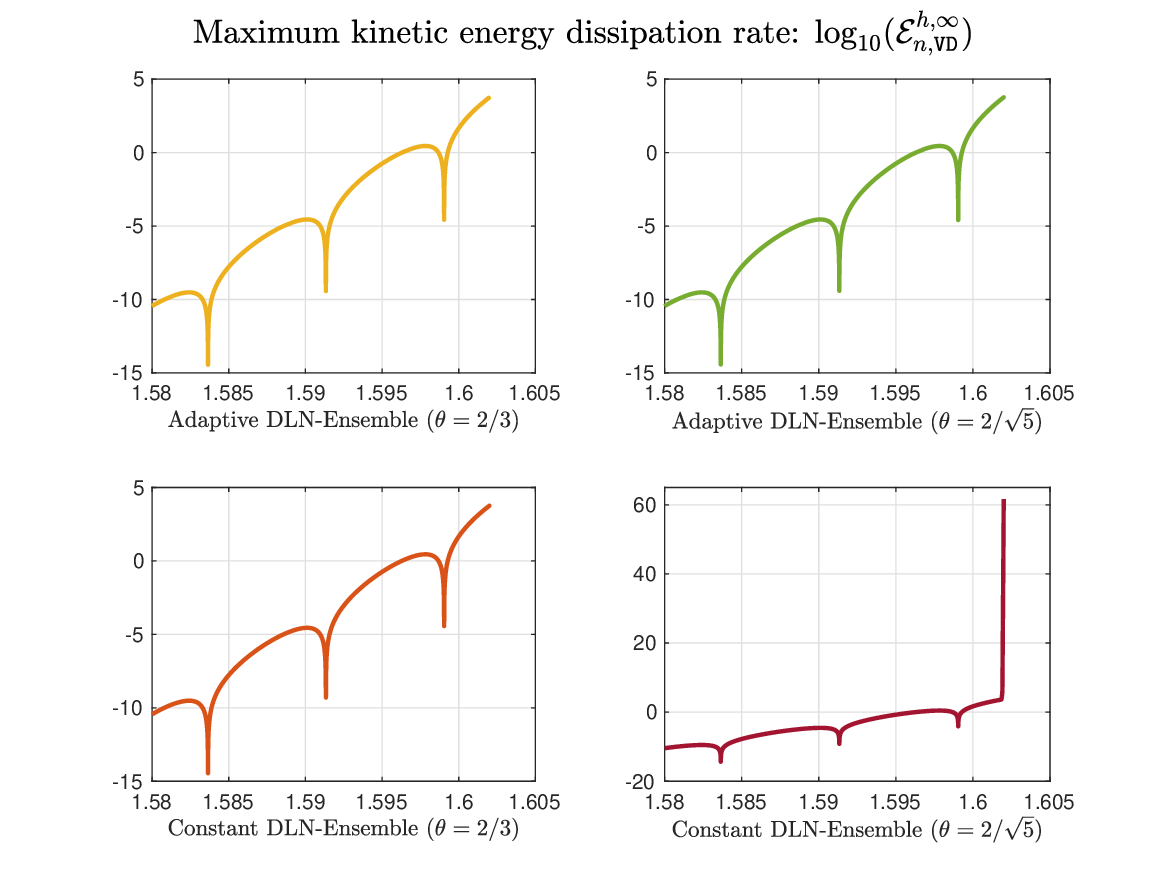}\\
				\vspace{0.02cm}
		\end{minipage}}  	
		\subfigure[$\widehat{T}_{n+1}$]{ \label{fig:EST-LTE-2nd}
			\hspace{-0.2cm}
			\begin{minipage}[t]{0.45\linewidth}
				\centering
				\includegraphics[width=2.8in,height=2.0in]{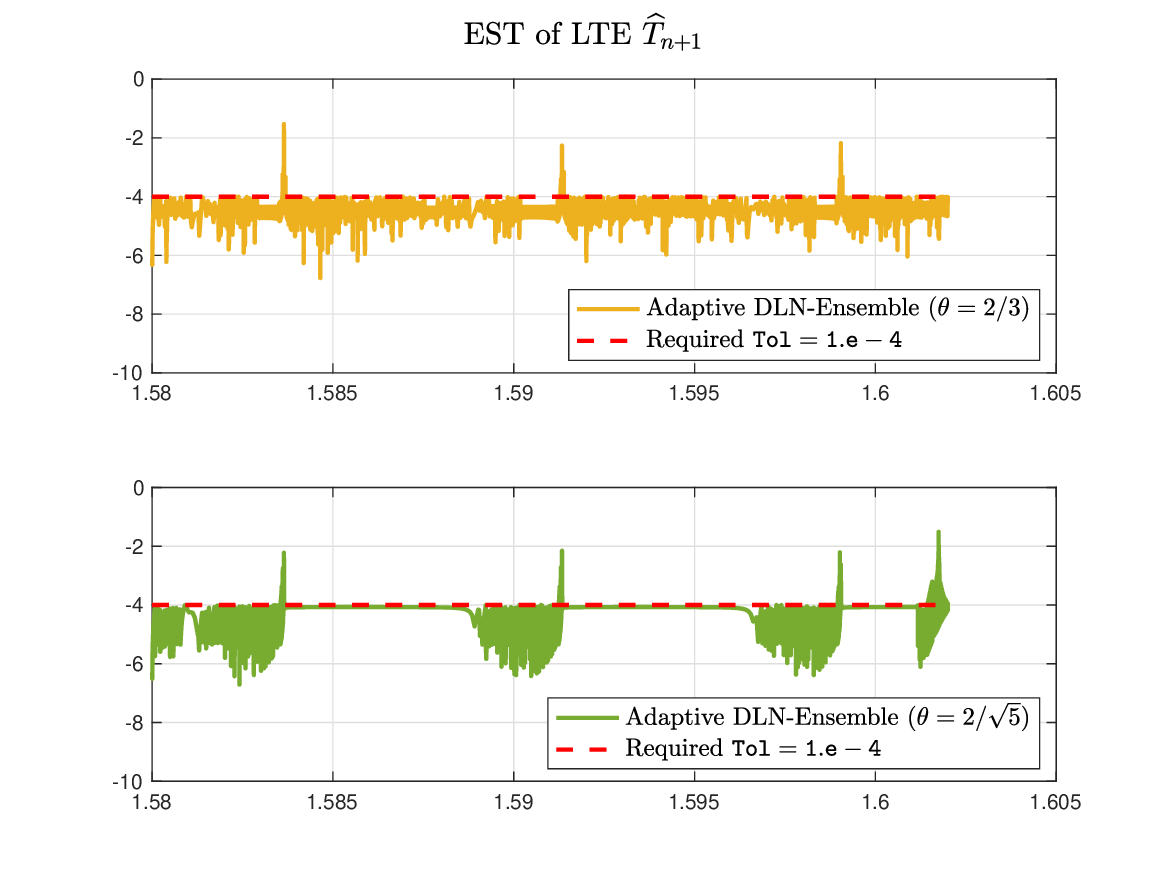}\\
				\vspace{0.02cm}
			\end{minipage}
			\quad}%
		\subfigure[$\log_{10} (k_{n})$]{ \label{fig:Step-2nd}
			\hspace{-0.2cm}
			\begin{minipage}[t]{0.45\linewidth}
				\centering
				\includegraphics[width=2.8in,height=2.0in]{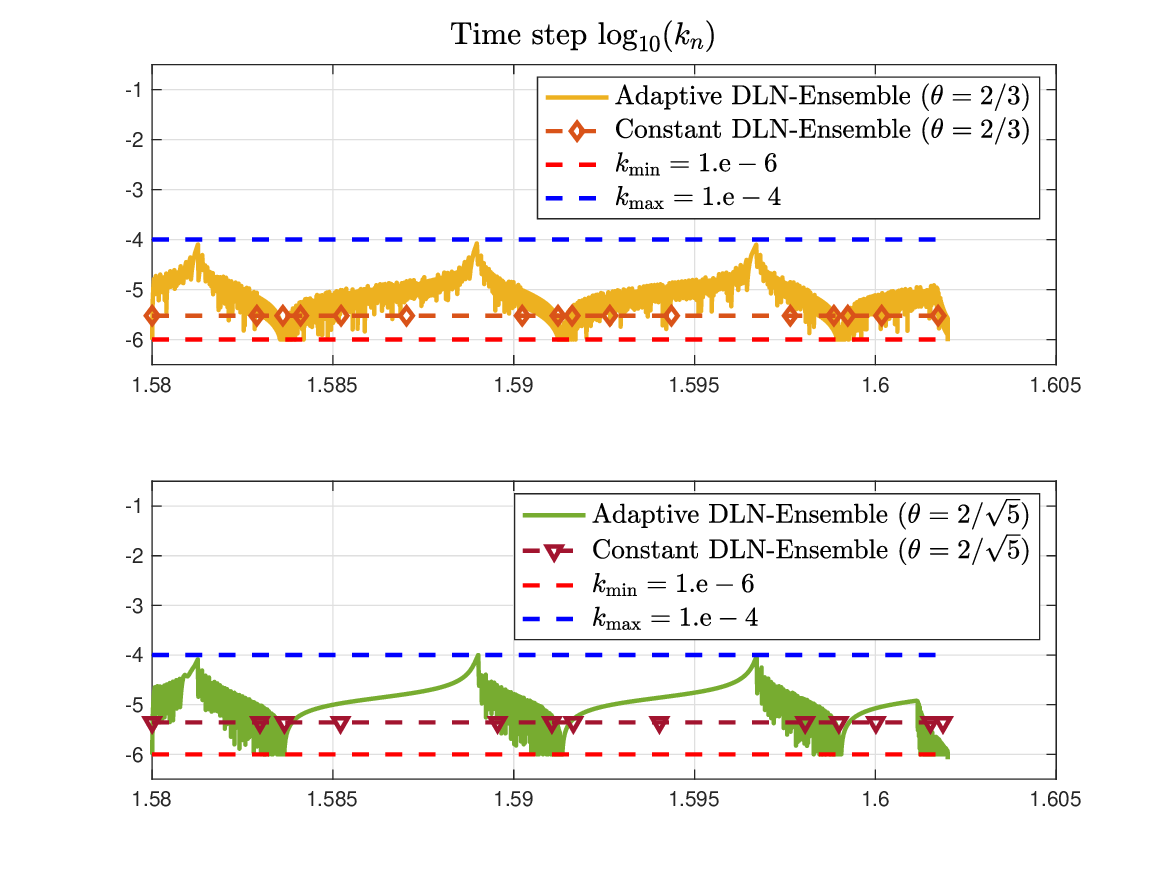}\\
				\vspace{0.02cm}
		\end{minipage}}  	
		\par
		\centering
		\vspace{-0.2cm}
		\caption{All the algorithms except the constant time-stepping algorithm with $\theta = 2/\sqrt{5}$ have similar patterns of numerical dissipation and kinetic energy dissipation rate. Adaptive algorithm with $\theta = 2/3$ is less efficient than adaptive algorithm with $\theta = 2/\sqrt{5}$ in that its $\widehat{T}_{n+1}$ reaches the required tolerance ${\tt{Tol}} = 1.\rm{e}-4$ more often and use the minimum time step size frequently.}
		\label{fig:ND-VD-LTE-Step-2nd}
	\end{figure}

	\section{Conclusions} The family of variable time-stepping DLN-Ensemble algorithms with $\theta \in (0,1)$ for multiple systems of NSE are stable and second-order convergent under the CFL-like conditions and other mild restrictions about time step ratios and diameter.  
	In practice, the algorithms can be easily implemented by refactorizing the BE-Ensemble algorithm.
	The adaptive DLN-Ensemble algorithms, utilizing the fully explicit AB2-like scheme to estimate LTE 
	with almost no additional costs, outperform the corresponding constant time-stepping algorithms in some highly stiff problems. 
	In the future, we would like to combine the family of DLN-Ensemble algorithms with other time-efficient algorithms, such as the scalar auxiliary variable (SAV) algorithm to have stable and convergent numerical solutions without the CFL-like conditions.

	\begin{appendices}
		\section{Proof of Lemma \ref{lemma:DLN-consistency}} 
		\label{appendixA} \ \\
		\begin{lemma}
			Let $\{ t_{n} \}_{n=1}^{N}$ be the time grids on the time interval $[0,T]$, $u(\cdot,t)$ the mapping from $[0,T]$ to $H^{\ell}(\Omega)$ and $u_{n}$ the function $u(\cdot,t_{n})$ in $H^{\ell}(\Omega)$. Assuming the mapping $u(\cdot,t)$ is smooth about $t$, then for any $\theta \in [0,1)$
			\begin{align}
			\| u_{n,\beta} - u(t_{n,\beta}) \|_{\ell}^{2} 
			\leq& C(\theta) (k_{n} + k_{n-1})^{3} \int_{t_{n-1}}^{t_{n+1}} \| u_{tt} \|_{\ell}^{2} dt, 
			\label{eq:Appendix-consist-2nd-eq1} \\
			\| u_{n,\ast} - u(t_{n,\beta}) \|_{\ell}^{2} 
			\leq& C(\theta) (k_{n} + k_{n-1})^{3} \int_{t_{n-1}}^{t_{n+1}} \| u_{tt} \|_{\ell}^{2} dt, 
			\label{eq:Appendix-consist-2nd-eq2} \\
			\Big\| \frac{u_{n,\alpha}}{\widehat{k}_{n}} - u_{t}(t_{n,\beta})\Big\|_{\ell}^{2}  
			\leq& C(\theta) (k_{n} + k_{n-1})^{3} \int_{t_{n-1}}^{t_{n+1}} \| u_{ttt} \|_{\ell}^{2} dt.
			\label{eq:Appendix-consist-2nd-eq3}
			\end{align}
			For $\theta = 1$, the corresponding conclusions for the midpoint rule are 
			\begin{align*}
			\big\| u_{n,\beta} - u(t_{n,\beta}) \big\|_{\ell}^{2} 
			=& \Big\| \frac{u_{n+1} + u_{n}}{2} - u \big( \frac{t_{n+1} + t_{n}}{2} \big)\Big\|_{\ell}
			\leq C k_{n}^{3} \int_{t_{n}}^{t_{n+1}} \| u_{tt} \|_{\ell}^{2} dt, \\
			\| u_{n,\ast} - u(t_{n,\beta}) \|_{\ell}^{2}
			=& \Big\| \Big( 1 + \frac{k_{n}}{2k_{n-1}} \Big) u_{n} - \frac{k_{n}}{2k_{n-1}} u_{n-1} 
			- u \big( \frac{t_{n+1} + t_{n}}{2} \big) \Big\|_{\ell} \\
			\leq& C (k_{n} + k_{n-1})^{3} \int_{t_{n-1}}^{t_{n+1}} \| u_{tt} \|_{\ell}^{2} dt, \\
			\Big\| \frac{u_{n,\alpha}}{\widehat{k}_{n}} - u_{t}(t_{n,\beta})\Big\|_{\ell}^{2}  
			=& \Big\| \frac{u_{n+1} - u_{n}}{k_{n}} - u_{t} \big( \frac{t_{n+1} + t_{n}}{2} \big) \Big\|_{\ell}
			\leq C k_{n}^{3} \displaystyle\int_{t_{n}}^{t_{n+1}} \| u_{ttt} \|_{\ell}^{2} dt. 
			\end{align*}
		\end{lemma}
		\begin{proof}
			We first prove the case $\theta \in [0,1)$. It suffices to prove the case $\ell = 0$. 
			By Taylor's theorem with integral remainder
			\begin{align}
			u ( x, t_{n+1} ) &= u ( x, t_{n} ) + u_{t} ( x, t_{n} ) {k_{n}} 
			+ \int_{t_{n}}^{t_{n+1}} u_{tt} ( x, t ) ( t_{n+1} - t ) dt, 
			\label{eq:Int-remainder} \\
			u ( x, t_{n-1} ) &= u ( x, t_{n} ) - u_{t} ( x, t_{n} ) {k_{n-1}} 
			+ \int_{t_{n}}^{t_{n-1}} u_{tt} ( x, t ) ( t_{n-1} - t ) dt, \notag \\
			u ( x, t_{n,\beta} ) &= u ( x, t_{n} ) + u_{t} ( x, t_{n} ) ( t_{n,\beta} - t_{n} ) 
			+ \int_{t_{n}}^{t_{n,\beta}} u_{tt} ( x, t ) ( t_{n,\beta} - t ) dt \notag \\
			&= u ( x, t_{n} ) + u_{t} ( x, t_{n} ) \big( {\beta_{2}^{(n)}} k_{n} - {\beta_{0}^{(n)}} k_{n-1} \big) 
			+ \int_{t_{n}}^{t_{n,\beta}} u_{tt} ( x, t ) ( t_{n,\beta} - t ) dt. \notag 
			\end{align}
			We use \eqref{eq:Int-remainder}  and the fact that $\beta_{2}^{(n)} + \beta_{1}^{(n)} + \beta_{0}^{(n)} = 1$, 
			\begin{align}
			& u_{n,\beta}(x) - u(x,t_{n,\beta})   
			\label{eq:AppendA-eq1} \\
			=& {\beta_{2}^{(n)}} \!\!\! \int_{t_{n}}^{t_{n+1}} \!\!\! u_{tt} ( x, t ) ( t_{n\!+\!1} \!-\! t ) dt 
			\!+\! {\beta_{0}^{(n)}} \!\!\! \int_{t_{n}}^{t_{n-1}} \!\!\! u_{tt} ( x, t ) ( t_{n\!-\!1} \!-\! t ) dt
			\!- \!\!\! \int_{t_{n}}^{t_{n,\beta}} \!\! u_{tt} ( x, t ) ( t_{n,\beta} \!-\! t ) dt. \notag 
			\end{align}
			By \eqref{eq:AppendA-eq1}, \eqref{eq:bound-beta} and H\"older's inequality,
			\begin{align*}
			&\big\| u_{n,\beta}(\cdot) - u(\cdot,t_{n,\beta})  \big\|^{2} \notag \\
			\leq& C(\theta)\!\! \int_{\Omega} \Big[ \int_{t_{n}}^{t_{n+1}} | u_{tt} (x,t)| (t_{n\!+\!1} \!-\! t) dt 
			+ \int_{t_{n}}^{t_{n-1}} | u_{tt} (x,t)| (t_{n\!-\!1} \!-\! t) dt \\
			& \qquad + \Big| \int_{t_{n}}^{t_{n,\beta}} | u_{tt} (x,t) | |t_{n,\beta} \!-\! t| dt \Big| \Big]^{2} dx 
			\notag \\
			\leq& C(\theta)\!\! \int_{\Omega}\!\! \Big\{\! \Big[ \!\! \int_{t_{n}}^{t_{n\!+\!1}} \!\!\! |u_{tt} (x, t)|^{2} dt  \!\! \int_{t_{n}}^{t_{n\!+\!1}} \!\!\!(t_{n+1} \!-\! t)^{2} dt \Big]^{\!\frac{1}{2}}
			\!\!+\!\! \Big[ \!\! \int_{t_{n\!-\!1}}^{t_{n}}\!\!\! |u_{tt} (x,t) |^{2} dt 
			\!\! \int_{t_{n\!-\!1}}^{t_{n}}\!\!\! (t \!-\! t_{n\!-\!1})^{2} dt \Big]^{\!\frac{1}{2}}  \notag \\
			&\qquad \qquad + \Big[ \!\!\int_{t_{n}}^{t_{n,\beta}}\!\!\! | u_{tt} (x,t)|^{2} dt  \Big|\!\! \int_{t_{n}}^{t_{n,\beta}} \!\! ( t_{n,\beta} \!-\! t )^{2} dt \Big| \Big]^{\!\frac{1}{2}} \!\Big\}^{2} dx 
			\notag \\
			\leq& C(\theta) (k_{n} + k_{n-1})^{3} \int_{\Omega} \int_{t_{n-1}}^{t_{n+1}}  |u_{tt} (x,t) |^{2} dt dx,
			\notag 
			\end{align*}
			which implies \eqref{eq:Appendix-consist-2nd-eq1}. Still by \eqref{eq:Int-remainder} and 
			$\beta_{2}^{(n)} + \beta_{1}^{(n)} + \beta_{0}^{(n)} = 1$, 
			\begin{align}
				&u_{n,\ast}(x) - u(x,t_{n,\beta}) 
				\label{eq:AppendA-eq2} \\
				=& \Big( \beta _{0}^{(n)} - \beta _{2}^{(n)} \frac{k_{n}}{k_{n-1}} \Big) 
				\int_{t_{n-1}}^{t_{n}} u_{tt} ( x, t ) ( t - t_{n-1}) dt - \int_{t_{n}}^{t_{n,\beta}} u_{tt} ( x, t ) ( t_{n,\beta} - t ) dt. \notag
			\end{align}
			By \eqref{eq:AppendA-eq2}, \eqref{eq:bound-beta} and H\"older's inequality,
			\begin{align*}
				&\big\| u_{n,\ast}(\cdot) - u(\cdot,t_{n,\beta})  \big\|^{2} \notag \\
				\leq& \int_{\Omega} \! \Big[ \Big| \Big( \beta _{0}^{(n)} \!-\! \beta _{2}^{(n)} \frac{k_{n}}{k_{n-1}} \Big) \Big|
				\int_{t_{n-1}}^{t_{n}} \!\! |u_{tt} ( x, t )| ( t - t_{n-1}) dt 
				\!+\! \Big| \int_{t_{n}}^{t_{n,\beta}} \! \! | u_{tt} (x,t) | |t_{n,\beta} \!-\! t| dt \Big| \Big]^{2} dx \\
				\leq& \int_{\Omega} \Big\{ \Big| \Big( \beta _{0}^{(n)} \!-\! \beta _{2}^{(n)} \frac{k_{n}}{k_{n-1}} \Big) \Big| 
				\Big[ \!\! \int_{t_{n\!-\!1}}^{t_{n}}\!\!\! |u_{tt} (x,t) |^{2} dt 
				\!\! \int_{t_{n\!-\!1}}^{t_{n}}\!\!\! (t \!-\! t_{n\!-\!1})^{2} dt \Big]^{\!\frac{1}{2}}  \notag \\
				&\qquad \qquad + \Big[ \Big|\!\!\int_{t_{n}}^{t_{n,\beta}}\!\!\! | u_{tt} (x,t)|^{2} dt \Big| \Big| \!\! \int_{t_{n}}^{t_{n,\beta}} \!\! ( t_{n,\beta} \!-\! t )^{2} dt\Big| \Big]^{\!\frac{1}{2}} \!\Big\}^{2} dx 
				\notag \\
				\leq& C(\theta) (k_{n} + k_{n-1})^{3} \int_{\Omega} \int_{t_{n-1}}^{t_{n+1}}  |u_{tt} (x,t) |^{2} dt dx,
			\end{align*}
			which implies \eqref{eq:Appendix-consist-2nd-eq2}.
			\begin{confidential}
				\color{darkblue}
				\begin{align*}
					&u_{n,\ast}(x) - u(x,t_{n,\beta}) \\
					=& \Big( \beta _{2}^{(n)} \big( 1 + \frac{k_{n}}{k_{n-1}} \big) + \beta _{1}^{(n)} \Big) u_{n}(x)
					+ \Big( \beta _{0}^{(n)} - \beta _{2}^{(n)} \frac{k_{n}}{k_{n-1}} \Big) u_{n-1}(x)
					- u(x,t_{n,\beta}) \\ 
					=& \Big( \beta _{2}^{(n)} \big( 1 + \frac{k_{n}}{k_{n-1}} \big) + \beta _{1}^{(n)} \Big) u(x,t_{n})
					+ \Big( \beta _{0}^{(n)} - \beta _{2}^{(n)} \frac{k_{n}}{k_{n-1}} \Big) 
					\Big\{ u ( x, t_{n} ) - u_{t} ( x, t_{n} ) {k_{n-1}} + \int_{t_{n}}^{t_{n-1}} u_{tt} ( x, t ) ( t_{n-1} - t ) dt \Big\} \\
					&- \Big\{ u ( x, t_{n} ) + u_{t} ( x, t_{n} ) \big( {\beta_{2}^{(n)}} k_{n} - {\beta_{0}^{(n)}} k_{n-1} \big) 
					+ \int_{t_{n}}^{t_{n,\beta}} u_{tt} ( x, t ) ( t_{n,\beta} - t ) dt \Big\} \\
					=& \Big[  \beta _{2}^{(n)} + \beta _{2}^{(n)}\frac{k_{n}}{k_{n-1}} + \beta _{1}^{(n)} 
					+ \beta _{0}^{(n)} - \beta _{2}^{(n)} \frac{k_{n}}{k_{n-1}} - 1 \Big] u(x,t_{n})
					+ \Big[ - \big( \beta _{0}^{(n)} k_{n-1} - \beta _{2}^{(n)} k_{n} \Big)  
					- \big( {\beta_{2}^{(n)}} k_{n} - {\beta_{0}^{(n)}} k_{n-1} \big) \Big] u_{t} ( x, t_{n} ) \\
					&+ \Big( \beta _{0}^{(n)} - \beta _{2}^{(n)} \frac{k_{n}}{k_{n-1}} \Big) 
					\int_{t_{n-1}}^{t_{n}} u_{tt} ( x, t ) ( t - t_{n-1}) dt
					- \int_{t_{n}}^{t_{n,\beta}} u_{tt} ( x, t ) ( t_{n,\beta} - t ) dt.
				\end{align*}
				\begin{align*}
					&\int_{\Omega} \Big\{ \Big| \Big( \beta _{0}^{(n)} \!-\! \beta _{2}^{(n)} \frac{k_{n}}{k_{n-1}} \Big) \Big| 
					\Big[ \!\! \int_{t_{n\!-\!1}}^{t_{n}}\!\!\! |u_{tt} (x,t) |^{2} dt 
					\!\! \int_{t_{n\!-\!1}}^{t_{n}}\!\!\! (t \!-\! t_{n\!-\!1})^{2} dt \Big]^{\!\frac{1}{2}}  + \Big[ \Big|\!\! \int_{t_{n}}^{t_{n,\beta}}\!\!\! | u_{tt} (x,t)|^{2} dt \Big| \Big| \!\! \int_{t_{n}}^{t_{n,\beta}} \!\! ( t_{n,\beta} \!-\! t )^{2} dt \Big| \Big]^{\!\frac{1}{2}} \!\Big\}^{2} dx \\
					=& \int_{\Omega} \Big\{ \Big| \Big( \beta _{0}^{(n)} \!-\! \beta _{2}^{(n)} \frac{k_{n}}{k_{n-1}} \Big) \Big| 
					\Big[ \!\! \int_{t_{n\!-\!1}}^{t_{n}}\!\!\! |u_{tt} (x,t) |^{2} dt 
					\Big( \frac{1}{3} (t - t_{n-1} )^{3} \Big|_{t_{n-1}}^{t_{n}} \Big) \Big]^{\!\frac{1}{2}}  + \Big[ \Big| \!\!\int_{t_{n}}^{t_{n,\beta}}\!\!\! | u_{tt} (x,t)|^{2} dt \Big| \Big( \Big| \frac{1}{3} ( t - t_{n,\beta})^{3} \Big|_{t_{n}}^{t_{n,\beta}} \Big| \Big) \Big]^{\!\frac{1}{2}} \!\Big\}^{2} dx \\
					=& \int_{\Omega} \Big\{ \Big| \Big( \beta _{0}^{(n)} \!-\! \beta _{2}^{(n)} \frac{k_{n}}{k_{n-1}} \Big) \Big| 
					\Big[ \frac{k_{n-1}^{3}}{3} \int_{t_{n\!-\!1}}^{t_{n}}\!\!\! |u_{tt} (x,t) |^{2} dt  \Big]^{\!\frac{1}{2}}  
					+ \Big[  \frac{|t_{n,\beta} - t_{n}|^{3}}{3} \Big| \!\! \int_{t_{n}}^{t_{n,\beta}}\!\!\! | u_{tt} (x,t)|^{2} dt \Big| \Big]^{\!\frac{1}{2}} \!\Big\}^{2} dx \\
					=& \int_{\Omega} \Big\{ \sqrt{\frac{k_{n-1}}{3}} \Big|\beta _{2}^{(n)}k_{n} - \beta _{0}^{(n)} k_{n-1} \Big| 
					\Big[ \int_{t_{n\!-\!1}}^{t_{n}}\!\!\! |u_{tt} (x,t) |^{2} dt  \Big]^{\!\frac{1}{2}}  
					+ \frac{\Big|\beta _{2}^{(n)}k_{n} - \beta _{0}^{(n)} k_{n-1} \Big|^{3/2} }{\sqrt{3}}
					\Big[  \!\! \int_{t_{n}}^{t_{n,\beta}}\!\!\! | u_{tt} (x,t)|^{2} dt \Big| \Big]^{\!\frac{1}{2}} \!\Big\}^{2} dx \\
					\leq& C(\theta) \big( k_{n} + k_{n-1} \big)^{3}  \int_{\Omega} \Big\{ \Big[ \int_{t_{n\!-\!1}}^{t_{n}}\!\!\! |u_{tt} (x,t) |^{2} dt  \Big]^{\!\frac{1}{2}} + \Big[  \Big|\!\! \int_{t_{n}}^{t_{n,\beta}}\!\!\! | u_{tt} (x,t)|^{2} dt \Big| \Big]^{\!\frac{1}{2}} \Big\}^{2} dx \\
					\leq& C(\theta) \big( k_{n} + k_{n-1} \big)^{3} \int_{\Omega} \Big[ \int_{t_{n\!-\!1}}^{t_{n}}\!\!\! |u_{tt} (x,t) |^{2} dt +\Big| \int_{t_{n}}^{t_{n,\beta}}\!\!\! | u_{tt} (x,t)|^{2} dt \Big| \Big] dx 
					\leq C(\theta) \big( k_{n} + k_{n-1} \big)^{3}  
					\int_{\Omega} \int_{t_{n-1}}^{t_{n+1}} \!\!\! |u_{tt} (x,t) |^{2} dt dx.
				\end{align*}
				\normalcolor
			\end{confidential}
			For \eqref{eq:Appendix-consist-2nd-eq3}, we still use the Taylor's theorem with integral remainder and obtain
			\begin{align}
			u (x,t_{n+1}) \!&=\! u(x,t_{n}) \!+\! u_{t}(x, t_{n}) {k_{n}} 
			\!+\! u_{tt}(x,t_{n}) \frac{k_{n}^{2}}{2} \!+\! \int_{t_{n}}^{t_{n+1}} \!u_{ttt}(x,t)
			\frac{(t_{n\!+\!1}\!-\!t)^{2}}{2} dt, 
			\label{eq:Int-remainder-2} \\
			u(x,t_{n-1}) \!&=\! u(x,t_{n}) \!-\! u_{t}(x, t_{n}) {k_{n-1}} 
			\!+\! u_{tt}(x,t_{n}) \frac{k_{n-1}^{2}}{2} \!+\! \int_{t_{n}}^{t_{n-1}} \!u_{ttt} (x,t) 
			\frac{(t_{n-1} - t)^{2}}{2} dt, \notag \\
			u_{t}(x, t_{n,\beta}) &= u_{t} (x,t_{n}) \!+\! u_{tt} (x,t_{n}) ( {\beta_{2}^{(n)}} k_{n} - {\beta_{0}^{(n)}} k_{n-1} ) \!+\! \int_{t_{n}}^{t_{n,\beta}} \!u_{ttt} (x,t) (t_{n,\beta} - t) dt. \notag 
			\end{align}
			By \eqref{eq:Int-remainder-2} and the fact that $\alpha_{2} + \alpha_{1} + \alpha_{0} = 0$
			\begin{align}
			&\frac{u_{n,\alpha}(x)}{\widehat{k}_{n}} - u_{t}(t_{n,\beta}) 
			\label{eq:AppendA-eq3} \\
			=& \Big[ \frac{{\alpha_{2}} {k_{n}}^{2} + {\alpha_{0}} {k_{n-1}}^{2}}{2 \widehat{k}_{n}} - \big( {\beta_{2}^{(n)}} k_{n} - {\beta_{0}^{(n)}} k_{n-1} \big) \Big] u_{tt}(x,t_{n}) - \int_{t_{n}}^{t_{n,\beta}} \!u_{ttt} (x,t) (t_{n,\beta} - t) dt \notag \\
			&+ \frac{{\alpha_{2}}}{2 \widehat{k}_{n}} \!\!\int_{t_{n}}^{t_{n\!+\!1}} \!\!u_{ttt} (x,t) (t_{n\!+\!1} \!-\! t)^{2} dt 
			+ \frac{{\alpha_{0}}}{2 \widehat{k}_{n}} \!\!\int_{t_{n}}^{t_{n\!-\!1}} \!\!u_{ttt} (x,t) (t_{n\!-\!1} \!-\! t)^{2} dt. \notag 
			\end{align}
			It's easy to check
			\begin{gather}
			\frac{{\alpha_{2}} {k_{n}}^{2} + {\alpha_{0}} {k_{n-1}}^{2}}{2 \widehat{k}_{n}} - \big( {\beta_{2}^{(n)}} k_{n} - {\beta_{0}^{(n)}} k_{n-1} \big) = 0.
			\label{eq:AppendA-eq4}
			\end{gather}
			We combine \eqref{eq:AppendA-eq3} and \eqref{eq:AppendA-eq4}
			\begin{align}
			&\Big\| \frac{u_{n,\alpha}(\cdot)}{\widehat{k}_{n}} - u_{t} (x,t_{n,\beta}) \Big\|^{2} 
			\label{eq:AppendA-eq5} \\
			\leq& \frac{C(\theta)}{{\widehat{k}_{n}}^{2}} \int_{\Omega} \Big[ \int_{t_{n}}^{t_{n+1}} | u_{ttt} (x,t) | (t_{n+1} - t)^{2} dt + \int_{t_{n-1}}^{t_{n}} | u_{ttt} (x,t) | (t_{n-1} - t)^{2} dt \notag \\
			&\qquad \qquad + \widehat{k}_{n}  \Big| \int_{t_{n}}^{t_{n,\beta}} | u_{ttt} (x,t) | | t_{n,\beta} - t | dt \Big| \Big]^{2} dx \notag \\
			\leq& \!\frac{C(\theta)}{\widehat{k}_{n}^{2}} \!\! \int_{\Omega}\!\!\! \Big\{\! \Big[ \!\! \int_{t_{n}}^{t_{n\!+\!1}} \!\!\! |u_{ttt} (x,t)|^{2} dt  \!\!\! \int_{t_{n}}^{t_{n\!+\!1}} \!\!\!(t_{n+1} \!-\! t)^{4} dt \Big]^{\!\frac{1}{2}}
			\!\!+\!\! \Big[ \!\! \int_{t_{n\!-\!1}}^{t_{n}}\!\!\! |u_{ttt} (x,t)|^{2} dt 
			\!\!\! \int_{t_{n\!-\!1}}^{t_{n}}\!\!\! (t \!-\! t_{n\!-\!1})^{4} dt \Big]^{\!\frac{1}{2}}  \notag \\
			&\qquad \qquad + \widehat{k}_{n} \Big[ \!\!\int_{t_{n}}^{t_{n,\beta}}\!\!\! | u_{ttt} (x,t)|^{2} dt \!\!\! \int_{t_{n}}^{t_{n,\beta}} \!\! (t_{n,\beta} \!-\! t)^{2} dt \Big]^{\!\frac{1}{2}} \!\Big\}^{2} dx 
			\notag \\
			\leq& C(\theta) \Big(\frac{k_{n} + k_{n-1}}{\widehat{k}_{n}} \Big)^{2} (k_{n} + k_{n-1})^{3} \int_{\Omega} \int_{t_{n-1}}^{t_{n+1}} |u_{ttt} (x,t)|^{2} dt dx. \notag 
			\end{align}
			Since 
			\begin{gather*}
			\frac{1}{\widehat{k}_{n}} = \frac{1}{\frac{1+\theta}{2}k_{n} + \frac{1-\theta}{2}k_{n-1} }
			\leq \frac{2}{(1-\theta) (k_{n}+k_{n-1})},
			\end{gather*}
			we have \eqref{eq:Appendix-consist-2nd-eq2} from \eqref{eq:AppendA-eq5}. For $\theta=1$, the DLN method is reduced to midpoint rule and the corresponding conclusions are easy to verify.
			\begin{confidential}
				\color{darkblue}
			For $\theta = 1$, we denote $t_{n+\frac{1}{2}} = (t_{n+1} + t_{n})/2$ and use the Taylor's theorem with integral remainder
			\begin{align}
			u (x,t_{n+1}) &= u(x,t_{n+\frac{1}{2}}) \!+\! u_{t}(x,t_{n+\frac{1}{2}}) \frac{k_{n}}{2}
			\!+\! \int_{t_{n+\frac{1}{2}}}^{t_{n+1}} \!u_{tt}(x,t)(t_{n+1}-t) dt, 
			\label{eq:Int-remainder-3} \\
			u (x,t_{n+1}) &= u(x,t_{n+\frac{1}{2}}) \!+\! u_{t}(x,t_{n+\frac{1}{2}}) \frac{k_{n}}{2} 
			\!+\! u_{tt}(x,t_{n}) \frac{k_{n}^{2}}{8} \!+\! \int_{t_{n+\frac{1}{2}}}^{t_{n+1}} \!u_{ttt}(x,t)
			\frac{(t_{n+1}-t)^{2}}{2} dt, \notag \\
			u(x,t_{n}) &= u(x,t_{n+\frac{1}{2}}) \!-\! u_{t}(x,t_{n+\frac{1}{2}}) \frac{k_{n}}{2} 
			\!+\! \int_{t_{n+\frac{1}{2}}}^{t_{n}} \!u_{tt} (x,t) (t_{n} - t) dt, \notag \\
			u(x,t_{n}) &= u(x,t_{n+\frac{1}{2}}) \!-\! u_{t}(x,t_{n+\frac{1}{2}}) \frac{k_{n}}{2} 
			\!+\! u_{tt}(x,t_{n}) \frac{k_{n}^{2}}{8} \!+\! \int_{t_{n+\frac{1}{2}}}^{t_{n}} \!u_{ttt} (x,t) 
			\frac{(t_{n} - t)^{2}}{2} dt. \notag 
			\end{align}
			By \eqref{eq:Int-remainder-3}
			\begin{align}
			u_{n,\beta}(x) - u(x,t_{n,\beta})
			=& \frac{u_{n+1}(x) + u_{n}(x)}{2} - u \big(x, t_{n+1/2} \big) 
			\label{eq:AppendA-eq6} \\
			=& \frac{1}{2} \int_{t_{n+\frac{1}{2}}}^{t_{n+1}} \!u_{tt}(x,t)(t_{n+1}-t) dt
			+ \frac{1}{2} \int_{t_{n+\frac{1}{2}}}^{t_{n}} \!u_{tt} (x,t) (t_{n} - t) dt, 
			\notag \\
			\frac{u_{n,\alpha}(x)}{\widehat{k}_{n}} - u_{t}(x,t_{n,\beta}) 
			=& \frac{u_{n+1}(x) - u_{n}(x)}{k_{n}} - u_{t} (t_{n+\frac{1}{2}}) \notag \\
			=& \frac{1}{2k_{n}} \Big( \int_{t_{n+\frac{1}{2}}}^{t_{n+1}} \!u_{ttt}(x,t)(t_{n+1}-t)^{2} dt
			- \int_{t_{n+\frac{1}{2}}}^{t_{n}} \!u_{ttt} (x,t) (t_{n} - t)^{2} dt \Big).  \notag 
			\end{align}
			By \eqref{eq:AppendA-eq5}
			\begin{align*}
			&\big\| u_{n,\beta}(\cdot) - u(\cdot,t_{n,\beta}) \big\|^{2} \\
			\leq& \frac{1}{4} \int_{\Omega}\!\!\! \Big[ \int_{t_{n+\frac{1}{2}}}^{t_{n+1}} | u_{tt} (x,t) | |t_{n+1} - t| dt + \int_{t_{n}}^{t_{n+\frac{1}{2}}} | u_{tt} (x,t) | |t_{n} - t| dt \Big]^{2} dx \notag \\
			\leq& \!\frac{1}{4} \!\! \int_{\Omega}\!\!\! \Big\{\! \Big[ \!\! \int_{t_{n+\frac{1}{2}}}^{t_{n\!+\!1}} \!\!\! |u_{tt} (x,t)|^{2} dt  \!\!\! \int_{t_{n+\frac{1}{2}}}^{t_{n\!+\!1}} \!\!\!(t_{n+1} \!-\! t)^{2} dt \Big]^{\!\frac{1}{2}}
			\!\!+\!\! \Big[ \!\! \int_{t_{n}}^{t_{n+\frac{1}{2}}}\!\!\! |u_{tt} (x,t)|^{2} dt 
			\!\!\! \int_{t_{n}}^{t_{n+\frac{1}{2}}}\!\!\! (t \!-\! t_{n})^{2} dt \Big]^{\!\frac{1}{2}} \Big\}^{2} dx \\
			\leq& \!\frac{k_{n}^{3}}{48} \int_{\Omega} \int_{t_{n}}^{t_{n+1}} |u_{tt} (x,t)|^{2} dt dx,
			\\
			&\Big\| \frac{u_{n,\alpha}(\cdot)}{\widehat{k}_{n}} - u_{t} (x,t_{n,\beta}) \Big\|^{2} \\
			\leq& \frac{1}{4 k_{n}^{2}} \int_{\Omega} \Big[ \int_{t_{n+\frac{1}{2}}}^{t_{n+1}} | u_{ttt} (x,t) | (t_{n+1} - t)^{2} dt + \int_{t_{n}}^{t_{n+\frac{1}{2}}} | u_{ttt} (x,t) | (t_{n} - t)^{2} dt \Big]^{2} dx \notag \\
			\leq& \!\frac{1}{4 k_{n}^{2}} \!\! \int_{\Omega}\!\!\! \Big\{\! \Big[ \!\! \int_{t_{n+\frac{1}{2}}}^{t_{n\!+\!1}} \!\!\! |u_{ttt} (x,t)|^{2} dt  \!\!\! \int_{t_{n+\frac{1}{2}}}^{t_{n\!+\!1}} \!\!\!(t_{n+1} \!-\! t)^{4} dt \Big]^{\!\frac{1}{2}}
			\!\!+\!\! \Big[ \!\! \int_{t_{n}}^{t_{n+\frac{1}{2}}}\!\!\! |u_{ttt} (x,t)|^{2} dt 
			\!\!\! \int_{t_{n}}^{t_{n+\frac{1}{2}}}\!\!\! (t \!-\! t_{n})^{4} dt \Big]^{\!\frac{1}{2}} \Big\}^{2} dx \\
			=& \frac{k_{n}^{3}}{320} \int_{\Omega} \int_{t_{n}}^{t_{n+1}} |u_{ttt} (x,t)|^{2} dt dx, 
			\end{align*}
			which implies \eqref{eq:Appendix-consist-2nd-eq2} for the case $\theta = 1$.
			\normalcolor
			\end{confidential}
		\end{proof}

		\section{Error Analysis} 

		\subsection{Proof of Theorem \ref{thm:Error-L2}} 
		\label{appendixB-L2} \ \\
		\begin{theorem}
			If the velocity $u^{j}(x,t)$ and pressure $p^{j}(x,t)$ in the $j$-th NSE satisfy the following regularities:
			\begin{gather}
			u \in \ell^{\infty}(\{t_{n}\}_{n=0}^{N};H^{r}(\Omega))  \cap \ell^{\infty,\beta}(\{t_{n}\}_{n=0}^{N};H^{r+1}(\Omega)), \ p \in \ell^{2,\beta}\big( \{t_{n}\}_{n=0}^{N};H^{s+1}(\Omega) \big) \notag \\
			u_{t} \in L^{2}\big( 0,T;H^{r+1}(\Omega) \big),  \ \ 
			u_{tt} \in L^{2}\big(0,T;H^{r+1}(\Omega) \big), \ \ 
			u_{ttt} \in L^{2}\big( 0,T;X^{-1} \big),
			\label{eq:appendixB-L2-regularity}
			\end{gather}
			for all $j$, then under CFL-like conditions in \eqref{eq:CFL-like-cond} and time ratio bounds in \eqref{eq:time-ratio-cond}, the numerical solutions DLN-Ensemble algorithms in \eqref{eq:DLN-Ensemble-Alg} for all $\theta \in (0,1)$ satisfy
			\begin{gather}
			\max_{0 \leq n \leq N} \| u_{n}^{j} - u_{n}^{j,h} \| + \Big(  \nu  \sum_{n=1}^{N-1} \widehat{k}_{n} 
			\| \nabla \big( u_{n,\beta}^{j} - u_{n,\beta}^{j,h} \big) \|^{2} \Big)^{1/2} 
			\label{eq:appendixB-error-L2-conclusion} \\
			\leq  \mathcal{O} \big( h^{r} ,h^{s+1} , k_{\rm{max}}^{2}, h^{1/2} k_{\rm{max}}^{3/2} \big).
			\notag
			\end{gather}
		\end{theorem}
		\begin{proof}
			We devide the proof into four parts:
		\begin{enumerate}
				\item[1.] We combine $j$-th NSE in \eqref{eq:jth-NSE} at time $t_{n,\beta}$ and the DLN-Ensemble algorithms for $j$-th NSE in \eqref{eq:DLN-Ensemble-Alg} to derive the equation of pointwise error $e_{n+1}:= u_{n+1}^{j} - u_{n+1}^{j,h}$
				\item[2.] We set $I_{\rm{St}} u_{n}^{j}$ to be the velocity component of 
				Stokes projection of $(u_{n}^{j},0)$ onto $V^{h} \times Q^{h}$ and decompose the error to be $e_{n}^{j} = (u_{n}^{j} - I_{\rm{St}} u_{n}^{j} ) + (I_{\rm{St}} u_{n}^{j} - u_{n}^{j,h}) 
				:= \eta_{n}^{j} + \phi_{n}^{j,h}$.
			Then we transfer the error equation to be the new equation 
			in terms of $\{\eta_{n-1+\ell}^{j}\}_{\ell=0}^{2}$ and $\{ \phi_{n-1+\ell}^{j} \}_{\ell=0}^{2}$.
			\item[3.] We obtain the bound for $\phi_{n+1}^{j,h}$ by addressing the terms from
			\begin{itemize}
				\item the semi-implicit DLN algorithms \cite{Pei24_NM} for $j$-th NSE in \eqref{eq:jth-NSE},
				\item the DLN-Ensemble algorithms in \eqref{eq:DLN-Ensemble-Alg}.
			\end{itemize}
			\item[4.] We use the discrete Gr\"onwall inequality without time restrictions \cite[p.369]{HR90_SIAM_NA}, approximations for Stokes projection in \eqref{eq:Stoke-Approx} and approximations for $(X^{h},Q^{h})$ in \eqref{eq:approx-thm} to achieve convergence of numerical solutions 
			in $L^{2}$-norm.
		\end{enumerate}
		\ \\
		\noindent \textit{Part 1.} \\
			The exact solutions of $j$-th NSE at time $t_{n,\beta}$ satisfies 
			\begin{gather}
				\big( u_{t}^{j} ( t_{n,\beta}) , v^{h} \big) + b \big( u^{j} ( t_{n,\beta}) , u^{j} ( t_{n,\beta} ) , v^{h} \big) + \nu \big( \nabla u^{j} (t_{n,\beta}) , \nabla v^{h} \big)
				- \big( p^{j} (t_{n,\beta}) , \nabla \cdot v^{h} \big) \notag \\
				= \big( f^{j} ( t_{n,\beta}), v^{h} \big), \qquad \forall v^{h} \in V^{h}.
				\label{eq:NSE-jth-exact}
			\end{gather}
			We restrict $v^{h} \in V^{h}$ in \eqref{eq:DLN-Ensemble-Alg} and subtract first equation of \eqref{eq:DLN-Ensemble-Alg} from \eqref{eq:NSE-jth-exact}
			\begin{gather}
				\frac{1}{\widehat{k}_{n}} \big( u_{n,\alpha}^{j} - u_{n,\alpha}^{j,h} , v^{h} \big) + b\big( u^{j} (t_{n,\beta}) , u^{j} (t_{n,\beta}) , v^{h} \big)  
				- b \big( \langle u^{h} \rangle_{n,\ast}, u_{n,\beta}^{j,h} , v^{h} \big) \notag \\
				- b \big( u_{n,\ast}^{j,h} \!-\! \langle u^{h} \rangle_{n,\ast} , u_{n,\ast}^{j,h}, v^{h} \big) - \big( p^{j} (t_{n,\beta}) , \nabla \cdot v^{h} \big) 
				+ \nu \big( \nabla \big( u^{j} (t_{n,\beta}) -  u_{n,\beta}^{j,h} \big) , \nabla v^{h} \big) 
				\notag \\
				= \frac{1}{\widehat{k}_{n}} \big( u_{n,\alpha}^{j}, v^{h} \big)
				- \big( u_{t}^{j}( t_{n,\beta}) , v^{h} \big). 
				\label{eq:DLN-Error-Eq1}
			\end{gather}
			We denote the error of velocity of $j$-th NSE at time $t_{n}$ to be $e_{n}^{j} = u_{n}^{j} - u_{n}^{j,h}$. \eqref{eq:DLN-Error-Eq1} becomes 
			\begin{gather}
				\frac{1}{\widehat{k}_{n}} \big( e_{n,\alpha}^{j}, v^{h} \big) 
				\!+\! b\big( u^{j} (t_{n,\beta}) , u^{j} (t_{n,\beta}) , v^{h} \big) 
				\!-\! b \big( \langle u^{h} \rangle_{n,\ast}, u_{n,\beta}^{j,h} , v^{h} \big)
				\!-\! b \big( u_{n,\ast}^{j,h} \!-\! \langle u^{h} \rangle_{n,\ast} , u_{n,\ast}^{j,h}, v^{h} \big)
				\notag \\
				- \big( p^{j} (t_{n,\beta}) , \nabla \cdot v^{h} \big)  
				+ \nu \big( \nabla \big( u^{j} ( t_{n,\beta}) - u_{n,\beta}^{j} \big), \nabla v^{h} \big)
				+ \nu \big( \nabla e_{n,\beta}^{j}, \nabla v^{h} \big) \notag \\
				= \frac{1}{\widehat{k}_{n}} \big( u_{n,\alpha}^{j}, v^{h} \big)
				- \big( u_{t}^{j}( t_{n,\beta}) , v^{h} \big). 
				\label{eq:DLN-Error-Eq2} 
			\end{gather}
			\ \\
			\noindent \textit{Part 2.} \\
			Let $I_{\rm{St}} u_{n}^{j} \in V^{h}$ be the first component of Stokes projection of $(u_{n}^{j},0)$ onto $V^{h} \times Q^{h}$. We denote $\eta_{n}^{j} = u_{n}^{j} - I_{\rm{St}} u_{n}^{j}$
			and  $\phi_{n}^{j,h} = I_{\rm{St}} u_{n}^{j} - u_{n}^{j,h}$.
			Thus $e_{n}^{j} = \eta_{n}^{j} + \phi_{n}^{j,h}$ and \eqref{eq:DLN-Error-Eq2} becomes
			\begin{confidential}
				\color{darkblue}
				\begin{gather*}
					\frac{1}{\widehat{k}_{n}} \big( \eta_{n,\alpha}^{j} , v^{h} \big) 
					+ \frac{1}{\widehat{k}_{n}} \big( \phi_{n,\alpha}^{j,h} , v^{h} \big) + b \big( u^{j} ( t_{n,\beta}) , u^{j} (t_{n,\beta}) , v^{h} \big) \notag \\
					- b \big( u_{n,\ast}^{j,h} - \big( u_{n,\ast}^{j,h} \!-\! \langle u^{h} \rangle_{n,\ast} \big), u_{n,\beta}^{j,h}, v^{h} \big) 
					- b \big( u_{n,\ast}^{j,h} \!-\! \langle u^{h} \rangle_{n,\ast}, u_{n,\ast}^{j,h}, v^{h} \big) - \big( p^{j} (t_{n,\beta}), \nabla \cdot v^{h} \big) \notag \\
					+ \nu \big( \nabla \big( u^{j} ( t_{n,\beta}) - u_{n,\beta}^{j} \big) + \nabla \big(u_{n,\beta}^{j} - u_{n,\beta}^{j,h} \big), \nabla v^{h} \big) 
					= \frac{1}{\widehat{k}_{n}} \big( u_{n,\alpha}^{j}, v^{h} \big)
					- \big( u_{t}^{j} (t_{n,\beta}), v^{h} \big).   
				\end{gather*}
				\begin{gather*}
					\frac{1}{\widehat{k}_{n}}
					\big( \phi_{n,\alpha}^{j,h} , v^{h} \big)
					+ b \big( u^{j}( t_{n,\beta}), u^{j} (t_{n,\beta}) , v^{h} \big)  
					- b \big(  u_{n,\ast}^{j,h} - \big( u_{n,\ast}^{j,h} \!-\! \langle u^{h} \rangle_{n,\ast} \big), u_{n,\beta}^{j,h}, v^{h} \big) \notag \\
					- b \big( u_{n,\ast}^{j,h} \!-\! \langle u^{h} \rangle_{n,\ast}, u_{n,\ast}^{j,h}, v^{h} \big) 
					- \big( p^{j} (t_{n,\beta}), \nabla \cdot v^{h} \big) 
					+ \nu \big( \nabla \big(u^{j} (t_{n,\beta}) - u_{n,\beta}^{j} \big), \nabla v^{h} \big)
					\notag \\
					+ \nu \big( \nabla \big( \eta_{n,\beta}^{j} + \phi_{n,\beta}^{j,h} \big), \nabla v^{h} \big)
					= \Big( \frac{u_{n,\alpha}^{j}}{\widehat{k}_{n}} - u_{t}^{j} (t_{n,\beta}), v^{h} \Big)  
					- \frac{1}{\widehat{k}_{n}} \big( \eta_{n,\alpha}^{j}, v^{h} \big).   
				\end{gather*}
				\begin{align*}
					&\frac{1}{\widehat{k}_{n}} \big( \phi_{n,\alpha}^{j,h} , v^{h} \big)
					\!+\! b \big( u^{j}( t_{n,\beta}), u^{j} (t_{n,\beta}) , v^{h} \big)  
					\!+\! \nu \big( \nabla \phi_{n,\beta}^{j,h}, \nabla v^{h} \big)  
					\!-\! \big( p^{j} (t_{n,\beta}), \nabla \cdot v^{h} \big)
					\notag \\
					&- \!b \big( u_{n,\ast}^{j,h} - \big( u_{n,\ast}^{j,h} \!-\! \langle u^{h} \rangle_{n,\ast} \big), u_{n,\beta}^{j,h}, v^{h} \big)
					\!-\! b \big( u_{n,\ast}^{j,h} \!-\! \langle u^{h} \rangle_{n,\ast}, u_{n,\ast}^{j,h}, v^{h} \big) 
					\!+\! \nu \big( \nabla \big(u^{j} (t_{n,\beta}) \!-\! u_{n,\beta}^{j} \big), \nabla v^{h} \big)
					\notag \\
					&= \Big( \frac{u_{n,\alpha}^{j}}{\widehat{k}_{n}} - u_{t}^{j} (t_{n,\beta}), v^{h} \Big) 
					- \frac{1}{\widehat{k}_{n}} \big( \eta_{n,\alpha}^{j}, v^{h} \big)
					- \nu \big( \nabla \eta_{n,\beta}^{j} , \nabla v^{h} \big). 
				\end{align*}
	
				\normalcolor
			\end{confidential}
			\begin{align}
				&\frac{1}{\widehat{k}_{n}} \big( \phi_{n,\alpha}^{j,h} , v^{h} \big)
				\!+\! \nu \big( \nabla \phi_{n,\beta}^{j,h}, \nabla v^{h} \big)
				\label{eq:DLN-Error-Eq3} \\
				=& \Big( \frac{u_{n,\alpha}^{j}}{\widehat{k}_{n}} \!-\! u_{t}^{j} (t_{n,\beta}), v^{h} \Big) 
				\!-\! \frac{1}{\widehat{k}_{n}} \big( \eta_{n,\alpha}^{j}, v^{h} \big)
				\!+\! \nu \! \big( \nabla \big(u_{n,\beta}^{j} \!-\! u^{j} (t_{n,\beta}) \big), \nabla v^{h} \big)  
				+ \big( p^{j} (t_{n,\beta}), \nabla \cdot v^{h} \big) \notag \\
				&- \nu \big( \nabla \eta_{n,\beta}^{j} , \nabla v^{h} \big) 
				\!-\! b \big( u^{j}( t_{n,\beta}), u^{j} (t_{n,\beta}) , v^{h} \big) 
				\!+\!b \big( u_{n,\ast}^{j,h} - \big( u_{n,\ast}^{j,h} \!-\! \langle u^{h} \rangle_{n,\ast} \big), u_{n,\beta}^{j,h}, v^{h} \big) \notag \\
				&+\! b \big( u_{n,\ast}^{j,h} \!-\! \langle u^{h} \rangle_{n,\ast}, u_{n,\ast}^{j,h}, v^{h} \big). \notag 
			\end{align}
			We denote 
			\begin{gather*}
				\Phi_{n}^{j,h} = \frac{1 - \varepsilon_{n}}{2} \phi_{n+1}^{j,h} - \phi_{n}^{j,h} + \frac{1 + \varepsilon_{n}}{2} \phi_{n-1}^{j,h},
			\end{gather*}
			and set $v^{h} = \phi_{n,\beta}^{j,h}$ in \eqref{eq:DLN-Error-Eq3}. By the identity in \eqref{eq:Gstab-Id}, \eqref{eq:DLN-Error-Eq3} becomes
			\begin{confidential}
				\color{darkblue}
				\begin{align*}
					&\frac{1}{\widehat{k}_{n}} \Big( \begin{Vmatrix} {\phi_{n+1}^{j,h}} \\ {\phi_{n}^{j,h}} \end{Vmatrix}^{2}_{G(\theta)}-\begin{Vmatrix} {\phi_{n}^{j,h}} \\ {\phi_{n-1}^{j,h}} \end{Vmatrix}^{2}_{G(\theta)} + \frac{ \theta (1 - {\theta}^{2})}{2 ( 1 + \varepsilon_{n} \theta )^{2}} \| \Phi_{n}^{j,h} \|^{2} \Big) + \nu \| \nabla \phi_{n,\beta}^{j,h} \|^{2} \\
					= & \Big( \frac{u_{n,\alpha}^{j}}{\widehat{k}_{n}} \!-\! u_{t}^{j} (t_{n,\beta}), \phi_{n,\beta}^{j,h} \Big) 
					\!-\! \frac{1}{\widehat{k}_{n}} \big( \eta_{n,\alpha}^{j}, \phi_{n,\beta}^{j,h} \big) 
					\!+\! \nu \! \big( \nabla \big(u_{n,\beta}^{j} \!-\! u^{j} (t_{n,\beta}) \big), \nabla \phi_{n,\beta}^{j,h} \big) \notag \\
					+& \big( p^{j} (t_{n,\beta}), \nabla \cdot \phi_{n,\beta}^{j,h} \big) 
					\!- \!b \big( u^{j}( t_{n,\beta}), u^{j} (t_{n,\beta}) , \phi_{n,\beta}^{j,h} \big) 
					\!+ \!b \big( u_{n,\ast}^{j,h} - \big( u_{n,\ast}^{j,h} \!-\! \langle u^{h} \rangle_{n,\ast} \big), u_{n,\beta}^{j,h}, \phi_{n,\beta}^{j,h} \big) \notag \\
					+& b \big( u_{n,\ast}^{j,h} \!-\! \langle u^{h} \rangle_{n,\ast}, u_{n,\ast}^{j,h}, \phi_{n,\beta}^{j,h} \big). \notag
				\end{align*}
				\normalcolor
			\end{confidential}
			\begin{align}
				&\frac{1}{\widehat{k}_{n}} \Big( \begin{Vmatrix} {\phi_{n+1}^{j,h}} \\ {\phi_{n}^{j,h}} \end{Vmatrix}^{2}_{G(\theta)}-\begin{Vmatrix} {\phi_{n}^{j,h}} \\ {\phi_{n-1}^{j,h}} \end{Vmatrix}^{2}_{G(\theta)} + \frac{ \theta (1 - {\theta}^{2})}{2 ( 1 + \varepsilon_{n} \theta )^{2}} \| \Phi_{n}^{j,h} \|^{2} \Big) + \nu \| \nabla \phi_{n,\beta}^{j,h} \|^{2} 
				\label{eq:DLN-Error-Eq4} \\
				= & \Big( \frac{u_{n,\alpha}^{j}}{\widehat{k}_{n}} \!-\! u_{t}^{j} (t_{n,\beta}), \phi_{n,\beta}^{j,h} \Big) 
				\!-\! \frac{1}{\widehat{k}_{n}} \big( \eta_{n,\alpha}^{j}, \phi_{n,\beta}^{j,h} \big) 
				\!+\! \nu \! \big( \nabla \big(u_{n,\beta}^{j} \!-\! u^{j} (t_{n,\beta}) \big), \nabla \phi_{n,\beta}^{j,h} \big) \notag \\
				&+ \big( p^{j} (t_{n,\beta}), \nabla \cdot \phi_{n,\beta}^{j,h} \big) 
				\!+ b \big( u_{n,\ast}^{j,h}, u_{n,\beta}^{j,h}, \phi_{n,\beta}^{j,h} \big) 
				\!- b \big( u_{n,\ast}^{j}, u_{n,\beta}^{j}, \phi_{n,\beta}^{j,h} \big)
				+ b \big( u_{n,\ast}^{j}, u_{n,\beta}^{j}, \phi_{n,\beta}^{j,h} \big) \notag \\
				&- b \big( u^{j}( t_{n,\beta}), u^{j} (t_{n,\beta}) , \phi_{n,\beta}^{j,h} \big) 
				+ b \big( u_{n,\ast}^{j,h} - \langle u^{h} \rangle_{n,\ast}, u_{n,\ast}^{j,h} - u_{n,\beta}^{j,h}, \phi_{n,\beta}^{j,h} \big). \notag
			\end{align}
			where $\nu \big( \nabla \eta_{n,\beta}^{j} , \nabla \phi_{n,\beta}^{j,h} \big) = 0$ in \eqref{eq:DLN-Error-Eq3} is due to definition of Stokes projection in \eqref{eq:Stokes-def}. \\
			\ \\
			\noindent \textit{Part 3.}  Now we address all terms on the right hand side of \eqref{eq:DLN-Error-Eq4}.  \\
			\noindent $\bullet$ \textit{Terms from the semi-implicit DLN algorithms for $j$-th NSE in \eqref{eq:jth-NSE}}  \\
			By the definition of dual norm in \eqref{eq:dual-norm}, Young's inequality and \eqref{eq:consist-2nd-eq3} in Lemma \ref{lemma:DLN-consistency}
			\begin{align}
				\Big( \frac{u_{n,\alpha}^{j}}{\widehat{k}_{n}} \!-\! u_{t}^{j} (t_{n,\beta}), \phi_{n,\beta}^{j,h} \Big)
				\leq& \Big\| \frac{u_{n,\alpha}^{j}}{\widehat{k}_{n}} \!-\! u_{t}^{j} (t_{n,\beta}) \Big\|_{-1} \| \nabla \phi_{n,\beta}^{j,h} \|   
				\label{eq:error-rhs-term1} \\
				\leq& \frac{C(\theta)}{\nu} k_{\rm{max}}^{3} \int_{t_{n-1}}^{t_{n+1}} \| u_{ttt}^{j} \|_{-1}^{2} dt
				+ \frac{\nu}{16} \| \nabla \phi_{n,\beta}^{j,h} \|^{2}.  \notag 
			\end{align}
			By Poincar\'e inequality, Young's inequality, \eqref{eq:approx-thm}, \eqref{eq:Stoke-Approx} 
			and the fact 
			\begin{gather*}
				u_{n,\alpha}^{j} = \alpha_{2} \big( u_{n+1}^{j} - u_{n}^{j} \big) 
				- \alpha_{0} \big( u_{n}^{j} - u_{n-1}^{j} \big),
			\end{gather*}
			we have
			\begin{confidential}
				\color{darkblue}
				\begin{align*}
					\frac{1}{\widehat{k}_{n}} \big( \eta_{n,\alpha}^{j}, \phi_{n,\beta}^{j,h} \big)
					\leq \frac{1}{\widehat{k}_{n}} \| \eta_{n,\alpha}^{j} \| \| \phi_{n,\beta}^{j,h} \|
					\leq \frac{C(\Omega)}{\widehat{k}_{n}} \| \eta_{n,\alpha}^{j} \| \| \nabla \phi_{n,\beta}^{j,h} \|
				\end{align*}
				\normalcolor
			\end{confidential}
			\begin{align}
				\frac{1}{\widehat{k}_{n}}
				\big(\eta_{n,\alpha}^{j}, \phi_{n,\beta}^{j,h} \!\big)
				\leq& \frac{C(\Omega)}{\nu \widehat{k}^{2}_{n}} 
				\big\| \eta_{n,\alpha}^{j} \big\|^{2}
				+ \frac{\nu}{16} \| \nabla \phi_{n,\beta}^{j,h} \|^{2} 
				\label{eq:error-rhs-term2-eq1} \\
				\leq& \frac{C(\Omega) h^{2r+2}}{\nu \widehat{k}^{2}_{n}} \| u_{n,\alpha}^{j} \|_{r+1}^{2} + \frac{\nu}{16} \| \nabla \phi_{n,\beta}^{j,h} \|^{2} \notag \\
				\leq& \frac{C(\Omega,\theta) h^{2r\!+\!2}}{\nu \widehat{k}^{2}_{n}} \big( \|u_{n\!+\!1}^{j} \!-\! u_{n}^{j}\|_{r\!+\!1}^{2} \!+\! \|u_{n}^{j} \!-\! u_{n\!-\!1}^{j}\|_{r\!+\!1}^{2} \big)
				\!+\! \frac{\nu}{16} \| \!\nabla \phi_{n,\beta}^{j,h} \!\|^{2}. \notag 
			\end{align}
			By Holder's inequality
			\begin{confidential}
				\color{darkblue}
				\begin{align*}
					\big\| u_{n,\alpha}^{j} \big\|_{r+1}^{2}
					=& \big\| \alpha_{2} \big( u_{n+1}^{j} - u_{n}^{j} \big) 
					- \alpha_{0} \big( u_{n}^{j} - u_{n-1}^{j} \big) \big\|_{r+1}^{2} \\
					\leq& C(\theta) \big( \|u_{n+1}^{j} - u_{n}^{j}\|_{r+1}^{2} + \|u_{n+1}^{j} - u_{n-1}^{j}\|_{r+1}^{2} \big)
				\end{align*}
				\begin{align*}
					\|u_{n+1}^{j} - u_{n}^{j} \|_{r+1}^{2} 
					=& \Big\| \int_{t_{n}}^{t_{n+1}} \partial_{t} u^{j}(\cdot,t) dt \Big\|_{r+1}^{2} 
					= \int_{\Omega} \sum_{\ell \leq r+1} 
					\Big| D^{\ell} \Big( \int_{t_{n}}^{t_{n+1}} \partial_{t} u^{j}(x,t) dt \Big) \Big|^{2} dx  \\
					=& \int_{\Omega} \sum_{\ell \leq r+1}
					\Big| \int_{t_{n}}^{t_{n+1}} D^{\ell} \partial_{t} u^{j}(x,t) dt \Big|^{2} dx
					\leq \int_{\Omega} \sum_{\ell \leq r+1}
					\Big( \int_{t_{n}}^{t_{n+1}} \big| D^{\ell} \partial_{t} u^{j}(x,t) \big| dt \Big)^{2} dx \\
					\leq& \int_{\Omega} \sum_{\ell \leq r+1} 
					\Big[ \big( \int_{t_{n}}^{t_{n+1}} 1^2 dt \big) 
					\Big( \int_{t_{n}}^{t_{n+1}} \big| D^{\ell} \partial_{t} u^{j}(x,t) \big|^2 dt \Big) \Big] dx \\
					=& k_{n} \int_{\Omega} \int_{t_{n}}^{t_{n+1}} 
					\sum_{\ell \leq r+1} \big| D^{\ell} \partial_{t} u^{j}(x,t) \big|^2 dt dx \\
					=& k_{n} \int_{t_{n}}^{t_{n+1}} \int_{\Omega}
					\sum_{\ell \leq r+1} \big| D^{\ell} \partial_{t} u^{j}(x,t) \big|^2 dx dt 
					= k_{n} \int_{t_{n}}^{t_{n+1}} \| u_{t}^{j} \|_{r+1}^{2} dt.
				\end{align*}
				Similarly 
				\begin{gather*}
					\|u_{n}^{j} - u_{n-1}^{j} \|_{r+1}^{2}
					\leq k_{n-1} \int_{t_{n-1}}^{t_{n}} \| u_{t}^{j} \|_{r+1}^{2} dt
				\end{gather*}
				\begin{align*}
					\big\| u_{n,\alpha}^{j} \big\|_{r+1}^{2}
					\leq & C(\theta) \big( \|u_{n+1}^{j} - u_{n}^{j} \|_{r+1}^{2} + \|u_{n}^{j} - u_{n-1}^{j} \|_{r+1}^{2} \big) \\
					\leq & C(\theta) (k_{n}+k_{n-1}) \int_{t_{n-1}}^{t_{n+1}} \| u_{t}^{j} \|_{r+1}^{2} dt.
				\end{align*}
				\normalcolor
			\end{confidential}
			\begin{align}
				\|u_{n+1}^{j} - u_{n}^{j} \|_{r+1}^{2} 
				&= \Big\| \int_{t_{n}}^{t_{n+1}} u_{t}^{j}(\cdot,t) dt \Big\|_{r+1}^{2}
				\leq k_{n} \int_{t_{n}}^{t_{n+1}} \| u_{t}^{j} \|_{r+1}^{2} dt, 
				\label{eq:error-rhs-term2-eq2} \\
				\|u_{n}^{j} - u_{n-1}^{j} \|_{r+1}^{2}
				&= \Big\| \int_{t_{n-1}}^{t_{n}} u_{t}^{j}(\cdot,t) dt \Big\|_{r+1}^{2}
				\leq k_{n-1} \int_{t_{n-1}}^{t_{n}} \| u_{t}^{j} \|_{r+1}^{2} dt. \notag 
			\end{align}
			We combining \eqref{eq:error-rhs-term2-eq1}, \eqref{eq:error-rhs-term2-eq2} and use the fact
			\begin{gather*}
				\widehat{k}_{n} \geq \max \big\{ \frac{1+\theta}{2} k_{n}, \frac{1-\theta}{2} k_{n-1} \big\} > 0,
			\end{gather*}
			to obtain 
			\begin{align}
				\frac{1}{\widehat{k}_{n}} \big( \eta_{n,\alpha}^{j}, \phi_{n,\beta}^{j,h} \big) 
				\leq \frac{C(\theta) h^{2r+2}}{\nu \widehat{k}_{n}} 
				\int_{t_{n-1}}^{t_{n+1}} \| u_{t}^{j} \|_{r+1}^{2} dt + \frac{\nu}{16} \| \nabla \phi_{n,\beta}^{j,h} \|^{2}. 
				\label{eq:error-rhs-term2}
			\end{align}
			By Young's inequality and \eqref{eq:consist-2nd-eq1} in Lemma \ref{lemma:DLN-consistency}
			\begin{align}
				\nu \big( \nabla (u_{n,\beta}^{j} - u^{j}(t_{n,\beta})), \nabla \phi_{n,\beta}^{j,h} \big) 
				\leq& \nu \big\| \nabla (u_{n,\beta}^{j} - u^{j}(t_{n,\beta})) \big\| 
				\| \nabla \phi_{n,\beta}^{j,h} \|   \label{eq:error-rhs-term3}  \\
				\leq& C(\theta) \nu k_{\rm{max}}^{3} \int_{t_{n-1}}^{t_{n+1}} \| \nabla u_{tt}^{j} \|^{2} dt 
				+ \frac{\nu}{16} \| \nabla \phi_{n,\beta}^{j,h} \|^{2}. \notag 
			\end{align}
			We set $q_{n}^{h}$ to be $L^2$-projection of $p(t_{n,\beta})$ onto $Q^{h}$ and use \eqref{eq:approx-thm}
			\begin{align}
				\big(p(t_{n,\beta}), \nabla \cdot \phi_{n,\beta}^{h}\big) = 
				\big(p(t_{n,\beta}) - q_{n}^{h}, \nabla \cdot \phi_{n,\beta}^{h}\big)
				\leq&  \sqrt{d} \big\| p(t_{n,\beta}) - q_{n}^{h} \big\| \| \nabla \phi_{n,\beta}^{h} \| 
				\label{eq:error-rhs-term4}  \\
				\leq&  \frac{C(\Omega)h^{2s+2}}{\nu} \| p(t_{n,\beta}) \|_{s+1}^{2}
				+ \frac{\nu}{16} \| \nabla \phi_{n,\beta}^{h} \|^{2}.
			 	\notag 
			\end{align}
			By skew-symmetric property of $b$ 
			\begin{align*}
				&b \big( u_{n,\ast}^{j,h}, u_{n,\beta}^{j,h}, \phi_{n,\beta}^{j,h} \big) 
				\!- b \big( u_{n,\ast}^{j}, u_{n,\beta}^{j}, \phi_{n,\beta}^{j,h} \big) \\
				=& b \big( u_{n,\ast}^{j,h}, u_{n,\beta}^{j,h}, \phi_{n,\beta}^{j,h} \big)
				\!- b \big( u_{n,\ast}^{j,h}, u_{n,\beta}^{j}, \phi_{n,\beta}^{j,h} \big)
				\!+ b \big( u_{n,\ast}^{j,h}, u_{n,\beta}^{j}, \phi_{n,\beta}^{j,h} \big)
				\!- b \big( u_{n,\ast}^{j}, u_{n,\beta}^{j}, \phi_{n,\beta}^{j,h} \big) \\
				=& - b \big( u_{n,\ast}^{j,h}, \eta_{n,\beta}^{j}, \phi_{n,\beta}^{j,h} \big) 
				\!- b \big( \eta_{n,\ast}^{j}, u_{n,\beta}^{j}, \phi_{n,\beta}^{j,h} \big)
				\!- b \big( \phi_{n,\ast}^{j,h}, u_{n,\beta}^{j}, \phi_{n,\beta}^{j,h} \big) \\
				=& b \big( u_{n,\ast}^{j} - u_{n,\ast}^{j,h}, \eta_{n,\beta}^{j}, \phi_{n,\beta}^{j,h} \big) 
				\!- b \big( u_{n,\ast}^{j}, \eta_{n,\beta}^{j}, \phi_{n,\beta}^{j,h} \big)
				\!- b \big( \eta_{n,\ast}^{j}, u_{n,\beta}^{j}, \phi_{n,\beta}^{j,h} \big)
				\!- b \big( \phi_{n,\ast}^{j,h}, u_{n,\beta}^{j}, \phi_{n,\beta}^{j,h} \big) \\
				=& b \big( \eta_{n,\ast}^{j}, \eta_{n,\beta}^{j}, \phi_{n,\beta}^{j,h} \big) 
				\!+ b \big( \phi_{n,\ast}^{j,h}, \eta_{n,\beta}^{j}, \phi_{n,\beta}^{j,h} \big)
				\!- b \big( u_{n,\ast}^{j}, \eta_{n,\beta}^{j}, \phi_{n,\beta}^{j,h} \big)
				\!- b \big( \eta_{n,\ast}^{j}, u_{n,\beta}^{j}, \phi_{n,\beta}^{j,h} \big) \\
				&- b \big( \phi_{n,\ast}^{j,h}, u_{n,\beta}^{j}, \phi_{n,\beta}^{j,h} \big).
			\end{align*}
			By \eqref{eq:b-bound-1}, \eqref{eq:b-bound-2} in Lemma \ref{lemma:b-bound}, inverse inequality in \eqref{eq:inv-inequal}, 
			approximations in \eqref{eq:approx-thm}, approximation of 
			Stokes projection in \eqref{eq:Stoke-Approx}, Poincar\'e inequality, bounds 
			of $\{\beta_{\ell}^{(n)}\}_{\ell=0}^{2}$ in \eqref{eq:bound-beta}, time ratio bounds 
			in \eqref{eq:time-ratio-cond} and $h \leq 1$
			\begin{align}
				b \big( \eta_{n,\ast}^{j}, \eta_{n,\beta}^{j}, \phi_{n,\beta}^{j,h} \big) 
				\leq& C(\Omega,\theta) \big( \| \nabla \eta_{n}^{j} \| + \| \nabla \eta_{n-1}^{j} \| \big)
				\| \nabla \eta_{n,\beta}^{j} \| \| \nabla \phi_{n,\beta}^{j,h} \|                   
				\label{eq:non-linear-term1} \\ 
				\leq& C(\Omega,\theta) \big( \| \nabla u_{n}^{j} \| + \| \nabla u_{n-1}^{j} \| \big) 
				\| \nabla \eta_{n,\beta}^{j} \| \| \nabla \phi_{n,\beta}^{j,h} \|, \notag  \\
				b \big( \phi_{n,\ast}^{j,h}, \eta_{n,\beta}^{j}, \phi_{n,\beta}^{j,h} \big) 
				\leq& C(\Omega) \| \phi_{n,\ast}^{j,h} \|^{1/2} \| \nabla \phi_{n,\ast}^{j,h} \|^{1/2} 
				\| \nabla \eta_{n,\beta}^{j} \| \| \nabla \phi_{n,\beta}^{j,h} \|                  \notag \\
				\leq& C(\Omega) h \| \phi_{n,\ast}^{j,h} \|^{1/2} \| \nabla \phi_{n,\ast}^{j,h} \|^{1/2} 
				\| u_{n,\beta}^{j} \|_{2} \| \nabla \phi_{n,\beta}^{j,h} \|        \notag \\
				\leq& C(\Omega,\theta) h^{1/2} \big( \| \phi_{n}^{j,h} \| + \| \phi_{n-1}^{j,h} \| \big) 
				\| u_{n,\beta}^{j} \|_{2} \| \nabla \phi_{n,\beta}^{j,h} \|,                         \notag \\
				b \big( u_{n,\ast}^{j}, \eta_{n,\beta}^{j}, \phi_{n,\beta}^{j,h} \big) 
				\leq& C(\Omega,\theta) \big( \| \nabla u_{n}^{j} \| + \| \nabla u_{n-1}^{j} \| \big)
				\| \nabla \eta_{n,\beta}^{j} \| \| \nabla \phi_{n,\beta}^{j,h} \|,  \notag \\
				b \big( \eta_{n,\ast}^{j}, u_{n,\beta}^{j}, \phi_{n,\beta}^{j,h} \big) 
				\leq& C(\Omega) \| \nabla \eta_{n,\ast}^{j} \| \| \nabla u_{n,\beta}^{j} \| \| \nabla \phi_{n,\beta}^{j,h} \|, \notag \\
				b \big( \phi_{n,\ast}^{j,h}, u_{n,\beta}^{j}, \phi_{n,\beta}^{j,h} \big) 
				\leq& C(\Omega,\theta) \| u_{n,\beta}^{j} \|_{2} \big( \| \phi_{n}^{j,h} \| + \| \phi_{n-1}^{j,h} \| \big)
				\| \nabla \phi_{n,\beta}^{j,h} \|.  \notag 
			\end{align}
			We apply Young's inequality to all non-linear terms in \eqref{eq:non-linear-term1} 
			\begin{align}
				&b \big( u_{n,\ast}^{j,h}, u_{n,\beta}^{j,h}, \phi_{n,\beta}^{j,h} \big) 
				- b \big( u_{n,\ast}^{j}, u_{n,\beta}^{j}, \phi_{n,\beta}^{j,h} \big) 
				\label{eq:non-linear-eq1} \\
				\leq& \frac{C(\Omega,\theta)}{\nu} \!\Big[ \| u_{n,\beta}^{j} \|_{2}^{2} \!
				\big( \| \phi_{n}^{j,h} \|^{2} \!+\! \| \phi_{n-1}^{j,h} \|^{2} \big) 
				\!+\! \| | \nabla u^{j} | \|_{\infty,0}^{2} \| \nabla \eta_{n,\ast}^{j} \|^{2} 
				\!+\! \| | \nabla u^{j} | \|_{\infty,0}^{2} \| \nabla \eta_{n,\beta}^{j} \|^{2} \Big]
				\notag \\
				&+\! \frac{\nu}{8} \| \nabla \phi_{n,\beta}^{j,h} \|^{2}.  \notag   
			\end{align}
			\begin{confidential}
				\color{darkblue}
				\begin{align*}
					&b \big( u_{n,\ast}^{j,h}, u_{n,\beta}^{j,h}, \phi_{n,\beta}^{j,h} \big) 
					- b \big( u_{n,\ast}^{j}, u_{n,\beta}^{j}, \phi_{n,\beta}^{j,h} \big) \\
					\leq& C(\Omega,\theta) \Big( \| u_{n,\beta}^{j} \|_{2} \big( \| \phi_{n}^{j,h} \| + \| \phi_{n-1}^{j,h} \| \big) 
					+ \| \nabla \eta_{n,\ast}^{j} \| \| \nabla u_{n,\beta}^{j} \|
					+ \big( \| \nabla u_{n}^{j} \| + \| \nabla u_{n-1}^{j} \| \big) \| \nabla \eta_{n,\beta}^{j} \| \Big) 
					\| \nabla \phi_{n,\beta}^{j,h} \| \\
					\leq& \frac{C(\Omega,\theta)}{\nu} \Big( \| u_{n,\beta}^{j} \|_{2}^{2} 
					\big( \| \phi_{n}^{j,h} \|^{2} + \| \phi_{n-1}^{j,h} \|^{2} \big) 
					+ \| \nabla \eta_{n,\ast}^{j} \|^{2} \| \nabla u_{n,\beta}^{j} \|^{2}
					+ \big( \| \nabla u_{n}^{j} \|^{2} + \| \nabla u_{n-1}^{j} \|^{2} \big) \| \nabla \eta_{n,\beta}^{j} \|^{2} \Big)
					+ \frac{\nu}{16} \| \nabla \phi_{n,\beta}^{j,h} \|^{2} \\
					\leq& \frac{C(\Omega,\theta)}{\nu} \Big( \| u_{n,\beta}^{j} \|_{2}^{2} 
					\big( \| \phi_{n}^{j,h} \|^{2} + \| \phi_{n-1}^{j,h} \|^{2} \big) 
					+ \| | \nabla u^{j} | \|_{\infty,0}^{2} \| \nabla \eta_{n,\ast}^{j} \|^{2} 
					+ \| |\nabla u^{j} | \|_{\infty,0}^{2}  \| \nabla \eta_{n,\beta}^{j} \|^{2} \Big)
					+ \frac{\nu}{8} \| \nabla \phi_{n,\beta}^{j,h} \|^{2} \\
				\end{align*}
				\normalcolor
			\end{confidential}
			By \eqref{eq:approx-thm}, \eqref{eq:Stoke-Approx}, triangle inequality and \eqref{eq:consist-2nd-eq1}, \eqref{eq:consist-2nd-eq2} in Lemma \ref{lemma:DLN-consistency}
			\begin{align*}
				\| \nabla \eta_{n,\ast}^{j} \|^{2}
				\leq& C h^{2r} \| u_{n,\ast}^{j} \|_{r+1}^{2} 
				\leq C h^{2r} \big( \| u_{n,\ast}^{j} - u^{j}(t_{n,\beta}) \|_{r+1}^{2} 
				+ \| u^{j}(t_{n,\beta}) \|_{r+1}^{2} \big)  \\
				\leq& C(\theta) h^{2r} (k_{n}+k_{n-1})^{3} \int_{t_{n-1}}^{t_{n+1}} \| u_{tt}^{j} \|_{r+1}^{2} dt 
				+ C h^{2r} \| u^{j}(t_{n,\beta}) \|_{r+1}^{2}, \\
				\| \nabla \eta_{n,\beta}^{j} \|^{2}
				\leq& C h^{2r} \| {u}_{n,\beta}^{j} \|_{r+1}^{2} 
				\leq C h^{2r} \big( \| {u}_{n,\beta}^{j} - u^{j}(t_{n,\beta}) \|_{r+1}^{2} 
				+ \| u^{j}(t_{n,\beta}) \|_{r+1}^{2} \big)  \\
				\leq& C(\theta) h^{2r} (k_{n}+k_{n-1})^{3} \int_{t_{n-1}}^{t_{n+1}} \| u_{tt}^{j} \|_{r+1}^{2} dt 
				+ C h^{2r} \| u^{j}(t_{n,\beta}) \|_{r+1}^{2}. 
			\end{align*}
			\eqref{eq:non-linear-eq1} becomes 
			\begin{align}
				&b \big( u_{n,\ast}^{j,h}, u_{n,\beta}^{j,h}, \phi_{n,\beta}^{j,h} \big) 
				- b \big( u_{n,\ast}^{j}, u_{n,\beta}^{j}, \phi_{n,\beta}^{j,h} \big) 
				\label{eq:non-linear-eq2} \\
				\leq& \frac{C(\Omega,\theta) h^{2r} }{\nu}  \| | \nabla u^{j} | \|_{\infty,0}^{2}
				\Big( k_{\rm{max}}^{3} \int_{t_{n-1}}^{t_{n+1}} \!\! \| u_{tt}^{j} \|_{r+1}^{2}\! dt 
				+ \| u^{j}(t_{n,\beta}) \|_{r+1}^{2} \!\Big) 
				+ \frac{\nu}{8} \| \nabla \phi_{n,\beta}^{j,h} \|^{2}   \notag  \\
				&+ \frac{C(\Omega,\theta) \| u_{n,\beta}^{j} \|_{2}^{2} }{\nu} \! 
				\big( \| \phi_{n}^{j,h} \|^{2} + \| \phi_{n\!-\!1}^{j,h} \|^{2} \big). \notag
			\end{align}
			\begin{confidential}
				\color{darkblue}
				\begin{align*}
				&b \big( u_{n,\ast}^{j}, u_{n,\beta}^{j}, \phi_{n,\beta}^{j,h} \big)  
				- b \big( u^{j}( t_{n,\beta}), u^{j} (t_{n,\beta}) , \phi_{n,\beta}^{j,h} \big) \\
				=&  b \big( u_{n,\ast}^{j}, u_{n,\beta}^{j}, \phi_{n,\beta}^{j,h} \big)
				- b \big( u^{j}(t_{n,\beta}), u_{n,\beta}^{j}, \phi_{n,\beta}^{j,h} \big) 
				+ b \big( u^{j}(t_{n,\beta}), u_{n,\beta}^{j}, \phi_{n,\beta}^{j,h} \big)
				- b ( u^{j}(t_{n,\beta}), u^{j}(t_{n,\beta}), \phi_{n,\beta}^{j,h} )
				\end{align*}
				\normalcolor
			\end{confidential}
			By \eqref{eq:b-bound-1}, Poincar\'e inequality, \eqref{eq:bound-beta} and \eqref{eq:consist-2nd-eq1}, \eqref{eq:consist-2nd-eq2} in Lemma \ref{lemma:DLN-consistency}
			\begin{align}
				&b \big( u_{n,\ast}^{j}, u_{n,\beta}^{j}, \phi_{n,\beta}^{h} \big)
				- b ( u^{j}(t_{n,\beta}), u^{j}(t_{n,\beta}), \phi_{n,\beta}^{j,h} )  
				\label{eq:non-linear-term2}\\
				=& b \big( u_{n,\ast}^{j} - u^{j}(t_{n,\beta}), u_{n,\beta}^{j}, \phi_{n,\beta}^{h} \big)
				+ b \big( u^{j}(t_{n,\beta}), u_{n,\beta}^{j} - u^{j}(t_{n,\beta}), \phi_{n,\beta}^{j,h} \big) \notag \\
				\leq& C(\Omega) \| \nabla \big( u_{n,\ast}^{j} - u^{j}(t_{n,\beta}) \big) \|
				\| \nabla u_{n,\beta}^{j} \| \| \nabla \phi_{n,\beta}^{j,h} \| \notag \\
				&+ C(\Omega) \| \nabla u^{j}(t_{n,\beta}) \| \| \nabla \big( u_{n,\beta}^{j} - u^{j}(t_{n,\beta}) \big) \|
				\| \nabla \phi_{n,\beta}^{j,h} \|    \notag \\
				\leq& \frac{C(\Omega,\theta)}{\nu} k_{\rm{max}}^{3} \big( \| |u^{j}| \|_{\infty,1}^{2} 
				+ \| |u^{j}| \|_{\infty,1,\beta}^{2} \big)
				\int_{t_{n-1}}^{t_{n+1}} \| \nabla u_{tt}^{j} \|^{2} dt 
				+ \frac{\nu}{16} \| \nabla \phi_{n,\beta}^{j,h} \|^{2}. \notag 
			\end{align}
			\begin{confidential}
				\color{darkblue}
				\begin{align*}
					C(\Omega) \| \nabla \big( u_{n,\ast}^{j} - u^{j}(t_{n,\beta}) \big) \|
					\| \nabla u_{n,\beta}^{j} \| \| \nabla \phi_{n,\beta}^{j,h} \|      
					\leq& C(\Omega,\theta) \| \nabla \big( u_{n,\ast}^{j} - u^{j}(t_{n,\beta}) \big) \|
					\| |u^{j}| \|_{\infty,1} \| \nabla \phi_{n,\beta}^{j,h} \| \\
					\leq& \frac{\nu}{32} \| \nabla \phi_{n,\beta}^{j,h} \|^{2}     
					+ \frac{C(\Omega,\theta)}{\nu} \| \nabla \big( u_{n,\ast}^{j} - u^{j}(t_{n,\beta}) \big) \|^{2}
					\| |u^{j}| \|_{\infty,1}^{2} \\
					\leq& \frac{\nu}{32} \| \nabla \phi_{n,\beta}^{j,h} \|^{2}
					+ \frac{C(\Omega,\theta)}{\nu} (k_{n} + k_{n-1})^{3} \| |u^{j}| \|_{\infty,1}^{2}
					\int_{t_{n-1}}^{t_{n+1}} \| \nabla u_{tt}^{j} \|^{2} dt \\
					\leq& \frac{\nu}{32} \| \nabla \phi_{n,\beta}^{j,h} \|^{2}
					+ \frac{C(\Omega,\theta)}{\nu} (k_{n} + k_{n-1})^{3} \| |u^{j}| \|_{\infty,1}^{2}
					\int_{t_{n-1}}^{t_{n+1}} \| \nabla u_{tt}^{j} \|^{2} dt
				\end{align*}
				\begin{align*}
					C(\Omega) \| \nabla u^{j}(t_{n,\beta}) \| \| \nabla \big( u_{n,\beta}^{j} - u^{j}(t_{n,\beta}) \big) \|
					\| \nabla \phi_{n,\beta}^{j,h} \| 
					\leq& \frac{\nu}{32} \| \nabla \phi_{n,\beta}^{j,h} \|^{2} 
					+ \frac{C(\Omega)}{\nu} \| \nabla \big( u_{n,\beta}^{j} - u^{j}(t_{n,\beta}) \big) \|^{2}
					\| \nabla u^{j}(t_{n,\beta}) \|^{2}    \\
					\leq& \frac{\nu}{32} \| \nabla \phi_{n,\beta}^{j,h} \|^{2} 
					+ \frac{C(\Omega)}{\nu} \| |u^{j}| \|_{\infty,1,\beta}^{2}
					\| \nabla \big( u_{n,\beta}^{j} - u^{j}(t_{n,\beta}) \big) \|^{2} \\
					\leq& \frac{\nu}{32} \| \nabla \phi_{n,\beta}^{j,h} \|^{2} 
					+ \frac{C(\Omega,\theta)}{\nu} \| |u^{j}| \|_{\infty,1,\beta}^{2} (k_{n}+k_{n-1})^{3}
					\int_{t_{n-1}}^{t_{n+1}} \| \nabla u_{tt}^{j} \|^{2} dt 
				\end{align*}
				\normalcolor
			\end{confidential}

			\noindent $\bullet$ \textit{New terms arising from the DLN-Ensemble algorithms in \eqref{eq:DLN-Ensemble-Alg}}  \\
			By skew-symmetric property of $b$
			\begin{align}
				&b \big( u_{n,\ast}^{j,h} \!-\! \langle  u^{h} \rangle_{n,\ast}, u_{n,\ast}^{j,h} - u_{n,\beta}^{j,h}, \phi_{n,\beta}^{j,h} \big) 
				\label{eq:non-linear-eq3}\\
				=&\!  b \big( u_{n,\ast}^{j,h} \!-\! \langle  u^{h} \rangle_{n,\ast}, \big( u_{n,\ast}^{j,h} \!-\! u_{n,\beta}^{j,h} \big) \!-\! \big( u_{n,\ast}^{j} \!-\! u_{n,\beta}^{j} \big) , \phi_{n,\beta}^{j,h} \big) 
				\!+\! b \big( u_{n,\ast}^{j,h} \!-\! \langle  u^{h} \rangle_{n,\ast}, u_{n,\ast}^{j} \!-\! u_{n,\beta}^{j}, \phi_{n,\beta}^{j,h} \big)  \notag \\
				=&\! b \big( u_{n,\ast}^{j,h} \!-\! \langle  u^{h} \rangle_{n,\ast}, \eta_{n,\ast}^{j} - \eta_{n,\beta}^{j}, \phi_{n,\beta}^{j,h} \big) 
				+ b \big( u_{n,\ast}^{j,h} \!-\! \langle  u^{h} \rangle_{n,\ast}, \phi_{n,\ast}^{j,h} - \phi_{n,\beta}^{j,h}, \phi_{n,\beta}^{j,h} \big) \notag \\
				&+ b \big( u_{n,\ast}^{j,h} \!-\! \langle  u^{h} \rangle_{n,\ast}, u_{n,\ast}^{j} - u_{n,\beta}^{j}, \phi_{n,\beta}^{j,h} \big) \notag \\
				=& b \big( u_{n,\ast}^{j,h} \!-\! \langle  u^{h} \rangle_{n,\ast}, \eta_{n,\ast}^{j} - \eta_{n,\beta}^{j}, \phi_{n,\beta}^{j,h} \big) 
				- b \big( u_{n,\ast}^{j,h} \!-\! \langle  u^{h} \rangle_{n,\ast}, \phi_{n,\beta}^{j,h},
				\phi_{n,\ast}^{j,h} - \phi_{n,\beta}^{j,h} \big) \notag \\
				&+ b \big( u_{n,\ast}^{j,h} \!-\! \langle  u^{h} \rangle_{n,\ast}, u_{n,\ast}^{j} - u_{n,\beta}^{j}, \phi_{n,\beta}^{j,h} \big). \notag 
			\end{align}
			By \eqref{eq:b-bound-1} in Lemma \ref{lemma:b-bound}, Poincar$\acute{\rm{e}}$ inequality, \eqref{eq:approx-thm}, \eqref{eq:Stoke-Approx}, Young's inequality and CFL-like conditions in \eqref{eq:CFL-like-cond} 
			\begin{align}
				& b \big( u_{n,\ast}^{j,h} \!-\! \langle  u^{h} \rangle_{n,\ast}, \eta_{n,\ast}^{j} - \eta_{n,\beta}^{j}, \phi_{n,\beta}^{j,h} \big) 
				\label{eq:non-linear-eq3-term1} \\
				\leq& C(\Omega) \big\| \nabla \big( u_{n,\ast}^{j,h} \!-\! \langle  u^{h} \rangle_{n,\ast} \big) \big\|
				\big\| \nabla \big( \eta_{n,\ast}^{j} - \eta_{n,\beta}^{j} \big) \big\| \| \nabla \phi_{n,\beta}^{j,h} \|
				\notag \\
				\leq& \! C(\Omega) h^{r} \! \big\| \nabla \big( u_{n,\ast}^{j,h} \!-\! \langle  u^{h} \rangle_{n,\ast} \big) \big\|
				\big( \big\| u_{n,\ast}^{j} \!-\! u^{j}(t_{n,\beta}) \big\|_{r\!+\!1}
				\!+\! \big\| u_{n,\beta}^{j} \!-\! u^{j}(t_{n,\beta})  \big\|_{r\!+\!1} \big) \| \nabla \phi_{n,\beta}^{j,h} \| \notag \\
				\leq&\! \frac{C(\Omega)h^{2r}}{\nu} \big\| \nabla \big( u_{n,\ast}^{j,h} \!-\! \langle  u^{h} \rangle_{n,\ast} \big) \big\|^{2} \!\big( \big\| u_{n,\ast}^{j} \!-\! u^{j}(t_{n,\beta}) \big\|_{r\!+\!1}^{2}
				\!+\! \big\| u_{n,\beta}^{j} \!-\! u^{j}(t_{n,\beta})  \big\|_{r\!+\!1}^{2} \big) \notag \\
				&+\! \frac{\nu}{32} \| \nabla \phi_{n,\beta}^{j,h} \|^{2}
				\notag \\
				\leq& \! C(\Omega, \theta) \frac{h^{2r}}{\nu} \! \frac{h \nu }{\widehat{k}_{n}} \Big( \frac{1-\varepsilon_{n}}{1+\varepsilon_{n} \theta} \Big)^{2} k_{\rm{max}}^{3} 
				\int_{t_{n-1}}^{t_{n+1}} \| u_{tt}^{j} \|_{r+1}^{2} dt \!+\! \frac{\nu}{32} \| \nabla \phi_{n,\beta}^{j,h} \|^{2} \notag \\
				\leq& \! C(\Omega,\theta) \frac{h^{2r+1} k_{\rm{max}}^{3} }{\widehat{k}_{n} } 
				\int_{t_{n-1}}^{t_{n+1}} \| u_{tt}^{j} \|_{r+1}^{2} dt \!+\! \frac{\nu}{32} \| \nabla \phi_{n,\beta}^{j,h} \|^{2}. \notag 
			\end{align}
			Similarly,
			\begin{align}
				&b \big( u_{n,\ast}^{j,h} \!-\! \langle  u^{h} \rangle_{n,\ast}, u_{n,\ast}^{j} - u_{n,\beta}^{j}, \phi_{n,\beta}^{j,h} \big) 
				\label{eq:non-linear-eq3-term3} \\
				\leq& C(\Omega) \big\| \nabla \big( u_{n,\ast}^{j,h} \!-\! \langle  u^{h} \rangle_{n,\ast} \big) \big\|
				\big\| \nabla \big( u_{n,\ast}^{j}- u_{n,\beta}^{j} \big) \big\| \| \nabla \phi_{n,\beta}^{j,h} \|
				\notag \\
				\leq& \! \frac{C(\Omega)}{\nu} \! \big\| \nabla \big( u_{n,\ast}^{j,h} \!-\! \langle  u^{h} \rangle_{n,\ast} \big) \big\|^{2} \big( \big\| \nabla \big(u_{n,\ast}^{j} \!-\! u^{j}(t_{n,\beta}) \big) \big\|^{2}
				\!+\! \big\| \nabla \big(u_{n,\beta}^{j} \!-\! u^{j}(t_{n,\beta}) \big)  \big\|^{2} \big) \!+\! \frac{\nu}{32} \| \nabla \phi_{n,\beta}^{j,h} \|^{2}
				\notag \\
				\leq& \! C(\Omega,\theta) \frac{h k_{\rm{max}}^{3} }{\widehat{k}_{n} } 
				\int_{t_{n-1}}^{t_{n+1}} \| \nabla u_{tt}^{j} \|^{2} dt \!+\! \frac{\nu}{32} \| \nabla \phi_{n,\beta}^{j,h} \|^{2}. \notag
			\end{align}
			We use \eqref{eq:b-bound-2} in Lemma \ref{lemma:b-bound}, Poincar$\acute{\rm{e}}$ inequality, inverse inequality in \eqref{eq:inv-inequal}, Young's inequality and the fact 
			\begin{gather*}
			\phi_{n,\beta}^{j,h} - \phi_{n,\ast}^{j,h}
			= \frac{2 \beta_{2}^{(n)}}{1 - \varepsilon_{n}} \Phi_{n}^{j,h},
			\end{gather*}
			to obtain 
			\begin{align}
				&b \big( u_{n,\ast}^{j,h} \!-\! \langle  u^{h} \rangle_{n,\ast}, \phi_{n,\beta}^{j,h},
				\phi_{n,\ast}^{j,h} - \phi_{n,\beta}^{j,h} \big) 
				\label{eq:non-linear-eq3-term2} \\
				\leq& C(\Omega) \big\| \nabla \big( u_{n,\ast}^{j,h} \!-\! \langle  u^{h} \rangle_{n,\ast} \big) \big\|
				\| \nabla \phi_{n,\beta}^{j,h} \| \big\| \phi_{n,\ast}^{j,h}  - \phi_{n,\beta}^{j,h} \big\|^{1/2} 
				\big\| \nabla \big( \phi_{n,\ast}^{j,h}  - \phi_{n,\beta}^{j,h} \big) \big\|^{1/2}
				\notag \\
				\leq& \frac{C(\Omega,\theta) }{\sqrt{h} (1 - \varepsilon_{n})}  \big\| \nabla \big( u_{n,\ast}^{j,h} \!-\! \langle  u^{h} \rangle_{n,\ast} \big) \big\|
				\| \nabla \phi_{n,\beta}^{j,h} \| \big\| \Phi_{n}^{j,h} \big\| \notag \\
				\leq& \frac{C(\Omega,\theta) \widehat{k}_{n}}{h} \Big( \frac{1 + \varepsilon_{n} \theta }{1 - \varepsilon_{n}} \Big)^{2} 
				\big\| \nabla \big( u_{n,\ast}^{j,h} \!-\! \langle  u^{h} \rangle_{n,\ast} \big) \big\|^{2} \| \nabla \phi_{n,\beta}^{j,h} \|^{2} + \frac{\theta (1 - \theta^2)}{4 \widehat{k}_{n} (1 + \varepsilon_{n} \theta )^{2} } 
				\big\| \Phi_{n}^{j,h} \big\|^{2}. \notag 
			\end{align}
			\begin{confidential}
				\color{darkblue}
				\begin{align*}
					&b \big( u_{n,\ast}^{j,h} \!-\! \langle  u^{h} \rangle_{n,\ast}, \phi_{n,\beta}^{j,h},
					\phi_{n,\ast}^{j,h} - \phi_{n,\beta}^{j,h} \big) \\
					\leq& C(\Omega) \big\| \nabla \big( u_{n,\ast}^{j,h} \!-\! \langle  u^{h} \rangle_{n,\ast} \big) \big\|
					\| \nabla \phi_{n,\beta}^{j,h} \| \big\| \phi_{n,\ast}^{j,h} - \phi_{n,\beta}^{j,h} \big\|^{1/2} 
					\big\| \nabla \big( \phi_{n,\ast}^{j,h} - \phi_{n,\beta}^{j,h} \big) \big\|^{1/2}
					\notag \\
					\leq& C(\Omega) h^{-1/2} \big\| \nabla \big( u_{n,\ast}^{j,h} \!-\! \langle  u^{h} \rangle_{n,\ast} \big) \big\|
					\| \nabla \phi_{n,\beta}^{j,h} \| \big\| \phi_{n,\ast}^{j,h} - \phi_{n,\beta}^{j,h} \big\| \notag \\
					=& \frac{C(\Omega)}{\sqrt{h}}  \big\| \nabla \big( u_{n,\ast}^{j,h} \!-\! \langle  u^{h} \rangle_{n,\ast} \big) \big\|
					\| \nabla \phi_{n,\beta}^{j,h} \| \frac{2 \beta_{2}^{(n)}}{1 - \varepsilon_{n}}  \| \Phi_{n}^{j,h} \| \\
					=& C(\Omega,\theta) \Big[ \sqrt{\frac{\widehat{k}_{n}}{h}} \Big( \frac{1 + \varepsilon_{n} \theta }{1 - \varepsilon_{n}} \Big) 
					\big\| \nabla \big( u_{n,\ast}^{j,h} \!-\! \langle  u^{h} \rangle_{n,\ast} \big) \big\| \| \nabla \phi_{n,\beta}^{j,h} \| \Big] \Big( \frac{1}{\sqrt{\widehat{k}_{n}}(1 + \varepsilon_{n} \theta ) } \| \Phi_{n}^{j,h} \| \Big) \\
					\leq& \frac{C(\Omega,\theta) \widehat{k}_{n}}{h} \Big( \frac{1 + \varepsilon_{n} \theta }{1 - \varepsilon_{n}} \Big)^{2} 
					\big\| \nabla \big( u_{n,\ast}^{j,h} \!-\! \langle  u^{h} \rangle_{n,\ast} \big) \big\|^{2} \| \nabla \phi_{n,\beta}^{j,h} \|^{2} + \frac{\theta (1 - \theta^2)}{4 \widehat{k}_{n} (1 + \varepsilon_{n} \theta )^{2} } 
					\big\| \Phi_{n}^{j,h} \big\|^{2}.  
				\end{align*}
				\normalcolor
			\end{confidential}
			We combine \eqref{eq:non-linear-eq3}, \eqref{eq:non-linear-eq3-term1}, \eqref{eq:non-linear-eq3-term3}, \eqref{eq:non-linear-eq3-term2} to have
			\begin{align}
				&b \big( u_{n,\ast}^{j,h} \!-\! \langle  u^{h} \rangle_{n,\ast}, u_{n,\ast}^{j,h} - u_{n,\beta}^{j,h}, \phi_{n,\beta}^{j,h} \big) 
				\label{eq:non-linear-eq4} \\
				\leq& \frac{C(\Omega,\theta) h^{2r+1}  k_{\rm{max}}^{3} }{\widehat{k}_{n} } 
				\!\! \int_{t_{n-1}}^{t_{n+1}} \!\| u_{tt}^{j} \|_{r+1}^{2} dt 
				+ \frac{C(\Omega,\theta) h k_{\rm{max}}^{3} }{\widehat{k}_{n} } \!\! \int_{t_{n-1}}^{t_{n+1}} \!\| \nabla u_{tt}^{j} \|^{2} dt \!+\! \frac{\nu}{16} \| \nabla \phi_{n,\beta}^{j,h} \|^{2} \notag \\
				&+ \frac{\theta (1 - \theta^2)}{4 \widehat{k}_{n} (1 + \varepsilon_{n} \theta )^{2} } 
				\big\| \Phi_{n}^{j,h} \big\|^{2} 
				+ \frac{C(\theta) \widehat{k}_{n}}{h} \Big( \frac{1 + \varepsilon_{n} \theta }{1 - \varepsilon_{n}} \Big)^{2} 
				\big\| \nabla \big( u_{n,\ast}^{j,h} \!-\! \langle  u^{h} \rangle_{n,\ast} \big) \big\|^{2} \| \nabla \phi_{n,\beta}^{j,h} \|^{2}. \notag 
			\end{align}
			\ \\
			\noindent \textit{Part 4.} \ \\
			By \eqref{eq:error-rhs-term1}, \eqref{eq:error-rhs-term2}, \eqref{eq:error-rhs-term3}, \eqref{eq:error-rhs-term4}, \eqref{eq:non-linear-eq2}, \eqref{eq:non-linear-term2} and \eqref{eq:non-linear-eq4}, \eqref{eq:DLN-Error-Eq4} becomes
			\begin{align}
				& \begin{Vmatrix} {\phi_{n+1}^{j,h}} \\ {\phi_{n}^{j,h}} \end{Vmatrix}^{2}_{G(\theta)}-\begin{Vmatrix} {\phi_{n}^{j,h}} \\ {\phi_{n-1}^{j,h}} \end{Vmatrix}^{2}_{G(\theta)} + \frac{ \theta (1 - {\theta}^{2})}{4 ( 1 + \varepsilon_{n} \theta )^{2}} \| \Phi_{n}^{j,h} \|^{2} + \frac{\nu \widehat{k}_{n}}{4} \| \nabla \phi_{n,\beta}^{j,h} \|^{2} 
				\label{eq:DLN-Error-Eq5} \\
				&+ \frac{\nu \widehat{k}_{n}}{4} \Big[ 1 - C (\Omega,\theta) \Big( \frac{1 + \varepsilon_{n} \theta}{1 - \varepsilon_{n}} \Big)^{2} \frac{\widehat{k}_{n}}{ \nu h} \big\| \nabla \big( u_{n,\ast}^{j,h} \!-\! \langle  u^{h} \rangle_{n,\ast}  \big) \big\|^{2} \Big]
				\| \nabla \phi_{n,\beta}^{j,h} \|^{2} \notag \\ 
				\leq & \frac{C(\Omega,\theta) \widehat{k}_{n} \| u_{n,\beta}^{j} \|_{2}^{2} }{\nu} \! 
				\big( \| \phi_{n}^{j,h} \|^{2} + \| \phi_{n\!-\!1}^{j,h} \|^{2} \big) 
				+ \frac{C(\theta) k_{\rm{max}}^{4}}{\nu}  \int_{t_{n-1}}^{t_{n+1}} \| u_{ttt}^{j} \|_{-1}^{2} dt \notag \\
				&+ \frac{C(\Omega,\theta) h^{2r+2}}{\nu} \int_{t_{n-1}}^{t_{n+1}} \| u_{t}^{j} \|_{r+1}^{2} dt 
				+ C(\theta) \nu k_{\rm{max}}^{4} \int_{t_{n-1}}^{t_{n+1}} \| \nabla u_{tt}^{j} \|^{2} dt  \notag \\
				&+ \frac{C(\Omega)h^{2s+2}}{\nu} (k_{n} \!+\! k_{n-1}) \| p(t_{n,\beta}) \|_{s\!+\!1}^{2}  \notag  \\
				&+ \frac{C(\Omega,\theta) \!h^{2r} }{\nu} \! \| | \!\nabla u^{j} | \|_{\infty\!,\!0}^{2}
				\Big(\! k_{\rm{max}}^{4}\! \int_{t_{n\!-\!1}}^{t_{n\!+\!1}} \| u_{tt}^{j} \|_{r\!+\!1}^{2}\! dt 
				\!+\! (k_{n} \!+\! k_{n-1}) \| u^{j}(t_{n,\beta}) \|_{r\!+\!1}^{2} \!\Big) \notag \\
				&+\! \frac{C(\Omega,\theta)k_{\rm{max}}^{4}}{\nu} \big( \! \| |u^{j}| \|_{\infty,1}^{2} 
				\!+\! \| |u^{j}| \|_{\infty,1,\beta}^{2} \! \big) \!
				\int_{t_{n\!-\!1}}^{t_{n\!+\!1}} \! \| \nabla u_{tt}^{j} \|^{2} \! dt \notag \\
				&+ C(\Omega,\theta) h^{2r+1}  k_{\rm{max}}^{3} 
				\!\! \int_{t_{n-1}}^{t_{n+1}} \!\| u_{tt}^{j} \|_{r+1}^{2} dt 
				+ C(\Omega,\theta) h k_{\rm{max}}^{3} \!
				\int_{t_{n\!-\!1}}^{t_{n\!+\!1}} \! \| \nabla u_{tt}^{j} \|^{2} \! dt. \notag 
			\end{align}
			\begin{confidential}
				\color{darkblue}
				\begin{align*}
				&\frac{1}{\widehat{k}_{n}} \Big( \begin{Vmatrix} {\phi_{n+1}^{j,h}} \\ {\phi_{n}^{j,h}} \end{Vmatrix}^{2}_{G(\theta)}-\begin{Vmatrix} {\phi_{n}^{j,h}} \\ {\phi_{n-1}^{j,h}} \end{Vmatrix}^{2}_{G(\theta)} + \frac{ \theta (1 - {\theta}^{2})}{2 ( 1 + \varepsilon_{n} \theta )^{2}} \| \Phi_{n}^{j,h} \|^{2} \Big) + \nu \| \nabla \phi_{n,\beta}^{j,h} \|^{2}  \\
				\leq & \frac{C(\theta)}{\nu} k_{\rm{max}}^{3} \int_{t_{n-1}}^{t_{n+1}} \| u_{ttt}^{j} \|_{-1}^{2} dt 
				+ \frac{C(\Omega,\theta) h^{2r+2}}{\nu \widehat{k}_{n}} \int_{t_{n-1}}^{t_{n+1}} \| u_{t}^{j} \|_{r+1}^{2} dt \\
				&+ C(\theta) \nu k_{\rm{max}}^{3} \int_{t_{n-1}}^{t_{n+1}} \| \nabla u_{tt}^{j} \|^{2} dt 
				+ \frac{C(\Omega) h^{2s+2}}{\nu} \| p(t_{n,\beta}) \|_{s\!+\!1}^{2}  \\
				&+ \frac{C(\Omega,\theta) \!h^{2r} }{\nu} \! \| | \!\nabla u^{j} | \|_{\infty\!,\!0}^{2}
				\Big(\! k_{\rm{max}}^{3}\! \int_{t_{n\!-\!1}}^{t_{n\!+\!1}} \| u_{tt}^{j} \|_{r\!+\!1}^{2}\! dt 
				\!+\! \| u^{j}(t_{n\!,\!\beta}) \|_{r\!+\!1}^{2} \!\Big) 
				+ \frac{C(\Omega,\theta) \!\| u_{n\!,\!\beta}^{j} \|_{2}^{2} }{\nu} \! 
				\big( \| \phi_{n}^{j,h} \|^{2} + \| \phi_{n\!-\!1}^{j,h} \|^{2} \big) \\
				&+ \frac{C(\Omega,\theta)}{\nu} k_{\rm{max}}^{3} \big( \| |u^{j}| \|_{\infty,1}^{2} 
				+ \| |u^{j}| \|_{\infty,1,\beta}^{2} \big)
				\int_{t_{n-1}}^{t_{n+1}} \| \nabla u_{tt}^{j} \|^{2} dt 
				+ \frac{ C(\Omega,\theta) h^{2r+1}  k_{\rm{max}}^{3} }{\widehat{k}_{n} } 
				\int_{t_{n-1}}^{t_{n+1}} \| u_{tt}^{j} \|_{r+1}^{2} dt \\
				&+ \frac{C(\Omega,\theta) h k_{\rm{max}}^{3} }{\widehat{k}_{n} } 
				\int_{t_{n-1}}^{t_{n+1}} \| \nabla u_{tt}^{j} \|^{2} dt \\
				&+ \frac{\theta (1 - \theta^2)}{4 \widehat{k}_{n} (1 + \varepsilon_{n} \theta )^{2} } 
				\big\| \Phi_{n}^{j,h} \big\|^{2} 
				+ \frac{C(\theta) \widehat{k}_{n}}{h} \Big( \frac{1 + \varepsilon_{n} \theta }{1 - \varepsilon_{n}} \Big)^{2} 
				\big\| \nabla \big( u_{n,\ast}^{j,h} \!-\! \langle u^{h} \rangle_{n,\ast} \big) \big\|^{2} \| \nabla \phi_{n,\beta}^{j,h} \|^{2} + \big( \frac{1}{16} \cdot 8 \big) \nu \| \nabla \phi_{n,\beta}^{j,h} \|^{2}.
				\end{align*}
				\begin{align*}
				& \begin{Vmatrix} {\phi_{n+1}^{j,h}} \\ {\phi_{n}^{j,h}} \end{Vmatrix}^{2}_{G(\theta)}-\begin{Vmatrix} {\phi_{n}^{j,h}} \\ {\phi_{n-1}^{j,h}} \end{Vmatrix}^{2}_{G(\theta)} + \frac{ \theta (1 - {\theta}^{2})}{2 ( 1 + \varepsilon_{n} \theta )^{2}} \| \Phi_{n}^{j,h} \|^{2}  + \nu \widehat{k}_{n} \| \nabla \phi_{n,\beta}^{j,h} \|^{2}  \\
				\leq & \frac{C(\theta) k_{\rm{max}}^{4}}{\nu}  \int_{t_{n-1}}^{t_{n+1}} \| u_{ttt}^{j} \|_{-1}^{2} dt 
				+ \frac{C(\Omega,\theta) h^{2r+2}}{\nu} \int_{t_{n-1}}^{t_{n+1}} \| u_{t}^{j} \|_{r+1}^{2} dt
				+ C(\theta) \nu k_{\rm{max}}^{4} \int_{t_{n-1}}^{t_{n+1}} \| \nabla u_{tt}^{j} \|^{2} dt \\
				&+ \frac{C(\Omega)h^{2s\!+\!2}}{\nu} (k_{n} + k_{n-1}) \| p(t_{n,\beta}) \|_{s\!+\!1}^{2}  
				+ \frac{C(\Omega,\theta) \!h^{2r} }{\nu} \! \| | \!\nabla u^{j} | \|_{\infty\!,\!0}^{2}
				\Big(\! k_{\rm{max}}^{4}\! \int_{t_{n\!-\!1}}^{t_{n\!+\!1}} \| u_{tt}^{j} \|_{r\!+\!1}^{2}\! dt 
				\!+\! (k_{n} + k_{n-1}) \| u^{j}(t_{n\!,\!\beta}) \|_{r\!+\!1}^{2} \!\Big) \\
				&+ \frac{C(\Omega,\theta) \widehat{k}_{n} \| u_{n\!,\!\beta}^{j} \|_{2}^{2} }{\nu} \! 
				\big( \| \phi_{n}^{j,h} \|^{2} + \| \phi_{n\!-\!1}^{j,h} \|^{2} \big) 
				+ \frac{C(\Omega,\theta)}{\nu} k_{\rm{max}}^{4} \big( \| |u^{j}| \|_{\infty,1}^{2} 
				+ \| |u^{j}| \|_{\infty,1,\beta}^{2} \big)
				\int_{t_{n-1}}^{t_{n+1}} \| \nabla u_{tt}^{j} \|^{2} dt  \\
				&+ C(\Omega,\theta) h^{2r+1} k_{\rm{max}}^{3} \int_{t_{n-1}}^{t_{n+1}} \| u_{tt}^{j} \|_{r+1}^{2} dt 
				+ C(\Omega,\theta) h k_{\rm{max}}^{3} \int_{t_{n-1}}^{t_{n+1}} \| \nabla u_{tt}^{j} \|^{2} dt \\
				&+ \frac{\theta (1 - \theta^2)}{4 (1 + \varepsilon_{n} \theta )^{2} } 
				\big\| \Phi_{n}^{j,h} \big\|^{2} 
				+ \frac{C(\theta) \widehat{k}_{n}^{2}}{h} \Big( \frac{1 + \varepsilon_{n} \theta }{1 - \varepsilon_{n}} \Big)^{2} 
				\big\| \nabla \big( u_{n,\ast}^{j,h} \!-\! \langle u^{h} \rangle_{n,\ast} \big) \big\|^{2} \| \nabla \phi_{n,\beta}^{j,h} \|^{2} + \frac{\nu \widehat{k}_{n}}{2} \| \nabla \phi_{n,\beta}^{j,h} \|^{2}.
				\end{align*}
				\normalcolor
			\end{confidential}
			We apply CFL-like conditions in \eqref{eq:CFL-like-cond} for \eqref{eq:DLN-Error-Eq5} and sum \eqref{eq:DLN-Error-Eq5} over $n$ from $1$ to $M-1$ ($M=2,3, \cdots, N$)
			\begin{confidential}
				\color{darkblue}
				\begin{align*}
				& \begin{Vmatrix} {\phi_{M}^{j,h}} \\ {\phi_{M-1}^{j,h}} \end{Vmatrix}^{2}_{G(\theta)} 
				+ \sum_{n=1}^{M-1} \frac{ \theta (1 - {\theta}^{2})}{4 ( 1 + \varepsilon_{n} \theta )^{2}} \| \Phi_{n}^{j,h} \|^{2} 
				+ \sum_{n=1}^{M-1} \frac{\nu \widehat{k}_{n}}{4} \| \nabla \phi_{n,\beta}^{j,h} \|^{2} \\
				\leq & \frac{C(\Omega,\theta)}{\nu} \! \sum_{n\!=\!1}^{M-1}
				\widehat{k}_{n} \| u_{n,\beta}^{j} \|_{2}^{2} \big( \| \phi_{n}^{j,h} \|^{2} + \| \phi_{n\!-\!1}^{j,h} \|^{2} \big) 
				+\! \frac{C(\theta) k_{\rm{max}}^{4}}{\nu} \! \| u_{ttt}^{j} \|_{2,-1}^{2} \\
				&+ \frac{C(\Omega,\theta) h^{2r\!+\!2}}{\nu} \! \| u_{t}^{j} \|_{2,r\!+\!1}^{2}
				+ C(\theta) \nu k_{\rm{max}}^{4} \! \| \nabla u_{tt}^{j} \|_{2,0}^{2} \notag \\
				&+ \frac{C(\Omega) h^{2s\!+\!2}}{\nu} \| |p^{j}| \|_{2,s\!+\!1,\beta}^{2} 
				+ \frac{C(\Omega,\theta) \!h^{2r} }{\nu} \! \| | \!\nabla u^{j} | \|_{\infty\!,\!0}^{2}
				\Big(\! k_{\rm{max}}^{4}\! \| u_{tt}^{j} \|_{2,r\!+\!1}^{2} 
				\!+\! \| |u^{j}| \|_{2,r\!+\!1,\beta}^{2} \!\Big) \notag \\
				&+\! \frac{C(\Omega,\theta)k_{\rm{max}}^{4}}{\nu} \big( \! \| |u^{j}| \|_{\infty,1}^{2} 
				\!+\! \| |u^{j}| \|_{\infty,1,\beta}^{2} \! \big) \! \| \nabla u_{tt}^{j} \|_{2,0}^{2}
				\!+\! C(\Omega,\theta) \! h^{2r+1} \! k_{\rm{max}}^{3} \! \| u_{tt}^{j} \|_{2,r\!+\!1}^{2} \\
				&+ C(\Omega,\theta) \! h k_{\rm{max}}^{3} \! \| \nabla u_{tt}^{j} \|_{2,0}^{2}
				+ \begin{Vmatrix} {\phi_{1}^{j,h}} \\ {\phi_{0}^{j,h}} \end{Vmatrix}^{2}_{G(\theta)}.
				\end{align*}
				\normalcolor
			\end{confidential}
			\begin{align}
				&\| \phi_{M}^{j,h} \|^{2} + \frac{4}{1 + \theta} \sum_{n=1}^{M-1} \frac{\nu \widehat{k}_{n}}{4} \| \nabla \phi_{n,\beta}^{j,h} \|^{2} 
				\label{eq:DLN-Error-Eq6} \\
				\leq& \! \frac{C(\Omega,\theta)}{\nu}\! \Big[ \widehat{k}_{M-1} \| u_{M-1,\beta}^{j} \|_{2}^{2}
				\| \phi_{M-1}^{j,h} \|^{2} + \sum_{n=1}^{M-2} 
				\big( \widehat{k}_{n+1} \| u_{n+1,\beta}^{j} \|_{2}^{2} 
				+ \widehat{k}_{n} \| u_{n,\beta}^{j} \|_{2}^{2}  \big) \| \phi_{n}^{j,h} \|^{2} \notag \\
				&\qquad \qquad + \widehat{k}_{1} \| u_{1,\beta}^{j} \|_{2}^{2} \| \phi_{0}^{j,h} \|^{2} \Big]  \notag \\
				&+ \frac{C(\theta) k_{\rm{max}}^{4}}{\nu}  \| u_{ttt}^{j} \|_{2,-1}^{2} 
				+ \frac{C(\Omega,\theta) h^{2r+2}}{\nu} \| u_{t}^{j} \|_{2,r+1}^{2} 
				+ C(\theta) \nu k_{\rm{max}}^{4}  \| \nabla u_{tt}^{j} \|_{2,0}^{2} \notag \\
				&+ \frac{C(\Omega)h^{2s+2}}{\nu} \| |p^{j}| \|_{2,s+1,\beta}^{2} 
				+ C(\Omega,\theta) h k_{\rm{max}}^{3} \| u_{tt}^{j} \|_{2,r+1}^{2} \notag  \\
				&+ \frac{C(\Omega,\theta) h^{2r} }{\nu} \| | \nabla u^{j} | \|_{\infty,0}^{2}
				\Big( k_{\rm{max}}^{4} \| u_{tt}^{j} \|_{2,r+1}^{2} + \| |u^{j}| \|_{2,r+1,\beta}^{2} \Big) \notag \\
				&+ \frac{C(\Omega,\theta)k_{\rm{max}}^{4}}{\nu} \big( \| |u^{j}| \|_{\infty,1}^{2} 
				+ \| |u^{j}| \|_{\infty,1,\beta}^{2} \big) \| \nabla u_{tt}^{j} \|_{2,0}^{2} 
				+ C(\Omega,\theta) h^{2r+1} k_{\rm{max}}^{3} \| u_{tt}^{j} \|_{2,r+1}^{2} \notag \\
				&+ C(\Omega,\theta) h k_{\rm{max}}^{3} \| \nabla u_{tt}^{j} \|_{2,0}^{2}
				+ C(\theta) \big( \| \phi_{1}^{j} \|^{2} + \| \phi_{0}^{j} \|^{2} \big). \notag 
			\end{align}
			We denote
			\begin{align}
				F_{1} 
				= & \frac{C(\theta) k_{\rm{max}}^{4}}{\nu}  \| u_{ttt}^{j} \|_{2,-1}^{2} 
				+ \frac{C(\Omega,\theta) h^{2r+2}}{\nu} \| u_{t}^{j} \|_{2,r+1}^{2} 
				\label{eq:def-F1} \\
				+& C(\theta) \nu k_{\rm{max}}^{4}  \| \nabla u_{tt}^{j} \|_{2,0}^{2} 
				+ \frac{C(\Omega)h^{2s+2}}{\nu} \| |p^{j}| \|_{2,s+1,\beta}^{2} 
				+ C(\Omega,\theta) h k_{\rm{max}}^{3} \| u_{tt}^{j} \|_{2,r+1}^{2} \notag  \\
				+& \frac{C(\Omega,\theta) h^{2r} }{\nu} \| | \nabla u^{j} | \|_{\infty,0}^{2}
				\Big( k_{\rm{max}}^{4} \| u_{tt}^{j} \|_{2,r+1}^{2} + \| |u^{j}| \|_{2,r+1,\beta}^{2} \Big) \notag \\
				+& \frac{C(\Omega,\theta)k_{\rm{max}}^{4}}{\nu} \big( \| |u^{j}| \|_{\infty,1}^{2} 
				+ \| |u^{j}| \|_{\infty,1,\beta}^{2} \big) \| \nabla u_{tt}^{j} \|_{2,0}^{2} 
				+ C(\Omega,\theta) h^{2r+1} k_{\rm{max}}^{3} \| u_{tt}^{j} \|_{2,r+1}^{2} \notag \\
				+& C(\Omega,\theta) h k_{\rm{max}}^{3} \| \nabla u_{tt}^{j} \|_{2,0}^{2}
				+ C(\theta) \big( \| \phi_{1}^{j} \|^{2} + \| \phi_{0}^{j} \|^{2} \big). \notag
			\end{align}
			By discrete Gr$\rm{\ddot{o}}$nwall inequality without restrictions (\cite[p.369]{HR90_SIAM_NA}) and \eqref{eq:consist-2nd-eq1} in Lemma \ref{lemma:DLN-consistency}, \eqref{eq:DLN-Error-Eq6} becomes
			\begin{align}
				\| \phi_{M}^{h} \|^2
				+ \frac{\nu}{1+\theta} \sum_{n=1}^{M-1} \widehat{k}_{n} \| \nabla \phi_{n,\beta}^{h} \|^2   
				\leq& \exp \Big( \frac{C(\Omega,\theta)}{\nu} \sum_{n=1}^{N-1}  \widehat{k}_{n} \| u_{n,\beta} \|_{2}^{2} \Big) F_{1} \label{eq:DLN-Error-Eq7} \\
				\leq& \exp \Big( \frac{C(\Omega,\theta)}{\nu} \big(  k_{\rm{max}}^{4} \| u_{tt}^{j} \|_{2,2}^{2} 
			   	+ \| |u^{j}| \|_{2,2,\beta}^{2} \big)  \Big) F_{1}.
				\notag 
			\end{align}
			\begin{confidential}
				\color{darkblue}
				\begin{align*}
					\sum_{n=1}^{N-1} \widehat{k}_{n} \| u_{n,\beta}^{j} \|_{2}^{2}
					\leq& 2 \sum_{n=1}^{N-1} \widehat{k}_{n} \Big( \| u_{n,\beta}^{j} - u^{j}(t_{n,\beta}) \|_{2}^{2}
					+ \| u^{j}(t_{n,\beta}) \|_{2}^{2} \Big) \\
					\leq& C(\theta) \sum_{n\!=\!1}^{M-1} \Big( k_{\rm{max}}^{4} \int_{t_{n-1}}^{t_{n+1}} \| u_{tt}^{j} \|_{2}^{2} dt
					+ (k_{n} \!+\! k_{n-1}) \| u^{j}(t_{n,\beta}) \|_{2}^{2} \Big) 
				\end{align*}
				\normalcolor
			\end{confidential}
			By triangle inequality, \eqref{eq:approx-thm}, \eqref{eq:Stoke-Approx}, \eqref{eq:consist-2nd-eq1} 
			in Lemma \ref{lemma:DLN-consistency} and \eqref{eq:DLN-Error-Eq6}
			\begin{confidential}
				\color{darkblue}
				\begin{align*}
					&\Big( \nu \sum_{n=1}^{N-1} \widehat{k}_{n} \| \nabla e_{n,\beta}^{j} \|^{2} \Big)^{1/2}  \\
					\leq& \Big( 2 \nu \sum_{n=1}^{N-1} \widehat{k}_{n} \| \nabla \phi_{n,\beta}^{j,h} \|^{2} \Big)^{1/2}
					+ \Big( 2 \nu \sum_{n=1}^{N-1} \widehat{k}_{n} \| \nabla \eta_{n,\beta}^{j} \|^{2} \Big)^{1/2} \\
					\leq& \exp \Big[ \frac{C(\Omega,\theta)}{\nu} \big( k_{\rm{max}}^{4} \| u_{tt}^{j} \|_{2,2}^{2} 
					+ \| |u^{j}| \|_{2,2,\beta}^{2} \big) \Big] \sqrt{F_{1}} 
					+ C(\Omega)h^{r} \Big( \nu \sum_{n=1}^{N-1} \widehat{k}_{n} \| u_{n,\beta}^{j} \|_{r+1}^2 \Big)^{1/2} \\
					\leq& \exp \Big[ \frac{C(\Omega,\theta)}{\nu} \big( k_{\rm{max}}^{4} \| u_{tt} \|_{2,2}^{2} 
					+ \| |u| \|_{2,2,\beta}^{2} \big) \Big] \sqrt{F_{1}} \\
					&+ C(\Omega) \sqrt{\nu} h^{r} \Big( \sum_{n=1}^{N-1} \widehat{k}_{n} \| u_{n,\beta}^{j} - u^{j}(t_{n,\beta}) \|_{r+1}^2 + \sum_{n=1}^{N-1} \widehat{k}_{n} \| u^{j}(t_{n,\beta}) \|_{r+1}^2 \Big)^{1/2} \\
					\leq& \exp \Big[ \frac{C(\Omega,\theta)}{\nu} \big( k_{\rm{max}}^{4} \| u_{tt}^{j} \|_{2,2}^{2} 
					+ \| |u^{j}| \|_{2,2,\beta}^{2} \big) \Big] \sqrt{F_{1}} \\
					&+ C(\Omega)\sqrt{\nu} h^{r} \Big( C(\theta) \sum_{n=1}^{N-1} (k_{n}+k_{n-1})^{4} 
					\int_{t_{n-1}}^{t_{n+1}}\| u_{tt}^{j} \|_{r+1}^2 dt 
					+ \sum_{n=1}^{N-1} (k_{n}+k_{n-1}) \| u^{j}(t_{n,\beta}) \|_{r+1}^2 \Big)^{1/2} \\
					\leq& \exp \Big[ \frac{C(\Omega,\theta)}{\nu} \big( k_{\rm{max}}^{4} \| u_{tt}^{j} \|_{2,2}^{2} 
					+ \| |u^{j}| \|_{2,2,\beta}^{2} \big) \Big] \sqrt{F_{1}} 
					+ C(\Omega,\theta) \sqrt{\nu} h^{r} \Big( k_{\rm{max}}^{4} \| u_{tt}^{j} \|_{2,r+1}^{2}
					+ \| |u^{j}| \|_{2,r+1,\beta}^2 \Big)^{1/2} \\
					\leq& \exp \Big[ \frac{C(\Omega,\theta)}{\nu} \big( k_{\rm{max}}^{4} \| u_{tt}^{j} \|_{2,2}^{2} 
					+ \| |u^{j}| \|_{2,2,\beta}^{2} \big) \Big] \sqrt{F_{1}}
					+ C(\Omega,\theta) \sqrt{\nu} h^{r} \big( k_{\rm{max}}^{2} \| u_{tt}^{j} \|_{2,r+1}
					+ \| |u^{j}| \|_{2,r+1,\beta} \big),
				\end{align*}
				\normalcolor
			\end{confidential}
			\begin{align}
				\max_{0 \leq n \leq N}\| e_{n}^{j} \|  
				\leq& \max_{0 \leq n \leq N}\| \phi_{n}^{j,h} \| + \max_{0 \leq n \leq N}\| \eta_{n}^{j} \| 
				\label{eq:error-L2-inf-final-L2} \\
				\leq& \exp \Big[ \frac{C(\Omega,\theta)}{\nu} \big( k_{\rm{max}}^{4} \| u_{tt}^{j} \|_{2,2}^{2} 
				+ \| |u^{j}| \|_{2,2,\beta}^{2} \big) \Big] \sqrt{F_{1}} 
				+ C(\Omega)h^{r} \| |u^{j}| \|_{\infty,r}.  \notag 
			\end{align}
			\begin{align}
				\Big(\! \nu \sum_{n=1}^{N\!-\!1} \widehat{k}_{n} \| \nabla e_{n,\beta}^{j} \|^{2} \!\Big)^{1/2}
				\leq& \Big(\! 2 \nu \sum_{n=1}^{N\!-\!1} \widehat{k}_{n} \| \nabla \phi_{n,\beta}^{j,h} \|^{2} \!\Big)^{1/2}
				\!+\! \Big(\! 2 \nu \sum_{n=1}^{N\!-\!1} \widehat{k}_{n} \| \nabla \eta_{n,\beta}^{j} \|^{2} \!\Big)^{1/2} 
				\label{eq:error-L2-inf-final-H1} \\
				\leq& \exp \Big[ \frac{C(\Omega,\theta)}{\nu} \big( k_{\rm{max}}^{4} \| u_{tt}^{j} \|_{2,2}^{2} 
				+ \| |u^{j}| \|_{2,2,\beta}^{2} \big) \Big] \sqrt{F_{1}} \notag \\
				& \qquad \qquad \quad + C(\Omega,\theta) \sqrt{\nu} h^{r} \big( k_{\rm{max}}^{2} \| u_{tt}^{j} \|_{2,r+1}
				+ \| |u^{j}| \|_{2,r+1,\beta} \big). \notag 
			\end{align}
			We combine \eqref{eq:error-L2-inf-final-L2} and \eqref{eq:error-L2-inf-final-H1} and obtain 
			\eqref{eq:appendixB-error-L2-conclusion} in Theorem \ref{thm:Error-L2}.

		\end{proof}

		\subsection{Proof of Theorem \ref{thm:Error-H1}}
		\label{appendixB-H1} \ \\
		\begin{theorem}
			We assume that for the $j$-th NSE in \eqref{eq:jth-NSE}, the velocity $u^{j}(x,t)$ satisfies
			\begin{gather*}
			u^{j} \!\in\! \ell^{\infty}(\{ t_{n} \}_{n=0}^{N};H^{r}(\Omega) \cap H^{2}(\Omega) ) 
			\cap \ell^{\infty,\beta}(\{ t_{n} \}_{n=0}^{N};H^{2}(\Omega))
			\cap \ell^{2,\beta}(\{ t_{n} \}_{n=0}^{N};H^{r+1}(\Omega)), \\
			u_{t}^{j} \in L^{2}(0,T;H^{r+1}(\Omega)),  \ \ 
			u_{tt}^{j} \in L^{2}(0,T;H^{r+1}(\Omega)), \ \ 
			u_{ttt}^{j} \in L^{2}(0,T;X^{-1} \cap L^{2}(\Omega)), 
			\end{gather*} 
			and the pressure $p^{j}(x,t)$ satisfies
			\begin{gather*}
			p^{j} \in \ell^{\infty}(\{ t_{n} \}_{n=0}^{N};H^{s+1}(\Omega)) \cap \ell^{2,\beta}(\{ t_{n} \}_{n=0}^{N};H^{s+1}(\Omega)), \\
			p_{t}^{j} \in L^{2}(0,T;H^{s+1}(\Omega)), \ p_{tt}^{j} \in L^{2}(0,T;H^{s+1}(\Omega)),
			\end{gather*}
			for all $j$.  
			Under the CFL-like conditions in \eqref{eq:CFL-like-cond}, time ratio bounds in \eqref{eq:time-ratio-cond} and the time-diameter condition in \eqref{eq:time-diameter-relation},
			the numerical solutions of the DLN-Ensemble algorithms in \eqref{eq:DLN-Ensemble-Alg} for all $\theta \in (0,1)$ satisfy 
			\begin{align}
			&\max_{0 \leq n \leq N}\| u_{n}^{j} - u_{n}^{j,h} \|_{1}  
			\!+\! \sum_{n=1}^{N-1} \frac{\widehat{k}_{n}}{\nu} \Big\| \frac{ u_{n,\alpha}^{j} - u_{n,\alpha}^{j,h} }{ \widehat{k}_{n} } \Big\|^{2}
			\!\leq\! \mathcal{O} \big( h^{r},h^{s+1}, k_{\rm{max}}^{2}, h^{1/2}k_{\rm{max}}^{3/2} \big). \label{eq:error-H1-conclusion-appendix} 
			\end{align}
		\end{theorem}
		\begin{proof}
			We devide the proof into the following three parts:
		\begin{enumerate}
				\item[1.] We set $I_{\rm{St}} u_{n}^{j}$ to be the velocity component of 
				Stokes projection of $(u_{n}^{j},p_{n}^{j})$ onto $V^{h} \times Q^{h}$ and decompose the error to be $e_{n}^{j} = (u_{n}^{j} - I_{\rm{St}} u_{n}^{j} ) + (I_{\rm{St}} u_{n}^{j} - u_{n}^{j,h}) 
				:= \eta_{n}^{j} + \phi_{n}^{j,h}$.
				Then we transfer the error equation in \eqref{eq:DLN-Error-Eq2} to be the new equation 
				in terms of $\{\eta_{n-1+\ell}^{j}\}_{\ell=0}^{2}$ and $\{ \phi_{n-1+\ell}^{j} \}_{\ell=0}^{2}$.
			\item[2.] We obtain the bound for $\nabla \phi_{n+1}^{j,h}$ by addressing the terms from
			\begin{itemize}
				\item the semi-implicit DLN algorithms \cite{Pei24_NM} for $j$-th NSE in \eqref{eq:jth-NSE},
				\item the DLN-Ensemble algorithms in \eqref{eq:DLN-Ensemble-Alg}.
			\end{itemize}
			\item[3.] We use the discrete Gr\"onwall inequality without time restrictions \cite[p.369]{HR90_SIAM_NA}, approximations for Stokes projection in \eqref{eq:Stoke-Approx} and approximations for $(X^{h},Q^{h})$ in \eqref{eq:approx-thm} to achieve convergence of numerical solutions 
			in $H^{1}$-norm.
		\end{enumerate}
		\ \\
		\ \\
		\noindent \textit{Part 1.} \\
		We denote $I_{\rm{St}} u_{n}^{j}$ to be the velocity component of the Stokes projection of $(u_{n}^{j}, p_{n}^{j})$ and decompose the error of velocity at $t_{n}$ as
		\begin{gather*}
			e_{n}^{j} = u_{n}^{j} - u_{n}^{j,h} = \big(u_{n}^{j} - I_{\rm{St}} u_{n}^{j} \big)
			+ \big( I_{\rm{St}} u_{n}^{j} - u_{n}^{j,h} \big) = \eta_{n}^{j} + \phi_{n}^{j,h},  \\
			\eta_{n}^{j} = u_{n}^{j} - I_{\rm{St}} u_{n}^{j}, \qquad 
			\phi_{n}^{j,h} = I_{\rm{St}} u_{n}^{j} - u_{n}^{j,h},
		\end{gather*}
		and we still have \eqref{eq:DLN-Error-Eq2}.
		We take $v^{h} = \widehat{k}_{n}^{-1} \phi_{n,\alpha}^{j,h} \in V^{h}$ in \eqref{eq:DLN-Error-Eq2}. By the $G$-stability identity in \eqref{eq:Gstab-Id} and first equation in the definition of Stokes projection in \eqref{eq:Stokes-def}, \eqref{eq:DLN-Error-Eq2} becomes
		\begin{confidential}
			\color{darkblue}
			\begin{align*}
				&\Big\| \widehat{k}_{n}^{-1} \phi_{n,\alpha}^{j,h} \Big\|^{2} + 
				\frac{\nu}{\widehat{k}_{n}} \Big(\begin{Vmatrix}
				\nabla {\phi_{n+1}^{j,h}} \\
				\nabla {\phi_{n}^{j,h}}
				\end{Vmatrix}
				_{G(\theta)}^{2} -
				\begin{Vmatrix}
				\nabla {\phi_{n}^{j,h}} \\
				\nabla {\phi_{n-1}^{j,h}}
				\end{Vmatrix}%
				_{G(\theta)}^{2} 
				+ \Big\| \nabla \big(\sum_{\ell=0}^{2} \gamma_{\ell}^{(n)} \phi_{n\!-\!1\!+\!\ell}^{j,h} \big) \Big\| ^{2} \Big) \\
				=& \Big( \frac{u_{n,\alpha}^{j}}{\widehat{k}_{n}} - u_{t}^{j} (t_{n,\beta}), \widehat{k}_{n}^{-1} \phi_{n,\alpha}^{j,h} \Big) 
				\!-\! \frac{1}{\widehat{k}_{n}} \Big( \eta_{n,\alpha}^{j}, \widehat{k}_{n}^{-1} \phi_{n,\alpha}^{j,h} \Big) 
				+\! \nu \! \big( \nabla \Big(u_{n,\beta}^{j} \!-\! u^{j} (t_{n,\beta}) \big), \nabla \big(\widehat{k}_{n}^{-1} \phi_{n,\alpha}^{j,h} \big) \Big) \notag \\
				&+ \Big( p^{j} (t_{n,\beta}), \nabla \cdot \big(\widehat{k}_{n}^{-1} \phi_{n,\alpha}^{j,h} \big) \Big) 
				- \nu \Big( \nabla \eta_{n,\beta}^{j} , \nabla \big( \widehat{k}_{n}^{-1} \phi_{n,\alpha}^{j,h} \big) \Big) 
				- b \big( u^{j}( t_{n,\beta}), u^{j} (t_{n,\beta}) , \widehat{k}_{n}^{-1} \phi_{n,\alpha}^{j,h} \big) \\
				&+ b \big( u_{n,\ast}^{j,h} - \big( u_{n,\ast}^{j,h} \!-\! \langle u^{h} \rangle_{n,\ast} \big), u_{n,\beta}^{j,h}, \widehat{k}_{n}^{-1} \phi_{n,\alpha}^{j,h} \big) + b \big( u_{n,\ast}^{j,h} \!-\! \langle u^{h} \rangle_{n,\ast}, u_{n,\ast}^{j,h}, \widehat{k}_{n}^{-1} \phi_{n,\alpha}^{j,h} \big). \notag
			\end{align*}
			\begin{align*}
				&\Big\| \widehat{k}_{n}^{-1} \phi_{n,\alpha}^{j,h} \Big\|^{2} + 
				\frac{\nu}{\widehat{k}_{n}} \Big(\begin{Vmatrix}
				\nabla {\phi_{n+1}^{j,h}} \\
				\nabla {\phi_{n}^{j,h}}
				\end{Vmatrix}
				_{G(\theta)}^{2} -
				\begin{Vmatrix}
				\nabla {\phi_{n}^{j,h}} \\
				\nabla {\phi_{n-1}^{j,h}}
				\end{Vmatrix}%
				_{G(\theta)}^{2} 
				+ \Big\| \nabla \big(\sum_{\ell=0}^{2} \gamma_{\ell}^{(n)} \phi_{n\!-\!1\!+\!\ell}^{j,h} \big) \Big\| ^{2} \Big) \\
				=& \Big( \frac{u_{n,\alpha}^{j}}{\widehat{k}_{n}} - u_{t}^{j} (t_{n,\beta}), \widehat{k}_{n}^{-1} \phi_{n,\alpha}^{j,h} \Big) 
				\!-\! \frac{1}{\widehat{k}_{n}} \Big( \eta_{n,\alpha}^{j}, \widehat{k}_{n}^{-1} \phi_{n,\alpha}^{j,h} \Big) 
				+\! \nu \! \big( \nabla \Big(u_{n,\beta}^{j} \!-\! u^{j} (t_{n,\beta}) \big), \nabla \big(\widehat{k}_{n}^{-1} \phi_{n,\alpha}^{j,h} \big) \Big) \notag \\
				&+ \Big( p^{j} (t_{n,\beta}), \nabla \cdot \big(\widehat{k}_{n}^{-1} \phi_{n,\alpha}^{j,h} \big) \Big) 
				- \Big( p_{n,\beta}^{j}, \nabla \cdot \big(\widehat{k}_{n}^{-1} \phi_{n,\alpha}^{j,h} \big) \Big) \\  
				&+ \!b \big( u_{n,\ast}^{j,h}, u_{n,\beta}^{j,h}, \widehat{k}_{n}^{-1} \phi_{n,\alpha}^{j,h} \big) 
				\!-\! b \big( u_{n,\ast}^{j}, u_{n,\beta}^{j}, \widehat{k}_{n}^{-1} \phi_{n,\alpha}^{j,h}\big) 
				\!+\! b \big( u_{n,\ast}^{j}, u_{n,\beta}^{j}, \widehat{k}_{n}^{-1} \phi_{n,\alpha}^{j,h} \big)  \notag \\
				&- \!b \big( u^{j}( t_{n,\beta}), u^{j} (t_{n,\beta}) , \widehat{k}_{n}^{-1} \phi_{n,\alpha}^{j,h} \big) 
				\!+\! b \big( u_{n,\ast}^{j,h} \!-\! \langle u^{h} \rangle_{n,\ast}, u_{n,\ast}^{j,h} - u_{n,\beta}^{j,h}, \widehat{k}_{n}^{-1} \phi_{n,\alpha}^{j,h} \big). \notag 
			\end{align*}
			\normalcolor
		\end{confidential}
		\begin{align}
			&\Big\| \widehat{k}_{n}^{-1} \phi_{n,\alpha}^{j,h} \Big\|^{2} + 
			\frac{\nu}{\widehat{k}_{n}} \Big(\begin{Vmatrix}
			\nabla {\phi_{n+1}^{j,h}} \\
			\nabla {\phi_{n}^{j,h}}
			\end{Vmatrix}
			_{G(\theta)}^{2} \!\!\!-
			\begin{Vmatrix}
			\nabla {\phi_{n}^{j,h}} \\
			\nabla {\phi_{n-1}^{j,h}}
			\end{Vmatrix}%
			_{G(\theta)}^{2} 
			\!\!\!+ \Big\| \nabla \big(\sum_{\ell=0}^{2} \gamma_{\ell}^{(n)} \phi_{n\!-\!1\!+\!\ell}^{j,h} \big) \Big\| ^{2} \Big) 
			\label{eq:DLN-Error-H1-Eq1} \\
			=& \Big(\! \frac{u_{n,\alpha}^{j}}{\widehat{k}_{n}} \!-\! u_{t}^{j} (t_{n,\beta}), \widehat{k}_{n}^{-1} \phi_{n,\alpha}^{j,h} \!\Big) 
			\!-\! \frac{1}{\widehat{k}_{n}}\! \Big(\! \eta_{n,\alpha}^{j}, \widehat{k}_{n}^{-1} \phi_{n,\alpha}^{j,h} \! \Big) 
			\!+\! \nu \! \Big(\! \nabla \big(u_{n,\beta}^{j} \!-\! u^{j} (t_{n,\beta}) \big), \nabla \!\big(\widehat{k}_{n}^{-1} \phi_{n,\alpha}^{j,h} \big)\!\! \Big) \notag \\
			&+ \Big( p^{j} (t_{n,\beta}) - p_{n,\beta}^{j}, \nabla \cdot \big(\widehat{k}_{n}^{-1} \phi_{n,\alpha}^{j,h} \big) \Big) 
			+ b \big( u_{n,\ast}^{j,h}, u_{n,\beta}^{j,h}, \widehat{k}_{n}^{-1} \phi_{n,\alpha}^{j,h} \big) 
			- b \big( u_{n,\ast}^{j}, u_{n,\beta}^{j}, \widehat{k}_{n}^{-1} \phi_{n,\alpha}^{j,h}\big) 
			\notag \\
			&+ b \big( u_{n,\ast}^{j}, u_{n,\beta}^{j}, \widehat{k}_{n}^{-1} \phi_{n,\alpha}^{j,h} \big)  
			- b \big( u^{j}( t_{n,\beta}), u^{j} (t_{n,\beta}) , \widehat{k}_{n}^{-1} \phi_{n,\alpha}^{j,h} \big) \notag \\
			&+ b \big( u_{n,\ast}^{j,h} \!-\! \langle u^{h} \rangle_{n,\ast}, u_{n,\ast}^{j,h} - u_{n,\beta}^{j,h}, \widehat{k}_{n}^{-1} \phi_{n,\alpha}^{j,h} \big). \notag 
		\end{align}
			\ \\
			\noindent \textit{Part 2.}  Now we address all terms on the right hand side of \eqref{eq:DLN-Error-Eq4}.  \\
			\noindent $\bullet$ \textit{Terms from the semi-implicit DLN algorithms for $j$-th NSE in \eqref{eq:jth-NSE}}  \\
			By Cauchy-Schwarz inequality, Young's inequality and \eqref{eq:consist-2nd-eq3} in 
			Lemma \ref{lemma:DLN-consistency}
			\begin{align}
				\Big( \frac{u_{n,\alpha}^{j}}{\widehat{k}_{n}}  
				\!-\! u^{j}(t_{n,\beta}), \widehat{k}_{n}^{-1} \phi_{n,\alpha}^{j,h} \! \Big) 
				\leq& \Big\| \frac{u_{n,\alpha}^{j}}{\widehat{k}_{n}} - u^{j}(t_{n,\beta}) \Big\| 
				\| \widehat{k}_{n}^{-1} \phi_{n,\alpha}^{j,h} \|   
				\label{eq:DLN-Error-H1-Eq1-term1} \\
				\leq& C(\theta) k_{\rm{max}}^{3} \int_{t_{n-1}}^{t_{n+1}} \| u_{ttt}^{j} \|^{2} dt
				+ \frac{1}{16} \| \widehat{k}_{n}^{-1} \phi_{n,\alpha}^{j,h} \|^{2}.  \notag 
			\end{align}
			By Cauchy-Schwarz inequality, Young's inequality, \eqref{eq:approx-thm}, \eqref{eq:Stoke-Approx} and H$\ddot{\rm{o}}$lder's inequality
			\begin{confidential}
				\color{darkblue}
				\begin{align*}
				&\big( \widehat{k}_{n}^{-1} \eta_{n,\alpha}^{j}, \widehat{k}_{n}^{-1} \phi_{n,\alpha}^{j,h} \big)
				\leq \| \widehat{k}_{n}^{-1} \eta_{n,\alpha}^{j} \|\| \widehat{k}_{n}^{-1} \phi_{n,\alpha}^{j,h} \| 
				\leq \frac{C}{\widehat{k}_{n}^{2}} \| \eta_{n,\alpha}^{j} \|^{2}
				+ \frac{1}{16}\| \widehat{k}_{n}^{-1} \phi_{n,\alpha}^{j,h} \|^{2}  \\
				\leq& \frac{C}{\widehat{k}_{n}^{2}} \Big( h^{2r+2} \| u_{n,\alpha}^{j} \|_{r+1}^{2} 
				+ \frac{h^{2s+4}}{\nu^{2}} \| p_{n,\alpha}^{j} \|_{s+1}^{2} \Big) 
				+ \frac{1}{16}\| \widehat{k}_{n}^{-1} \phi_{n,\alpha}^{j,h} \|^{2}  \\
				\leq& \frac{C}{\widehat{k}_{n}^{2}} \Big( h^{2r+2} 
				C(\theta) (k_{n}+k_{n-1}) \int_{t_{n-1}}^{t_{n+1}} \| u_{t}^{j} \|_{r+1}^{2} dt
				+ \frac{h^{2s+4}}{\nu^{2}} C(\theta) (k_{n}+k_{n-1}) \int_{t_{n-1}}^{t_{n+1}} \| p_{t}^{j} \|_{s+1}^{2} dt \Big) 
				+ \frac{1}{16}\| \widehat{k}_{n}^{-1} \phi_{n,\alpha}^{j,h} \|^{2}
				\end{align*}
			\normalcolor
			\end{confidential}
			\begin{align}
				&\big( \widehat{k}_{n}^{-1} \eta_{n,\alpha}^{j}, \widehat{k}_{n}^{-1} \phi_{n,\alpha}^{j,h} \big)  
				\label{eq:DLN-Error-H1-Eq1-term2} \\
				\leq& \frac{C}{\widehat{k}_{n}^{2}} \Big( h^{2r+2} \| u_{n,\alpha}^{j} \|_{r+1}^{2} 
				+ \frac{h^{2s+4}}{\nu^{2}} \| p_{n,\alpha}^{j} \|_{s+1}^{2} \Big) 
				+ \frac{1}{16}\| \widehat{k}_{n}^{-1} \phi_{n,\alpha}^{j,h} \|^{2} \notag \\
				\leq& \frac{C(\theta)}{\widehat{k}_{n}} \Big( h^{2r+2} 
				\int_{t_{n-1}}^{t_{n+1}} \| u_{t}^{j} \|_{r+1}^{2} dt 
				+ \frac{h^{2s+4}}{\nu^{2}} \int_{t_{n-1}}^{t_{n+1}} \| p_{t}^{j} \|_{s+1}^{2} dt \Big) 
				+ \frac{1}{16}\| \widehat{k}_{n}^{-1} \phi_{n,\alpha}^{j,h} \|^{2}.   \notag 
			\end{align}
			By Gauss diverence theorem, Cauchy-Schwarz inequality, Young's inequality and \eqref{eq:consist-2nd-eq1} in 
			Lemma \ref{lemma:DLN-consistency}
			\begin{align}
				\nu \big( \nabla (u_{n,\beta} - u(t_{n,\beta})), \nabla \widehat{k}_{n}^{-1} \phi_{n,\alpha}^{h} \big) 
				=& \nu \big( \Delta (u_{n,\beta} - u(t_{n,\beta})), \widehat{k}_{n}^{-1} \phi_{n,\alpha}^{h} \big) 
				\label{eq:DLN-Error-H1-Eq1-term3} \\
				\leq& C(\theta) \nu^{2} k_{\rm{max}}^{3} \int_{t_{n-1}}^{t_{n+1}} \| u_{tt} \|_{2}^{2} dt 
				+ \frac{1}{16} \| \widehat{k}_{n}^{-1} \phi_{n,\alpha}^{h} \|^{2}. \notag  \\
				\big(p_{n,\beta}^{j} - p^{j}(t_{n,\beta}), \nabla \cdot \widehat{k}_{n}^{-1} \phi_{n,\alpha}^{j,h} \big)
				\leq& C \big\| \nabla (p_{n,\beta}^{j} - p^{j}(t_{n,\beta})) \big\|^{2}
				+ \frac{1}{16} \| \widehat{k}_{n}^{-1} \phi_{n,\alpha}^{j,h} \|^{2} 
				\label{eq:DLN-Error-H1-Eq1-term4} \\
				\leq& C(\theta) k_{\rm{max}}^3 \int_{t_{n-1}}^{t_{n+1}} \| \nabla p_{tt}^{j} \|^{2} dt 
				+ \frac{1}{16} \| \widehat{k}_{n}^{-1} \phi_{n,\alpha}^{j,h} \|^{2}.   \notag
			\end{align}
			\begin{confidential}
				\color{darkblue}
				\begin{align*}
				\big(p_{n,\beta}^{j} - p^{j}(t_{n,\beta}), \nabla \cdot \widehat{k}_{n}^{-1} \phi_{n,\alpha}^{j,h} \big) 
				=& - \big(\nabla (p_{n,\beta}^{j} - p^{j}(t_{n,\beta})), \widehat{k}_{n}^{-1} \phi_{n,\alpha}^{j,h} \big) \\
				\leq & C \big\| \nabla (p_{n,\beta}^{j} - p^{j}(t_{n,\beta})) \big\|^{2}
				+ \frac{1}{16} \| \widehat{k}_{n}^{-1} \phi_{n,\alpha}^{j,h} \|^{2} \\
				\leq & C(\theta) (k_{n} + k_{n-1})^3 \int_{t_{n-1}}^{t_{n+1}} \| \nabla p_{tt}^{j} \|^{2} dt 
				+ \frac{1}{16} \| \widehat{k}_{n}^{-1} \phi_{n,\alpha}^{j,h} \|^{2}.
				\end{align*}
				\normalcolor
			\end{confidential}
			For non-linear terms, it's easy to check 
			\begin{confidential}
				\color{darkblue}
				\begin{align*}
				&  b \big( u_{n,\ast}^{j,h}, u_{n,\beta}^{j,h}, \widehat{k}_{n}^{-1} \phi_{n,\alpha}^{j,h} \big)
				- b \big( u_{n,\ast}^{j}, u_{n,\beta}^{j}, \widehat{k}_{n}^{-1} \phi_{n,\alpha}^{j,h} \big)\\
				=& b \big( u_{n,\ast}^{j,h}, u_{n,\beta}^{j,h}, \widehat{k}_{n}^{-1} \phi_{n,\alpha}^{j,h} \big) 
				- b \big( u_{n,\ast}^{j}, u_{n,\beta}^{j,h}, \widehat{k}_{n}^{-1} \phi_{n,\alpha}^{j,h} \big) 
				+ b \big(  u_{n,\ast}^{j}, u_{n,\beta}^{j,h}, \widehat{k}_{n}^{-1} \phi_{n,\alpha}^{j,h} \big) 
				- b \big( u_{n,\ast}^{j}, u_{n,\beta}^{j}, \widehat{k}_{n}^{-1} \phi_{n,\alpha}^{j,h} \big) \\
				=& - b \big( e_{n,\ast}^{j}, u_{n,\beta}^{j,h}, \widehat{k}_{n}^{-1} \phi_{n,\alpha}^{j,h} \big)
				- b \big( u_{n,\ast}^{j}, e_{n,\beta}^{j}, \widehat{k}_{n}^{-1} \phi_{n,\alpha}^{j,h} \big) \\
				=& - b \big( e_{n,\ast}^{j}, u_{n,\beta}^{j,h}, \widehat{k}_{n}^{-1} \phi_{n,\alpha}^{j,h} \big) 
				+ b \big( e_{n,\ast}^{j}, u_{n,\beta}^{j}, \widehat{k}_{n}^{-1} \phi_{n,\alpha}^{j,h} \big)
				- b \big( e_{n,\ast}^{j}, u_{n,\beta}^{j}, \widehat{k}_{n}^{-1} \phi_{n,\alpha}^{j,h} \big)
				- b \big( u_{n,\ast}^{j}, e_{n,\beta}^{j}, \widehat{k}_{n}^{-1} \phi_{n,\alpha}^{j,h} \big) \\
				=& b \big( e_{n,\ast}^{j}, u_{n,\beta}^{j}, \widehat{k}_{n}^{-1} \phi_{n,\alpha}^{j,h} \big)
				- b \big( u_{n,\ast}^{j}, e_{n,\beta}^{j}, \widehat{k}_{n}^{-1} \phi_{n,\alpha}^{j,h} \big)
				+ b \big( e_{n,\ast}^{j}, e_{n,\beta}^{j}, \widehat{k}_{n}^{-1} \phi_{n,\alpha}^{j,h} \big) 
				\end{align*}
				\normalcolor
			\end{confidential}
			\begin{align*}
				& b \big( u_{n,\ast}^{j,h}, u_{n,\beta}^{j,h}, \widehat{k}_{n}^{-1} \phi_{n,\alpha}^{j,h} \big)
				- b \big( u_{n,\ast}^{j}, u_{n,\beta}^{j}, \widehat{k}_{n}^{-1} \phi_{n,\alpha}^{j,h} \big)
				\\
				= & b \big( e_{n,\ast}^{j}, u_{n,\beta}^{j}, \widehat{k}_{n}^{-1} \phi_{n,\alpha}^{j,h} \big)
				- b \big( u_{n,\ast}^{j}, e_{n,\beta}^{j}, \widehat{k}_{n}^{-1} \phi_{n,\alpha}^{j,h} \big)
				+ b \big( e_{n,\ast}^{j}, e_{n,\beta}^{j}, \widehat{k}_{n}^{-1} \phi_{n,\alpha}^{j,h} \big).
			\end{align*}
			By \eqref{eq:b-bound-1}, \eqref{eq:b-bound-3}, Poincar\'e inequality and inverse inequality in \eqref{eq:inv-inequal}
			\begin{align*}
				b \big( e_{n,\ast}^{j}, u_{n,\beta}^{j}, \widehat{k}_{n}^{-1} \phi_{n,\alpha}^{j,h} \big) 
				=& b \big( \phi_{n,\ast}^{j,h}, u_{n,\beta}^{j}, \widehat{k}_{n}^{-1} \phi_{n,\alpha}^{j,h} \big)
				+ b \big( \eta_{n,\ast}^{j}, u_{n,\beta}^{j}, \widehat{k}_{n}^{-1} \phi_{n,\alpha}^{j,h} \big) 
				\notag \\
				\leq& C(\Omega) \big( \| \phi_{n,\ast}^{j,h} \|_{1} + \| \eta_{n,\ast}^{j} \|_{1} \big) 
				\| u_{n,\beta}^{j} \|_{2} \| \widehat{k}_{n}^{-1} \phi_{n,\alpha}^{j,h} \| \\
				b \big( u_{n,\ast}^{j}, e_{n,\beta}^{j}, \widehat{k}_{n}^{-1} \phi_{n,\alpha}^{j,h} \big)
				\leq& C(\Omega) \| u_{n,\ast}^{j} \|_{2} \| e_{n,\beta}^{j} \|_{1} 
				\| \widehat{k}_{n}^{-1} \phi_{n,\alpha}^{j,h} \|,  \notag \\
				b \big( e_{n,\ast}^{j}, e_{n,\beta}^{j}, \widehat{k}_{n}^{-1} \phi_{n,\alpha}^{j,h} \big)
				=& b \big( \phi_{n,\ast}^{j,h}, e_{n,\beta}^{j}, \widehat{k}_{n}^{-1} \phi_{n,\alpha}^{j,h} \big)
				+ b \big( \eta_{n,\ast}^{j}, e_{n,\beta}^{j}, \widehat{k}_{n}^{-1} \phi_{n,\alpha}^{j,h} \big)
				\notag \\
				\leq& C(\Omega) h^{-1/2} \big( \| \phi_{n,\ast}^{j,h} \|_{1} + \| \eta_{n,\ast}^{j}\|_{1} \big) \| e_{n,\beta}^{j} \|_{1} \| \widehat{k}_{n}^{-1} \phi_{n,\alpha}^{j,h} \|.
			\end{align*}
			Thus 
			\begin{align}
				&b \big( u_{n,\ast}^{j,h}, u_{n,\beta}^{j,h}, \widehat{k}_{n}^{-1} \phi_{n,\alpha}^{j,h} \big)
				- b \big( u_{n,\ast}^{j}, u_{n,\beta}^{j}, \widehat{k}_{n}^{-1} \phi_{n,\alpha}^{j,h} \big) 
				\label{eq:DLN-Error-H1-Eq1-term5-6} \\
				\leq& C(\Omega) \big( \| u_{n,\beta}^{j} \|_{2} + h^{-1/2} \| e_{n,\beta}^{j} \|_{1} \big) 
				\| \phi_{n,\ast}^{j,h} \|_{1} \| \widehat{k}_{n}^{-1} \phi_{n,\alpha}^{j,h} \| 
				+ C(\Omega) \| \eta_{n,\ast}^{j} \|_{1} \| u_{n,\beta}^{j} \|_{2} 
				\| \widehat{k}_{n}^{-1} \phi_{n,\alpha}^{j,h} \|  \notag \\
				&+ C(\Omega) \| u_{n,\ast}^{j} \|_{2} \| e_{n,\beta}^{j} \|_{1} 
				\| \widehat{k}_{n}^{-1} \phi_{n,\alpha}^{j,h} \| 
				+ C(\Omega) h^{-1/2} \| \eta_{n,\ast}^{j} \|_{1} \| e_{n,\beta}^{j} \|_{1} 
				\| \widehat{k}_{n}^{-1} \phi_{n,\alpha}^{j,h} \|.     \notag 
			\end{align}
			By Cauchy-Schwarz inequality, Young's inequality, Poincar\'e inequality, \eqref{eq:approx-thm}, \eqref{eq:Stoke-Approx} and \eqref{eq:consist-2nd-eq2}
			\begin{align}
				& C(\Omega) \big( \| u_{n,\beta}^{j} \|_{2} + h^{-1/2} \| e_{n,\beta}^{j} \|_{1} \big) 
				\| \phi_{n,\ast}^{j,h} \|_{1} \| \widehat{k}_{n}^{-1} \phi_{n,\alpha}^{j,h} \| 
				\label{eq:DLN-Error-H1-Eq1-term5-6-sub-term1} \\
				\leq & C(\Omega,\theta) \big( \| u_{n,\beta}^{j} \|_{2}+ h^{-1/2} \| e_{n,\beta}^{j} \|_{1} \big)
				\big( \| \nabla \phi_{n}^{j,h} \| + \| \nabla \phi_{n-1}^{j,h} \| \big)
				+ \frac{1}{128} \| \widehat{k}_{n}^{-1} \phi_{n,\alpha}^{j,h} \|^{2} \notag \\
				\leq & C(\Omega,\theta) \big( \| u_{n,\beta}^{j} \|_{2}^{2} + h^{-1}\| e_{n,\beta}^{j} \|_{1}^{2} \big)
				\big( \| \nabla \phi_{n}^{j,h} \|^{2} + \| \nabla \phi_{n-1}^{j,h} \|^{2} \big)
				+ \frac{1}{128} \| \widehat{k}_{n}^{-1} \phi_{n,\alpha}^{j,h} \|^{2}, \notag
			\end{align}
			\begin{align}
				&C(\Omega) \| \eta_{n,\ast}^{j} \|_{1} \| u_{n,\beta}^{j} \|_{2}  
				\| \widehat{k}_{n}^{-1} \phi_{n,\alpha}^{j,h} \|  
				\label{eq:DLN-Error-H1-Eq1-term5-6-sub-term2} \\
				\leq& C(\Omega,\theta) \| |u^{j}| \|_{\infty,2}^{2} \Big( \frac{h^{2s+2}}{\nu^{2}} 
				\| p_{n,\ast}^{j} \|_{s+1}^{2} + h^{2r} \| u_{n,\ast}^{j} \|_{r+1}^{2} \Big) 
				+ \frac{1}{128} \| \widehat{k}_{n}^{-1} \phi_{n,\alpha}^{j,h} \|^{2} \notag \\
				\leq& C(\Omega,\theta) \| |u^{j}| \|_{\infty,2}^{2} \Big[ \frac{h^{2s+2}}{\nu^{2}} 
				\big( \| p_{n,\ast}^{j} - p^{j}(t_{n,\beta}) \|_{s+1}^{2} + \| p^{j}(t_{n,\beta})\|_{s+1}^{2} \big) \notag \\
				&\qquad \qquad \qquad + h^{2r} \big( \| u_{n,\ast}^{j} - u^{j}(t_{n,\beta}) \|_{r+1}^{2} + \| u^{j}(t_{n,\beta}) \|_{r+1}^{2} \big) \Big]  
				+ \frac{1}{128} \| \widehat{k}_{n}^{-1} \phi_{n,\alpha}^{j,h} \|^{2}  \notag \\ 
				\leq& C(\Omega,\theta) \| |u^{j}| \|_{\infty,2}^{2} \Big[ \frac{h^{2s\!+\!2}}{\nu^{2}} 
				\Big( C(\theta) k_{\rm{max}}^{3} \int_{t_{n-1}}^{t_{n+1}} \| p_{tt}^{j} \|_{s\!+\!1}^{2} dt 
				+ \| p^{j}(t_{n,\beta})\|_{s\!+\!1}^{2}  \Big) \notag \\
				&\qquad \qquad \qquad + h^{2r} \Big( C(\theta) k_{\rm{max}}^{3} \int_{t_{n\!-\!1}}^{t_{n\!+\!1}} \| u_{tt}^{j} \|_{r\!+\!1}^{2} dt + \| u^{j}(t_{n,\beta})\|_{r\!+\!1}^{2} \Big) \Big]  
				+ \frac{1}{128} \| \widehat{k}_{n}^{-1} \phi_{n,\alpha}^{j,h} \|^{2}, \notag 
			\end{align}
			\begin{confidential}
				\color{darkblue}
				\begin{align*}
				&C(\Omega) \| \eta_{n,\ast}^{j} \|_{1} \| u_{n,\beta}^{j} \|_{2} 
				\| \widehat{k}_{n}^{-1} \phi_{n,\alpha}^{j,h} \|  
				\leq C(\Omega) \| \eta_{n,\ast}^{j} \|_{1}^{2} \| u_{n,\beta}^{j} \|_{2}^{2}
				+ \frac{1}{128} \| \widehat{k}_{n}^{-1} \phi_{n,\alpha}^{j,h} \|^{2} \\
				\leq& C(\Omega,\theta) \| |u^{j}| \|_{\infty,2}^{2} \Big( \frac{h^{2s+2}}{\nu^{2}} 
				\| p_{n,\ast}^{j} \|_{s+1}^{2} + h^{2r} \| u_{n,\ast}^{j} \|_{r+1}^{2} \Big) 
				+ \frac{1}{128} \| \widehat{k}_{n}^{-1} \phi_{n,\alpha}^{j,h} \|^{2} \\
				\leq& C(\Omega,\theta) \| |u^{j}| \|_{\infty,2}^{2} \Big[ \frac{h^{2s+2}}{\nu^{2}} 
				\big( \| p_{n,\ast}^{j} - p^{j}(t_{n,\beta}) \|_{s+1}^{2} + \| p^{j}(t_{n,\beta})\|_{s+1}^{2} \big) \\
				&+ h^{2r} \big( \| u_{n,\ast}^{j} - u^{j}(t_{n,\beta}) \|_{r+1}^{2} + \| u^{j}(t_{n,\beta}) \|_{r+1}^{2} \big) \Big]  
				+ \frac{1}{128} \| \widehat{k}_{n}^{-1} \phi_{n,\alpha}^{j,h} \|^{2} \\
				\leq& C(\Omega,\theta) \| |u^{j}| \|_{\infty,2}^{2} \Big[ \frac{h^{2s+2}}{\nu^{2}} 
				\Big( C(\Omega,\theta) k_{\rm{max}}^{3} \int_{t_{n-1}}^{t_{n+1}} \| p_{tt}^{j} \|_{s+1}^{2} dt 
				+ \| p^{j}(t_{n,\beta})\|_{s+1}^{2}  \Big) \\
				&+ h^{2r} \big( C(\theta) k_{\rm{max}}^{3} \int_{t_{n-1}}^{t_{n+1}} \| u_{tt}^{j} \|_{r+1}^{2} dt + \| u^{j}(t_{n,\beta})\|_{r+1}^{2} \big) \Big]  
				+ \frac{1}{128} \| \widehat{k}_{n}^{-1} \phi_{n,\alpha}^{j,h} \|^{2}
				\end{align*}
				\normalcolor
			\end{confidential}
			\begin{align}
				C(\Omega) \| u_{n,\ast}^{j} \|_{2} \| e_{n,\beta}^{j} \|_{1} 
				\| \widehat{k}_{n}^{-1} \phi_{n,\alpha}^{j,h} \|
				\leq& C(\Omega) \| u_{n,\ast}^{j} \|_{2}^{2} \| e_{n,\beta}^{j} \|_{1}^{2}
				+ \frac{1}{128} \| \widehat{k}_{n}^{-1} \phi_{n,\alpha}^{j,h} \|^{2} 
				\label{eq:DLN-Error-H1-Eq1-term5-6-sub-term3}
				\\
				\leq& C(\Omega,\theta) \| |u^{j}| \|_{\infty,2}^{2} \| e_{n,\beta}^{j} \|_{1}^{2}
				+ \frac{1}{128} \| \widehat{k}_{n}^{-1} \phi_{n,\alpha}^{j,h} \|^{2}, \notag 
				\end{align}
				\begin{align}
				&C(\Omega) h^{-1/2} \| \eta_{n,\ast}^{j} \|_{1} \| e_{n,\beta}^{j} \|_{1} 
				\| \widehat{k}_{n}^{-1} \phi_{n,\alpha}^{j,h} \|  
				\label{eq:DLN-Error-H1-Eq1-term5-6-sub-term4} \\
				\leq& C(\Omega) h^{-1} \| \eta_{n,\ast}^{j} \|_{1}^{2} \| e_{n,\beta}^{j} \|_{1}^{2}
				+ \frac{1}{128} \| \widehat{k}_{n}^{-1} \phi_{n,\alpha}^{j,h} \|^{2}  \notag \\
				\leq& C(\Omega,\theta) \big( \frac{h^{2s+1}}{\nu^{2}} \| |p^{j}| \|_{\infty,s+1}^{2} 
				+ h^{2r-1} \| |u^{j}| \|_{\infty,r+1}^{2} \big) \| e_{n,\beta}^{j} \|_{1}^{2}
				+ \frac{1}{128} \| \widehat{k}_{n}^{-1} \phi_{n,\alpha}^{j,h} \|^{2}. \notag 
			\end{align}
			\begin{confidential}
				\color{darkblue}
				\begin{align*}
					C(\Omega) h^{-1/2} \| \eta_{n,\ast}^{j} \|_{1} \| e_{n,\beta}^{j} \|_{1} 
					\| \widehat{k}_{n}^{-1} \phi_{n,\alpha}^{j,h} \| 
					\leq& C(\Omega) h^{-1} \| \eta_{n,\ast}^{j} \|_{1}^{2} \| e_{n,\beta}^{j} \|_{1}^{2}
					+ \frac{1}{128} \| \widehat{k}_{n}^{-1} \phi_{n,\alpha}^{j,h} \|^{2}  \\
					\leq& C(\Omega) h^{-1} \Big( \frac{h^{2s+2}}{\nu^{2}} \| p_{n,\ast}^{j} \|_{s+1}^{2} 
					+ h^{2r} \| u_{n,\ast}^{j} \|_{r+1}^{2} \Big) \| e_{n,\beta}^{j} \|_{1}^{2}
					+ \frac{1}{128} \| \widehat{k}_{n}^{-1} \phi_{n,\alpha}^{j,h} \|^{2} \\
					\leq& C(\Omega,\theta) \Big( \frac{h^{2s+1}}{\nu^{2}} \| |p^{j}| \|_{\infty,s+1}^{2} 
					+ h^{2r-1} \| |u^{j}| \|_{\infty,r+1}^{2} \Big) \| e_{n,\beta}^{j} \|_{1}^{2}
					+ \frac{1}{128} \| \widehat{k}_{n}^{-1} \phi_{n,\alpha}^{j,h} \|^{2}.
				\end{align*}
				\normalcolor
			\end{confidential}
			By \eqref{eq:DLN-Error-H1-Eq1-term5-6-sub-term1}, \eqref{eq:DLN-Error-H1-Eq1-term5-6-sub-term2}, 
			\eqref{eq:DLN-Error-H1-Eq1-term5-6-sub-term3} and \eqref{eq:DLN-Error-H1-Eq1-term5-6-sub-term4}, \eqref{eq:DLN-Error-H1-Eq1-term5-6} becomes 
			\begin{align}
				&b \big( u_{n,\ast}^{j,h}, u_{n,\beta}^{j,h}, \widehat{k}_{n}^{-1} \phi_{n,\alpha}^{j,h} \big)
				- b \big( u_{n,\ast}^{j}, u_{n,\beta}^{j}, \widehat{k}_{n}^{-1} \phi_{n,\alpha}^{j,h} \big)  
				\label{eq:DLN-Error-H1-term5-6-eq2} \\
				\leq& C(\Omega,\theta) \big( \| u_{n,\beta}^{j} \|_{2}^{2} + h^{-1}\| e_{n,\beta}^{j} \|_{1}^{2} \big)
				\big( \| \nabla \phi_{n}^{j,h} \|^{2} + \| \nabla \phi_{n-1}^{j,h} \|^{2} \big) 
				+ C(\Omega,\theta) \| |u^{j}| \|_{\infty,2}^{2} \| e_{n,\beta}^{j} \|_{1}^{2} \notag \\
				&+ C(\Omega,\theta) \Big( \frac{h^{2s+1}}{\nu^{2}} \| |p^{j}| \|_{\infty,s+1}^{2} 
				+ h^{2r-1} \| |u^{j}| \|_{\infty,r+1}^{2} \Big) \| e_{n,\beta}^{j} \|_{1}^{2}
				+ \frac{1}{32} \| \widehat{k}_{n}^{-1} \phi_{n,\alpha}^{j,h} \|^{2} \notag \\
				&+ C(\Omega,\theta) \| |u^{j}| \|_{\infty,2}^{2} \Big[ \frac{h^{2s\!+\!2}}{\nu^{2}} 
				\Big( k_{\rm{max}}^{3} \int_{t_{n-1}}^{t_{n+1}} \| p_{tt}^{j} \|_{s\!+\!1}^{2} dt 
				+ \| p^{j}(t_{n,\beta})\|_{s\!+\!1}^{2}  \Big) \notag \\
				& \qquad \qquad \qquad \qquad \qquad \qquad 
				+ h^{2r} \Big( k_{\rm{max}}^{3} \int_{t_{n\!-\!1}}^{t_{n\!+\!1}} \| u_{tt}^{j} \|_{r\!+\!1}^{2} dt 
				+ \| u^{j}(t_{n,\beta})\|_{r\!+\!1}^{2} \Big) \Big]. 
				\notag 
			\end{align}
			By \eqref{eq:b-bound-1}, \eqref{eq:b-bound-3} and \eqref{eq:consist-2nd-eq1},\eqref{eq:consist-2nd-eq2} in Lemma \ref{lemma:DLN-consistency}
			\begin{align}
				&b \big( u_{n,\ast}^{j}, u_{n,\beta}, \widehat{k}_{n}^{-1} \phi_{n,\alpha}^{j,h} \big)
				- b ( u^{j}(t_{n,\beta}), u^{j}(t_{n,\beta}), \widehat{k}_{n}^{-1} \phi_{n,\alpha}^{j,h} )  
				\label{eq:DLN-Error-H1-term7-8} \\
				=& b \big( u_{n,\ast}^{j} - u(t_{n,\beta}), u_{n,\beta}^{j}, \widehat{k}_{n}^{-1} \phi_{n,\alpha}^{j,h} \big)
				+ b \big( u^{j}(t_{n,\beta}), u_{n,\beta}^{j} - u^{j}(t_{n,\beta}), \widehat{k}_{n}^{-1} \phi_{n,\alpha}^{j,h} \big) \notag \\
				\leq& C(\Omega) \| u_{n,\ast}^{j} - u^{j}(t_{n,\beta}) \|_{1}
				\| u_{n,\beta}^{j} \|_{2} \| \widehat{k}_{n}^{-1} \phi_{n,\alpha}^{j,h} \| 
				+ C(\Omega) \| u^{j}(t_{n,\beta}) \|_{2} \| u_{n,\beta}^{j} - u^{j}(t_{n,\beta}) \|_{1}
				\| \widehat{k}_{n}^{-1} \phi_{n,\alpha}^{j,h} \|    \notag \\
				\leq& C(\Omega,\theta) k_{\rm{max}}^{3} \big( \| |u^{j}| \|_{\infty,2}^{2} 
				+ \| |u^{j}| \|_{\infty,2,\beta}^{2}\big) \int_{t_{n-1}}^{t_{n+1}} \| u_{tt}^{j} \|_{1}^{2} dt 
				+ \frac{1}{32} \| \widehat{k}_{n}^{-1} \phi_{n,\alpha}^{j,h} \|^{2}. \notag 
			\end{align}
			\begin{confidential}
				\color{darkblue}
				\begin{align*}
					&b \big( u_{n,\ast}^{j}, u_{n,\beta}^{j}, \widehat{k}_{n}^{-1} \phi_{n,\alpha}^{j,h} \big)
					- b ( u^{j}(t_{n,\beta}), u^{j}(t_{n,\beta}), \widehat{k}_{n}^{-1} \phi_{n,\alpha}^{j,h} ) \\
					=& b \big( u_{n,\ast}^{j}, u_{n,\beta}^{j}, \widehat{k}_{n}^{-1} \phi_{n,\alpha}^{j,h} \big)
					- b \big( u^{j}(t_{n,\beta}), u_{n,\beta}^{j}, \widehat{k}_{n}^{-1} \phi_{n,\alpha}^{j,h} \big) 
					+ b \big( u^{j}(t_{n,\beta}), u_{n,\beta}^{j}, \widehat{k}_{n}^{-1} \phi_{n,\alpha}^{j,h} \big)
					- b ( u^{j}(t_{n,\beta}), u^{j}(t_{n,\beta}), \widehat{k}_{n}^{-1} \phi_{n,\alpha}^{j,h} )
				\end{align*}
				\normalcolor
			\end{confidential}

			\noindent $\bullet$ \textit{New terms arising from the DLN-Ensemble algorithms in \eqref{eq:DLN-Ensemble-Alg}}  \\
			Similar to \eqref{eq:non-linear-eq3}
			\begin{align}
				&b \big( u_{n,\ast}^{j,h} \!-\! \langle u^{h} \rangle_{n,\ast}, u_{n,\ast}^{j,h} - u_{n,\beta}^{j,h}, \widehat{k}_{n}^{-1} \phi_{n,\alpha}^{j,h}  \big) 
				\label{eq:DLN-Error-H1-Eq1-term9}   \\
				=& b \big( u_{n,\ast}^{j,h} \!-\! \langle u^{h} \rangle_{n,\ast}, \eta_{n,\ast}^{j} - \eta_{n,\beta}^{j}, \widehat{k}_{n}^{-1} \phi_{n,\alpha}^{j,h}  \big) 
				+ b \big( u_{n,\ast}^{j,h} \!-\! \langle u^{h} \rangle_{n,\ast}, \phi_{n,\ast}^{j,h} - \phi_{n,\beta}^{j,h}, \widehat{k}_{n}^{-1} \phi_{n,\alpha}^{j,h} \big) \notag \\
				&+ b \big( u_{n,\ast}^{j,h} \!-\! \langle u^{h} \rangle_{n,\ast}, u_{n,\ast}^{j} - u_{n,\beta}^{j}, \widehat{k}_{n}^{-1} \phi_{n,\alpha}^{j,h}  \big). \notag 
			\end{align}
			\begin{confidential}
				\color{darkblue}
				\begin{align*}
				&b \big( u_{n,\ast}^{j,h} \!-\! \langle u^{h} \rangle_{n,\ast}, u_{n,\ast}^{j,h} - u_{n,\beta}^{j,h}, \widehat{k}_{n}^{-1} \phi_{n,\alpha}^{j,h} \big) \\
				=&  b \big( u_{n,\ast}^{j,h} \!-\! \langle u^{h} \rangle_{n,\ast}, \big( u_{n,\ast}^{j,h} - u_{n,\beta}^{j,h} \big) - \big( u_{n,\ast}^{j} - u_{n,\beta}^{j} \big) , \widehat{k}_{n}^{-1} \phi_{n,\alpha}^{j,h} \big) 
				+ b \big( u_{n,\ast}^{j,h} \!-\! \langle u^{h} \rangle_{n,\ast}, u_{n,\ast}^{j} - u_{n,\beta}^{j}, \widehat{k}_{n}^{-1} \phi_{n,\alpha}^{j,h} \big)  \\
				=& b \big( u_{n,\ast}^{j,h} \!-\! \langle u^{h} \rangle_{n,\ast}, \big( \eta_{n,\ast}^{j} - \eta_{n,\beta}^{j} \big) + \big( \phi_{n,\ast}^{j,h} - \phi_{n,\beta}^{j,h} \big) , \widehat{k}_{n}^{-1} \phi_{n,\alpha}^{j,h} \big) 
				+ b \big( u_{n,\ast}^{j,h} \!-\! \langle u^{h} \rangle_{n,\ast}, u_{n,\ast}^{j} - u_{n,\beta}^{j}, \widehat{k}_{n}^{-1} \phi_{n,\alpha}^{j,h} \big) 
				\end{align*}
				\normalcolor
			\end{confidential}
			By \eqref{eq:b-bound-3} in Lemma \ref{lemma:b-bound}, Poincar\'e inequality, \eqref{eq:approx-thm}, 
			\eqref{eq:Stoke-Approx}, Young's inequality, CFL-like conditions in \eqref{eq:CFL-like-cond} 
			and \eqref{eq:consist-2nd-eq1}, \eqref{eq:consist-2nd-eq2} in Lemma \ref{lemma:DLN-consistency}
			\begin{align}
				& b \big( u_{n,\ast}^{j,h} \!-\! \langle u^{h} \rangle_{n,\ast}, \eta_{n,\ast}^{j} - \eta_{n,\beta}^{j}, \widehat{k}_{n}^{-1} \phi_{n,\alpha}^{j,h} \big) 
				\label{eq:DLN-Error-H1-Eq1-term9-term1} \\
				\leq& C(\Omega) h^{-1/2} \big\| \nabla \big( u_{n,\ast}^{j,h} \!-\! \langle u^{h} \rangle_{n,\ast} \big) \big\|
				\big\| \eta_{n,\ast}^{j} - \eta_{n,\beta}^{j} \big\|_{1} \| \widehat{k}_{n}^{-1} \phi_{n,\alpha}^{j,h} \|
				\notag \\
				\leq& C(\Omega) h^{-1/2} \big\| \nabla \big( u_{n,\ast}^{j,h} \!-\! \langle u^{h} \rangle_{n,\ast} \big) \big\|
				\Big( \frac{h^{s+1}}{\nu} \| p_{n,\ast}^{j} \!-\! p_{n,\beta}^{j} \|_{s\!+\!1} 
				\!+\! h^{r} \| u_{n,\ast}^{j} \!-\! u_{n,\beta}^{j} \|_{r\!+\!1} \Big) \| \widehat{k}_{n}^{-1} \phi_{n,\alpha}^{j,h} \|
				\notag \\
				\leq& C(\Omega) h^{-1} \big\| \nabla \big( u_{n,\ast}^{j,h} \!-\! \langle u^{h} \rangle_{n,\ast} \big) \big\|^{2}
				\Big( \frac{h^{2s+2}}{\nu^{2}} \| p_{n,\ast}^{j} \!-\! p_{n,\beta}^{j} \|_{s+1}^{2} 
				\!+\! h^{2r} \| u_{n,\ast}^{j} \!-\! u_{n,\beta}^{j} \|_{r+1}^{2} \Big)  \notag \\
				&+ \frac{1}{32} \| \widehat{k}_{n}^{-1} \phi_{n,\alpha}^{j,h} \|^{2} \notag \\
				\leq& C(\Omega,\theta) \frac{ \nu }{\widehat{k}_{n}} \Big( \frac{1-\varepsilon_{n}}{1+\varepsilon_{n} \theta} \Big)^{2} 
				\Big( \frac{h^{2s+2}}{\nu^{2}} \| p_{n,\ast}^{j} \!-\! p_{n,\beta}^{j} \|_{s\!+\!1}^{2} 
				\!+\! h^{2r} \| u_{n,\ast}^{j} \!-\! u_{n,\beta}^{j} \|_{r\!+\!1}^{2} \Big) 
				\!+\! \frac{1}{32} \| \widehat{k}_{n}^{-1} \phi_{n,\alpha}^{j,h} \|^{2} \notag \\
				\leq&\frac{C(\Omega,\theta) k_{\rm{max}}^{3}  \nu }{\widehat{k}_{n}} \!\Big( \frac{1\!-\!\varepsilon_{n}}{1\!+\!\varepsilon_{n} \theta} \Big)^{2} \!\!
				\Big( \frac{h^{2s\!+\!2}}{\nu^{2}} \!\!\! \int_{t_{n\!-\!1}}^{t_{n\!+\!1}} \!\| p_{tt}^{j} \|_{s\!+\!1}^{2} dt 
				\!+\! h^{2r} \!\!\!  \int_{t_{n\!-\!1}}^{t_{n\!+\!1}} \!\| u_{tt}^{j} \|_{r\!+\!1}^{2} dt  \Big) \!+\! \frac{1}{32} \| \widehat{k}_{n}^{-1} \phi_{n,\alpha}^{j,h} \|^{2}. \notag 
			\end{align}
			\begin{confidential}
				\color{darkblue}
				\begin{align*}
					& b \big( u_{n,\ast}^{j,h} \!-\! \langle u^{h} \rangle_{n,\ast}, \eta_{n,\ast}^{j} - \eta_{n,\beta}^{j}, \widehat{k}_{n}^{-1} \phi_{n,\alpha}^{j,h} \big) \\
					\leq& C(\Omega) \big\| \nabla \big( u_{n,\ast}^{j,h} \!-\! \langle u^{h} \rangle_{n,\ast} \big) \big\|
					\big\| \eta_{n,\ast}^{j} - \eta_{n,\beta}^{j} \big\|_{1} \| \widehat{k}_{n}^{-1} \phi_{n,\alpha}^{j,h} \|^{1/2}
					\| \nabla \widehat{k}_{n}^{-1} \phi_{n,\alpha}^{j,h} \|^{1/2} 
					\notag \\
					\leq& C(\Omega) h^{-1/2} \big\| \nabla \big( u_{n,\ast}^{j,h} \!-\! \langle u^{h} \rangle_{n,\ast} \big) \big\|
					\big\| \eta_{n,\ast}^{j} - \eta_{n,\beta}^{j} \big\|_{1} \| \widehat{k}_{n}^{-1} \phi_{n,\alpha}^{j,h} \| \\
					\leq& C(\Omega) h^{-1/2} \big\| \nabla \big( u_{n,\ast}^{j,h} \!-\! \langle u^{h} \rangle_{n,\ast} \big) \big\|
					\Big( \frac{h^{s+1}}{\nu} \| p_{n,\ast}^{j} - p_{n,\beta}^{j} \|_{s+1} 
					+ h^{r} \| u_{n,\ast}^{j} - u_{n,\beta}^{j} \|_{r+1} \Big) \| \widehat{k}_{n}^{-1} \phi_{n,\alpha}^{j,h} \| \\
					\leq& C(\Omega)h^{-1} \big\| \nabla \big( u_{n,\ast}^{j,h} \!-\! \langle u^{h} \rangle_{n,\ast} \big) \big\|^{2}
					\Big( \frac{h^{2s+2}}{\nu^{2}} \| \widetilde{p}_{n}^{j} - p_{n,\beta}^{j} \|_{s+1}^{2} 
					+ h^{2r} \| u_{n,\ast}^{j} - u_{n,\beta}^{j} \|_{r+1}^{2} \Big) 
					+ \frac{1}{32} \| \widehat{k}_{n}^{-1} \phi_{n,\alpha}^{j,h} \|^{2} \\
					\leq& C(\Omega) h^{-1} \Big[ C(\Omega,\theta) \frac{h \nu }{\widehat{k}_{n}} \Big( \frac{1-\varepsilon_{n}}{1+\varepsilon_{n} \theta} \Big)^{2} \Big] \Big( \frac{h^{2s+2}}{\nu^{2}} C(\theta) k_{\rm{max}}^{3} \int_{t_{n-1}}^{t_{n+1}} \| p_{tt}^{j} \|_{s+1}^{2} dt 
					+ h^{2r} C(\theta) k_{\rm{max}}^{3} \int_{t_{n-1}}^{t_{n+1}} \| u_{tt}^{j} \|_{r+1}^{2} dt \Big) \\
					&+ \frac{1}{32} \| \widehat{k}_{n}^{-1} \phi_{n,\alpha}^{j,h} \|^{2} \\
					\leq& C(\Omega,\theta) \frac{h \nu }{\widehat{k}_{n}} \Big( \frac{1-\varepsilon_{n}}{1+\varepsilon_{n} \theta} \Big)^{2} 
					\Big( \frac{h^{2s+1}}{\nu^{2}} k_{\rm{max}}^{3} \int_{t_{n-1}}^{t_{n+1}} \| p_{tt}^{j} \|_{s+1}^{2} dt 
					+ h^{2r-1} k_{\rm{max}}^{3} \int_{t_{n-1}}^{t_{n+1}} \| u_{tt}^{j} \|_{r+1}^{2} dt  \Big) \\
					&+ \frac{1}{32} \| \widehat{k}_{n}^{-1} \phi_{n,\alpha}^{j,h} \|^{2}.
				\end{align*}
				\normalcolor
			\end{confidential}
			Similarly,
			\begin{align}
				&b \big( u_{n,\ast}^{j,h} \!-\! \langle u^{h} \rangle_{n,\ast}, \phi_{n,\ast}^{j,h} - \phi_{n,\beta}^{j,h},
				\widehat{k}_{n}^{-1} \phi_{n,\alpha}^{j,h} \big) 
				\label{eq:DLN-Error-H1-Eq1-term9-term2} \\
				\leq& C(\Omega) h^{-1/2} \big\| \nabla \big( u_{n,\ast}^{j,h} \!-\! \langle u^{h} \rangle_{n,\ast} \big) \big\|
				\big\| \nabla \big( \phi_{n,\ast}^{j,h} - \phi_{n,\beta}^{j,h} \big) \big\| 
				\| \widehat{k}_{n}^{-1} \phi_{n,\alpha}^{j,h} \| \notag \\
				\leq& C(\Omega) h^{-1/2} \frac{2 \beta_{2}^{(n)}}{1 - \varepsilon_{n}}
				\big\| \nabla \big( u_{n,\ast}^{j,h} \!-\! \langle u^{h} \rangle_{n,\ast} \big) \big\|
				\| \nabla \Phi_{n}^{j,h} \|  \| \widehat{k}_{n}^{-1} \phi_{n,\alpha}^{j,h} \| \notag \\
				\leq& \frac{C(\Omega,\theta) \widehat{k}_{n}}{\nu h} \Big( \frac{1 + \varepsilon_{n} \theta }{1 - \varepsilon_{n}} \Big)^{2} 
				\big\| \nabla \big( u_{n,\ast}^{j,h} \!-\! \langle u^{h} \rangle_{n,\ast} \big) \big\|^{2} 
				\| \widehat{k}_{n}^{-1} \phi_{n,\alpha}^{j,h} \|^{2} 
				+ \frac{\nu \theta (1 - \theta^2)}{4 \widehat{k}_{n} (1 + \varepsilon_{n} \theta )^{2} } 
				\big\| \nabla \Phi_{n}^{j,h} \big\|^{2}. \notag
			\end{align}
			\begin{align}
				&b \big( u_{n,\ast}^{j,h} \!-\! \langle u^{h} \rangle_{n,\ast}, u_{n,\ast}^{j} - u_{n,\beta}^{j}, \widehat{k}_{n}^{-1} \phi_{n,\alpha}^{j,h} \big) 
				\label{eq:DLN-Error-H1-Eq1-term9-term3} \\
				\leq& C(\Omega) \big\| \nabla \big( u_{n,\ast}^{j,h} \!-\! \langle u^{h} \rangle_{n,\ast} \big) \big\|
				\| u_{n,\ast}^{j} - u_{n,\beta}^{j} \|_{2} \| \widehat{k}_{n}^{-1} \phi_{n,\alpha}^{j,h} \|
				\notag \\
				\leq&  C(\Omega) \big\| \nabla \big( u_{n,\ast}^{j,h} \!-\! \langle u^{h} \rangle_{n,\ast} \big) \big\|^{2}
				\| u_{n,\ast}^{j} - u_{n,\beta}^{j} \|_{2}^{2}  + \frac{1}{32} \| \widehat{k}_{n}^{-1} \phi_{n,\alpha}^{j,h} \|^{2} \notag \\
				\leq& C(\Omega) \Big[ C(\Omega,\theta) \frac{h \nu }{\widehat{k}_{n}} \Big( \frac{1-\varepsilon_{n}}{1+\varepsilon_{n} \theta} \Big)^{2} \Big]
				\Big( C(\theta) k_{\rm{max}}^{3} \int_{t_{n-1}}^{t_{n+1}} \| u_{tt}^{j} \|_{2}^{2} dt \Big) 
				+ \frac{1}{32} \| \widehat{k}_{n}^{-1} \phi_{n,\alpha}^{j,h} \|^{2} \notag \\
				\leq& C(\Omega,\theta) \frac{h \nu }{\widehat{k}_{n}}  k_{\rm{max}}^{3} \int_{t_{n-1}}^{t_{n+1}} \| u_{tt}^{j} \|_{2}^{2} dt 
				+ \frac{1}{32} \| \widehat{k}_{n}^{-1} \phi_{n,\alpha}^{j,h} \|^{2}.  \notag
			\end{align} 
			\begin{confidential}
				\color{darkblue}
				\begin{align*}
					&b \big( u_{n,\ast}^{j,h} \!-\! \langle u^{h} \rangle_{n,\ast}, u_{n,\ast}^{j} - u_{n,\beta}^{j}, \widehat{k}_{n}^{-1} \phi_{n,\alpha}^{j,h} \big) \\
					\leq& C(\Omega) \| u_{n,\ast}^{j,h} \!-\! \langle u^{h} \rangle_{n,\ast} \|_{1} 
					\| u_{n,\ast}^{j} - u_{n,\beta}^{j} \|_{2} \| \widehat{k}_{n}^{-1} \phi_{n,\alpha}^{j,h} \| \\
					\leq& C(\Omega) \big\| \nabla \big( u_{n,\ast}^{j,h} \!-\! \langle u^{h} \rangle_{n,\ast} \big) \big\|
					\| u_{n,\ast}^{j} - u_{n,\beta}^{j} \|_{2} \| \widehat{k}_{n}^{-1} \phi_{n,\alpha}^{j,h} \| \\
					\leq& C(\Omega) \big\| \nabla \big( u_{n,\ast}^{j,h} \!-\! \langle u^{h} \rangle_{n,\ast} \big) \big\|^{2}
					\| u_{n,\ast}^{j} - u_{n,\beta}^{j} \|_{2}^{2}  + \frac{1}{32} \| \widehat{k}_{n}^{-1} \phi_{n,\alpha}^{j,h} \|^{2} \\
					\leq& C(\Omega) \Big[ C(\Omega,\theta) \frac{h \nu }{\widehat{k}_{n}} \Big( \frac{1-\varepsilon_{n}}{1+\varepsilon_{n} \theta} \Big)^{2} \Big]
					\Big( C(\theta) k_{\rm{max}}^{3} \int_{t_{n-1}}^{t_{n+1}} \| u_{tt}^{j} \|_{2}^{2} dt \Big) 
					+ \frac{1}{32} \| \widehat{k}_{n}^{-1} \phi_{n,\alpha}^{j,h} \|^{2} \\
					\leq& C(\Omega,\theta) \frac{h \nu }{\widehat{k}_{n}}  k_{\rm{max}}^{3} \int_{t_{n-1}}^{t_{n+1}} \| u_{tt}^{j} \|_{2}^{2} dt 
					+ \frac{1}{32} \| \widehat{k}_{n}^{-1} \phi_{n,\alpha}^{j,h} \|^{2}  
				\end{align*}
				\begin{align*}
					&b \big( u_{n,\ast}^{j,h} \!-\! \langle u^{h} \rangle_{n,\ast}, \phi_{n,\ast}^{j,h} - \phi_{n,\beta}^{j,h},
					\widehat{k}_{n}^{-1} \phi_{n,\alpha}^{j,h} \big) \\
					\leq& C(\Omega) \| u_{n,\ast}^{j,h} \!-\! \langle u^{h} \rangle_{n,\ast} \|_{1}  \| \phi_{n,\ast}^{j,h} - \phi_{n,\beta}^{j,h} \|_{1} \| \widehat{k}_{n}^{-1} \phi_{n,\alpha}^{j,h} \|^{1/2}
					\| \nabla \widehat{k}_{n}^{-1} \phi_{n,\alpha}^{j,h} \|^{1/2} \\
					\leq& C(\Omega) h^{-1/2} \big\| \nabla \big( u_{n,\ast}^{j,h} \!-\! \langle u^{h} \rangle_{n,\ast} \big) \big\|
					\big\| \nabla \big( \phi_{n,\ast}^{j,h}  - \phi_{n,\beta}^{j,h} \big) \big\| 
					\| \widehat{k}_{n}^{-1} \phi_{n,\alpha}^{j,h} \| \\
					\leq& C(\Omega) h^{-1/2} \big\| \nabla \big( u_{n,\ast}^{j,h} \!-\! \langle u^{h} \rangle_{n,\ast} \big) \big\|
					\Big[ \frac{2 \beta_{2}^{(n)}}{1 - \varepsilon_{n}} \| \nabla \Phi_{n}^{j,h} \| \Big] \| \widehat{k}_{n}^{-1} \phi_{n,\alpha}^{j,h} \| \\
					\leq& \frac{C(\Omega,\theta) }{\sqrt{h} (1 - \varepsilon_{n})} \big\| \nabla \big( u_{n,\ast}^{j,h} \!-\! \langle u^{h} \rangle_{n,\ast} \big) \big\| \| \nabla \Phi_{n}^{j,h} \| \| \widehat{k}_{n}^{-1} \phi_{n,\alpha}^{j,h} \| \\
					\leq& \frac{C(\Omega,\theta) \widehat{k}_{n}}{\nu h} \Big( \frac{1 + \varepsilon_{n} \theta }{1 - \varepsilon_{n}} \Big)^{2} 
					\big\| \nabla \big( u_{n,\ast}^{j,h} \!-\! \langle u^{h} \rangle_{n,\ast} \big) \big\|^{2} 
					\| \widehat{k}_{n}^{-1} \phi_{n,\alpha}^{j,h} \|^{2} 
					+ \frac{\nu \theta (1 - \theta^2)}{4 \widehat{k}_{n} (1 + \varepsilon_{n} \theta )^{2} } 
					\big\| \nabla \Phi_{n}^{j,h} \big\|^{2}
				\end{align*}
				\normalcolor
			\end{confidential}
			\ \\
			\noindent \textit{Part 3.} \ \\
			We combine \eqref{eq:DLN-Error-H1-Eq1-term1}, \eqref{eq:DLN-Error-H1-Eq1-term2}, 
			\eqref{eq:DLN-Error-H1-Eq1-term3}, \eqref{eq:DLN-Error-H1-Eq1-term4}, 
			\eqref{eq:DLN-Error-H1-term5-6-eq2}, \eqref{eq:DLN-Error-H1-term7-8}, 
			\eqref{eq:DLN-Error-H1-Eq1-term9-term1}, \eqref{eq:DLN-Error-H1-Eq1-term9-term2} and \eqref{eq:DLN-Error-H1-Eq1-term9-term3}
			\begin{confidential}
				\color{darkblue}
				\begin{align*}
				&\Big\| \widehat{k}_{n}^{-1} \phi_{n,\alpha}^{j,h} \Big\|^{2} + 
				\frac{\nu}{\widehat{k}_{n}} \Big(\begin{Vmatrix}
				\nabla {\phi_{n+1}^{j,h}} \\
				\nabla {\phi_{n}^{j,h}}
				\end{Vmatrix}
				_{G(\theta)}^{2} -
				\begin{Vmatrix}
				\nabla {\phi_{n}^{j,h}} \\
				\nabla {\phi_{n-1}^{j,h}}
				\end{Vmatrix}
				_{G(\theta)}^{2} 
				+ \frac{ \theta (1 - {\theta}^{2})}{2 ( 1 + \varepsilon_{n} \theta )^{2}} \big\| \nabla \Phi_{n}^{j,h} \big\|^{2}  \Big) \\
				= & C(\theta) k_{\rm{max}}^{3} \int_{t_{n-1}}^{t_{n+1}} \| u_{ttt}^{j} \|^{2} dt
				+ \frac{1}{32} \| \widehat{k}_{n}^{-1} \phi_{n,\alpha}^{j,h} \|^{2} \\
				&+ \frac{C(\theta)}{\widehat{k}_{n}} \Big( h^{2r+2} \int_{t_{n-1}}^{t_{n+1}} \| u_{t}^{j} \|_{r+1}^{2} dt 
				+ \frac{h^{2s+4}}{\nu^{2}} \int_{t_{n-1}}^{t_{n+1}} \| p_{t}^{j} \|_{s+1}^{2} dt \Big) 
				+ \frac{1}{32}\| \widehat{k}_{n}^{-1} \phi_{n,\alpha}^{j,h} \|^{2} \\
				&+ C(\theta) \nu^{2} k_{\rm{max}}^{3} \int_{t_{n-1}}^{t_{n+1}} \| u_{tt} \|_{2}^{2} dt 
				+ \frac{1}{32} \| \widehat{k}_{n}^{-1} \phi_{n,\alpha}^{h} \|^{2} 
				+ C(\theta) k_{\rm{max}}^3 \int_{t_{n-1}}^{t_{n+1}} \| \nabla p_{tt}^{j} \|^{2} dt 
				+ \frac{1}{32} \| \widehat{k}_{n}^{-1} \phi_{n,\alpha}^{j,h} \|^{2} \\
				&+C(\Omega,\theta) \big( \| u_{n,\beta}^{j} \|_{2}^{2} + h^{-1}\| e_{n,\beta}^{j} \|_{1}^{2} \big)
				\big( \| \nabla \phi_{n}^{j,h} \|^{2} + \| \nabla \phi_{n-1}^{j,h} \|^{2} \big) 
				+ C(\Omega,\theta) \| |u^{j}| \|_{\infty,2}^{2} \| e_{n,\beta}^{j} \|_{1}^{2} \notag \\
				&+ C(\Omega,\theta) \big( \frac{h^{2s+1}}{\nu^{2}} \| |p^{j}| \|_{\infty,s+1}^{2} 
				+ h^{2r-1} \| |u^{j}| \|_{\infty,r+1}^{2} \big) \| e_{n,\beta}^{j} \|_{1}^{2}
				+ \frac{1}{32} \| \widehat{k}_{n}^{-1} \phi_{n,\alpha}^{j,h} \|^{2} \notag \\
				&+ C(\Omega,\theta) \| |u^{j}| \|_{\infty,2}^{2} \Big[ \frac{h^{2s\!+\!2}}{\nu^{2}} 
				\Big( k_{\rm{max}}^{3} \int_{t_{n-1}}^{t_{n+1}} \| p_{tt}^{j} \|_{s\!+\!1}^{2} dt 
				+ \| p^{j}(t_{n,\beta})\|_{s\!+\!1}^{2}  \Big) 
				+ h^{2r} \Big( k_{\rm{max}}^{3} \int_{t_{n\!-\!1}}^{t_{n\!+\!1}} \| u_{tt}^{j} \|_{r\!+\!1}^{2} dt 
				+ \| u^{j}(t_{n,\beta})\|_{r\!+\!1}^{2} \Big) \Big] \\
				&+C(\Omega,\theta) k_{\rm{max}}^{3} \big( \| |u^{j}| \|_{\infty,2}^{2} 
				+ \| |u^{j}| \|_{\infty,2,\beta}^{2}\big) \int_{t_{n-1}}^{t_{n+1}} \| u_{tt}^{j} \|_{1}^{2} dt 
				+ \frac{1}{32} \| \widehat{k}_{n}^{-1} \phi_{n,\alpha}^{j,h} \|^{2} \\
				&+ \frac{C(\Omega,\theta) k_{\rm{max}}^{3}  \nu }{\widehat{k}_{n}} \!\Big( \frac{1\!-\!\varepsilon_{n}}{1\!+\!\varepsilon_{n} \theta} \Big)^{2} \!\!
				\Big( \frac{h^{2s\!+\!2}}{\nu^{2}} \!\!\! \int_{t_{n\!-\!1}}^{t_{n\!+\!1}} \!\| p_{tt}^{j} \|_{s\!+\!1}^{2} dt 
				\!+\! h^{2r} \!\!\!  \int_{t_{n\!-\!1}}^{t_{n\!+\!1}} \!\| u_{tt}^{j} \|_{r\!+\!1}^{2} dt  \Big) \!+\! \frac{1}{32} \| \widehat{k}_{n}^{-1} \phi_{n,\alpha}^{j,h} \|^{2} \\
				&+ \frac{C(\Omega,\theta) \widehat{k}_{n}}{\nu h} \Big( \frac{1 + \varepsilon_{n} \theta }{1 - \varepsilon_{n}} \Big)^{2} 
				\big\| \nabla \big( u_{n,\ast}^{j,h} \!-\! \langle u^{h} \rangle_{n,\ast} \big) \big\|^{2} 
				\| \widehat{k}_{n}^{-1} \phi_{n,\alpha}^{j,h} \|^{2} 
				+ \frac{\nu \theta (1 - \theta^2)}{4 \widehat{k}_{n} (1 + \varepsilon_{n} \theta )^{2} } 
				\big\| \nabla \Phi_{n}^{j,h} \big\|^{2} \\
				&+ C(\Omega,\theta) \frac{h \nu }{\widehat{k}_{n}}  k_{\rm{max}}^{3} \int_{t_{n-1}}^{t_{n+1}} \| u_{tt}^{j} \|_{2}^{2} dt 
				+ \frac{1}{32} \| \widehat{k}_{n}^{-1} \phi_{n,\alpha}^{j,h} \|^{2}
				\end{align*}
				\begin{align*}
				&\begin{Vmatrix}
				\nabla {\phi_{n+1}^{j,h}} \\
				\nabla {\phi_{n}^{j,h}}
				\end{Vmatrix}
				_{G(\theta)}^{2} -
				\begin{Vmatrix}
				\nabla {\phi_{n}^{j,h}} \\
				\nabla {\phi_{n-1}^{j,h}}
				\end{Vmatrix}
				_{G(\theta)}^{2} 
				+ \frac{ \theta (1 - {\theta}^{2})}{4 ( 1 + \varepsilon_{n} \theta )^{2}} \big\| \nabla \Phi_{n}^{j,h} \big\|^{2}
				+ \frac{3\widehat{k}_{n}}{4 \nu} \Big\| \widehat{k}_{n}^{-1} \phi_{n,\alpha}^{j,h} \Big\|^{2} \\
				&- \frac{C(\Omega,\theta) \widehat{k}_{n}^{2}}{\nu^{2} h} \Big( \frac{1 + \varepsilon_{n} \theta }{1 - \varepsilon_{n}} \Big)^{2} 
				\big\| \nabla \big( u_{n,\ast}^{j,h} \!-\! \langle u^{h} \rangle_{n,\ast} \big) \big\|^{2} 
				\Big\| \widehat{k}_{n}^{-1} \phi_{n,\alpha}^{j,h} \Big\|^{2} \notag \\
				\leq& \frac{C(\theta)k_{\rm{max}}^{4}}{\nu} \! \int_{t_{n-1}}^{t_{n+1}} \| u_{ttt}^{j} \|^{2} dt
				\!+\! \frac{C(\theta)}{\nu} \Big( h^{2r+2} \! \int_{t_{n-1}}^{t_{n+1}} \| u_{t}^{j} \|_{r+1}^{2} dt 
				\!+\! \frac{h^{2s+4}}{\nu^{2}} \! \int_{t_{n-1}}^{t_{n+1}} \| p_{t}^{j} \|_{s+1}^{2} dt \Big) \notag \\
				&+\! C(\theta) \nu k_{\rm{max}}^{4} \!\int_{t_{n-1}}^{t_{n+1}} \| u_{tt} \|_{2}^{2} dt 
				\!+\!\frac{C(\theta)k_{\rm{max}}^{4}}{\nu} \! \int_{t_{n-1}}^{t_{n+1}} \| \nabla p_{tt}^{j} \|^{2} dt \notag \\
				&+\frac{C(\Omega,\theta) \widehat{k}_{n}}{\nu} \big( \| u_{n,\beta}^{j} \|_{2}^{2} + h^{-1}\| e_{n,\beta}^{j} \|_{1}^{2} \big)
				\big( \| \nabla \phi_{n}^{j,h} \|^{2} + \| \nabla \phi_{n-1}^{j,h} \|^{2} \big) 
				+ \frac{C(\Omega,\theta)}{\nu} \| |u^{j}| \|_{\infty,2}^{2} \widehat{k}_{n} \| e_{n,\beta}^{j} \|_{1}^{2} \notag \\
				&+ C(\Omega,\theta) \big( \frac{h^{2s+1}}{\nu^{3}} \| |p^{j}| \|_{\infty,s+1}^{2} 
				+ \frac{h^{2r-1}}{\nu} \| |u^{j}| \|_{\infty,r+1}^{2} \big) \widehat{k}_{n} \| e_{n,\beta}^{j} \|_{1}^{2} \notag \\
				&+ C(\Omega,\theta) \| |u^{j}| \|_{\infty,2}^{2} \frac{h^{2s\!+\!2}}{\nu^{3}} 
				\Big( k_{\rm{max}}^{4} \int_{t_{n-1}}^{t_{n+1}} \| p_{tt}^{j} \|_{s\!+\!1}^{2} dt 
				+ \widehat{k}_{n} \| p^{j}(t_{n,\beta})\|_{s\!+\!1}^{2}  \Big) \notag \\
				&+C(\Omega,\theta) \| |u^{j}| \|_{\infty,2}^{2} \frac{h^{2r}}{\nu} 
				\Big( k_{\rm{max}}^{4} \int_{t_{n\!-\!1}}^{t_{n\!+\!1}} \| u_{tt}^{j} \|_{r\!+\!1}^{2} dt 
				+ \widehat{k}_{n}  \| u^{j}(t_{n,\beta})\|_{r\!+\!1}^{2} \Big) \notag \\
				&+ \frac{C(\theta) k_{\rm{max}}^{4}}{\nu} \big( \| |u^{j}| \|_{\infty,2}^{2} 
				+ \| |u^{j}| \|_{\infty,2,\beta}^{2}\big) \! \int_{t_{n-1}}^{t_{n+1}} \| u_{tt}^{j} \|_{1}^{2} dt  \notag \\
				&+\! C(\Omega,\theta) k_{\rm{max}}^{3} \!
				\Big(\! \frac{h^{2s\!+\!2}}{\nu^{2}} \!\!\! \int_{t_{n\!-\!1}}^{t_{n\!+\!1}} \!\!\| p_{tt}^{j} \|_{s\!+\!1}^{2} dt 
				\!+\! h^{2r} \!\!\! \int_{t_{n\!-\!1}}^{t_{n\!+\!1}} \! \!\| u_{tt}^{j} \|_{r\!+\!1}^{2} dt \! \Big) 
				\!+\! C(\theta) h  k_{\rm{max}}^{3} \!\! \int_{t_{n\!-\!1}}^{t_{n\!+\!1}} \! \!\| u_{tt}^{j} \|_{2}^{2} dt. \notag 
				\end{align*}
				\normalcolor
			\end{confidential}
			\begin{align}
			&\begin{Vmatrix}
			\nabla {\phi_{n+1}^{j,h}} \\
			\nabla {\phi_{n}^{j,h}}
			\end{Vmatrix}
			_{G(\theta)}^{2} -
			\begin{Vmatrix}
			\nabla {\phi_{n}^{j,h}} \\
			\nabla {\phi_{n-1}^{j,h}}
			\end{Vmatrix}
			_{G(\theta)}^{2} 
			+ \frac{ \theta (1 - {\theta}^{2})}{4 ( 1 + \varepsilon_{n} \theta )^{2}} \big\| \nabla \Phi_{n}^{j,h} \big\|^{2}
			+ \frac{\widehat{k}_{n}}{4 \nu} \Big\| \widehat{k}_{n}^{-1} \phi_{n,\alpha}^{j,h} \Big\|^{2} \notag \\
			& + \frac{\widehat{k}_{n}}{2\nu} \Big[ 1 - C(\Omega,\theta)  \Big( \frac{1 + \varepsilon_{n} \theta }{1 - \varepsilon_{n}} \Big)^{2} \frac{\widehat{k}_{n}}{\nu h}  \big\| \nabla \big( u_{n,\ast}^{j,h} \!-\! \langle u^{h} \rangle_{n,\ast} \big) \big\|^{2} \Big] \Big\| \widehat{k}_{n}^{-1} \phi_{n,\alpha}^{j,h} \Big\|^{2} 
			\label{eq:DLN-Error-H1-eq2} \\
			\leq& \frac{C(\theta)k_{\rm{max}}^{4}}{\nu} \! \int_{t_{n-1}}^{t_{n+1}} \| u_{ttt}^{j} \|^{2} dt
			\!+\! \frac{C(\theta)}{\nu} \Big( h^{2r+2} \! \int_{t_{n-1}}^{t_{n+1}} \| u_{t}^{j} \|_{r+1}^{2} dt 
			\!+\! \frac{h^{2s+4}}{\nu^{2}} \! \int_{t_{n-1}}^{t_{n+1}} \| p_{t}^{j} \|_{s+1}^{2} dt \Big) \notag \\
			+&\! C(\theta) \nu k_{\rm{max}}^{4} \!\int_{t_{n-1}}^{t_{n+1}} \| u_{tt}^{j} \|_{2}^{2} dt 
			\!+\!\frac{C(\theta)k_{\rm{max}}^{4}}{\nu} \! \int_{t_{n-1}}^{t_{n+1}} \| \nabla p_{tt}^{j} \|^{2} dt 
			\!+\! \frac{C(\Omega,\theta)}{\nu} \| |u^{j}| \|_{\infty,2}^{2} \widehat{k}_{n} \| e_{n,\beta}^{j} \|_{1}^{2}
			\notag \\
			+&\frac{C(\Omega,\theta) \widehat{k}_{n}}{\nu} \big( \| u_{n,\beta}^{j} \|_{2}^{2} + h^{-1}\| e_{n,\beta}^{j} \|_{1}^{2} \big)
			\big( \| \nabla \phi_{n}^{j,h} \|^{2} + \| \nabla \phi_{n-1}^{j,h} \|^{2} \big)  \notag \\
			+& C(\Omega,\theta) \big( \frac{h^{2s+1}}{\nu^{3}} \| |p^{j}| \|_{\infty,s+1}^{2} 
			+ \frac{h^{2r-1}}{\nu} \| |u^{j}| \|_{\infty,r+1}^{2} \big) \widehat{k}_{n} \| e_{n,\beta}^{j} \|_{1}^{2} \notag \\
			+& C(\Omega,\theta) \| |u^{j}| \|_{\infty,2}^{2} \frac{h^{2s\!+\!2}}{\nu^{3}} 
			\Big( k_{\rm{max}}^{4} \int_{t_{n-1}}^{t_{n+1}} \| p_{tt}^{j} \|_{s\!+\!1}^{2} dt 
			+ \widehat{k}_{n} \| p^{j}(t_{n,\beta})\|_{s\!+\!1}^{2}  \Big) \notag \\
			+&C(\Omega,\theta) \| |u^{j}| \|_{\infty,2}^{2} \frac{h^{2r}}{\nu} 
			\Big( k_{\rm{max}}^{4} \int_{t_{n\!-\!1}}^{t_{n\!+\!1}} \| u_{tt}^{j} \|_{r\!+\!1}^{2} dt 
			+ \widehat{k}_{n}  \| u^{j}(t_{n,\beta})\|_{r\!+\!1}^{2} \Big) \notag \\
			+& \frac{C(\Omega,\theta) k_{\rm{max}}^{4}}{\nu} \big( \| |u^{j}| \|_{\infty,2}^{2} 
			+ \| |u^{j}| \|_{\infty,2,\beta}^{2}\big) \! \int_{t_{n-1}}^{t_{n+1}} \| u_{tt}^{j} \|_{1}^{2} dt  \notag \\
			+&\! C(\Omega,\theta) k_{\rm{max}}^{3} \!
			\Big(\! \frac{h^{2s\!+\!2}}{\nu^{2}} \!\!\! \int_{t_{n\!-\!1}}^{t_{n\!+\!1}} \!\!\| p_{tt}^{j} \|_{s\!+\!1}^{2} dt 
			\!+\! h^{2r} \!\!\! \int_{t_{n\!-\!1}}^{t_{n\!+\!1}} \! \!\| u_{tt}^{j} \|_{r\!+\!1}^{2} dt \! \Big) 
			\!+\! C(\theta) h  k_{\rm{max}}^{3} \!\! \int_{t_{n\!-\!1}}^{t_{n\!+\!1}} \! \!\| u_{tt}^{j} \|_{2}^{2} dt. \notag 
			\end{align}
			We sum \eqref{eq:DLN-Error-H1-eq2} over $n$ from $1$ to $M$ ($M=2, \cdots, N$) and use CFL-like conditions in \eqref{eq:CFL-like-cond} to obtain
			\begin{align}
			&\| \nabla \phi_{M}^{j,h} \|^{2} 
			+ \sum_{n=1}^{N-1} \frac{ \theta (1 - {\theta}^{2})}{4 ( 1 + \varepsilon_{n} \theta )^{2}} \big\| \nabla \Phi_{n}^{j,h} \big\|^{2}
			+ \frac{1}{\nu (1+\theta) } \sum_{n=1}^{M-1} \widehat{k}_{n} \Big\| \widehat{k}_{n}^{-1} \phi_{n,\alpha}^{j,h} \Big\|^{2} 
			\label{eq:DLN-Error-H1-eq3}   \\
			\leq& \frac{C(\Omega,\theta)}{\nu} \sum_{n=1}^{N-1}\big( \widehat{k}_{n} \| u_{n,\beta}^{j} \|_{2}^{2} 
			+ \frac{1}{h \nu} \nu \widehat{k}_{n} \| e_{n,\beta}^{j} \|_{1}^{2} \big)
			\big( \| \nabla \phi_{n}^{j,h} \|^{2} + \| \nabla \phi_{n-1}^{j,h} \|^{2} \big) \notag \\
			+& \frac{C(\Omega,\theta)}{\nu^{2}} \| |u^{j}| \|_{\infty,2}^{2} \sum_{n=1}^{N-1} 
			\nu \widehat{k}_{n} \| e_{n,\beta}^{j} \|_{1}^{2}    \notag \\
			+&\! C(\Omega,\theta) \big( \frac{h^{2s\!+\!1}}{\nu^{4}} \| |p^{j}| \|_{\infty,s\!+\!1}^{2} 
			\!+\! \frac{h^{2r\!-\!1}}{\nu^{2}} \| |u^{j}| \|_{\infty,r\!+\!1}^{2} \big) \sum_{n\!=\!1}^{N\!-\!1} 
			\nu \widehat{k}_{n} \| e_{n,\beta}^{j} \|_{1}^{2}  
			\!+\! \frac{C(\theta)k_{\rm{max}}^{4}}{\nu} \| u_{ttt}^{j} \|_{2,0}^{2} \notag \\
			+&\frac{C(\theta)}{\nu} \Big( h^{2r+2} \| u_{t}^{j} \|_{2,r+1}^{2} \!+\! \frac{h^{2s+4}}{\nu^{2}} \| p_{t}^{j} \|_{2,s+1}^{2}  \Big)
			+\! C(\theta) \nu k_{\rm{max}}^{4} \| u_{tt}^{j} \|_{2,2}^{2}  \notag \\
			+&\!\frac{C(\theta)k_{\rm{max}}^{4}}{\nu} \| \nabla p_{tt}^{j} \|_{2,0}^{2} 
			+ C(\Omega,\theta) \| |u^{j}| \|_{\infty,2}^{2} \frac{h^{2s\!+\!2}}{\nu^{3}} 
			\Big( k_{\rm{max}}^{4} \| p_{tt}^{j} \|_{2,s\!+\!1}^{2} + \| |p^{j}| \|_{2,s\!+\!1,\beta}^{2}  \Big) \notag \\
			+& C(\Omega,\theta) \| |u^{j}| \|_{\infty,2}^{2} \frac{h^{2r}}{\nu} 
			\Big( k_{\rm{max}}^{4} \| u_{tt}^{j} \|_{2,r\!+\!1}^{2}  
			+ \| |u^{j}| \|_{2,r\!+\!1,\beta}^{2} \Big) \notag \\
			+& \frac{C(\Omega,\theta) k_{\rm{max}}^{4}}{\nu} \big( \| |u^{j}| \|_{\infty,2}^{2} 
			+ \| |u^{j}| \|_{\infty,2,\beta}^{2}\big)  \| u_{tt}^{j} \|_{2,1}^{2} 
			+ C(\Omega,\theta) h  k_{\rm{max}}^{3} \| u_{tt}^{j} \|_{2,2}^{2} \notag \\
			+& C(\Omega,\theta) k_{\rm{max}}^{3} \!
			\Big(\! \frac{h^{2s\!+\!2}}{\nu^{2}} \| p_{tt}^{j} \|_{2,s\!+\!1}^{2} 
			\!+\! h^{2r} \| u_{tt}^{j} \|_{2,r\!+\!1}^{2} \Big)
			\!+\! C(\theta) \big( \| \nabla \phi_{1}^{j,h} \|^{2} + \| \nabla \phi_{0}^{j,h} \|^{2} \big).  \notag 
			\end{align}
			Similar to \eqref{eq:DLN-Stab-H1-eq4}, we apply \eqref{eq:error-DLN-L2H1} and discrete Gronwall inequality to \eqref{eq:DLN-Error-H1-eq3} 
			\begin{confidential}
				\color{darkblue}
				\begin{align*}
				&\| \nabla \phi_{N}^{j,h} \|^{2}   
				+ \sum_{n=1}^{N-1} \frac{ \theta (1 - {\theta}^{2})}{4 ( 1 + \varepsilon_{n} \theta )^{2}} \big\| \nabla \Phi_{n}^{j,h} \big\|^{2}
				+ \frac{1}{\nu(1+\theta)} \sum_{n=1}^{N-1} \widehat{k}_{n} 
				\Big\| \widehat{k}_{n}^{-1} \phi_{n,\alpha}^{j,h} \Big\|^{2} \\
				\leq & \frac{C(\Omega,\theta)}{\nu} 
				\Big[ \big( \widehat{k}_{N-1} \| u_{N-1,\beta}^{j} \|_{2}^{2} 
				+ \frac{1}{h \nu} \nu \widehat{k}_{N-1} \| e_{N-1,\beta}^{j} \|_{1}^{2}  \big) 
				\| \nabla \phi_{N-1}^{j,h} \|^{2} \\
				+& \sum_{n=2}^{N-2}  \big(\widehat{k}_{n+1} \| u_{n+1,\beta}^{j} \|_{2}^{2} 
				+ \widehat{k}_{n} \| u_{n,\beta}^{j} \|_{2}^{2} 
				+ \frac{1}{h \nu} \nu \widehat{k}_{n+1} \| e_{n+1,\beta}^{j} \|_{1}^{2} 
				+ \frac{1}{h \nu} \nu \widehat{k}_{n} \| e_{n,\beta}^{j} \|_{1}^{2} \big) 
				\| \nabla \phi_{n}^{j,h} \|^{2} \\
				+& \big( \widehat{k}_{1} \| u_{1,\beta}^{j} \|_{2}^{2} 
				+ \frac{1}{h \nu} \nu \widehat{k}_{1} \| e_{1,\beta}^{j} \|_{1}^{2}  \big) 
				\| \nabla \phi_{0}^{j,h} \|^{2} \Big] 
				+ \frac{C(\Omega,\theta)(1 + \nu T)}{\nu^{2}} \| |u^{j}| \|_{\infty,2}^{2} F_{2}^{2} 	\\
				+& C(\Omega,\theta)(1 + \nu T) \big( \frac{h^{2s+1}}{\nu^{4}} \| |p^{j}| \|_{\infty,s+1}^{2} 
				+ \frac{h^{2r-1}}{\nu^{2}} \| |u^{j}| \|_{\infty,r+1}^{2} \big) 
				F_{2}^{2} 
				\!+\! \frac{C(\theta)k_{\rm{max}}^{4}}{\nu} \| u_{ttt}^{j} \|_{2,0}^{2} \\
				+&\frac{C(\theta)}{\nu} \Big( h^{2r+2} \| u_{t}^{j} \|_{2,r+1}^{2} \!+\! \frac{h^{2s+4}}{\nu^{2}} \| p_{t}^{j} \|_{2,s+1}^{2}  \Big)
				+\! C(\theta) \nu k_{\rm{max}}^{4} \| u_{tt}^{j} \|_{2,2}^{2}   \notag \\
				+&\!\frac{C(\theta)k_{\rm{max}}^{4}}{\nu} \| \nabla p_{tt}^{j} \|_{2,0}^{2} 
				+ C(\Omega,\theta) \| |u^{j}| \|_{\infty,2}^{2} \frac{h^{2s\!+\!2}}{\nu^{3}} 
				\Big( k_{\rm{max}}^{4} \| p_{tt}^{j} \|_{2,s\!+\!1}^{2} + \| |p^{j}| \|_{2,s\!+\!1,\beta}^{2}  \Big) \notag \\
				+& C(\Omega,\theta) \| |u^{j}| \|_{\infty,2}^{2} \frac{h^{2r}}{\nu} 
				\Big( k_{\rm{max}}^{4} \| u_{tt}^{j} \|_{2,r\!+\!1}^{2}  
				+ \| |u^{j}| \|_{2,r\!+\!1,\beta}^{2} \Big) \notag \\
				+& \frac{C(\Omega,\theta) k_{\rm{max}}^{4}}{\nu} \big( \| |u^{j}| \|_{\infty,2}^{2} 
				+ \| |u^{j}| \|_{\infty,2,\beta}^{2}\big)  \| u_{tt}^{j} \|_{2,1}^{2} 
				+ C(\Omega,\theta) h  k_{\rm{max}}^{3} \| u_{tt}^{j} \|_{2,2}^{2} \notag \\
				+& C(\Omega,\theta) k_{\rm{max}}^{3} \!
				\Big(\! \frac{h^{2s\!+\!2}}{\nu^{2}} \| p_{tt}^{j} \|_{2,s\!+\!1}^{2} 
				\!+\! h^{2r} \| u_{tt}^{j} \|_{2,r\!+\!1}^{2} \Big)
				\!+\! C(\theta) \big( \| \nabla \phi_{1}^{j,h} \|^{2} + \| \nabla \phi_{0}^{j,h} \|^{2} \big).
				\end{align*}
				\normalcolor
			\end{confidential}
			\begin{align}
			&\| \nabla \phi_{M}^{j,h} \|^{2}  
			+ \sum_{n=1}^{M-1} \frac{ \theta (1 - {\theta}^{2})}{4 ( 1 + \varepsilon_{n} \theta )^{2}} \big\| \nabla \Phi_{n}^{j,h} \big\|^{2}
			+ \frac{1}{\nu(1+\theta)} \sum_{n=1}^{M-1} \widehat{k}_{n} 
			\Big\| \widehat{k}_{n}^{-1} \phi_{n,\alpha}^{j,h} \Big\|^{2}        
			\label{eq:DLN-Error-H1-eq4} \\
			&\leq  \exp \Big[ \frac{C(\Omega,\theta)}{\nu}  
			\Big( k_{\rm{max}}^{4} \| u_{tt}^{j} \|_{2,2}^{2} 
			+ \| |u^{j}|\|_{2,2,\beta}^{2} + \frac{ (C\nu T + 1) F_{2}^{2} }{h \nu} \Big) \Big] F_{3},               \notag 
			\end{align}
			where
			\begin{align*}
			F_{3} 
			&= C(\Omega,\theta)(1 + \nu T) \big( \frac{h^{2s+1}}{\nu^{4}} \| |p^{j}| \|_{\infty,s+1}^{2} 
			+ \frac{h^{2r-1}}{\nu^{2}} \| |u^{j}| \|_{\infty,r+1}^{2} \big) F_{2}^{2} 
			\!+\! \frac{C(\theta)k_{\rm{max}}^{4}}{\nu} \| u_{ttt}^{j} \|_{2,0}^{2} \\
			&+\frac{C(\Omega,\theta)}{\nu} \Big( h^{2r+2} \| u_{t}^{j} \|_{2,r+1}^{2} \!+\! \frac{h^{2s+4}}{\nu^{2}} \| p_{t}^{j} \|_{2,s+1}^{2} \Big)
			+\! C(\theta) \nu k_{\rm{max}}^{4} \| u_{tt}^{j} \|_{2,2}^{2}  \notag \\
			&+\!\frac{C(\theta)k_{\rm{max}}^{4}}{\nu} \| \nabla p_{tt}^{j} \|_{2,0}^{2} 
			+ C(\Omega,\theta) \| |u^{j}| \|_{\infty,2}^{2} \frac{h^{2s\!+\!2}}{\nu^{3}} 
			\Big( k_{\rm{max}}^{4} \| p_{tt}^{j} \|_{2,s\!+\!1}^{2} + \| |p^{j}| \|_{2,s\!+\!1,\beta}^{2}  \Big) \notag \\
			&+ C(\Omega,\theta) \| |u^{j}| \|_{\infty,2}^{2} \frac{h^{2r}}{\nu} 
			\Big( k_{\rm{max}}^{4} \| u_{tt}^{j} \|_{2,r\!+\!1}^{2}  
			\!+\! \| |u^{j}| \|_{2,r\!+\!1,\beta}^{2} \Big) 
			\!+\! \frac{C(\Omega,\theta)(1 + \nu T)}{\nu^{2}} \| |u^{j}| \|_{\infty,2}^{2} F_{2}^{2} \notag \\
			&+ \frac{C(\Omega,\theta) k_{\rm{max}}^{4}}{\nu} \big( \| |u^{j}| \|_{\infty,2}^{2} 
			+ \| |u^{j}| \|_{\infty,2,\beta}^{2}\big)  \| u_{tt}^{j} \|_{2,1}^{2} 
			+ C(\Omega,\theta) h  k_{\rm{max}}^{3} \| u_{tt}^{j} \|_{2,2}^{2} \notag \\
			&+ C(\Omega,\theta) k_{\rm{max}}^{3} \!
			\Big(\! \frac{h^{2s\!+\!2}}{\nu^{2}} \| p_{tt}^{j} \|_{2,s\!+\!1}^{2} 
			\!+\! h^{2r} \| u_{tt}^{j} \|_{2,r\!+\!1}^{2} \Big)
			\!+\! C(\theta) \big( \| \nabla \phi_{1}^{j,h} \|^{2} + \| \nabla \phi_{0}^{j,h} \|^{2} \big).
			\end{align*}
			By the time-diameter restriction in \eqref{eq:time-diameter-relation}, $h^{-1} F_{2}^{2}$ is bounded. Thus 
			\begin{align}
			\max_{0 \leq n \leq N} \| \nabla e_{n}^{j} \|  
			\leq& \max_{0 \leq n \leq N} \| \nabla \eta_{n}^{j} \| 
			+ \max_{0 \leq n \leq N} \| \nabla \phi_{n}^{j,h} \|   
			\label{eq:DLN-Error-H1-eq5} \\
			\leq& C(\Omega)h^{r} \| |u^{j}| \|_{\infty,r+1} \!+\! \frac{C(\Omega)h^{s+1}}{\nu} \| |p^{j}| \|_{\infty,s+1}  
			\notag \\
			&+ \exp \Big[ \frac{C(\Omega,\theta)}{\nu}  \big( k_{\rm{max}}^{4} \| u_{tt}^{j} \|_{2,2}^{2} 
			+ \| |u^{j}|\|_{2,2,\beta}^{2} + \frac{ (C \nu T + 1) F_{2}^{2}  }{h \nu} \big)   \Big] \sqrt{F_{3}}.   \notag 
			\end{align}
			By \eqref{eq:approx-thm}, \eqref{eq:Stoke-Approx} and H$\ddot{\rm{o}}$lder's inequality 
			\begin{confidential}
				\color{darkblue}
				\begin{align*}
				&\sum_{n=1}^{N-1} \frac{\widehat{k}_{n}}{\nu} 
				\Big\| \widehat{k}_{n}^{-1} \eta_{n,\alpha}^{j} \Big\|^{2}
				= \sum_{n=1}^{N-1} \frac{1}{\nu \widehat{k}_{n}} \| \eta_{n,\alpha}^{j} \|^{2} 
				\leq \sum_{n=1}^{N-1} \frac{C(\Omega)}{\nu \widehat{k}_{n}} 
				\Big( \frac{h^{2s+2}}{\nu^{2}} \| p_{n,\alpha}^{j} \|_{s+1}^{2} 
				+ h^{2r} \| u_{n,\alpha}^{j} \|_{r}^{2}  \Big) \\
				\leq& \sum_{n=1}^{N-1} \frac{C(\Omega)}{\nu \widehat{k}_{n}}
				\Big[ \frac{h^{2s+2}}{\nu^{2}} C(\theta)(k_{n} + k_{n-1}) \int_{t_{n-1}}^{t_{n+1}} \| p_{t}^{j} \|_{s+1}^{2} dt 
				+ h^{2r} C(\theta)(k_{n} + k_{n-1}) \int_{t_{n-1}}^{t_{n+1}} \| u_{t}^{j} \|_{r}^{2} dt  \Big] \\
				\leq& \sum_{n=1}^{N-1} \frac{C(\Omega,\theta)}{\nu} \Big[ \frac{h^{2s+2}}{\nu^{2}} 
				\int_{t_{n-1}}^{t_{n+1}} \| p_{t}^{j} \|_{s+1}^{2} dt 
				+ h^{2r} \int_{t_{n-1}}^{t_{n+1}} \| u_{t}^{j} \|_{r}^{2} dt  \Big]
				= \frac{C(\theta)}{\nu} 
				\Big( \frac{h^{2s+2}}{\nu^{2}} \| p_{t}^{j} \|_{2,s+1}^{2} 
				+ h^{2r} \| u_{t}^{j} \|_{2,r}^{2}   \Big)
				\end{align*}
				\begin{align*}
				\sum_{n=1}^{N-1} \frac{\widehat{k}_{n}}{\nu} 
				\Big\| \widehat{k}_{n}^{-1} e_{n,\alpha}^{j} \Big\|^{2} 
				\leq& 2 \sum_{n=1}^{N-1} \frac{\widehat{k}_{n}}{\nu} 
				\Big\| \widehat{k}_{n}^{-1} \eta_{n,\alpha}^{j} \Big\|^{2}
				+ 2 \sum_{n=1}^{N-1} \frac{\widehat{k}_{n}}{\nu} 
				\Big\| \widehat{k}_{n}^{-1} \phi_{n,\alpha}^{j,h} \Big\|^{2} \\
				\leq& \frac{C(\Omega,\theta)}{\nu} \Big( \frac{h^{2s+2}}{\nu^{2}} \| p_{t}^{j} \|_{2,s+1}^{2} 
				+ h^{2r} \| u_{t}^{j} \|_{2,r}^{2}  \Big) \\
				+& \exp \Big[ \frac{C(\Omega,\theta)}{\nu} \! 
				\Big( k_{\rm{max}}^{4} \| u_{tt}^{j} \|_{2,2}^{2} 
				\!+\! \| |u^{j}|\|_{2,2,\beta}^{2} \!+\! \frac{ (C \nu T \!+\! 1)F_{2}^{2}  }{h \nu} \Big)  \Big] \cdot F_{3}
				\end{align*}
				\normalcolor
			\end{confidential}
			\begin{align}
			\sum_{n=1}^{N-1} \frac{\widehat{k}_{n}}{\nu} 
			\Big\| \widehat{k}_{n}^{-1} \eta_{n,\alpha}^{j} \Big\|^{2}
			\leq& \sum_{n=1}^{N-1} \frac{C(\Omega)}{\nu \widehat{k}_{n}} 
			\Big( \frac{h^{2s+2}}{\nu^{2}} \| p_{n,\alpha}^{j} \|_{s+1}^{2} + h^{2r} \| u_{n,\alpha}^{j} \|_{r}^{2} \Big) 
			\label{eq:diff-eta-L2} \\
			\leq& \sum_{n=1}^{N-1} \frac{C(\Omega,\theta)}{\nu} \Big[ \frac{h^{2s+2}}{\nu^{2}} 
			\int_{t_{n-1}}^{t_{n+1}} \| p_{t}^{j} \|_{s+1}^{2} dt 
			+ h^{2r} \int_{t_{n-1}}^{t_{n+1}} \| u_{t}^{j} \|_{r}^{2} dt  \Big] \notag \\
			=& \frac{C(\Omega,\theta)}{\nu} \Big( \frac{h^{2s+2}}{\nu^{2}} \| p_{t}^{j} \|_{2,s+1}^{2} 
			+ h^{2r} \| u_{t}^{j} \|_{2,r}^{2} \Big).      \notag 
			\end{align}
			By \eqref{eq:DLN-Error-H1-eq4}, \eqref{eq:diff-eta-L2} and triangle inequality
			\begin{align}
			\sum_{n=1}^{N-1} \frac{\widehat{k}_{n}}{\nu} 
			\Big\| \widehat{k}_{n}^{-1} e_{n,\alpha}^{j} \Big\|^{2} 
			\leq&\! \frac{C(\Omega,\theta)}{\nu} \!\Big(\! \frac{h^{2s\!+\!2}}{\nu^{2}} 
			\| p_{t}^{j} \|_{2,s\!+\!1}^{2} 
			\!+\! h^{2r} \| u_{t}^{j} \|_{2,r}^{2}  \!\Big)  
			\label{eq:error-diff-time-L2} \\
			&+ \exp \Big[ \frac{C(\Omega,\theta)}{\nu} \! 
			\Big( k_{\rm{max}}^{4} \| u_{tt}^{j} \|_{2,2}^{2} 
			\!+\! \| |u^{j}|\|_{2,2,\beta}^{2} \!+\! \frac{ (C \nu T \!+\! 1) F_{2}^{2}  }{h \nu} \Big) \Big]  F_{3},      \notag
			\end{align}
			By \eqref{eq:error-L2-conclusion} in Theorem \ref{thm:Error-L2}, \eqref{eq:DLN-Error-H1-eq5} 
			and \eqref{eq:error-diff-time-L2}, we have \eqref{eq:error-H1-conclusion-appendix}. 
		\end{proof}
		\ \\

		\subsection{Proof of Theorem \ref{thm:Error-Pressure}}
		\label{appendixB-Pressure} \ 
		\begin{theorem}
			We assume that for the $j$-th NSE in \eqref{eq:jth-NSE}, the velocity $u^{j}(x,t)$ satisfies
			\begin{gather*}
			u \in \ell^{\infty}(\{ t_{n} \}_{n=0}^{N};H^{r+1}(\Omega)) \cap \ell^{\infty,\beta}(\{ t_{n} \}_{n=0}^{N};H^{1}(\Omega))
			\cap \ell^{2,\beta}(\{ t_{n} \}_{n=0}^{N};H^{r+1}(\Omega)), \\
			u_{t} \in L^{2}(0,T;H^{r+1}(\Omega)),  \ \ 
			u_{tt} \in L^{2}(0,T;H^{r+1}(\Omega)), \ \ 
			u_{ttt} \in L^{2}(0,T;X^{-1} \cap L^{2}(\Omega)), 
			\end{gather*} 
			and the pressure $p^{j}(x,t)$ satisfies
			\begin{gather*}
			p \in \ell^{\infty}(\{ t_{n} \}_{n=0}^{N};H^{s+1}) \cap \ell^{2,\beta}(\{ t_{n} \}_{n=0}^{N};H^{s+1}), \\
			p_{t} \in L^{2}(0,T;H^{s+1}(\Omega)), \ \ p_{tt} \in L^{2}(0,T;H^{s+1}(\Omega)),
			\end{gather*}
			for all $j$.  
			Under the CFL-like conditions in \eqref{eq:CFL-like-cond}, time ratio bounds in \eqref{eq:time-ratio-cond} and the time-diameter condition in \eqref{eq:time-diameter-relation},
			the numerical solutions of the DLN-Ensemble algorithms in \eqref{eq:DLN-Ensemble-Alg} for all $\theta \in (0,1)$ satisfy
			\begin{align}
			\Big( \sum_{n=1}^{N-1} \widehat{k}_{n} \| p_{n,\beta}^{j} - p_{n,\beta}^{j,h} \|^{2} \Big)^{1/2}
			\leq \mathcal{O} \big( h^{r},h^{s+1}, k_{\rm{max}}^{2}, h^{1/2}k_{\rm{max}}^{3/2}  \big). 
			\label{eq:Error-Pressure-Conclusion-Appendix} 
			\end{align}
		\end{theorem}
		\begin{proof}
			As in the proof of Theorem \ref{thm:Error-H1}, 
			we still denote $I_{\rm{St}} u_{n}^{j}$ to be the velocity component of Stokes projection of $(u_{n}^{j}, p_{n}^{j})$  and decompose the error of velocity at $t_{n}$ as
			\begin{gather*}
				e_{n}^{j} = u_{n}^{j} - u_{n}^{j,h} = \big(u_{n}^{j} - I_{\rm{St}} u_{n}^{j} \big)
				+ \big( I_{\rm{St}} u_{n}^{j} - u_{n}^{j,h} \big) = \eta_{n}^{j} + \phi_{n}^{j,h},  \\
				\eta_{n}^{j} = u_{n}^{j} - I_{\rm{St}} u_{n}^{j}, \qquad 
				\phi_{n}^{j,h} = I_{\rm{St}} u_{n}^{j} - u_{n}^{j,h}.
			\end{gather*}
			We set $v^{h} \in X^{h}$ in \eqref{eq:NSE-jth-exact} and subtract first equation of \eqref{eq:DLN-Ensemble-Alg} from \eqref{eq:NSE-jth-exact}
			\begin{confidential}
				\color{darkblue}
				\begin{gather*}
				\frac{1}{\widehat{k}_{n}} \big( u_{n,\alpha}^{j} - u_{n,\alpha}^{j,h}, v^{h} \big) + b\big( u^{j} (t_{n,\beta}) , u^{j} (t_{n,\beta}) , v^{h} \big)  
				- b \big( \langle u^{h} \rangle_{n,\ast}, u_{n,\beta}^{j,h} , v^{h} \big) \notag \\
				- b \big( u_{n,\ast}^{j,h} - \langle u^{h} \rangle_{n,\ast}, u_{n,\ast}^{j,h}, v^{h} \big) 
				- \big( p^{j} (t_{n,\beta}) , \nabla \cdot v^{h} \big) + \big( p_{n,\beta}^{j,h}, \nabla \cdot v^{h} \big)
				+ \nu \big( \nabla \big( u^{j} (t_{n,\beta}) -  u_{n,\beta}^{j,h} \big) , \nabla v^{h} \big) 
				\notag \\
				= \frac{1}{\widehat{k}_{n}} \big( u_{n,\alpha}^{j}, v^{h} \big)
				- \big( u_{t}^{j}( t_{n,\beta}) , v^{h} \big). 
				\end{gather*}
				\begin{gather*}
				\frac{1}{\widehat{k}_{n}} \big( e_{n,\alpha}^{j}, v^{h} \big) 
				+ \nu \big( \nabla \big( u^{j} (t_{n,\beta}) - u_{n,\beta}^{j} \big) , \nabla v^{h} \big) 
				+ \nu \big( \nabla \big( u_{n,\beta}^{j} - u_{n,\beta}^{j,h} \big) , \nabla v^{h} \big)
				\\
				- \big( p^{j} (t_{n,\beta}) - p_{n,\beta}^{j} + p_{n,\beta}^{j} - q^{j,h} + q^{j,h} - p_{n,\beta}^{j,h}, \nabla \cdot v^{h} \big)
				\\
				+ b\big( u^{j} (t_{n,\beta}) , u^{j} (t_{n,\beta}) , v^{h} \big)  
				- b \big( \langle u^{h} \rangle_{n,\ast}, u_{n,\beta}^{j,h}, v^{h} \big) 
				- b \big( u_{n,\ast}^{j,h} - \langle u^{h} \rangle_{n,\ast}, u_{n,\ast}^{j,h}, v^{h} \big)  \\
				= \frac{1}{\widehat{k}_{n}} \big( u_{n,\alpha}^{j}, v^{h} \big)
				- \big( u_{t}^{j}( t_{n,\beta}) , v^{h} \big). 
				\end{gather*}
				\begin{align*}
				&b\big( u^{j} (t_{n,\beta}) , u^{j} (t_{n,\beta}) , v^{h} \big)  
				- b \big( \langle u^{h} \rangle_{n,\ast}, u_{n,\beta}^{j,h}, v^{h} \big) 
				- b \big( u_{n,\ast}^{j,h} - \langle u^{h} \rangle_{n,\ast}, u_{n,\ast}^{j,h}, v^{h} \big) \\
				=& b\big( u^{j} (t_{n,\beta}) , u^{j} (t_{n,\beta}) , v^{h} \big)  
				- b \big( u_{n,\ast}^{j,h} - \big( u_{n,\ast}^{j,h} - \langle u^{h} \rangle_{n,\ast} \big), u_{n,\beta}^{j,h}, v^{h} \big) 
				- b \big( u_{n,\ast}^{j,h} - \langle u^{h} \rangle_{n,\ast}, u_{n,\ast}^{j,h}, v^{h} \big) \\
				=& b\big( u^{j} (t_{n,\beta}) , u^{j} (t_{n,\beta}) , v^{h} \big)  
				- b \big( u_{n,\ast}^{j,h}, u_{n,\beta}^{j,h}, v^{h} \big) 
				+ b \big( u_{n,\ast}^{j,h} - \langle u^{h} \rangle_{n,\ast}, u_{n,\beta}^{j,h}, v^{h} \big) 
				- b \big( u_{n,\ast}^{j,h} - \langle u^{h} \rangle_{n,\ast}, u_{n,\ast}^{j,h}, v^{h} \big) \\
				=& b\big( u^{j} (t_{n,\beta}) , u^{j} (t_{n,\beta}) , v^{h} \big)  
				- b \big( u_{n,\ast}^{j,h}, u_{n,\beta}^{j,h}, v^{h} \big) 
				+ b \big( u_{n,\ast}^{j,h} - \langle u^{h} \rangle_{n,\ast}, u_{n,\beta}^{j,h} - u_{n,\ast}^{j,h}, v^{h} \big) \\
				=& b\big( u^{j} (t_{n,\beta}) , u^{j} (t_{n,\beta}) , v^{h} \big) 
				- b \big( u_{n,\ast}^{j}, u_{n,\beta}^{j}, v^{h} \big) 
				+ b \big( u_{n,\ast}^{j}, u_{n,\beta}^{j}, v^{h} \big) 
				- b \big( u_{n,\ast}^{j,h}, u_{n,\beta}^{j,h}, v^{h} \big) 
				+ b \big( u_{n,\ast}^{j,h} - \langle u^{h} \rangle_{n,\ast}, u_{n,\beta}^{j,h} - u_{n,\ast}^{j,h}, v^{h} \big) \\ 
				\end{align*}
				\begin{gather*}
				\frac{1}{\widehat{k}_{n}} \Big( e_{n,\alpha}^{j} , v^{h} \Big) 
				+ \nu \big( \nabla \big( u^{j} (t_{n,\beta}) - u_{n,\beta}^{j} \big) , \nabla v^{h} \big) 
				+ \nu \big( \nabla e_{n,\beta}^{j} , \nabla v^{h} \big) \\
				- \big( p^{j} (t_{n,\beta}) - p_{n,\beta}^{j}, \nabla \cdot v^{h} \big)
				- \big( p_{n,\beta}^{j} - q^{j,h}, \nabla \cdot v^{h} \big)
				+ \big( p_{n,\beta}^{j,h} - q^{j,h}, \nabla \cdot v^{h} \big) \\
				+ b\big( u^{j} (t_{n,\beta}) , u^{j} (t_{n,\beta}) , v^{h} \big) 
				- b \big( u_{n,\ast}^{j}, u_{n,\beta}^{j}, v^{h} \big) 
				+ b \big( u_{n,\ast}^{j}, u_{n,\beta}^{j}, v^{h} \big) 
				- b \big( u_{n,\ast}^{j,h}, u_{n,\beta}^{j,h}, v^{h} \big) 
				+ b \big( u_{n,\ast}^{j,h} - \langle u^{h} \rangle_{n,\ast}, u_{n,\beta}^{j,h} - u_{n,\ast}^{j,h}, v^{h} \big) \\ 
				= \frac{1}{\widehat{k}_{n}} \big( u_{n,\alpha}^{j}, v^{h} \big)
				- \big( u_{t}^{j}( t_{n,\beta}) , v^{h} \big). 
				\end{gather*}
				\normalcolor
			\end{confidential}
			\begin{align}
				( p_{n,\beta}^{j,h} - q_{n}^{j,h}, \nabla \cdot v^{h} )
				=& - \big( \widehat{k}_{n}^{-1} e_{n,\alpha}^{j} , v^{h} \big) - \nu \big( \nabla e_{n,\beta}^{j}, \nabla v^{h} \big) 
				\label{eq:DLN-Error-P-L2L2-eq1} \\
				&+ \nu \big( \nabla \big( u_{n,\beta}^{j} \!-\! u^{j} (t_{n,\beta}) \big) , \nabla v^{h} \big)  
				+ \Big( \frac{u_{n,\alpha}^{j}}{\widehat{k}_{n}} - u_{t}^{j}( t_{n,\beta}), v^{h} \Big) \notag \\
				&+ \big( p^{j} (t_{n,\beta}) - p_{n,\beta}^{j}, \nabla \cdot v^{h} \big)
				+ \big( p_{n,\beta}^{j} - q_{n}^{j,h}, \nabla \cdot v^{h} \big) \notag \\
				&+ b \big( u_{n,\ast}^{j}, u_{n,\beta}^{j}, v^{h} \big)  
				- b\big( u^{j} (t_{n,\beta}) , u^{j} (t_{n,\beta}) , v^{h} \big) 
				+ b \big( u_{n,\ast}^{j,h}, u_{n,\beta}^{j,h}, v^{h} \big)  \notag \\
				&- b \big( u_{n,\ast}^{j}, u_{n,\beta}^{j}, v^{h} \big) 
				- b \big( u_{n,\ast}^{j,h} - \langle u^{h} \rangle_{n,\ast}, u_{n,\beta}^{j,h} - u_{n,\ast}^{j,h}, v^{h} \big). \notag 
			\end{align}
			where $q_{n}^{j,h}$ is the $L^{2}$-projection of $p_{n,\beta}^{j}$ onto $Q^{h}$.
			By Cauchy-Schwarz inequality, and Poincar$\acute{\rm{e}}$ inequality,
			\begin{gather}
				- \big( \widehat{k}_{n}^{-1} e_{n,\alpha}^{j} , v^{h} \big) \!-\! \nu \big( \nabla e_{n,\beta}^{j}, \nabla v^{h} \big) 
				\leq C(\Omega) \big( \big\| \widehat{k}_{n}^{-1} e_{n,\alpha}^{j} \big\| + \nu  \| \nabla e_{n,\beta}^{j} \| \big) \| \nabla v^{h} \|.
				\label{eq:DLN-Error-P-L2L2-eq1-term1-2}
			\end{gather}
			By Cauchy-Schwarz inequality, definition of dual norm in \eqref{eq:dual-norm}, \eqref{eq:consist-2nd-eq1} and \eqref{eq:consist-2nd-eq3} in Lemma \ref{lemma:DLN-consistency}
			\begin{align}
				\nu \big( \nabla (u_{n,\beta}^{j} - u^{j}(t_{n,\beta})), \nabla v^{h} \big) 
				\leq&  \nu \big\| \nabla (u_{n,\beta}^{j} - u^{j}(t_{n,\beta})) \big\| \| \nabla v^{h} \|  
				\label{eq:DLN-Error-P-L2L2-eq1-term3} \\
				\leq&  C(\theta)\nu \Big( k_{\rm{max}}^{3} 
				\int_{t_{n-1}}^{t_{n+1}} \| \nabla u_{tt}^{j} \| dt \Big)^{1/2}
				\| \nabla v^{h} \|, \notag \\
				\Big( \frac{ u_{n,\alpha}^{j}}{\widehat{k}_{n}}  
				- u^{j}(t_{n,\beta}), v^{h} \Big) 
				\leq& \Big\| \frac{u_{n,\alpha}^{j}}{\widehat{k}_{n}} 
				- u^{j}(t_{n,\beta}) \Big\|_{-1} \| \nabla v^{h} \|  
				\label{eq:DLN-Error-P-L2L2-eq1-term4} \\
				\leq& C(\theta) \Big(  k_{\rm{max}}^{3} \int_{t_{n-1}}^{t_{n+1}} \| u_{ttt}^{j} \|_{-1}^{2} dt \Big)^{1/2} \| \nabla v^{h} \|.  \notag 
			\end{align}
			By Cauchy-Schwarz inequality,  \eqref{eq:approx-thm} and \eqref{eq:consist-2nd-eq1} 
			in Lemma \ref{lemma:DLN-consistency}
			\begin{gather}
				(p^{j}(t_{n,\beta}) \!-\! p_{n,\beta}^{j}, \nabla \cdot v^{h})
				\!\leq\! \sqrt{d} \| p^{j}(t_{n,\beta}) \!-\! p_{n,\beta}^{j} \| \| \nabla v^{h} \|
				\!\leq\! C(\theta) \Big( \!k_{\rm{max}}^{3} \!\! \!
				\int_{t_{n\!-\!1}}^{t_{n\!+\!1}} \!\| p_{tt}^{j} \| dt \!\Big)^{\frac{1}{2}} \| \nabla v^{h} \|, \notag \\
				( p_{n,\beta}^{j} - q_{n}^{j,h}, \nabla \cdot v^{h} )
				\leq \sqrt{d} \| p_{n,\beta}^{j} - q^{j,h} \| \| \nabla v^{h} \|
				\leq \frac{C(\Omega)h^{s+1}}{\nu} \| p_{n,\beta}^{j} \|_{s+1} \| \nabla v^{h} \|.
				\label{eq:DLN-Error-P-L2L2-eq1-term5-6}
			\end{gather}
			By \eqref{eq:b-bound-1} in Lemma \ref{lemma:b-bound}, \eqref{eq:consist-2nd-eq1} 
			in Lemma \ref{lemma:DLN-consistency} and Poincar\'e inequality
			\begin{align}
				&b \big( u_{n,\ast}^{j}, u_{n,\beta}^{j}, v^{h} \big)
				- b ( u^{j}(t_{n,\beta}), u^{j}(t_{n,\beta}), v^{h} )  
				\label{eq:DLN-Error-P-L2L2-eq1-term7-8} \\
				=& b \big( u_{n,\ast}^{j} - u^{j}(t_{n,\beta}), u_{n,\beta}^{j}, v^{h} \big)
				+ b \big( u^{j}(t_{n,\beta}), u_{n,\beta}^{j} - u^{j}(t_{n,\beta}), v^{h} \big) \notag \\
				\leq& C(\Omega) \big( \| \nabla \big( u_{n,\ast}^{j} - u^{j}(t_{n,\beta}) \big) \|
				\| \nabla u_{n,\beta}^{j} \| 
				+ \| \nabla u^{j}(t_{n,\beta}) \| \| \nabla \big( u_{n,\beta}^{j} - u^{j}(t_{n,\beta}) \big) \| \big) \| \nabla v^{h} \|   \notag \\
				\leq& C(\Omega,\theta) \big( \| |u^{j}| \|_{\infty,1} 
				+ \| |u^{j}| \|_{\infty,1,\beta} \big) 
				\Big( k_{\rm{max}}^{3} 
				\int_{t_{n-1}}^{t_{n+1}} \| \nabla u_{tt}^{j} \| dt \Big)^{1/2} \| \nabla v^{h} \|. \notag 
			\end{align}
			By \eqref{eq:b-bound-1} in Lemma \ref{lemma:b-bound} and Poincar$\acute{\rm{e}}$ inequality
			\begin{align}
				& b \big( u_{n,\ast}^{j,h}, u_{n,\beta}^{j,h}, v^{h} \big)
				- b \big( u_{n,\ast}^{j}, u_{n,\beta}^{j}, v^{h} \big)
				\label{eq:DLN-Error-P-L2L2-eq1-term9-10} \\
				= & b \big( e_{n,\ast}^{j}, u_{n,\beta}^{j}, v^{h} \big)
				- b \big( u_{n,\ast}^{j}, e_{n,\beta}^{j}, v^{h} \big)
				+ b \big( e_{n,\ast}^{j}, e_{n,\beta}^{j}, v^{h} \big) \notag \\
				\leq& C(\Omega) \big( \| e_{n,\ast}^{j} \|_{1} \| \nabla u_{n,\beta}^{j} \|
				+ \| u_{n,\ast}^{j} \|_{1} \| \nabla e_{n,\beta}^{j} \| 
				+ \| e_{n,\ast}^{j} \|_{1} \| \nabla e_{n,\beta}^{j} \| \big)
				\| \nabla v^{h} \|.   \notag 
			\end{align}
			By \eqref{eq:b-bound-1} in Lemma \ref{lemma:b-bound} and Poincar$\acute{\rm{e}}$ inequality
			\begin{align}
				&b \big( u_{n,\ast}^{j,h} - \langle u^{h} \rangle_{n,\ast}, u_{n,\beta}^{j,h} - u_{n,\ast}^{j,h}, v^{h} \big) 
				\label{eq:DLN-Error-P-L2L2-eq1-term11} \\
				=& b \big( u_{n,\ast}^{j,h} - \langle u^{h} \rangle_{n,\ast}, \eta_{n,\ast}^{j} - \eta_{n,\beta}^{j}, v^{h} \big) 
				+ b \big( u_{n,\ast}^{j,h} - \langle u^{h} \rangle_{n,\ast}, \frac{2 \beta_{2}^{(n)}}{1 - \varepsilon_{n}}  \Phi_{n}^{j,h}, v^{h} \big) \notag \\
				&+ b \big( u_{n,\ast}^{j,h} - \langle u^{h} \rangle_{n,\ast}, u_{n,\ast}^{j} - u_{n,\beta}^{j}, v^{h} \big) \notag \\
				\leq& C(\theta) \big\| \nabla \big( u_{n,\ast}^{j,h} - \langle u^{h} \rangle_{n,\ast} \big) \big\| 
				\Big( \big\| \eta_{n,\ast}^{j} \!-\! \eta_{n,\beta}^{j} \big\|_{1} +  \frac{1}{1 \!-\! \varepsilon_{n}} \| \nabla \Phi_{n}^{j,h} \|
				+ \big\| \nabla \big( u_{n,\ast}^{j} - u_{n,\beta}^{j} \big) \big\| \Big) \| \nabla v^{h} \|. \notag 
			\end{align}
			We combine \eqref{eq:DLN-Error-P-L2L2-eq1}, \eqref{eq:DLN-Error-P-L2L2-eq1-term1-2}, 
			\eqref{eq:DLN-Error-P-L2L2-eq1-term3}, \eqref{eq:DLN-Error-P-L2L2-eq1-term4}, 
			\eqref{eq:DLN-Error-P-L2L2-eq1-term5-6}, \eqref{eq:DLN-Error-P-L2L2-eq1-term7-8}, 
			\eqref{eq:DLN-Error-P-L2L2-eq1-term9-10}, \eqref{eq:DLN-Error-P-L2L2-eq1-term11}, and use the discrete inf-dup condition in \eqref{eq:inf-sup-cond}
			\begin{confidential}
				\color{darkblue}
				\begin{align*}
					|( p_{n,\beta}^{j,h} - q_{n}^{j,h}, \nabla \cdot v^{h} )|
					\leq& C(\Omega) \big( \big\| \widehat{k}_{n}^{-1} e_{n,\alpha}^{j} \big\| + \nu  \| \nabla e_{n,\beta}^{j} \| \big) \| \nabla v^{h} \|
					+ C(\theta)\nu \Big( k_{\rm{max}}^{3} \int_{t_{n\!-\!1}}^{t_{n\!+\!1}} \| \nabla u_{tt}^{j} \| dt \Big)^{1/2} \| \nabla v^{h} \| 
					\\
					&+ C(\theta) \Big(  k_{\rm{max}}^{3} \int_{t_{n-1}}^{t_{n+1}} \| u_{ttt}^{j} \|_{-1}^{2} dt \Big)^{1/2} \| \nabla v^{h} \|
					\\
					&+ C(\theta) \Big( \!k_{\rm{max}}^{3} \!\! \!\int_{t_{n\!-\!1}}^{t_{n\!+\!1}} \!\| p_{tt}^{j} \| dt \!\Big)^{\frac{1}{2}} \| \nabla v^{h} \| + \frac{C(\Omega)h^{s+1}}{\nu} \| p_{n,\beta}^{j} \|_{s+1} \| \nabla v^{h} \| \\
					&+ C(\Omega,\theta) \big( \| |u^{j}| \|_{\infty,1} + \| |u^{j}| \|_{\infty,1,\beta} \big) 
					\Big( k_{\rm{max}}^{3} \int_{t_{n-1}}^{t_{n+1}} \| \nabla u_{tt}^{j} \| dt \Big)^{1/2} \| \nabla v^{h} \| \\
					&+ C(\Omega) \big( \| e_{n,\ast}^{j} \|_{1} \| \nabla u_{n,\beta}^{j} \|
					+ \| u_{n,\ast}^{j} \|_{1} \| \nabla e_{n,\beta}^{j} \| 
					+ \| e_{n,\ast}^{j} \|_{1} \| \nabla e_{n,\beta}^{j} \| \big) \| \nabla v^{h} \| \\
					&+ C(\Omega,\theta) \big\| \nabla \big( u_{n,\ast}^{j,h} - \langle u^{h} \rangle_{n,\ast} \big) \big\| 
					\Big( \big\| \eta_{n,\ast}^{j} \!-\! \eta_{n,\beta}^{j} \big\|_{1} +  \frac{1}{1 \!-\! \varepsilon_{n}} \| \nabla \Phi_{n}^{j,h} \|
					+ \big\| \nabla \big( u_{n,\ast}^{j} - u_{n,\beta}^{j} \big) \big\| \Big) \| \nabla v^{h} \|. 
				\end{align*}
				\normalcolor
			\end{confidential}
			\begin{align}
				&C_{\tt{is}} \| p_{n,\beta}^{j,h} \!-\! q_{n}^{j,h} \| 
				\label{eq:DLN-Error-P-L2L2-eq2} \\
				\leq & C(\Omega) \big(\big\| \widehat{k}_{n}^{-1} e_{n,\alpha}^{j} \big\| 
				+\! \nu  \| \nabla e_{n,\beta}^{j} \| \big)
				+\! C(\theta)\nu \Big( k_{\rm{max}}^{3} \int_{t_{n\!-\!1}}^{t_{n\!+\!1}} \| \nabla u_{tt}^{j} \| dt \Big)^{\frac{1}{2}} \notag \\
				+& C(\theta) \Big(  k_{\rm{max}}^{3} \int_{t_{n-1}}^{t_{n+1}} \| u_{ttt}^{j} \|_{-1}^{2} dt \Big)^{\frac{1}{2}} 
				+ C(\theta) \Big( \!k_{\rm{max}}^{3} \!\! \!\int_{t_{n\!-\!1}}^{t_{n\!+\!1}} \!\| p_{tt}^{j} \| dt \!\Big)^{\frac{1}{2}} 
				+ \frac{C(\Omega) h^{s+1}}{\nu} \| p_{n,\beta}^{j} \|_{s+1} \notag \\
				+& C(\Omega,\theta) \big( \| |u^{j}| \|_{\infty,1} + \| |u^{j}| \|_{\infty,1,\beta} \big) 
				\Big( k_{\rm{max}}^{3} \int_{t_{n-1}}^{t_{n+1}} \| \nabla u_{tt}^{j} \| dt \Big)^{\frac{1}{2}} \notag \\
				+& C(\Omega) \big( \| e_{n,\ast}^{j} \|_{1} \| \nabla u_{n,\beta}^{j} \|
				+ \| u_{n,\ast}^{j} \|_{1} \| \nabla e_{n,\beta}^{j} \| 
				+ \| e_{n,\ast}^{j} \|_{1} \| \nabla e_{n,\beta}^{j} \| \big) \| \nabla v^{h} \| \notag \\
				+& C(\Omega,\theta) \big\| \nabla \big( u_{n,\ast}^{j,h} \!-\! \langle u^{h} \rangle_{n,\ast} \big) \big\| 
				\Big( \big\| \eta_{n,\ast}^{j} \!-\! \eta_{n,\beta}^{j} \big\|_{1} +  \frac{1}{1 \!-\! \varepsilon_{n}} \| \nabla \Phi_{n}^{j,h} \|
				\!+\! \big\| \nabla \big( u_{n,\ast}^{j} \!-\! u_{n,\beta}^{j} \big) \big\| \Big) \| \nabla v^{h} \|. \notag 
			\end{align}
			By triangle inequality \eqref{eq:approx-thm} and \eqref{eq:consist-2nd-eq1} in Lemma \ref{lemma:DLN-consistency}, \eqref{eq:DLN-Error-P-L2L2-eq2} becomes
			\begin{confidential}
				\color{darkblue}
				\begin{align*}
					&\sum_{n=1}^{N-1} \widehat{k}_{n} \| p_{n,\beta}^{j} - p_{n,\beta}^{j,h} \|^{2} \\
					\leq& 2 \sum_{n=1}^{N-1} \widehat{k}_{n} \| p_{n,\beta}^{j} - q_{n}^{j,h} \|^{2}
					+ 2 \sum_{n=1}^{N-1} \widehat{k}_{n} \| p_{n,\beta}^{j,h} - q_{n}^{j,h} \|^{2} \\
					\leq& C(\Omega) h^{2s+2}  \sum_{n=1}^{N-1} \widehat{k}_{n} \| p_{n,\beta}^{j} \|_{s+1}^{2}
					+ C(\Omega) \sum_{n=1}^{N-1} \widehat{k}_{n} \| \widehat{k}_{n}^{-1} e_{n,\alpha}^{j} \|^{2}
					+ C(\Omega) \sum_{n=1}^{N-1} \nu \widehat{k}_{n} \| \nabla e_{n,\beta}^{j} \|^{2}   \notag \\
					+& C(\theta) \nu k_{\rm{max}}^{4} \sum_{n=1}^{N-1} \int_{t_{n\!-\!1}}^{t_{n\!+\!1}} \| \nabla u_{tt}^{j} \| dt
					+ C(\theta) k_{\rm{max}}^{4} \sum_{n=1}^{N-1} \int_{t_{n-1}}^{t_{n+1}} \| u_{ttt}^{j} \|_{-1}^{2} dt \\
					+& C(\theta) k_{\rm{max}}^{4} \sum_{n=1}^{N-1} \int_{t_{n-1}}^{t_{n+1}} \| p_{tt}^{j} \| dt 
					+ \frac{C(\Omega) h^{2s+2}}{\nu^{2}}  \sum_{n=1}^{N-1} \widehat{k}_{n} \| p_{n,\beta}^{j} \|_{s+1} \\
					+& C(\Omega,\theta) \big( \| |u^{j}| \|_{\infty,1} \!+\! \| |u^{j}| \|_{\infty,1,\beta} \big) k_{\rm{max}}^{4} 
					\sum_{n=1}^{N-1} \int_{t_{n-1}}^{t_{n+1}} \| \nabla u_{tt}^{j} \| dt \\
					+& C(\Omega,\theta) \| | u| \|_{\infty,1}^{2} \| |e_{n}^{j}| \|_{\infty,1}^{2} \big( \sum_{n=1}^{N-1} \widehat{k}_{n} \big)
					+ C(\Omega,\theta) \| |e_{n}^{j}| \|_{\infty,1}^{2} \big( \sum_{n=1}^{N-1} \widehat{k}_{n} \big) \\
					+& C(\Omega,\theta) \sum_{n=1}^{N-1} \widehat{k}_{n} 
					\big\| \nabla \big( u_{n,\ast}^{j,h} \!-\! \langle u^{h} \rangle_{n,\ast} \big) \big\|^{2}  
					\Big( \big\| \eta_{n,\ast}^{j} \!-\! \eta_{n,\beta}^{j} \big\|_{1}^{2} 
					+ \frac{\| \nabla \Phi_{n}^{j,h} \|^{2}}{(1 \!-\! \varepsilon_{n})^{2}} 
					+ \big\| \nabla \big( u_{n,\ast}^{j} \!-\! u_{n,\beta}^{j} \big)  \big\|^{2} \Big).
				\end{align*}
				\normalcolor
			\end{confidential}
			\begin{align}
				&\sum_{n=1}^{N-1} \widehat{k}_{n} \| p_{n,\beta}^{j} - p_{n,\beta}^{j,h} \|^{2} 
				\label{eq:DLN-Error-P-L2L2-eq3} \\
				\leq& 2 \sum_{n=1}^{N-1} \widehat{k}_{n} \| p_{n,\beta}^{j} - q_{n}^{j,h} \|^{2}
				+ 2 \sum_{n=1}^{N-1} \widehat{k}_{n} \| p_{n,\beta}^{j,h} - q_{n}^{j,h} \|^{2}  \notag \\
				\leq& C(\theta) h^{2s+2} \big( k_{\rm{max}}^{4} \| p_{tt}^{j} \|_{2,s\!+\!1}^{2} \!+\! \| |p^{j}| \|_{2,s\!+\!1,\beta}^{2} \big)
				\!+\! C(\Omega) \!\sum_{n=1}^{N-1} \widehat{k}_{n} \| \widehat{k}_{n}^{-1} e_{n,\alpha}^{j} \|^{2}
				\notag \\
				+& C(\Omega) \!\sum_{n=1}^{N-1} \nu \widehat{k}_{n} \| \nabla e_{n,\beta}^{j} \|^{2}   
				+ C(\theta) \nu k_{\rm{max}}^{4} \| \nabla u_{tt}^{j} \|_{2,0}^{2} + C(\theta) k_{\rm{max}}^{4} \big( \| u_{ttt}^{j} \|_{2,-1}^{2}
				+ \| p_{tt}^{j} \|_{2,0}^{2} \big) \notag \\
				+& \frac{C(\Omega,\theta) h^{2s+2}}{\nu^{2}} \!\big( k_{\rm{max}}^{4} \| p_{tt}^{j} \|_{2,s\!+\!1}^{2} \!+\! \| |p^{j}| \|_{2,s\!+\!1,\beta}^{2} \big) 
				+ C(\Omega,\theta) T \| | u| \|_{\infty,1}^{2} \| |e_{n}^{j}| \|_{\infty,1}^{2} 
				\notag \\
				+& C(\Omega,\theta) \big( \| |u^{j}| \|_{\infty,1} \!+\! \| |u^{j}| \|_{\infty,1,\beta} \big) k_{\rm{max}}^{4} \| \nabla u_{tt}^{j} \|_{2,0}^{2}
				\!+\! C(\Omega,\theta) T \| |e_{n}^{j}| \|_{\infty,1}^{2} \notag \\
				+& C(\Omega,\theta) \sum_{n=1}^{N-1} \widehat{k}_{n} 
				\big\| \nabla \big( u_{n,\ast}^{j,h} \!-\! \langle u^{h} \rangle_{n,\ast} \big) \big\|^{2}  
				\Big( \big\| \eta_{n,\ast}^{j} \!-\! \eta_{n,\beta}^{j} \big\|_{1}^{2} 
				+ \frac{\| \nabla \Phi_{n}^{j,h} \|^{2}}{(1 \!-\! \varepsilon_{n})^{2}} 
				+ \big\| \nabla \big( u_{n,\ast}^{j} \!-\! u_{n,\beta}^{j} \big)  \big\|^{2} \Big).  \notag 
			\end{align}
			By \eqref{eq:error-L2-inf-final-H1} in the proof of Theorem \ref{thm:Error-L2}  
			and \eqref{eq:error-diff-time-L2} in the proof of Theorem \ref{thm:Error-H1}
			\begin{align}
				&\sum_{n=1}^{M-1} \widehat{k}_{n} \| \widehat{k}_{n}^{-1} e_{n,\alpha}^{j} \|^{2}
				+\sum_{n=1}^{M-1} \nu \widehat{k}_{n} \| \nabla e_{n,\beta}^{j} \|^{2}    
				\label{eq:DLN-Error-P-L2L2-eq3-error-terms}     \\
				\leq& C(\Omega,\theta) \!\Big(\! \frac{h^{2s\!+\!2}}{\nu^{2}}  \| p_{t}^{j} \|_{2,s\!+\!1}^{2} \!+\! h^{2r} \| u_{t}^{j} \|_{2,r}^{2}  \!\Big) 
				\notag \\
				+& \nu \exp \Big[\! \frac{C(\Omega,\theta)}{\nu} \! \Big(\! k_{\rm{max}}^{4} \| u_{tt}^{j} \|_{2,2}^{2} 
				\!+\! \| |u^{j}|\|_{2,2,\beta}^{2} \!+\! \frac{ (C \nu T \!+\! 1) F_{2}^{2} }{h \nu} \Big)  \! \Big]  F_{3}  \notag \\
				+&\! \exp \!\Big[ \frac{C(\Omega,\theta)}{\nu} \big( k_{\rm{max}}^{4} \| u_{tt}^{j} \|_{2,2}^{2} 
				+ \| |u^{j}| \|_{2,2,\beta}^{2} \big) \Big] \! F_{1}
				\!+\! C(\Omega,\theta)  \nu h^{2r} \!\big( \!k_{\rm{max}}^{4} \!\| u_{tt}^{j} \|_{2,r\!+\!1}^{2}
				\!+\! \| |u^{j}| \|_{2,r\!+\!1,\beta}^{2} \!\big). \notag 
			\end{align}
			Similar to \eqref{eq:DLN-Error-H1-Eq1-term9-term1}, we have 
			\begin{align}
				&\sum_{n=1}^{N-1} \widehat{k}_{n} 
				\big\| \nabla \big( u_{n,\ast}^{j,h} \!-\! \langle u^{h} \rangle_{n,\ast} \big) \big\|^{2} 
				\big\| \eta_{n,\ast}^{j} - \eta_{n,\beta}^{j} \big\|_{1}^{2} 
				\label{eq:DLN-Error-P-L2L2-eq3-avg-term1}   \\
				\leq& C(\Omega) \sum_{n=1}^{N-1} \widehat{k}_{n} 
				\big\| \nabla \big( u_{n,\ast}^{j,h} \!-\! \langle u^{h} \rangle_{n,\ast} \big) \big\|^{2} 
				\Big( \frac{h^{2s\!+\!2}}{\nu^{2}} \| p_{n,\ast}^{j} \!-\! p_{n,\beta}^{j} \|_{s\!+\!1}^{2} 
				+ h^{2r} \| u_{n,\ast}^{j} \!-\! u_{n,\beta}^{j} \|_{r\!+\!1}^{2} \Big)  \notag \\
				\leq& C(\theta) \sum_{n=1}^{N-1} \widehat{k}_{n} \Big[ C(\Omega,\theta) \frac{ \nu h }{\widehat{k}_{n}} \Big( \frac{1-\varepsilon_{n}}{1+\varepsilon_{n} \theta} \Big)^{2} \Big]
				\Big( \frac{h^{2s\!+\!2}}{\nu^{2}} \!\!\! \int_{t_{n\!-\!1}}^{t_{n\!+\!1}} \!\| p_{tt}^{j} \|_{s\!+\!1}^{2} dt 
				\!+\! h^{2r} \!\!\!  \int_{t_{n\!-\!1}}^{t_{n\!+\!1}} \!\| u_{tt}^{j} \|_{r\!+\!1}^{2} dt  \Big) \notag \\
				\leq& C(\Omega,\theta) \nu\big( \frac{h^{2s\!+\!3}}{\nu^{2}} \| p_{tt}^{j} \|_{2,s\!+\!1}^{2} 
				+ h^{2r\!+\!1} \| u_{tt}^{j} \|_{2,r\!+\!1}^{2} \big). \notag 
			\end{align}
			By the CFL-like conditions in \eqref{eq:CFL-like-cond} and \eqref{eq:DLN-Error-H1-eq4} in the proof of Theorem \ref{thm:Error-H1}
			\begin{align}
				&\sum_{n=1}^{N-1} \widehat{k}_{n} 
				\big\| \nabla \big( u_{n,\ast}^{j,h} \!-\! \langle u^{h} \rangle_{n,\ast} \big) \big\|^{2} 
				\frac{\| \nabla \Phi_{n}^{j,h} \|^{2}}{(1 \!-\! \varepsilon_{n})^{2}}  
				\label{eq:DLN-Error-P-L2L2-eq3-avg-term2} \\
				\leq& \sum_{n=1}^{N-1} \widehat{k}_{n}  
				\Big[ C(\Omega,\theta) \frac{ \nu h }{\widehat{k}_{n}} \Big( \frac{1-\varepsilon_{n}}{1+\varepsilon_{n} \theta} \Big)^{2} \Big] 
				\frac{\| \nabla \Phi_{n}^{j,h} \|^{2}}{(1 \!-\! \varepsilon_{n})^{2}}  \notag \\
				\leq& C(\Omega,\theta) \nu h k_{\rm{max}}  
				\sum_{n=1}^{N-1} \frac{ \theta (1 - {\theta}^{2})}{4 ( 1 + \varepsilon_{n} \theta )^{2}} \| \nabla \Phi_{n}^{j,h} \|^{2} \notag \\
				\leq& C(\Omega,\theta) \nu h k_{\rm{max}}  
				\exp \Big[ \frac{C(\Omega,\theta)}{\nu}  \big( k_{\rm{max}}^{4} \| u_{tt}^{j} \|_{2,2}^{2} 
				+ \| |u^{j}|\|_{2,2,\beta}^{2} + \frac{ (C \nu T + 1) F_{2}^{2} }{h \nu} \big)  \Big] F_{3}. \notag 
			\end{align}
			By the CFL-like conditions in \eqref{eq:CFL-like-cond} and \eqref{eq:consist-2nd-eq1},
			\eqref{eq:consist-2nd-eq2} 
			in Lemma \ref{lemma:DLN-consistency},
			\begin{align}
				&\sum_{n=1}^{M-1} \widehat{k}_{n} 
				\big\| \nabla \big( u_{n,\ast}^{j,h} \!-\! \langle u^{h} \rangle_{n,\ast} \big) \big\|^{2}  
				\big\| \nabla \big( u_{n,\ast}^{j} \!-\! u_{n,\beta}^{j} \big)  \big\|^{2} 
				\label{eq:DLN-Error-P-L2L2-eq3-avg-term3} \\
				\leq& \sum_{n=1}^{M-1} \widehat{k}_{n} \Big[ C(\Omega,\theta) \frac{ \nu h }{\widehat{k}_{n}} \Big( \frac{1-\varepsilon_{n}}{1+\varepsilon_{n} \theta} \Big)^{2} \Big] 
				\Big( C(\theta) k_{\rm{max}}^{3} \int_{t_{n\!-\!1}}^{t_{n\!+\!1}} \!\| \nabla u_{tt}^{j} \|^{2} dt  \Big) \notag \\
				\leq& C(\Omega,\theta) \nu h k_{\rm{max}}^{3} \| \nabla u_{tt}^{j} \|_{2,0}^{2}. \notag 
			\end{align}
			We combine \eqref{eq:DLN-Error-P-L2L2-eq3}, \eqref{eq:DLN-Error-P-L2L2-eq3-error-terms}, \eqref{eq:DLN-Error-P-L2L2-eq3-avg-term1}, \eqref{eq:DLN-Error-P-L2L2-eq3-avg-term2}, and \eqref{eq:DLN-Error-P-L2L2-eq3-avg-term3},
			\begin{align}
				&\sum_{n=1}^{N-1} \widehat{k}_{n} \| p_{n,\beta}^{j} - p_{n,\beta}^{j,h} \|^{2} 
				\label{eq:DLN-Error-P-L2L2-eq4}  \\
				\leq& C(\Omega,\theta) h^{2s+2} \big( k_{\rm{max}}^{4} \| p_{tt}^{j} \|_{2,s\!+\!1}^{2} \!+\! \| |p^{j}| \|_{2,s\!+\!1,\beta}^{2} \big)
				+C(\Omega,\theta) \!\Big(\! \frac{h^{2s\!+\!2}}{\nu^{2}}  \| p_{t}^{j} \|_{2,s\!+\!1}^{2} \!+\! h^{2r} \| u_{t}^{j} \|_{2,r}^{2}  \!\Big) 
				\notag \\
				+& \nu \exp \Big[\! \frac{C(\Omega,\theta)}{\nu} \! \big(\! k_{\rm{max}}^{4} \| u_{tt}^{j} \|_{2,2}^{2} 
				\!+\! \| |u^{j}|\|_{2,2,\beta}^{2} \!+\! \frac{ (C\nu T \!+\! 1) F_{2}^{2} }{h \nu} \big) \! \Big]  F_{3}  \notag \\
				+&\! \exp \!\Big[ \frac{C(\Omega,\theta)}{\nu} \big( k_{\rm{max}}^{4} \| u_{tt}^{j} \|_{2,2}^{2} 
				+ \| |u^{j}| \|_{2,2,\beta}^{2} \big) \Big] \! F_{1}
				\!+\! C(\Omega,\theta) \nu h^{2r} \!\big( \!k_{\rm{max}}^{4} \!\| u_{tt}^{j} \|_{2,r\!+\!1}^{2}
				\!+\! \| |u^{j}| \|_{2,r\!+\!1,\beta}^{2} \!\big) \notag \\
				+& \!C(\theta) \nu k_{\rm{max}}^{4} \| \nabla u_{tt}^{j} \|_{2,0}^{2} + C(\theta) k_{\rm{max}}^{4} \big( \| u_{ttt}^{j} \|_{2,-1}^{2}
				+ \| p_{tt}^{j} \|_{2,0}^{2} \big) \notag \\
				+& \frac{C(\Omega,\theta) h^{2s\!+\!2}}{\nu^{2}} \!\big( k_{\rm{max}}^{4} \| p_{tt}^{j} \|_{2,s\!+\!1}^{2} \!+\! \| |p^{j}| \|_{2,s\!+\!1,\beta}^{2} \big)  \notag \\
				+& C(\Omega,\theta) \big( \| |u^{j}| \|_{\infty,1} \!+\! \| |u^{j}| \|_{\infty,1,\beta} \big) k_{\rm{max}}^{4} \| \nabla u_{tt}^{j} \|_{2,0}^{2} \notag \\
				+& C(\theta) T  (\| | u^{j}| \|_{\infty,1}^{2} \!+\! 1 ) \Big\{ C(\Omega) h^{2r} \| |u^{j}| \|_{\infty,r+1}^{2} 
				+ \frac{C(\Omega)h^{2s+2}}{\nu^{2}} \| |p^{j}| \|_{\infty,s+1}^{2}  \notag \\
				& \qquad \qquad \qquad \qquad \quad + \exp \Big[ \frac{C(\Omega,\theta)}{\nu}  \Big( k_{\rm{max}}^{4} \| u_{tt}^{j} \|_{2,2}^{2} 
				+ \| |u^{j}|\|_{2,2,\beta}^{2} + \frac{ (C \nu T + 1)F_{2}^{2} }{h \nu} \Big) \Big] F_{3} \Big\} \notag \\
				+&C(\Omega,\theta) \nu\big( \frac{h^{2s\!+\!3}}{\nu^{2}} \| p_{tt}^{j} \|_{2,s\!+\!1}^{2} 
				+ h^{2r\!+\!1} \| u_{tt}^{j} \|_{2,r\!+\!1}^{2} \big) + C(\theta) \nu h k_{\rm{max}}^{3} \| \nabla u_{tt}^{j} \|_{2,0}^{2} \notag \\
				+&C(\Omega,\theta) \nu h k_{\rm{max}}  
				\exp \Big[ \frac{C(\Omega,\theta)}{\nu}  \big( k_{\rm{max}}^{4} \| u_{tt}^{j} \|_{2,2}^{2} 
				+ \| |u^{j}|\|_{2,2,\beta}^{2} + \frac{ (C\nu T + 1) F_{2}^{2} }{h \nu} \big) \Big] F_{3}, \notag 
			\end{align}
			which implies \eqref{eq:Error-Pressure-Conclusion-Appendix}.

		\end{proof}

	\end{appendices}

\bibliographystyle{abbrv}
\bibliography{ReferenceEnsemble}

\end{document}